\documentclass[a4paper,11pt,twoside]{article}

\usepackage[utf8]{inputenc}

\usepackage{amsmath}  
\usepackage{mathrsfs}  
\usepackage{amsthm}
\usepackage{amsfonts}
\usepackage{amssymb, latexsym}
\usepackage{mhequ}
\usepackage{verbatim}
\usepackage{ifthen}
\usepackage{esint}

\usepackage{mathptmx}
\DeclareMathAlphabet{\mathcal}{OMS}{cmsy}{m}{n} 

\usepackage{graphicx}
\usepackage{color}
\usepackage{xcolor}
\usepackage{tikz}
\usetikzlibrary{calc,intersections,through,backgrounds}
\usetikzlibrary{decorations.pathreplacing}
\usetikzlibrary{shapes}

\usepackage{float} 
\usepackage{enumerate}
\usepackage[margin=1in, marginparwidth=.8in]{geometry}

\usepackage[
 pdfauthor={Pablo Linares, Felix Otto, Markus Tempelmayr, Pavlos Tsatsoulis},
 pdftitle={A diagram-free approach to the stochastic estimates in regularity structures}, pdftex]{hyperref}
\hypersetup{
  colorlinks=false,
  pdfborder={0 0 0}
  }

\usepackage{tocloft}
\usepackage{fancyhdr}

\usepackage[]{appendix}

\newcommand{\E}{\mathbb{E}}
\newcommand{\N}{\mathbb{N}}

\newcommand{\R}{\mathbb{R}}

\newcommand{\ignore}[1]{}

\newcommand{\HH}{\mathbb{H}}

\newcommand{\dd}{\mathrm{d}}
\newcommand{\pp}{\mathrm{p.p.}}

\newcommand{\n}{\mathbf{n}}
\newcommand{\m}{\mathbf{m}}

\newcommand{\length}[1]{| #1 |_{\prec}}
\newcommand{\0}{\mathbf{0}}
\newcommand{\z}{\mathsf{z}}
\usepackage{hhline}

\theoremstyle{plain}
\newtheorem{theorem}{Theorem}[section]
\newtheorem*{theorem*}{Theorem}
\newtheorem{proposition}[theorem]{Proposition}

\newtheorem{lemma}[theorem]{Lemma}
\newtheorem{remark}[theorem]{Remark}

\theoremstyle{definition}

\newtheorem{assumption}[theorem]{Assumption}

\numberwithin{equation}{section}

\title{A diagram-free approach to the stochastic estimates \\ in regularity structures}
\author{Pablo Linares, Felix Otto, Markus Tempelmayr, Pavlos Tsatsoulis}
\date{}

 \fancypagestyle{plain}{

  \fancyhf{}
  \fancyfoot[L]{\footnotesize \today}
  \fancyfoot[C]{\thepage}
}

\renewenvironment{abstract}
{
\begin{center}
\begin{minipage}{.9\textwidth}\small\textbf{Abstract}.\noindent \smallskip
}
{
\end{minipage}
\end{center}
}

\newenvironment{keywords}
{
\begin{center}
\begin{minipage}{.9\textwidth}\small\textbf{Keywords}:\noindent
}
{
\end{minipage}
\end{center}
}

\newenvironment{msc}
{
\begin{center}
\begin{minipage}{.9\textwidth}\small\textbf{MSC 2020}:\noindent
}
{
\end{minipage}
\end{center}
}

\begin{document}

\maketitle

\begin{abstract}

In this paper, we explore the version of Hairer's regularity structures based on
a greedier index set than trees, as introduced in \cite{OSSW21} and algebraically
characterized in \cite{LOT21}. More precisely, we construct and stochastically estimate
the renormalized model postulated in \cite{OSSW21}, avoiding the use of Feynman diagrams
but still in a fully automated, i.~e.~inductive way.
This is carried out for a class of quasi-linear parabolic PDEs driven by noise
in the full singular but renormalizable range.

\smallskip

We assume a spectral gap inequality on the (not necessarily Gaussian) noise ensemble.
The resulting control on the variance of the model naturally complements 
its vanishing expectation arising from the BPHZ-choice of renormalization.
We capture the gain in regularity on the level of the Malliavin derivative of the model
by describing it as a modelled distribution. 
Symmetry is an important guiding principle and built-in on the level of the renormalization Ansatz.
Our approach is analytic and top-down rather than combinatorial and bottom-up.

\end{abstract}



\begin{msc} 
60H17; 60L30; 60H07; 81T16
\end{msc}

\tableofcontents

\section{Introduction}\label{Intro}

We continue the program started in \cite{OSSW21} of replacing trees 
by multi-indices as a more parsimonious but equally natural index set, 
within the framework of Hairer's regularity structures.
Like in \cite{LOT21}, we implement this for quasi-linear 
parabolic\footnote{treating (\ref{eq:sqle}) as an (anisotropic) elliptic equation, 
we denote by $x_1$ the space-like variable and by $x_2$ the time-like variable;
this allows us to use $t$ for the semi-group convolution (\ref{eq:semi_group}) parameter}
equations of the form 
\begin{align}\label{eq:sqle}
(\partial_2-\partial_1^2)u =a(u)\partial_1^2u+ \xi,
\end{align}
driven by a stationary\footnote{i.~e.~invariant in law under space-time shifts}
noise $\xi$ in the regime where the product $a(u)\partial_1^2u$ is singular
but renormalizable. 
This is the case when\footnote{
as a consequence of the single space dimension, there is an additional constraint} a general solution $u$ to (\ref{eq:sqle}) with $a\equiv0$
is H\"older continuous with an exponent $\alpha\in(0,1)$, 
which means that $\xi$ is in the (negative) H\"older space $C^{\alpha-2}$.
However, we believe that our program applies to a much larger class\footnote{In
particular, the restriction to a single spatial variable is just for convenience,
and has the advantage of making white noise renormalizable with $\alpha=\frac{1}{2}-$. 
However, as discussed after Assumption \ref{ass:spectral_gap} in Subsection \ref{S:SG}, there is 
a didactic advantage in allowing for general space dimension, 
replacing $\partial_1^2$ by the Laplacian, as done in \cite{OSSW21}.}  
of nonlinear PDEs.

\medskip

The investigation of (\ref{eq:sqle}) started in \cite{OW19}
sparked some activity, at first in the mildly singular range \cite{FG19,BDH19}
of $\alpha\in(\frac{2}{3},1)$, and then up to the white-noise level ($\alpha>\frac{1}{2}$) \cite{GH19,OSSW18},
and finally including the white-noise level:
\cite{GH19} actually deals with the regime $\alpha>0$, 
however the obtained counterterm is only shown to be local up to $\alpha>\frac{1}{2}$, 
which was extended in \cite{Ge20} to $\alpha>\frac{1}{3}$;
 \cite{BGN24} put the calculations performed in \cite{Ge20} on an algebraic footing 
as a step towards an extension to $\alpha>0$;
\cite{BM19} deals with the analytic part of the solution theory up to $\alpha>\frac{2}{5}$, 
and so does \cite{BHK23} down to $\alpha>0$, 
however again with the caveat that the obtained counterterm 
can only be shown to be local up to $\alpha>\frac{2}{5}$.
In \cite{OSSW21}, local Schauder estimates were established for $\alpha\in(0,1)$, based on the notion of
modelled distributions, postulating the existence and estimation of a suitable renormalized model.
In \cite{LOT21}, the Hopf-theoretic nature of the structure group based on multi-indices
was uncovered, which is rather Lie-geometric than combinatorial in the sense
that it provides a representation of natural actions on the solution manifold.
In this paper, we construct the BPHZ-renormalized model 
and provide its stochastic estimates, as the input to \cite{OSSW21}.
Our multi-index approach is analytic, meaning that it is based
on taking derivatives w.~r.~t.~the nonlinearity $a$ and the noise $\xi$, whereas the tree-based
approach is combinatorial, using Feynman diagrams. Our approach is fully automated 
since it proceeds by induction over the index set, with all induction
steps having the same structure\footnote{up to the distinction between singular
and regular multi-indices}.
Loosely speaking, \cite{OSSW21} can be seen as an analogue\footnote{however partial and restricted to (\ref{eq:sqle})} of\footnote{including aspects of \cite{BCCH21} in the sense that the notion
of the structure group is less generic than in \cite{Ha14}} 
\cite{Ha14} in the sense that it deals with the 
analytic solution theory, this paper corresponds to \cite{CH16}
by establishing the stochastic estimates on the model,
while \cite{LOT21} works out the Hopf-algebraic structure in the spirit of \cite{BHZ19,BCCH21}.
Let us stress that this paper is essentially self-contained w.~r.~t.~\cite{OSSW21}
and \cite{LOT21}, and can be read and appreciated independently:
It serves as an input for \cite{OSSW21}, while \cite{LOT21}
provides a deeper algebraic understanding not required in this paper\footnote{besides a few finiteness properties and identities explicitly spelled out}.

\medskip

In terms of renormalization, the multi-index approach is top-down
rather than bottom-up, in the sense that for the renormalized equation
\begin{align}\label{ao04}
(\partial_2-\partial_1^2)u+h=a(u)\partial_1^2u+\xi
\end{align}
we postulate on the counter term $h$:
\begin{itemize}
\item $h$ is local, i.~e.~it depends on the solution $u$ only through its value\footnote{this 
particularly simple form relies on the assumption that $\xi$
is invariant in law under spatial reflection} $u(x)$
at the current space-time point $x$ and
\item $h$ is homogeneous, i.~e.~not explicitly dependent on $x$, and
thus\footnote{in view of the stationarity of the noise} deterministic,
i.~e.~not explicitly dependent on $\xi$. Both conditions imply that
$h=h(u(x))$ for some deterministic nonlinearity $h$.
\item The most subtle postulate relates $a$ to $h$: If we replace $a$ by $a(\cdot+u)$ 
for some $u$-shift $u\in\mathbb{R}$, $h$ is replaced by $h(\cdot+u)$. This means that the
renormalization is independent of the choice of the origin in $u$-space.
It implies that 
\begin{align}\label{ao05}
h(u)=c[a(\cdot+u)]\quad\mbox{for}\;u\in\mathbb{R}
\end{align}
for some deterministic functional\footnote{here and in the sequel, $[\cdot]$
denotes the functional dependence on a field} $c$ on $a$-space,
which typically diverges as the regularization
of $\xi$ fades away.
 
\end{itemize}
As opposed to the tree-based treatment of quasi-linear equations by regularity
structures \cite{GH19,Ge20}, we thus do not have to show a posteriori 
that $h$ is local. Some symmetries are built-in, like the independence from the choice 
of the origin in $u$-space and $a$-space. Other symmetries, like the invariance in law
under $\xi\mapsto-\xi$, are easily seen to transmit
to the model\footnote{however, our approach is oblivious
to symmetries arising from a Gaussian nature of the noise}.
As a consequence, our more greedy approach based on multi-indices
rather than trees reduces the number of divergent constants contained in $c$.
The comparison for the quasi-linear equation (\ref{eq:sqle}) is complicated
by the non-local nature of the tree-based treatment in \cite{GH19,Ge20},
and the fact that the divergent constants are better treated as functions of
a place-holder $a_0$ for the elliptic coefficient $1+a$, 
see Subsection \ref{sec:OSSW}. The comparison is easier
for the semi-linear multiplicative stochastic heat 
equation $(\partial_2-\partial_1^2)u=a(u)\xi$, as treated by tree-based regularity structures
in \cite{HP15} for $\alpha=\frac{1}{2}-$. Here it is clear that for, e.~g., $\alpha=\frac{1}{4}+$, 
the number of divergent constants decreases from 85 to 30, see also \cite[Section 7]{LOT21} and \cite[Section 5]{BL23}.

\medskip

The two main conceptual merits of our approach are:
\begin{itemize}
\item Spectral gap (SG) inequality. Our main assumption on the law of $\xi$, 
next to invariance under translation\footnote{i.~e.~stationarity} 
and reflection\footnote{in the spatial variable(s)} is a spectral gap inequality. 
The SG inequality is specified by a Hilbert norm on $\xi$-space, 
which provides the analogue of the Cameron-Martin space from the Gaussian case,
and here is $L^2(\mathbb{R}^2)$-based. 
This Hilbert norm is chosen in agreement with $\xi\in C^{\alpha-2}$.
While this includes non-Gaussian ensembles, 
the main benefit is that the SG inequality 
naturally complements the BPHZ-choice of renormalization:
On the one hand, a SG inequality, which we apply to the negative-homogeneity part
$\Pi^{-}$ of the model, 
estimates the variance of a random variable. On the other hand, the BPHZ-choice
of renormalization is just made to annihilate the expectation $\mathbb{E}\Pi^{-}$. 
We refer to Subsections \ref{BPHZMall} and \ref{sec:Pi_minus_constr} for the details.
\item Malliavin derivative as modelled distribution. 
The use of a SG inequality requires the control of the
(first-order) Malliavin derivative of $\Pi^{-}$, 
which is the Fr\'echet derivative of $\Pi^{-}=\Pi^{-}[\xi]$ w.~r.~t.~the noise $\xi$. 
It is convenient to think of it in terms of the directional derivative
$\delta\Pi^{-}$ for some arbitrary infinitesimal noise 
perturbation\footnote{an element of what would be the Cameron-Martin space in the Gaussian case}
$\delta\xi$.
Since $\Pi^{-}$ is multi-linear in $\xi$, passing to $\delta\Pi^{-}$ amounts
to replacing one of the instances of $\xi$ by $\delta\xi$.
This leads only to a subtle gain in regularity, 
which is conveniently expressed after integration\footnote{in Hairer's jargon}, 
i.~e.~on the level of $\delta\Pi$.
It is captured by describing $\delta\Pi$ as a modelled distribution\footnote{we use
the notation ${\rm d}\Gamma^*$ to indicate that it has some structural similarities with
the Malliavin derivative $\delta\Gamma^*$ of the change-of-base-point transformation
$\Gamma^*$; the star is a reminder
of the fact that $\Gamma^*$ is the transpose of Hairer's, see the discussion at
the beginning of Subsection \ref{sec:postGamma}}
${\rm d}\Gamma^*$
w.~r.~t.~$\Pi$ itself. 
Hence surprisingly, the notion of a modelled distribution with its intrinsic continuity
property, which was introduced in \cite[Definition 3.1]{Ha14} 
for the deterministic Schauder theory given the model, 
here plays a role in the stochastic estimation of the model itself.
Crucially, as opposed to $\Pi^{-}$ itself, the representation for $\delta\Pi^{-}$, 
or rather of its rough-path increment $\delta\Pi^{-}-{\rm d}\Gamma^*\Pi^{-}$, 
in terms of $\Pi$ and $\delta\Pi-{\rm d}\Gamma^*\Pi$ does not involve the divergent $c$.
This ultimately allows for reconstruction of $\delta\Pi^{-}$.
We refer to Subsection \ref{thirdblock} for details.
\end{itemize}

\medskip

Two more technical merits of our approach are:
\begin{itemize}
\item Scaling as a guiding principle. In order not to break it, we
work on the whole space-time, which because of potential infra-red divergences is interesting 
in its own right. As a collateral damage, in order to avoid critical cases, 
we have to generalize from white noise to a more general noise $\xi$
with a fractional (negative) Sobolev norm playing the role of the Cameron-Martin space. Like
in \cite{CMW19}, this has the positive side effect of allowing to 
explore the limits of the approach. In order not to break scaling, 
we work with annealed instead of quenched estimates. By this jargon\footnote{from statistical
physics and ultimately from metallurgy} we mean that 
the inner norm is an $L^p$-norm in probability, 
while the outer norm is a H\"older norm in space-time; Subsection \ref{BPHZMall} is the
place where this transition is made on the level of $\delta\xi$ .
Estimates in annealed norms have the advantage of coming without a (marginal) 
loss in the exponent\footnote{as is well-known from Brownian motion 
and its law of iterated logarithm}. 
\item H\"older vs.~$L^2$-topologies.
The SG inequality and Malliavin calculus rely on $L^2$-based space-time norms 
(the analogue of the Cameron-Martin space is an $L^2$-norm of a negative fractional derivative 
of $\delta\xi$) 
whereas the Schauder calculus of modelled distributions like our ${\rm d}\Gamma^*$
is based on H\"older norms.
We introduce a weight that is singular (but integrable)
in a secondary base point $z$ into the
$L^2$-based norms to emulate a H\"older norm localized in $z$.
Averaging over $z$ recovers the original Cameron-Martin norm.
We refer to Subsection \ref{thirdblock} for details.
\end{itemize}

\medskip

We now comment on related work. 
Hairer's regularity structures triggered a rapid development in the field of singular SPDEs. 
They provide a framework for local well-posedness for a large class of semi-linear SPDEs, 
as worked out in \cite{Ha14, BHZ19, CH16, BCCH21}. 
As mentioned, this approach has been extended to quasi-linear SPDEs\footnote{
at the expense of loosing the local nature of the problem inside the proof} 
in \cite{GH19, Ge20}. 
Gubinelli, Imkeller and Perkowski's paracontrolled calculus \cite{GPI15} 
provides an alternative approach, based on Littlewood-Paley decomposition,
to (stochastic) estimates and renormalization. 
While it does not provide a general framework by itself, 
see however \cite{BH21},
paracontrolled calculus has been efficiently applied to a variety of singular SPDEs \cite{CC18, GP17}.
Furthermore, it naturally extends to quasi-linear equations, see \cite{FG19, BDH19},
and to dispersive equations, see \cite{GKO18}.
Kupiainen appealed to Wilsonian renormalization \cite{Ku16,KM17} to treat some semi-linear SPDEs.
Duch \cite{Du21} used Wilsonian renormalization in the continuum form of the Polchinski 
flow equation, see below. 
Our approach has also similarities with the one of Epstein and Glaser \cite[Subsection~3.1]{Sch95}
in the sense that it is inductive, and that it uses a coarser index set than trees,
meaning that it only monitors specific linear combinations of trees. 
With its modern version \cite{HW05}, it shares the
guiding principle of scaling and symmetries (covariance).

\medskip

Although the stochastic estimates on the centered model obtained in \cite{CH16} 
are identical to the ones in the present paper,
the assumption of cumulant bounds on the noise,
and the methodology to obtain stochastic estimates via multi-scale analysis 
differ radically from our approach.
As opposed to our method, 
the analysis in \cite{CH16} builds up on the well-established 
physics approach to Feynman diagrams 
into which it incorporates positive renormalization.
Problems like overlapping sub-divergences, 
or counter terms necessary to cure divergences at one scale 
that could potentially interfere at another scale,
do not come up in our recursive approach.

\medskip

The approach of \cite{Ku16, Du21}
is not in the framework of regularity structures,
with its conceptual separation between the task of constructing and stochastically estimating
the centered model on the one hand, and a deterministic fixed point argument to
solve a given initial/boundary value problem on the other.
It directly constructs solutions for given (periodic) boundary conditions, 
while the centered model constructed and estimated in \cite{BHZ19, CH16}
and in this paper describes the entire solution manifold. 
Nevertheless \cite{Du21} has similarities with the present work
in the sense that it recursively constructs and estimates multi-linear functionals of the noise
in such a way that overlapping sub-divergences do not play a role.
However, the assumption of cumulant bounds on the noise is more closely related to \cite{CH16}
than to the present work. 
The recursive structure in \cite{Du21} arises from a book-keeping parameter 
in front of the non-linearity, which leads to a formal power series expansion in this parameter,
into which polynomials are incorporated like in the present work.
This leads to an even more parsimonious index set than in the present work.

\medskip

In the use of Malliavin calculus, the spectral gap inequality, and annealed estimates, 
this paper is inspired by recent developments in quantitative stochastic homogenization 
\cite{GO11, DO20, JO22}. Malliavin calculus has been used, within the framework of regularity structures, 
in \cite{CFG17, GL20, Sch18}, however in its original purpose, namely for the existence of probability densities.
In the context of stochastic estimates for singular SPDEs, Malliavin calculus has been used in \cite{FG19ii} to 
estimate non-polynomial functionals of the noise through the Wiener chaos decomposition, 
and in combination with the spectral gap inequality in \cite{IORT20} to estimate the first non-linear term in case of a non-local operator.

\medskip

Since posting this work, the spectral gap inequality has been used
in several works to establish stochastic estimates. 
In \cite[Appendix C]{KT22} the authors obtained stochastic estimates for the simple case of $\Phi^4_2$.
First algebraic steps towards extending this work to the tree-based setting were 
made in \cite{BN23}, 
whereas \cite{HS24} obtained stochastic estimates in the tree-based setting for a large class of semi-linear equations. 
Also \cite{BH23} revisited the arguments of the present work and applied it to the tree-based setting, 
and \cite{BB23} obtained estimates in the tree-based setting for the generalized KPZ 
equation in one spatial dimension by appealing to a higher order version of the spectral gap inequality, 
however still relying on diagrammatic tools to estimate higher order Malliavin derivatives. 
The present work has been extended in \cite{GT23} and \cite{BOT24} (within the multi-index setting) 
to a fourth-order quasi-linear equation with multiplicative noise, and a semi-linear equation with polynomial non-linearity and additive noise, respectively.
Furthermore, \cite{Tem23} used Malliavin calculus and a spectral gap assumption to give a characterization of models.
The spectral gap inequality has also been applied in a more classical setting of rough paths \cite{GK23}.

\medskip

On the algebraic side, the approach based on multi-indices has been generalized to a large class of semi-linear SPDEs in \cite{BL23}, which also attempts to systematize the algebraic structure of the top-down approach to renormalization developed here (see also \cite{Lin23} for an algebraic construction based on multi-indices consistent with \cite{BL23} and connected to the algebraic renormalization of rough paths \cite{BCFP19}). Moreover, algebraic structures based on multi-indices have triggered the study of post-Lie and Novikov algebras in the context of regularity structures \cite{BD23, BK23, JZ23} and, more recently, they have been introduced in numerical analysis \cite{BEFH24}.

\section{Assumptions and statement of result}

\subsection{Spectral gap inequality}\label{S:SG}

In this subsection, we motivate and state our assumptions on the law of 
random Schwartz distributions $\xi$ on space-time $\mathbb{R}^2$.
The crucial assumption is that of a spectral gap (SG) inequality. 
The structure underlying a SG inequality is a Hilbert
norm on the space of space-time fields, which plays the role of the Cameron-Martin norm
from the Gaussian case. In the same way white noise has $L^2(\mathbb{R}^2)$
as Cameron-Martin space, our norm will be an $L^2(\mathbb{R}^2)$-based norm. 
Because our law is shift-invariant, it is natural to choose a translation-invariant norm. 
Since we are aiming at $\xi$'s that are almost surely in the negative H\"older space
$C^{\alpha-2-}$, by Kolmogorov's criterion, it should be an $L^2$-based Sobolev norm of
the fractional order $\frac{D}{2}+\alpha-2$, where $D$ is the effective dimension 
(see \cite[Lemma 10.2]{Ha14}).
In view of the parabolic nature, both H\"older and Sobolev
norms need to be anisotropic: If the spatial variable $x_1$ sets the unit,
the time variable $x_2$ is worth two units. In particular, we have for the effective dimension $D=1+2=3$ so that the order
of fractional derivative should be $\alpha-\frac{1}{2}$.
For the H\"older norm anisotropy 
means that it is based on the parabolic Carnot-Carath\'eodory distance
\begin{align}\label{ao18}
|y-x|:=|y_1-x_1|+\sqrt{|y_2-x_2|}.
\end{align}
It is convenient to express the anisotropic
version of the Sobolev norm in terms of the space-time elliptic operator $\partial_1^4-\partial_2^2$,
which is of order four:
\begin{align}\label{ao01}
\big(\int_{\mathbb{R}^2}dx\big((\partial_1^4-\partial_2^2)
^{\frac{1}{4}(\alpha-\frac{1}{2})}\delta\xi\big)^2\big)^\frac{1}{2},
\end{align}
where here and in the sequel, we think of $\delta\xi$ as an infinitesimal perturbation of $\xi$.

\medskip

Having motivated the Hilbert norm (\ref{ao01}), we return to the notion of a
SG inequality. A SG inequality amounts to a
Poincar\'e inequality (with mean value zero)
on the space of space-time fields endowed with a probability measure and a Hilbert norm\footnote{
In analogy to the standard Poincar\'e inequality providing a lower bound on the spectral gap of the 
(Neumann) Laplacian, the SG inequality provides a lower bound on
the spectral gap of the generator of the stochastic process defined through the Dirichlet form
arising from the Hilbertian structure, for which the given ensemble is in detailed balance
by construction.}.
It is formulated in terms of a generic random variable $F$, 
which is an integrable function(al) on the space of $\xi$'s, i.~e.~$F=F[\xi]$. 
The notion of a gradient of $F$ and its (squared) norm relies on the
Hilbertian structure (\ref{ao01}). 
We momentarily consider $F=F[\xi]$ that are Fr\'echet 
differentiable w.~r.~t.~(\ref{ao01});
meaning that the differential $dF[\xi]$ in a configuration $\xi$, which is
a linear form on the space of infinitesimal perturbation
$\delta\xi$'s, is bounded w.~r.~t.~(\ref{ao01}). Representing
the differential $dF[\xi]$ in terms of 
\begin{align}\label{ao22}
\delta F:=dF[\xi].\delta\xi=\int_{\mathbb{R}^2}dx\,\,\delta\xi\,\,\frac{\partial F}{\partial\xi}[\xi],
\end{align}
this means that the $L^2(\mathbb{R}^2)$-dual norm of
the Malliavin derivative $\frac{\partial F}{\partial\xi}=\frac{\partial F}{\partial\xi}[\xi](x)$
is finite:\footnote{To ease the notation we suppress the dependence of $\|\cdot\|_*$ on $\alpha$.} 
\begin{align}\label{as02}
\|\frac{\partial F}{\partial\xi}[\xi]\|_*:= \left(\int_{\mathbb{R}^2}dx\big((\partial_1^4-\partial_2^2)
^{\frac{1}{4}(\frac{1}{2}-\alpha)}\frac{\partial F}{\partial\xi}[\xi]\big)^2\right)^{\frac{1}{2}}<\infty.
\end{align}

A functional-analytic subtlety arises from the fact that Fr\'echet differentiability of $F$
is not enough to give an a priori meaning to (\ref{as02}) for almost every
realization\footnote{consider $F[\xi]=\frac{1}{2}\int dx\xi^2$;
on the one hand, $F$ is Fr\'echet differentiable for $\alpha=\frac{1}{2}$ 
with $\frac{\partial F}{\partial\xi}[\xi]=\xi$;
on the other hand, white noise almost surely has infinite $L^2(\mathbb{R}^2)$-norm.} $\xi$.
Therefore one restricts to cylindrical functionals 
\begin{align}\label{eq:cyl_fnct}
F[\xi]=\bar F((\xi,\zeta_1),\dots,(\xi,\zeta_N))
\end{align}
for some smooth function $\bar F$ on $\R^N$ and Schwartz functions $\zeta_1,\dots,\zeta_N$, where 
$(\cdot,\cdot)$ here stands for the pairing between a Schwartz distribution and a Schwartz function.
These cylindrical functionals are obviously Fr\'echet differentiable (even on the
space of Schwartz distributions) with
\begin{align}\label{cw56}
\frac{\partial F}{\partial\xi}[\xi]
=\sum_{n=1}^N\partial_n\bar F((\xi,\zeta_1),\dots,(\xi,\zeta_N))\zeta_n.
\end{align}

\begin{assumption}\label{ass:spectral_gap} 
The law $\mathbb{E}$ of the Schwartz distribution $\xi$ is invariant under space-time shift
and spatial reflection\footnote{meaning $x_1\mapsto -x_1$}.
It is centered\footnote{meaning $\mathbb{E}\xi(x)=0$, 
an assumption just made for convenience and w.~l.~o.~g.} and for an 
$\alpha\in(\frac{1}{4},\frac{1}{2})-\mathbb{Q}$ it satisfies the spectral
gap inequality
\begin{equs}\label{ao03}
\E(F- \E F)^2 
\leq \E\|\frac{\partial F}{\partial\xi}\|_*^2,
\end{equs}
for all integrable cylindrical functionals $F$. In addition, we assume that the
operator (\ref{cw56}) is closable\footnote{in the Gaussian case, this is automatic
\cite[Proposition~1.2.1]{Nu06}} with respect to the topologies\footnote{meaning that
the closure of the graph of $F\mapsto \frac{\partial F}{\partial\xi}$ 
on the space of cylindrical functionals
w.~r.~t.~the product topology remains a graph}
of $\mathbb{E}^\frac{1}{2}|\cdot|^2$ and $\mathbb{E}^\frac{1}{2}\|\cdot\|_*^2$; here and in the sequel we use for $p\geq 1$ the shorthand notation 
$\E^\frac{1}{p}|\cdot|^p$ to denote the stochastic $\mathbb{L}^p$-norm.
\end{assumption}

We learn from a (parabolic) rescaling of space-time that there is no loss in generality
in assuming that the constant in (\ref{ao03}) is unity.
We note that any Gaussian ensemble with a Cameron-Martin norm that dominates (\ref{ao01})
satisfies (\ref{ao03}), see \cite[Theorem 5.5.1, eq.~(5.5.2)]{Bo98}.
In particular, this applies to any stationary Gaussian ensemble 
with a covariance function of which the Fourier transform satisfies
${\mathcal F}c(k)\le(k_1^4+k_2^2)^{\frac{1}{2}(\frac{1}{2}-\alpha)}$.
Let us comment on the constraints on $\alpha$: Recall that white noise
corresponds to $\alpha=\frac{1}{2}$; however, because of the Schauder theory
involved in integration, we need to avoid rational $\alpha$. For convenience, we restrict
ourselves to the more singular side by assuming $\alpha<\frac{1}{2}$.
This does include white noise, provided it is tamed by an infra-red cut-off, 
which can be achieved by cutting off the large-scale Fourier modes to satisfy
the above-mentioned estimate (while preserving stationarity).
In the case of rational $\alpha$, 
like $\alpha=\frac{1}{3}$, and without an infra-red cut-off, logarithms in the
estimates are unavoidable; these are not captured in this work. 
Reconstruction imposes $\alpha>\frac{1}{4}$,
a constraint already present on the level of the first counter term,
see \cite[Proposition 4.2]{OW19}, and specific to the single space dimension;
a similar restriction arises already in rough path theory, where fractional Brownian 
motion can be (canonically) lifted to a rough path only for Hurst parameter $H>\frac{1}{4}$ \cite{CQ02}.
We will discuss in Subsection~\ref{thirdblock} that
this is the only reason for restricting the $\alpha$-range away from $\alpha=0$
and thus the limit of renormalizability.

\medskip

We note that assumption (\ref{ao03})
implies that $\xi$ has an annealed (parabolic)
H\"older regularity of exponent $\alpha-2$, as expressed by (\ref{eq:pi_minus_generic}) 
for $\beta=0$, and thus almost every realization $\xi$ has a quenched local H\"older regularity 
for any exponent $<\alpha-2$, but not better. In order to give a classical sense
to the model, $\xi$ will enter its definition only after mollification\footnote{
the resulting pathwise smoothness is not used}, 
see (\ref{eq:Pi_minus_def}).
The key insight is that the estimates of Theorem~\ref{thm:main}
do not depend on the scale of this ultra-violet cut-off. 
The companion work \cite{Tem23} builds upon the tools developed here 
to also pass to the limit of vanishing mollification,
and to give an independent characterization of the latter. 
This uniqueness, which relies on the assumption $\alpha\not\in\mathbb{Q}$, 
amounts to universality in the spirit of \cite[Proposition~1.9]{IORT20}: 
The limit is independent of the specific regularization.
The assumption of reflection invariance in law is crucial for the (simple) form
of the renormalized equation; if omitted, we expect an additional counter term
of the form $\tilde h(u)\partial_1u$. 


\subsection{Definition of the model \texorpdfstring{$(\Pi_{x},\Gamma_{yx})$}{Pi x, Gamma yx}}
\label{sec:model}

In this subsection, we motivate and define the objects of
our main result,
Theorem \ref{thm:main}.
We first develop an algebraic perspective on the counter-term,
then introduce coordinates on the solution manifold,
then define the centered model $\Pi_x$ proper, as a linear
map on the abstract model space $\mathsf{T}$, 
introduce the grading of $\mathsf{T}$ by the homogeneity $|\cdot|$,
and finally the re-centering transformations $\Gamma_{yx}\in{\rm End}(\mathsf{T})$.
We follow the concepts, language, and notation of regularity structures.
A more in-depth treatment is provided in \cite{LOT21},
more motivations are provided in \cite{LO22, OST23}, but this text is self-contained.

\subsubsection{The counter term \texorpdfstring{$c$}{c} as an element of 
\texorpdfstring{$\mathbb{R}[[\mathsf{z}_k]]$}{R[[z k]]}}\label{sss}
On the space of analytic functions $a$ in the variable $u$,
coordinates are given by
\begin{align}\label{ao07a}
\mathsf{z}_{k}[a]=\frac{1}{k!}\frac{d^ka}{du^k}(0)\quad\mbox{for}\;k\in\mathbb{N}_0.
\end{align}
If $a$ is a polynomial, as denoted by $a\in\mathbb{R}[u]$, we obtain by Taylor
\begin{align}\label{Taylor_for_a}
a(u)=\sum_{k\ge 0}u^k\mathsf{z}_k[a].
\end{align}
For a multi-index $\beta$, which associates a frequency $\beta(k)\in\mathbb{N}_0$
to every $k\ge 0$ such that all but finitely many $\beta(k)$'s vanish,
the monomial $\mathsf{z}^\beta$ $=\prod_{k\ge 0}\mathsf{z}_k^{\beta(k)}$ 
defines a (nonlinear) functional on the space of analytic $a$'s via (\ref{ao07a}).
In fact, (\ref{ao07a}) 
naturally extends to the space $\mathbb{R}[[u]]$ of formal power series $a$ in $u$.
Hence we may identify the algebra $\mathbb{R}[\mathsf{z}_k]$ 
of polynomials $\sum_\beta \pi_\beta \z^\beta$ in the variables $\{\mathsf{z}_k\}_{k\ge 0}$ 
with a sub-algebra of the algebra of function(al)s on $\mathbb{R}[[u]]$. 

\medskip

In view of (\ref{ao05}), we are interested in the action of $u$-shift
$a\mapsto a(\cdot+v)$ for $v\in\mathbb{R}$, which is also well-defined as an endomorphism of $\mathbb{R}[[u]]$. 
By pull back, this endomorphism of $\mathbb{R}[[u]]$
lifts to an endomorphism of the algebra of functionals $\pi$ on $\mathbb{R}[[u]]$.
We consider its infinitesimal generator $D^{({\bf 0})}$ defined 
on the sub-algebra $\mathbb{R}[\mathsf{z}_k]$ through $(D^{({\bf 0})}\pi)[a]$ 
$:=\frac{d}{dv}_{|v=0}\pi[a(\cdot+v)]$
and claim that
\begin{align}\label{eq:def_Dnull}
D^{({\bf 0})}=\sum_{k\ge 0}(k+1)\mathsf{z}_{k+1}\partial_{\mathsf{z}_k},
\end{align}
noting that the sum is effectively finite when applied to $\pi\in\mathbb{R}[\mathsf{z}_k]$.
The elementary argument is given in Section \ref{sec:algebraic_aspects}.

\medskip

Returning to (\ref{ao05}), we note that the pull back of 
$a\mapsto a(\cdot+v)$ can be expressed in terms of
its infinitesimal generator $D^{({\bf 0})}$ via the exponential formula
\begin{align}\label{expfirst}
c[a(\cdot+v)]=\big(\sum_{l\ge 0}\frac{1}{l!}v^l(D^{({\bf 0})})^lc\big)[a].
\end{align}
In view of (\ref{eq:def_Dnull}), 
$\sum_{l\ge 0}\frac{1}{l!}v^l(D^{({\bf 0})})^l$ maps $\mathsf{z}_k$ onto the 
{\it infinite}\footnote{which is still effectively
finite when evaluated at an $a\in\mathbb{R}[u]$}
linear combination $\sum_{l\ge 0}$ $\binom{l+k}{l}$ $v^l$ $\mathsf{z}_{l+k}$. This
motivates to pass from the algebra of polynomials $\mathbb{R}[\mathsf{z}_k]$ 
to the algebra of formal power series $\mathbb{R}[[\mathsf{z}_k]]$.
The matrix representation of (\ref{eq:def_Dnull}) w.~r.~t.~the monomial basis $\{\z^\beta\}_\beta$ is given by\footnote{$e_k$ denotes the multi-index with $e_k(l)=\delta_k^l$} 
\begin{align}\label{def_Dnull_rig}
(D^{({\bf 0})})^\gamma_\beta
=(D^{(\0)}\z^\gamma)_\beta
=\sum_{k\ge 0}\left\{
\begin{array}{cc}(k+1)\gamma(k)&\mbox{provided}\;\gamma+e_{k+1}=\beta+e_k\\
0&\mbox{else}\end{array}\right\},
\end{align}
and has the property that for every multi-index $\beta$, it vanishes
for all but finitely many multi-indices $\gamma$. Hence (\ref{eq:def_Dnull})
extends to an endomorphism\footnote{in fact, a derivation} on $\mathbb{R}[[\mathsf{z}_k]]$. Moreover, for the scaled length 
$[\beta]=\sum_{k\ge 0}k\beta(k)$ we learn from (\ref{def_Dnull_rig}) by induction in $l$ that
\begin{align}\label{Dopop}
((D^{({\bf 0})})^l)_\beta^\gamma=0\quad\mbox{unless}\quad[\beta]=[\gamma]+l,
\end{align}
so that the sum
\begin{align}\label{expsecond}
\sum_{l\ge 0}\frac{1}{l!}v^l(D^{({\bf 0})})^lc
\end{align}
is effectively finite, meaning that it is finite on the level of a component $\beta$. 
A collateral damage of this extension from 
$\mathbb{R}[\mathsf{z}_k]$ to $\mathbb{R}[[\mathsf{z}_k]]$ is that 
$c \in\mathbb{R}[[\mathsf{z}_k]]$ can no longer be identified with a functional on 
$\mathbb{R}[u]$, so that (\ref{expfirst}) becomes formal.

\subsubsection{Coordinates \texorpdfstring{$\mathsf{z}_k$}{z k} and \texorpdfstring{$\mathsf{z}_{\bf n}$}{z n} for the solution manifold}
The next (heuristic) task is to endow the solution manifold for (\ref{ao04}) with coordinates.
In case of $a\equiv0$, which by (\ref{ao05}) entails $h\equiv const$,
the manifold of solutions $u$ of (\ref{ao04}) obviously is an affine space over the linear space of
space-time functions $p$ with $(\partial_2-\partial_1^2)p=0$; 
those functions $p$ are analytic\footnote{in fact, entire}. 
It is convenient to free oneself from the constraint $(\partial_2-\partial_1^2)p=0$
by extending the manifold to all space-time functions $u$ that satisfy (\ref{ao04})
up to a space-time analytic function\footnote{it turns out that this extension is not explored, see (\ref{eq:model})}. 
In view of the Cauchy-Kovalevskaya theorem, one expects that for analytic $a$,
the space of analytic space-time functions $p$ still provides a 
parameterization of the (extended) nonlinear solution manifold
-- at least for sufficiently small $a$ and locally near a base-point $x\in\mathbb{R}^2$.

\medskip

According to (\ref{ao05}), if $u$ solves (\ref{ao04}), then for any constant $v$, $u-v$ solves
(\ref{ao04}) with $a$ replaced by $a(\cdot+v)$.
Hence the manifold of all space-time functions $u$ modulo constants
that satisfy (\ref{ao04}) up to a space-time analytic function -- 
for some analytic nonlinearity $a$ -- is parameterized by the tuple $(a,p)$ 
with $p$ modulo constants. We think of $p$ as providing a germ at the
base-point $x=0$ so that
\begin{align}\label{ao07p}
\mathsf{z}_{\bf n}[p]=\frac{1}{{\bf n}!}\frac{\partial^{\bf n}p}{\partial y^{\bf n}}(0)
\quad\mbox{for}\;{\bf n}\in\mathbb{N}_0^2-\{(0,0)\},
\end{align}
which are coordinates for $\{p\in\mathbb{R}[[x_1,x_2]]\,|\,p(0)=0\}$,
are natural for the parameterization near $x=0$.
Hence the union of (\ref{ao07a}) and (\ref{ao07p}) is expected
to provide coordinates for the above solution manifold.


\subsubsection{Definition of \texorpdfstring{$\Pi_x$}{Pi x}}
The coordinates allow us to identify the general solution $u$
with an element $\Pi_x$ $\in C^2[[\mathsf{z}_k,\mathsf{z}_{\bf n}]]$, where
$C^2[[\mathsf{z}_k,\mathsf{z}_{\bf n}]]$ denotes the space of formal power series
in $\mathsf{z}_k$, $\mathsf{z}_{\bf n}$ with coefficients given by space-time
functions that are twice continuously differentiable in $y_1$ and  
continuously differentiable in $y_2$.
On a formal level, the relationship between $u$ and $\Pi_x$ is the following:
On the one hand, $u$ can be recovered from $\Pi_x$ via the series\footnote{which has
no reason to converge; and with $p$ depending on $x$}
\begin{align}\label{formal}
u(\cdot)-u(x)=\sum_{\beta}\Pi_{x\beta}(\cdot)\mathsf{z}^\beta[a,p],
\end{align}
where the sum runs over all multi-indices $\beta$
that associate a frequency to both $k\ge 0$ and ${\bf n}\not={\bf 0}$,
and the monomial $\z^\beta=\prod_{k\geq0, \n\neq\0}\z_k^{\beta(k)}\z_\n^{\beta(\n)}$ is evaluated at $(a,p)$ according to \eqref{ao07a} and \eqref{ao07p}.
On the other hand, $\Pi_{x\beta}$ can be recovered from $u$ by taking 
the partial derivative  w.~r.~t.~the variables $\mathsf{z}_k$, $\mathsf{z}_{\bf n}$ corresponding to
the multi-index $\beta$, and then evaluating at $\mathsf{z}_k=\mathsf{z}_{\bf n}=0$.
In particular, we have
\begin{align}\label{new}
\Pi_{x\beta}(x)=0.
\end{align}
The expansion (\ref{formal}) can be seen as a PDE version of a Butcher series\footnote{
with a more parsimonious index set, which arises from combining terms that depend on the
nonlinearity $a$ in the same way}
as extended to rough paths in \cite[Section 5]{Gu10}. 

\medskip

We note that by (\ref{Taylor_for_a}),
$a(u)\partial_1^2u$ formally corresponds to
$\sum_{k\ge 0}\mathsf{z}_k\Pi_x^k\partial_1^2\Pi_x$, 
an effectively finite sum.
Likewise, in view of (\ref{ao05}) and (\ref{expfirst}),
the counter term $h$ formally corresponds to
$\sum_{l\ge 0}\frac{1}{l!}\Pi_x^l(D^{({\bf 0})})^lc$. 
For any $c\in\mathbb{R}[[\mathsf{z}_k]]$, 
this sum is effectively finite according to (\ref{Dopop}).
As announced in Subsection \ref{S:SG}, we replace $\xi$ by a mollified version 
$\xi_\tau\in C^0$.
Hence to $\Pi_x\in C^2[[\mathsf{z}_k,\mathsf{z}_{\bf n}]]$ we associate the r.~h.~s.
\begin{align}\label{eq:Pi_minus_def}
\Pi^{-}_{x\beta}:=\big(\sum_{k\ge 0}\mathsf{z}_k\Pi^k_x\partial_1^2\Pi_x
-\sum_{l\ge 0}\frac{1}{l!}\Pi^{l}_x(D^{({\bf 0})})^{l} c
+\xi_\tau\mathsf{1}\big)_\beta\;\in\;C^0,
\end{align}
where $\mathsf{1}$ is the
neutral element of the algebra $\mathbb{R}[[\mathsf{z}_k,\mathsf{z}_{\bf n}]]$.

\subsubsection{Definition of \texorpdfstring{$\mathsf{\bar T}$}{T bar} and \texorpdfstring{$\mathsf{T}$}{T}, purely polynomial and populated multi-indices \texorpdfstring{$\beta$}{beta}}
A special role is played by the multi-indices of the form
\begin{align}\label{defpp}
\beta=e_{\bf n}\quad\mbox{for some}\;{\bf n}\not={\bf 0},
\quad\mbox{which we call ``purely polynomial''}.
\end{align}
In view of (\ref{ao07p}), the corresponding linear subspace
\begin{align}\label{defTbar}
\mathsf{\bar T}^*=\{\;\pi\in\mathbb{R}[[\mathsf{z}_k,\mathsf{z}_{\bf n}]]\;|\;
\pi_\beta=0\;\mbox{unless $\beta$ is purely polynomial}\;\}
\end{align}
is the algebraic dual of $\mathsf{\bar T} \cong \mathbb{R}[y_1,y_2] / \mathbb{R}$, the space of space-time polynomials (i.~e. the polynomial sector in \cite[Remark 2.23]{Ha14}) modulo constants. By definitions (\ref{ao07a}) \& (\ref{ao07p}), for $a\equiv0$,
(\ref{formal}) collapses to $u(\cdot)-u(x)$ $=\Pi_{x0}(\cdot)$
$+\sum_{{\bf n}\not={\bf 0}}$ $\Pi_{xe_{\bf n}}(\cdot)$ $\frac{1}{{\bf n}!}$
$\frac{\partial^{\bf n}p}{\partial y^{\bf n}}(0)$.
Specifying the affine parameterization\footnote{which is consistent with $p(0)=0$ and (\ref{new})} to be $u=u(x)+p(\cdot-x)+\Pi_{x0}$ 
for $a\equiv0$, this implies that $p(\cdot-x)$ 
$=\sum_{{\bf n}\not={\bf 0}}\Pi_{xe_{\bf n}}(\cdot)
\frac{1}{{\bf n}!}\frac{\partial^{\bf n}p}{\partial y^{\bf n}}(0)$.
Reproducing \cite[Assumption 5.3]{Ha14}, this leads us to postulate
\begin{align}\label{eq:pi_purely_pol}
\Pi_{x\beta}=(\cdot-x)^{\bf n}\quad\mbox{for $\beta$ of the form (\ref{defpp})}.
\end{align}

\medskip

A special role is played by the additive function
\begin{align}\label{def[]}
[\beta]:=\sum_{k\ge 0}k\beta(k)-\sum_{{\bf n}\not={\bf 0}}\beta({\bf n})
\end{align}
and by the subset of multi-indices
\begin{align}\label{ao09}
&\beta=e_{\bf n}\;\mbox{for some ${\bf n}\not={\bf 0}$}\quad\mbox{or}\quad[\beta]\ge 0,\\
&\mbox{which we subsume by saying ``$\beta$ is populated''}.\nonumber
\end{align}
Indeed, 
we claim that from (\ref{eq:Pi_minus_def}) and (\ref{eq:pi_purely_pol}) we obtain
\begin{align}\label{mapping}
\lefteqn{\Pi_{x\gamma}\equiv 0\;\mbox{unless $\gamma$ is populated}}\nonumber\\
&\quad\Longrightarrow\quad
\Pi_{x\beta}^{-}\in{\mathbb{R}[y_1,y_2]}\;\mbox{unless $[\beta]\ge 0$}.
\end{align}
For the reader's convenience,
the elementary argument for (\ref{mapping}) is provided in Section \ref{sec:algebraic_aspects}.
Since 
we impose the PDE (\ref{ao04}) only up to analytic space-time functions, we learn from
(\ref{mapping}) that it is self-consistent to postulate the l.~h.~s.~of (\ref{mapping}).
Hence the space-time function $\Pi_x$ will have values in 
\begin{align}\label{defT}
\mathsf{T}^*=\{\;\pi\in\mathbb{R}[[\mathsf{z}_k,\mathsf{z}_{\bf n}]]\;|\;
\pi_\beta=0\;\mbox{unless $\beta$ is populated}\;\},
\end{align}
the (algebraic) dual of 
the direct sum indexed by all populated $\beta$'s, 
which we assimilate to Hairer's abstract model space $\mathsf{T}\supset\mathsf{\bar T}$.


\subsubsection{Homogeneity \texorpdfstring{$|\beta|$}{|beta|} of a multi-index \texorpdfstring{$\beta$}{beta}}
The homogeneity $|\beta|$ of a populated
multi-index $\beta$ is motivated by a scaling invariance in law of the
manifold of solutions to (\ref{ao04}): We start with a parabolic rescaling of
space and time according to $x_1=\lambda\hat x_1$ and $x_2=\lambda^2\hat x_2$.
Our assumption (\ref{ao03}) on the noise ensemble is consistent with\footnote{
which for $\alpha=\frac{1}{2}$ turns into the well-known invariance of white noise} 
$\xi(x)=_{\rm law}\lambda^{\alpha-2}\hat{\xi}(\hat{x})$. 
This translates into 
the desired $u(x)=_{\rm law} \lambda^\alpha\hat{u}(\hat{x})$, provided we transform the nonlinearities
according to $a(u)=\hat a(\lambda^{-\alpha}u)$ and
$h(u)=\lambda^{\alpha-2}\hat h(\lambda^{-\alpha}u)$. On the level of the coordinates
(\ref{ao07a}) the former translates into $\mathsf{z}_k$ 
$=\lambda^{-\alpha k}\mathsf{\hat z}_k$. 
When it comes to the parameter $p$ it is consistent with the above\footnote{
in particular $u=u(x)+p(\cdot-x)+\Pi_{x 0}$ for $a\equiv 0$}
that it scales like $u$, i.~e.~$p(x)=\lambda^\alpha\hat{p}(\hat{x})$, 
so that the coordinates (\ref{ao07p}) transform
according to $\mathsf{z}_{\bf n}$ 
$=\lambda^{\alpha-|{\bf n}|}\mathsf{\hat z}_{\bf n}$, where 
\begin{align}\label{ao11}
|{\bf n}|=n_1+2n_2.
\end{align}
Hence we read off (\ref{formal}) that $\Pi_{x\beta}(y)$ 
$=_{\rm law}\lambda^{|\beta|}\hat\Pi_{\hat x\beta}(\hat{y})$, where
\begin{align}\label{ao12}
|\beta|:=\alpha(1+[\beta])+|\beta|_p\quad\mbox{with}\quad
|\beta|_p:=\sum_{{\bf n}\not={\bf 0}}|{\bf n}|\beta({\bf n}).
\end{align}

\medskip

In agreement with regularity structures \cite[Definition 2.1]{Ha14},
because of $\alpha>0$, the set $\mathsf{A}$ of homogeneities is locally finite and bounded from below, in fact by $\alpha$ in our range
of $\alpha< 1$. Moreover, (\ref{ao11}) and (\ref{ao12})
are consistent in the sense of
\begin{align}\label{ao20}
|e_{\bf n}|=|{\bf n}|\quad\mbox{for}\;{\bf n}\not={\bf 0}.
\end{align}
Thanks to our assumption of $\alpha\not\in\mathbb{Q}$, there is a reverse of (\ref{ao20}):
\begin{align}\label{irrational}
|\beta|\in\mathbb{N}\quad\Longrightarrow\quad\beta\;\mbox{is purely polynomial}.
\end{align}
Both reproduce \cite[Assumption 5.3]{Ha14}; 
the implication (\ref{irrational}), in its negated form, 
plays a crucial role in the Liouville argument Proposition~\ref{prop:change_of_basepointII}.

\subsubsection{Postulates on \texorpdfstring{$\Gamma_{yx}$}{Gamma yx}}\label{sec:postGamma}
While Hairer thinks of $\Pi_x$ as a linear map from the abstract model space $\mathsf{T}$
into the space of functions\footnote{in his case rather distributions, so
linear forms themselves} of space-time, we equivalently interpret $\Pi_x$ as a 
space-time function with values in $\mathsf{T}^*$. Next to $\Pi_x$, 
Hairer's notion of a model features $\Gamma_{yx}$ $\in{\rm End}(\mathsf{T})$
encoding the transformation from base-point $y$ to base-point $x$ in the sense of
$\Pi_x=\Pi_y\Gamma_{yx}$. We express this in terms of the (algebraic) transpose\footnote{
We note that since $\mathsf{T}$ is infinite-dimensional, not every element of
${\rm End}(\mathsf{T}^*)$ is a transpose of an element in ${\rm End}(\mathsf{T})$.
We will construct $\Gamma_{xy}^*$ via its matrix representation
with the property that for every $\beta$, $(\Gamma_{xy}^*)_\beta^\gamma$ $:= (\Gamma_{xy}^* \z^\gamma)_\beta$ $=0$
for all but finitely many $\gamma$'s, which means that we have access to 
$\Gamma_{yx}$ even if it is not used in our approach.}
$\Gamma_{xy}^*\in{\rm End}(\mathsf{T}^*)$ as\footnote{
Note the notational abuse of swapping $xy$ when transposing, which is done here to obtain more intuitive re-centering formulas.}
\begin{align}\label{eq:recenter_Pi}
\Pi_x=\Gamma^*_{xy}\Pi_y\quad\mbox{modulo a space-time constant},
\end{align} 
where the modulo is a consequence of (\ref{new}).

\medskip

In particular, (\ref{eq:recenter_Pi}) suggests transitivity $\Gamma_{yx}=\Gamma_{yz}\Gamma_{zx}$, 
which following \cite[Definition 2.17]{Ha14} we postulate:
\begin{align}\label{ao15}
\Gamma_{xy}^*=\Gamma_{xz}^*\Gamma_{zy}^*
\quad\mbox{and}\quad\Gamma_{xx}^*={\rm id}.
\end{align}
Likewise, following \cite[Assumption 3.20]{Ha16} we postulate 
that $\Gamma_{yx}$ preserves the polynomial sector 
$\mathsf{\bar T}$ $\cong\mathbb{R}[[x_1,x_2]]/\mathbb{R}$ and acts on it according to 
$(\cdot-y)^{\bf n}\Gamma_{yx}$ $=(\cdot-x)^{\bf n}$. We claim that this
postulate translates into
\begin{align}\label{ao21}
(\Gamma^*_{xy})_\beta^\gamma
&=\left\{\begin{array}{cl}
{\binom{\bf n}{\bf m}}(y-x)^{{\bf n}-{\bf m}}&
\mbox{if}\;\;\gamma=e_{\bf m}\;\;\mbox{for some}\;{\bf m}\not={\bf 0}\\
0&\mbox{else}\end{array}\right\},
\end{align}
provided that $\beta$ is of the form (\ref{defpp}),
with the understanding that ${\binom{\bf n}{\bf m}}$ vanishes
if the componentwise ${\bf m}\le {\bf n}$ is violated. 
We refer to Section 
\ref{sec:algebraic_aspects} for the argument. Note that (\ref{ao21})
is consistent with (\ref{eq:pi_purely_pol}) and (\ref{eq:recenter_Pi}).

\medskip

Finally, the strict triangularity
of $\Gamma_{yx}-{\rm id}$ with respect to the grading of $\mathsf{T}$
induced by the homogeneity (\ref{ao12}), as postulated in \cite[Definition 2.1]{Ha14},
amounts in our setting to
\begin{align}\label{triangle}
(\Gamma^*_{xy}-{\rm id})_{\beta}^\gamma=0\quad\mbox{unless}\;|\gamma|<|\beta|.
\end{align}
Note that in view of (\ref{ao20}), (\ref{ao21}) is consistent with (\ref{triangle}).


\subsection{Statement of main result: construction and estimates on \texorpdfstring{$(\Pi_x,\Gamma_{yx})$}{(Pi x, Gamma yx)}}\label{State}

Our main result Theorem \ref{thm:main} provides a construction 
of a model  $(\Pi_x,\Gamma_{yx})$ that satisfies 
all the postulates of the previous Subsection \ref{sec:model}.

\begin{theorem}\label{thm:main} Under Assumption~\ref{ass:spectral_gap} the following holds:

\medskip

For every populated $\beta$, there exists\footnote{the last condition amounts
	to $c\in\mathbb{R}[[\mathsf{z}_k]]$} 
\begin{equation}
\mbox{a deterministic}\,\, c_\beta\in\mathbb{R}\,\,
\mbox{satisfying}\quad c_\beta = 0 \quad  \text{unless}\quad|\beta|<2\quad
\text{and}\quad\beta({\bf n}) =0\;\;\text{for all}\;\;{\bf n}\neq{\bf 0}, \label{eq:c_pop_cond}
\end{equation}
and for every $x\in \mathbb{R}^2$ there exists a random 
$\Pi_{x\beta}\in C^2(\mathbb{R}^2)$ such that almost surely
\begin{align}
(\partial_2-\partial_1^2)\Pi_{x\beta}&=\Pi_{x\beta}^{-}\quad
\mbox{unless $\beta$ is purely polynomial},\label{eq:model}
\end{align}
with $\Pi_{x\beta}^{-}$ defined in (\ref{eq:Pi_minus_def}),
and which is given by (\ref{eq:pi_purely_pol}) for $\beta$ purely polynomial.

\medskip

Moreover, for every $x,y\in \mathbb{R}^2$ 
there exists a random 
$\Gamma_{yx}\in{\rm End}(\mathsf{T})$
such that almost surely we have (\ref{eq:recenter_Pi}), (\ref{ao15}),
(\ref{ao21}), and (\ref{triangle}).

\medskip

Finally, we have for all $p<\infty$
\begin{align}
\mathbb{E}^\frac{1}{p}|\Pi_{x\beta}(y)|^p & \lesssim|y-x|^{|\beta|},\label{eq:pi_generic}\\
\mathbb{E}^\frac{1}{p}|(\Gamma^*_{xy})_{\beta}^{\gamma}|^p & \lesssim |y-x|^{|\beta|-|\gamma|}
\quad\mbox{for all populated}\;\gamma,
\label{eq:gamma}
\end{align}
where here and in the sequel, $\lesssim$ means $\le C$ with a constant $C$
only depending\footnote{note that there is no dependence on the law, since
we normalized the constant in (\ref{ao03}); the dependence
on $\alpha$ could be subsumed into the one on $\beta$} on $\alpha$, $\beta$ and $p$.
\end{theorem}

The crucial point is that the estimates (\ref{eq:pi_generic}) and (\ref{eq:gamma})
are independent of the mollification $\xi_\tau$ of $\xi$ in (\ref{eq:Pi_minus_def}).
For pure convenience, we fix the mollification to be
the semi-group convolution $(\cdot)_\tau$, 
so that the ultra-violet cut-off scale is given by $\sqrt[4]{\tau}$,
see Subsection \ref{sec:semi_group}.
Our proof extends to more general classes of
ultra-violet cut-off: For instance, instead of the mollification, one could strengthen
the norm (\ref{ao01}) by adding, on the level of the squared norm, the term
$\tau\int_{\mathbb{R}^2}dx((\partial_1^4-\partial_2^2)^{\frac{1}{4}(\alpha+\frac{3}{2})}
\delta\xi)^2$. It would even be sufficient to impose 
the annealed H\"older regularity (\ref{cw06}) on the level of $\xi$ itself 
by means of an additional approximation argument.
Since $\Pi_{x\beta}\in C^2$ almost surely, it follows from (\ref{eq:pi_generic})
together with $|\beta|\ge\alpha>0$ that (\ref{new}) holds.


\subsection{BPHZ-choice of \texorpdfstring{$c$}{c} and divergent bounds}\label{divc}

In this subsection, we argue that the infra-red part of (\ref{eq:pi_generic})
enforces a canonical choice of $c$, for given regularization parameter $\tau$. 
In fact, our inductive construction algorithm for $(\Pi_x,c)$ is unique, as a benefit
of working on the whole space-time and avoiding a non-canonical
infra-red truncation of the heat kernel $(\partial_2-\partial_1^2)^{-1}$. 
See \cite[Theorem~1.3]{Tem23} for a much stronger uniqueness statement.

\medskip

For the rest of this subsection we fix a populated 
and not purely polynomial $\beta$. 
We start by observing that 
the PDE (\ref{eq:model}) combined with the estimate (\ref{eq:pi_generic}) 
uniquely determines $\Pi_{x\beta}$ given $\Pi_{x\beta}^{-}$.
Indeed, the difference $w$ of two solutions satisfies $(\partial_2-\partial_1^2)w=0$.
Appealing to estimate (\ref{eq:pi_generic}) for $|y-x|\uparrow\infty$,
a Liouville argument shows that $w$ is a (random) polynomial of (parabolic) degree $\le|\beta|$. 
By (\ref{irrational}),
this strengthens to being a polynomial of degree $<|\beta|$.
Appealing to (\ref{eq:pi_generic}) for $|y-x|\downarrow0$, we learn that
this polynomial must vanish. We refer to the proof of Proposition \ref{prop:change_of_basepointII} for details on an annealed version of this Liouville argument.

\medskip

Our inductive algorithm is based on an ordering $\prec$ on multi-indices,
see Subsection \ref{induct}; we now argue that $c_\beta$ is unique given 
$(\Pi_{x\gamma},c_\gamma)$ for $\gamma\prec\beta$. 
Indeed, we obtain from the main estimate (\ref{eq:pi_generic}), 
with help of the kernel estimate (\ref{cw61}), that the ensemble and space-time average
$\lim_{t\uparrow\infty}\mathbb{E}(\partial_2-\partial_1^2)\Pi_{x\beta\,t}(x)$ vanishes
for $|\beta|<2$, where $2$ is the order of the differential operator $\partial_2-\partial_1^2$; 
here $(\cdot)_t$  denotes the semi-group convolution, see Subsection
	\ref{sec:semi_group}. 
By the PDE (\ref{eq:model}) this implies 
\begin{align}\label{eq:BPHZ}
\lim_{t\uparrow\infty}\mathbb{E}\Pi_{x\beta\,t}^{-}(x)=0\quad\mbox{for}\;|\beta|<2.
\end{align}
Writing $\Pi_{x\beta}^{-}=-c_\beta+\tilde\Pi_{x\beta}^{-}$, where 
$\tilde\Pi_{x\beta}^{-}$ does not contain the $(l=0)$-term in (\ref{eq:Pi_minus_def}),
we learn from (\ref{eq:BPHZ}) and the fact that $c_\beta$ is deterministic (and
a space-time constant)
\begin{align}\label{use}
c_\beta=\lim_{t\uparrow\infty}\mathbb{E}\tilde\Pi_{x\beta\,t}^{-}(x)
\quad\mbox{for}\;|\beta|<2.
\end{align}
Note that by the first part of the population condition (\ref{eq:c_pop_cond}),
we don't have to consider $|\beta|\ge 2$.
It thus remains to note that
$\tilde\Pi_{x\beta}^{-}$ depends on $(\Pi_{x\gamma},c_\gamma)$ only through $\gamma\prec\beta$, 
which we shall establish in Section~\ref{sec:algebraic_aspects}.
This shows that our algorithm uniquely determines $c$, in the spirit
of a BPHZ-choice of renormalization, 
see \cite[Theorem 6.18 and eq. (6.25)]{BHZ19} for the form BPHZ renormalization
takes within regularity structures.

\medskip

Let us sketch why (\ref{use}) is consistent with the
second part of the population condition (\ref{eq:c_pop_cond}) on $c$,
referring to Subsection \ref{sec:Pi_minus_constr} for details.
The shift invariance in law of $\xi$, see Assumption \ref{ass:spectral_gap},
which by the above uniqueness transmits
to $\Pi_{x\beta}$, ensures that the r.~h.~s.~of (\ref{use}) is independent of $x$.
The reflection parity in law of $\xi$, which transmits to $\Pi_{x\beta}$, 
ensures that the r.~h.~s.~of (\ref{use}) vanishes for odd 
$\sum_{{\bf n}\not={\bf 0}}n_1\beta({\bf n})$, which in view of $|\beta|_p\le|\beta|<2$,
cf.~(\ref{ao12}), implies the second population constraint in (\ref{eq:c_pop_cond}).

\medskip

Theorem \ref{thm:main} only states those estimates that are independent 
of the ultra-violet cut-off provided by the convolution
of the driver $\xi$ in (\ref{eq:Pi_minus_def}).
In particular, $c$ is expected to diverge as the cut-off scale\footnote{as always, measured
in units of $x_1$} $\sqrt[4]{\tau}$ goes to zero. 
The following proposition provides the scaling-wise natural estimate on $c_\beta$, an annealed
and weighted $C^{2,\alpha}$-estimate on $\Pi_{x\beta}$, and a similar
$C^{\alpha}$-estimate on $\Pi_{x\beta}^{-}$. By Kolmogorov's criterion,
see for instance \cite[proof of Lemma 4.1]{OW19}, 
the latter ensures that $\Pi_{x\beta}\in C^2$ almost surely, 
in line with Theorem~\ref{thm:main}.
We stress that these estimates are not used in a quantitative way in the proof of Theorem~\ref{thm:main} 
-- the estimates of Theorem~\ref{thm:main} are orthogonal to the divergence of $c$.
However, qualitative boundedness of $c$ plays a role when addressing analyticity in $\mathsf{z}_0$, 
see Subsection~\ref{sec:OSSW}, and qualitative continuity is used in reconstruction, 
see Proposition~\ref{prop:reconstr_reg} and Proposition~\ref{prop:recIII}. Moreover, 
qualitative boundedness of $\partial_1^2 \Pi_x$ is convenient when rigorously establishing Malliavin differentiability of $\Pi^-_x$, see Subsection~\ref{sec:reconstr_approx}.
\begin{proposition}[Divergent bounds I]\label{rem:1} Under Assumption~\ref{ass:spectral_gap} the following holds for every populated $\beta$:
\begin{align}
|c_\beta|&\lesssim(\sqrt[4]{\tau})^{|\beta|-2},\label{cw03}\\
\mathbb{E}^\frac{1}{p}|\partial_1^2\Pi_{x\beta}(y)|^p + \mathbb{E}^\frac{1}{p}|\partial_2\Pi_{x\beta}(y)|^p
&\lesssim(\sqrt[4]{\tau})^{\alpha-2}(\sqrt[4]{\tau}+|y-x|)^{|\beta|-\alpha}.\label{cw60}
\end{align}
Furthermore, we have
\begin{align}
&\mathbb{E}^\frac{1}{p}|\partial_1^2\Pi_{x\beta}(y)-\partial_1^2\Pi_{x\beta}(z)|^p
+\mathbb{E}^\frac{1}{p}|\partial_2\Pi_{x\beta}(y)-\partial_2\Pi_{x\beta}(z)|^p \label{cw60_cont}\\
&+\mathbb{E}^\frac{1}{p}|\Pi_{x\beta}^-(y)-\Pi_{x\beta}^-(z)|^p\label{cw60_minus_cont}\\
&\lesssim(\sqrt[4]{\tau})^{-2}(\sqrt[4]{\tau}+|y-z|+|z-x|)^{|\beta|-\alpha} 
|y-z|^\alpha.\nonumber
\end{align}
\end{proposition}
The proof of Proposition~\ref{rem:1} is given in Section~\ref{sec:proof_rem}.


\subsection{Exponential formula for \texorpdfstring{$\Gamma_{xy}^*$}{Gamma* xy} and structure group \texorpdfstring{$\mathsf{G}$}{G}, definition of \texorpdfstring{$\pi_{xy}^{({\bf n})}$}{pi n xy}}

In Subsection \ref{sec:Gamma_construct}, we will construct the change-of-basepoint 
transformation $\Gamma^*_{xy}$ $\in{\rm End}(\mathsf{T}^*)$,
see (\ref{eq:recenter_Pi}), by inductively tilting the $\mathsf{T}^*$-valued model $\Pi_x$
with help of a space-time polynomial in order
to achieve the appropriate order of vanishing of $\Pi_y$ 
in $y$. The coefficients of these space-time polynomials are collected in 
$\{\pi_{xy}^{({\bf n})}\}_{\bf n}\subset\mathsf{T}^*$,
so that the ${\bf n}$'s range\footnote{we will always state explicitly 
if we only mean ${\bf n}\not={\bf 0}$, unless it is used in 
$\mathbb{R}[[\mathsf{z}_k,\mathsf{z}_{\bf n}]]$} over $\mathbb{N}_0^2$.

\medskip

Following Hairer we adopt a more abstract point of view 
on the purely algebraic map $\{\pi^{(\n)}\}_{{\bf n}}$ $\mapsto\Gamma^*$ 
$\in{\rm End}(\mathsf{T}^*)$. It is given by the exponential-type formula
\begin{align}\label{exp}
\Gamma^*=\sum_{k\geq0}\tfrac{1}{k!}\sum_{\n_1,\dots,\n_k}
\pi^{(\n_1)}\cdots\pi^{(\n_k)} D^{(\n_1)}\cdots D^{(\n_k)},
\end{align}
where we have set for convenience
\begin{align}\label{eq:def_Dn}
D^{({\bf n})}:=\partial_{\mathsf{z}_{\bf n}}\quad\mbox{for}\;{\bf n}\not={\bf 0},
\end{align}
which like $D^{({\bf 0})}$ defines a derivation on 
the algebra $\mathbb{R}[[\mathsf{z}_k,\mathsf{z}_{\bf n}]]$. We note that 
(\ref{exp}) is an extension of (\ref{expfirst}) from $\mathbb{R}[[\mathsf{z}_k]]$ to
$\mathbb{R}[[\mathsf{z}_k,\mathsf{z}_{\bf n}]]$ and from $v\in\mathbb{R}$ 
to $\pi^{({\bf 0})}\in\mathsf{T}^*$. A collateral damage of the latter is that (\ref{exp}) 
can no longer be interpreted as a standard matrix exponential, since multiplication by
$\pi^{({\bf 0})}$ and the derivation $D^{({\bf 0})}$ do not commute\footnote{however,
the derivations commute among themselves, like of course the multiplication operators do}.

\medskip

In line with the order of vanishing of $\Pi_{y\beta}$ to be achieved, and thus the 
degree of the tilting space-time polynomial, one imposes the population condition 
\begin{align}\label{eq:pi_n_pop}
\pi^{({\bf n})}_{\beta}=0\quad\mbox{unless}\;|{\bf n}|<|\beta|.
\end{align}
The constraint (\ref{eq:pi_n_pop}) 
also ensures that the sum (\ref{exp})
over $k$ and $\n_1, \ldots,\n_k$ is effectively\footnote{i.~e.~on 
the level of the matrix entries} finite, which is not obvious,
see \cite[Subsection 5.1 and eq. (5.16)]{LOT21} 
or \cite[Lemma 3.12]{LO22}.

\medskip

As shown in \cite[Subsection 5.1]{LOT21}, 
provided the purely polynomial part of the $\pi^{({\bf n})}$'s is
of the form\footnote{where ${\bf m}>{\bf n}$ means ($\m\ge\n$ and ${\bf m}\not={\bf n}$),
so that (\ref{eq:def_pi_n_poly}) is consistent with (\ref{eq:pi_n_pop}) in view of
(\ref{ao20})}
\begin{align}\label{eq:def_pi_n_poly}
\pi_{e_{\bf m}}^{({\bf n})}=\left\{\begin{array}{cl}
{\binom{\bf m}{\bf n}}\,h^{{\bf m}-{\bf n}}&
\mbox{provided}\;{\bf m}>{\bf n}\\
0&\mbox{else}\end{array}\right\}
\end{align}
for some space-time shift vector $h\in\mathbb{R}^2$, the corresponding set of
transposed endomorphisms $\Gamma\in{\rm End}(\mathsf{T})$
can be assimilated to the structure group 
$\mathsf{G}\subset{\rm Aut}(\mathsf{T})$ in the sense of Hairer.
In particular, such $\Gamma$'s meet the postulates of regularity structures:
They are strictly triangular in the sense of \eqref{triangle}, see \cite[eq. (5.10)]{LOT21},
and they respect the polynomial sector $\mathsf{\bar T}$ 
in the sense of \cite[Assumption 3.20]{Ha16}, i.~e.~
$\Gamma (\cdot)^\n = (\cdot +h)^\n$ (modulo constants) with the $h$ from (\ref{eq:def_pi_n_poly}), 
see \cite[eq. (5.11)]{LOT21}; in particular we have
\begin{equation}\label{mt44}
\Gamma\mathsf{\bar{T}} \subset \mathsf{\bar T}.
\end{equation}

\medskip

Our dual perspective has the advantage of revealing that
$\Gamma^*$ is multiplicative\footnote{due to the presence of
the polynomial sector $\mathsf{\bar T}$, cf.~(\ref{defTbar}),
$\mathsf{T}^*$ is not closed under multiplication, hence $\Gamma^*$ cannot be called
an algebra endomorphism.}:
\begin{align}\label{eq:mult}
\Gamma^*\pi\pi'=(\Gamma^*\pi)(\Gamma^*\pi')\quad\mbox{provided}\;\;\pi\pi'\in\mathsf{T}^*,
\quad\mbox{and}\quad\Gamma^*\mathsf{1}=\mathsf{1},
\end{align}
see \cite[Proposition 5.1 (ii)]{LOT21}. Hence\footnote{also appealing to
the finiteness properties that ensure the existence of the transpose} 
$\Gamma^*$ is determined by its value
on the coordinates $\mathsf{z}_k,\mathsf{z}_{\bf n}$;
it is straightforward to check from (\ref{eq:def_Dnull})
and (\ref{eq:def_Dn}) that those are given by\footnote{a version of (\ref{eq:Gamma_z_k}) 
already appeared in the discussion of Subsubsection \ref{sss}}
\begin{align}
\Gamma^*\mathsf{z}_{k}&=\sum_{l\ge 0} \binom{k+l}{k} (\pi^{({\bf 0})})^{l}
\mathsf{z}_{k+l}\quad\mbox{for}\;k\ge 0,\label{eq:Gamma_z_k}\\
\Gamma^*\mathsf{z}_{\bf n}&=\mathsf{z}_{\bf n}+\pi^{({\bf n})}\quad\mbox{for}\;{\bf n}\not={\bf 0}.
\label{eq:Gamma_z_n}
\end{align}
Moreover, the group structure becomes more apparent on the dual level\footnote{we recall that
by $\{\pi^{(\n)}\}_\n \mapsto \Gamma^*$ we understand that 
$\{\pi^{(\n)}\}_\n$ gives rise to $\Gamma^*$ via \eqref{exp}}:
\begin{align}\label{eq:monoid}
\{\pi^{(\n)}+\Gamma^*\bar\pi^{(\n)}\}_\n \mapsto \Gamma^*\bar\Gamma^*
\quad \mbox{provided} \quad
\{\pi^{(\n)}\}_\n \mapsto \Gamma^* 
\mbox{ and }
\{\bar\pi^{(\n)}\}_\n \mapsto \bar\Gamma^*,
\end{align}
see \cite[Proposition 5.1 (iii)]{LOT21}. For the purpose of establishing
transitivity (\ref{ao15}) we retain that 
because of the obvious $\{0\}_\n\mapsto{\rm id}$, we obtain from
(\ref{eq:monoid}) and (qualitative) invertibility in a first stage that
$\{-\Gamma^{-*}\pi^{(\n)}\}_\n \mapsto\Gamma^{-*}$, and in a second stage that
\begin{align}\label{eq:monoid_inverse}
\{\pi^{(\n)}-\Gamma^*\bar\Gamma^{-*}\bar\pi^{(\n)}\}_\n \mapsto \Gamma^*\bar\Gamma^{-*}
\quad \mbox{provided} \quad
\{\pi^{(\n)}\}_\n \mapsto \Gamma^* 
\mbox{ and }
\{\bar\pi^{(\n)}\}_\n \mapsto \bar\Gamma^*.
\end{align}

\medskip

The construction of $\{\pi_{xy}^{(\n)}\}_\n$ that gives rise to $\Gamma^*_{xy}$ 
which satisfy \eqref{eq:recenter_Pi} and \eqref{ao15} is carried out in 
Subsection~\ref{sec:Gamma_construct}. 
The purely polynomial part of $\pi_{xy}^{(\n)}$ is forced upon us:
By \eqref{eq:def_pi_n_poly} only $h$ needs to be chosen, 
and by \eqref{mt44} and \eqref{eq:Gamma_z_n} we see that the choice $h=y-x$ 
is necessary and sufficient to obtain \eqref{ao21}:
\begin{align}\label{ho06}
\pi_{xy e_{\bf m}}^{({\bf n})}=\left\{\begin{array}{cl}
{\binom{\bf m}{\bf n}}\,(y-x)^{{\bf m}-{\bf n}}&
\mbox{provided}\;{\bf m}>{\bf n}\\
0&\mbox{else}\end{array}\right\}.
\end{align}
Our estimate (\ref{eq:pi_generic})
of $\Gamma_{xy}^*$ will be derived from an estimate of $\{\pi^{(\n)}_{xy}\}_{\bf n}$, 
see the upcoming proposition, which 
crucially uses the effective finiteness of the sum in (\ref{exp}).

\begin{proposition}\label{rem:pi_n}
Under Assumption \ref{ass:spectral_gap} the following holds for every populated $\beta$:
\begin{align}\label{eq:pi_n}
\mathbb{E}^\frac{1}{p}|\pi_{xy\beta}^{(\mathbf{n})}|^p
\lesssim|y-x|^{|\beta|-|{\bf n}|}.
\end{align}
\end{proposition}

Note that \eqref{ho06} is consistent with (\ref{eq:pi_n}) in view of (\ref{ao20}).


\subsection{\texorpdfstring{Augmenting the model space with $\mathsf{\tilde T}$ and the model with $\Pi_x^-$}{Augmenting the model space with tildeT and the model with Pi-}}\label{sec:Hairer}
There are essentially only semantic differences between our results
and Hairer's postulates, as stated in \cite[Definition 2.17]{Ha14}.
The algebraic aspects of this are worked out in \cite[Section 5.3]{LOT21},
of which we now give a synopsis: 
Hairer's abstract model space actually corresponds to what in our notation is
\begin{align}\label{eq:direct_sum}
\mathbb{R}\oplus\mathsf{T}\oplus\mathsf{\tilde T}
=\mathbb{R}\oplus\mathsf{\bar T}\oplus\mathsf{\tilde T}\oplus\mathsf{\tilde T},
\end{align}
where $\mathsf{\tilde T}$ is the direct sum over all populated,
not purely polynomial multi-indices,
cf.~(\ref{ao09}), a linear complement to $\mathsf{\bar T}$ in $\mathsf{T}$.
Hairer's abstract integration operator ${\mathcal I}$, cf.~\cite[Definition 5.7]{Ha14},
is in our notation
given by the identification of the second $\mathsf{\tilde T}$-component with the first
$\mathsf{\tilde T}$-component in (\ref{eq:direct_sum}),
see also \cite[(5.36)]{LOT21}.
Endowing the $\beta$-component of the second $\mathsf{\tilde T}$-contribution with the
homogeneity $|\beta|-2$  meets \cite[Definition 5.7]{Ha14}
for our second-order integration operator $(\partial_2-\partial_1^2)^{-1}$.
The $\mathbb{R}$-component, which is endowed with homogeneity $0$,
is the placeholder for the constant functions factored out in our approach
to $\mathsf{\bar T}\cong\mathbb{R}[y_1,y_2]/\mathbb{R}$.
Loosely speaking, our set-up is minimalistic.

\medskip

To be consistent with (\ref{eq:direct_sum}), 
our model $\Pi_x$ needs to be extended by the constant function
of value 1, and by $\Pi_{x}^{-}$, which in agreement with \eqref{eq:Pi_minus_def} is
given by\footnote{
with the abuse of notation $\Pi_x^-\in\mathsf{\tilde T}^*$, meaning that $\Pi_x^-$ is a $\mathsf{\tilde T}^*$-valued function} 
\begin{align}\label{eq:Pi_minus_def_alt}
\Pi^{-}_{x}=P\sum_{k\ge 0}\mathsf{z}_k\Pi^k_x\partial_1^2\Pi_x
-\sum_{l\ge 0}\frac{1}{l!}\Pi^l_x(D^{({\bf 0})})^lc
+\xi_\tau\mathsf{1} \ \in \ \mathsf{\tilde T}^*,
\end{align}
where $P$ denotes the projection\footnote{
note that \eqref{eq:Tbar_def} differs from
\cite{OSSW21} and \cite{LOT21} where $P$ denotes the projection
onto $\mathsf{\bar T}^*$}
\begin{align}\label{eq:Tbar_def}
P\quad\mbox{of}\;
\mathbb{R}[[\mathsf{z}_k,\mathsf{z}_{\bf n}]]\;\mbox{on}\;\mathsf{\tilde T}^*,
\end{align}
on the algebraic dual\footnote{
which as a linear space is the direct product over the 
populated not purely polynomial multi-indices}
$\mathsf{\tilde T}^*$ of $\mathsf{\tilde T}$.
Let us point out that the role of $P$ in \eqref{eq:Pi_minus_def_alt} is rather 
to project $\mathsf{z}_k\Pi_x^k\partial_1^2\Pi_x$ from 
$\mathbb{R}[[\mathsf{z}_k,\mathsf{z}_\n]]$ into $\mathsf{T}^*$;
the further restriction to $\mathsf{\tilde T}^*$ is automatic 
due to the presence of $\mathsf{z}_k$.
The second and third contributions to \eqref{eq:Pi_minus_def_alt} 
belong to $\mathsf{\tilde T}^*$,
which is obvious for the third contribution and 
follows from \eqref{ord19} for the second contribution. 
For later use, we note that by \eqref{mt44} and the definitions (\ref{eq:def_Dnull}) \& (\ref{eq:def_Dn})
\begin{align}\label{eq:algIII5}
\Gamma^*\mathsf{\tilde T}^*\subset\mathsf{\tilde T}^*
\quad\mbox{and}\quad
D^{({\bf n})}\mathsf{T}^*\subset\mathsf{\tilde T}^*.
\end{align}

\medskip

Like \eqref{eq:recenter_Pi} for $\Pi_x$ we obtain for $\Pi_x^-$ that
\begin{align}\label{eq:recenter_Pi_minus}
(\Pi_{x}^{-}-\Gamma_{xy}^*\Pi_{y}^{-})_\beta\quad
\mbox{is a random polynomial of degree}\;\le|\beta|-2,
\end{align}
see \eqref{eq:recenter_Pi_minus_specific} below for a more specific form.
To be consistent with (\ref{eq:direct_sum}), also $\Gamma$ needs to be
extended to an endomorphism of (\ref{eq:direct_sum}).
The extension of $\Gamma$ in \cite[(5.34)]{LOT21} precisely
incorporates the polynomial correction terms in
(\ref{eq:recenter_Pi}) \& (\ref{eq:recenter_Pi_minus}), in now full agreement with
\cite[Definition 2.17]{Ha14}.

\begin{proposition}\label{rem:6}
The transformations (\ref{eq:recenter_Pi}) and (\ref{eq:recenter_Pi_minus}) specify to
\begin{align}
\Pi_x&=\Gamma_{xy}^*\Pi_y+\pi^{({\bf 0})}_{xy},\label{eq:recenter_Pi_specific}\\
\Pi_x^{-}&=\Gamma_{xy}^*\Pi_y^{-}
+P\sum_{k\ge 0}\mathsf{z}_k\big(\Gamma_{xy}^*({\rm id}-P)\Pi_y+\pi^{({\bf 0})}_{xy}\big)^k
\partial_1^2\big(\Gamma_{xy}^*({\rm id}-P)\Pi_y+\pi^{({\bf 0})}_{xy}\big).
\label{eq:recenter_Pi_minus_specific}
\end{align}
where $\pi^{({\bf n})}_{xy}$ is related to $\Gamma_{xy}^*$ via (\ref{exp}).
\end{proposition}
The re-centering properties \eqref{eq:recenter_Pi_specific} and \eqref{eq:recenter_Pi_minus_specific} are established
in Propositions \ref{prop:change_of_basepointII} and \ref{prop:change_of_basepoint}, respectively.

\medskip

An alternative strengthening of \eqref{eq:recenter_Pi_minus}, 
which we will not make any use of, is given in \cite[(5.29)]{LOT21} in the form
\begin{align*}
\Pi_x^{-}&=\Gamma_{xy}^*\Pi_y^{-}
+\sum_{\n\neq\0}(P\pi_{xy}^{(\n)})(\partial_2-\partial_1^2)(\cdot -y)^{\n} .
\end{align*}
To show this, first, in view of (\ref{eq:pi_purely_pol}), 
(\ref{eq:model}) may be extended to all multi-indices:
\begin{equation*}
	(\partial_2-\partial_1^2)\Pi_{y}
	=\Pi_{y}^{-}+\sum_{{\bf n}\not={\bf 0}}\mathsf{z}_{\bf n}
	(\partial_2-\partial_1^2)(\cdot-y)^{\bf n}. 
\end{equation*}
We apply $\Gamma^*_{xy}$ to this identity, and use (\ref{eq:recenter_Pi}) on the l.~h.~s.~and 
\eqref{eq:Gamma_z_n} on the r.~h.~s. We then apply $P$ to the resulting identity, and use (\ref{eq:model}) on the l.~h.~s.~and 
the first item of (\ref{eq:algIII5}) on the r.~h.~s.

\medskip

We learn from evaluating (\ref{eq:recenter_Pi_specific}) and using \eqref{new} that
\begin{align}\label{eq:Pipi}
\Pi_x(y)=\pi^{({\bf 0})}_{xy},
\end{align}
which will play a major role in the proof. 
With help of (\ref{eq:model}),
it is easy to extend (\ref{eq:pi_generic}) from $\Pi_x$ to\footnote{In fact, the actual proof proceeds the other way around, namely passing from $\Pi_x^-$ to $\Pi_x$, cf. Proposition \ref{prop:int_pi_minus} below.} $\Pi_x^{-}$:

\begin{proposition}\label{rem:5}
Under Assumption \ref{ass:spectral_gap} the following holds for every populated $\beta$:
\begin{equation}\label{eq:pi_minus_generic}
\mathbb{E}^{\frac{1}{p}}|\Pi_{x\beta\,t}^-(y)|^p\lesssim
(\sqrt[4]{t})^{\alpha-2}(\sqrt[4]{t}+|y-x|)^{|\beta|-\alpha}.
\end{equation}
\end{proposition}

The estimate (\ref{eq:pi_minus_generic}) actually holds
for any convolution kernel in the sense of \cite[Definition~2.17]{Ha14}, 
as can be seen from the upcoming  argument.
The setting of \cite[Definition~2.17]{Ha14} corresponds to $|y-x|\le\sqrt[4]{t}$ here, 
in which case the r.~h.~s.~of \eqref{eq:pi_minus_generic} is $\sim(\sqrt[4]{t})^{|\beta|-2}$, 
in line with the postulated homogeneity $|\beta|-2$ 
of the $\beta$-component of $\mathsf{\tilde T}$ at the beginning of this subsection.
For (\ref{eq:pi_minus_generic}) (and thus $\beta$ not purely polynomial thanks
to $P$ in (\ref{eq:Pi_minus_def_alt})), we appeal to (\ref{eq:model}), to which
we apply the convolution operator and then $\mathbb{E}^\frac{1}{p}|\cdot|^p$:
\begin{align*}
\mathbb{E}^\frac{1}{p}|\Pi_{x\beta\,t}^{-}(y)|^p
\le\int dz|(\partial_2-\partial_1^2)\psi_t(y-z)|\mathbb{E}^\frac{1}{p}|\Pi_{x\beta}(z)|^p.
\end{align*}
Hence (\ref{eq:pi_minus_generic}) follows from (\ref{eq:pi_generic})
via the moment bounds (\ref{cw61}) on $\psi_t$.


\subsection{Relation to model in \texorpdfstring{\cite{OSSW21}}{OSSW21}}\label{sec:OSSW}

In this subsection we connect our notion of model constructed in Theorem \ref{thm:main} to the one postulated in \cite{OSSW21}. The following is specific to the quasi-linear problem and unrelated to the proof of our main result, and thus may be skipped.

\medskip

The quasi-linear equation (\ref{ao04}) differs from a semi-linear equation,
like the multiplicative heat equation $(\partial_2-\partial_1^2)u+h(u)=a(u)\xi$, 
by the absence of the
invariance in law $u=\lambda\hat u$, $a=\lambda\hat a$, $h=\lambda\hat h$.
This scale invariance would be encoded by the tighter population condition $\sum_{k}(k-1)\beta(k)$ 
$+\sum_{{\bf n}}\beta({\bf n})=-1$, compare with (\ref{ao09}), see \cite[eq. (7.2)]{LOT21} for details
on the latter. 
However, this lack of the tighter population condition for the quasi-linear equation (\ref{ao04})
is compensated by the special role of $\mathsf{z}_0$:
A priori only formal power series in $\mathsf{z}_0$ are convergent power series.
Indeed, in view of (\ref{ao07a}), changing the origin of $a$-space from $0$ to some $a_0-1$ 
amounts to the actual replacement
\begin{align}\label{ud01}
\partial_2-\partial_1^2\quad\leadsto\quad\partial_2-a_0\partial_1^2
\end{align}
next to the formal replacement $\mathsf{z}_0\leadsto\mathsf{z}_0+(a_0-1)$.
Hence we expect analyticity as long as the real part of $a_0\in\mathbb{C}$ is positive.
Replacing the formal variable $\mathsf{z}_0$ by the actual variable $a_0-1$
will allow us to eventually restrict to multi-indices $\beta$ with
\begin{align}\label{cw14}
\beta(k=0)=0.
\end{align}

\medskip
 
Let us be more precise: 
Any formal power series $\pi$ in the variables $\{\mathsf{z}_k\}_{k\ge 1}$ and
$\{\mathsf{z}_{\bf n}\}_{{\bf n}\not={\bf 0}}$
with coefficients that are analytic functions in $a_0\in\mathbb{C}$ with positive real part 
can be identified with an element of $\mathbb{C}[[\{\mathsf{z}_k\}_{k\ge 0},
\{\mathsf{z}_{\bf n}\}_{{\bf n}\not={\bf 0}}]]$ via
\begin{align}\label{cw42}
\pi_{\beta+le_0}=\frac{1}{l!}\frac{d^l\pi_\beta}{da_0^l}(a_0=1)\quad\mbox{for}\;\beta\;
\mbox{satisfying}\;(\ref{cw14})\;\mbox{and}\;l\ge 0.
\end{align}
Hence this new algebra
is canonically (and strictly) embedded into our $\mathbb{C}[[\mathsf{z}_k,\mathsf{z}_{\bf n}]]$.
Hence we will pass from $\mathsf{T}^*$ to its intersection with 
this new algebra. As a linear space, this smaller $\mathsf{T}^*$ amounts to
the direct product of the space of analytic functions in $a_0$, indexed by populated
multi-indices satisfying (\ref{cw14}). 
Under the identification (\ref{cw42}), the derivation $D^{({\bf 0})}$ defined in 
(\ref{eq:def_Dnull}) restricts to a derivation on the new algebra (formally) given by
\begin{align}\label{cw24}
D^{({\bf 0})}=\mathsf{z}_1\partial_{a_0}+\sum_{k\ge 1}(k+1)\mathsf{z}_{k+1}\partial_{\mathsf{z}_k}.
\end{align}
Analogous to \eqref{def_Dnull_rig}, 
as an endomorphism on the tighter version of $\mathsf{T}^*$, it is (rigorously)
given through its matrix entries
\begin{align}\label{cw43}
(D^{({\bf 0})})_{\beta}^\gamma
&=\left\{\begin{array}{cc}
\partial_{a_0}&\mbox{provided}\;\gamma+e_1=\beta\\
0&\mbox{else}\end{array}\right\}\nonumber\\
&+\sum_{k\ge 1}\left\{\begin{array}{cc}
(k+1)\gamma(k)&\mbox{provided}\;\gamma+e_{k+1}=\beta+e_k\\
0&\mbox{else}\end{array}\right\},
\end{align}
where both $\beta$ and $\gamma$ satisfy (\ref{cw14}). Given elements $\pi^{({\bf n})}$ 
of the tighter version of $\mathsf{T}^*$ satisfying the population condition (\ref{eq:pi_n_pop}),
(\ref{exp}) then defines a triangular automorphism  on the tighter version of $\mathsf{T}^*$.
This gives rise to a re-interpretation of $\mathsf{G}^*$ $\in{\rm End}(\mathsf{T}^*)$
as the dual of the structure group. 


\begin{proposition}\label{rem:2}
When replacing (\ref{eq:model}) by 
\begin{align}\label{cw20}
(\partial_2-a_0\partial_1^2)\Pi_{x\beta}=\Pi_{x\beta}^{-}
\quad\mbox{for $\beta$ not purely polynomial}
\end{align}
with $a_0\in\mathbb{C}$ of positive real part,
all statements of Theorem \ref{thm:main} hold locally uniformly in $a_0$.
Also the estimates (\ref{cw03}) \& (\ref{cw60}) on $c$ \& $\partial_1^2\Pi_x$,
as well as the estimate   (\ref{eq:pi_n}) on $\pi^{({\bf n})}_{xy}$ hold locally uniformly.

\medskip

For any base point $x,y$, any ${\bf n}$ and any populated multi-index $\beta$, 
\begin{align}
c_{\beta}&\;\mbox{is analytic},\label{cw58ter}\\ 
\Pi_{x\beta}&\;\mbox{is analytic w.~r.~t.~the norm given by 
(\ref{eq:pi_generic})}\label{cw58bis},\\
\pi^{({\bf n})}_{xy\beta}&\;\mbox{is analytic w.~r.~t.~the norm given by (\ref{eq:pi_n})},
\label{cw59bis}
\end{align}
all w.~r.~t.~the variable $a_0$.

\medskip

For these three objects $\pi=c,\pi^{({\bf n})}_{xy},\Pi_x$, (\ref{cw42})
holds for all $a_0$:
\begin{align}\label{cw44}
\pi_{\beta+le_0}=\frac{1}{l!}\frac{d^l\pi_\beta}{da_0^l}\quad\mbox{for}\;\beta\;
\mbox{satisfying}\;(\ref{cw14})\;\mbox{and}\;l\ge 0.
\end{align}
\end{proposition}


Noting that (\ref{eq:Pi_minus_def}) takes the form
\begin{align}\label{eq:def_Pi_minus_a0}
\Pi_{x\beta}^{-}=\big(\sum_{k\ge 1}\mathsf{z}_k\Pi_{x}^k\partial_1^2\Pi_{x}
-\sum_{k\ge 0}\frac{1}{k!}\Pi_x^k(D^{({\bf 0})})^kc + \xi_\tau\mathsf{1}\big)_\beta
\quad\mbox{for}\;\beta\;\mbox{satisfying}\;(\ref{cw14}),
\end{align}
we learn from (\ref{cw44}) that 
when equipped with the interpretation (\ref{cw43}) of $D^{({\bf 0})}$ and
the corresponding interpretation of $\Gamma_{xy}^*$ via (\ref{exp}), 
the restriction of the $a_0$-dependent model to multi-indices of the form
(\ref{cw14}) is closed. Together with (\ref{cw20}),
this is the input for \cite{OSSW21}.

\medskip

On the above class (\ref{cw14})
of $\beta$'s, the homogeneity $|\cdot|$ defined in (\ref{ao12}) is actually coercive. 
In the case of $\alpha=\frac{1}{2}-$
relevant for white noise, we are left with 10 multi-indices satisfying $|\beta|<2$.
Ordered by increasing homogeneity they are given in Figure~\ref{fig:sing_mult}.
\begin{figure}[H]
\begin{center}
\begin{tabular}{c|l}
homogeneity & multi-indices \\
\hline
$\alpha$ & $0$ \\
$2\alpha$ & $e_1$ \\
$3\alpha$ & $e_2,\;2e_1$ \\
$\alpha+1$ & $e_1+e_{(1,0)}$ \\
$4\alpha$ & $e_3,\;e_1+e_2,\;3e_1$ \\
$2\alpha+1$ & $e_2+e_{(1,0)},\;2e_1+e_{(1,0)}$
\end{tabular}
\end{center}
\caption{Multi-indices $\beta$ with $|\beta|<2$ for $\alpha\in(\frac{2}{5},\frac{1}{2})$. \label{fig:sing_mult}}
\end{figure}

For the 10 multi-indices in Figure~\ref{fig:sing_mult}, the combination of (\ref{cw20}) and (\ref{eq:def_Pi_minus_a0})
takes the form
\begin{align*}
(\partial_2-a_0\partial_1^2)\Pi_{x0}&=\xi_\tau-c_0,\\
(\partial_2-a_0\partial_1^2)\Pi_{xe_1}&=\Pi_{x0}\partial_1^2\Pi_{x0}-c_{e_1},\\
(\partial_2-a_0\partial_1^2)\Pi_{x e_2} &= \Pi_{x0}^2\partial_1^2\Pi_{x0} 
- (c_{e_2}+ 2\Pi_{x0} c_{e_1}),\\
(\partial_2-a_0\partial_1^2)\Pi_{x 2e_1} &= \Pi_{x0}\partial_1^2\Pi_{xe_1}+\Pi_{xe_1}\partial_1^2\Pi_{x0} 
- (c_{2e_1}+ \Pi_{x0} \partial_{a_0}c_{e_1}),\\
(\partial_2-a_0\partial_1^2)\Pi_{x e_3} &= \Pi_{x0}^3\partial_1^2\Pi_{x0} 
- (c_{e_3}+ 3\Pi_{x0} c_{e_2}+3\Pi_{x0}^2c_{e_1}), \\
(\partial_2-a_0\partial_1^2)\Pi_{x e_1+e_2} &= \Pi_{x0}\partial_1^2\Pi_{xe_2}+\Pi_{xe_2}\partial_1^2\Pi_{x0} +2\Pi_{x0}\Pi_{xe_1}\partial_1^2\Pi_{x0}+\Pi_{x0}^2\partial_1^2\Pi_{xe_1} \\
&\quad- (c_{e_1+e_2}+ 2\Pi_{xe_1} c_{e_1}+\Pi_{x0}\partial_{a_0}c_{e_2}+4\Pi_{x0}c_{2e_1}+3\Pi_{x0}^2\partial_{a_0}c_{e_1}),\\
(\partial_2-a_0\partial_1^2)\Pi_{x 3e_1} &= \Pi_{x0}\partial_1^2\Pi_{x2e_1}+\Pi_{x2e_1}\partial_1^2\Pi_{x0} +\Pi_{xe_1}\partial_1^2\Pi_{xe_1} \\
&\quad-(c_{3e_1}+ \Pi_{xe_1} \partial_{a_0}c_{e_1}+\Pi_{x0}\partial_{a_0}c_{2e_1}+\Pi_{x0}^2\tfrac{1}{2}\partial_{a_0}^2c_{e_1}),\\
(\partial_2-a_0\partial_1^2)\Pi_{xe_1+e_{(1,0)}}&=\Pi_{xe_{(1,0)}}\partial_1^2\Pi_{x0},\\
(\partial_2-a_0\partial_1^2)\Pi_{xe_2+e_{(1,0)}}&=2\Pi_{xe_{(1,0)}}\Pi_{x0}\partial_1^2\Pi_{x0}-2\Pi_{xe_{(1,0)}}c_{e_1},\\
(\partial_2-a_0\partial_1^2)\Pi_{x2e_1+e_{(1,0)}}&=\Pi_{x0}\partial_1^2\Pi_{xe_1+e_{(1,0)}} + \Pi_{xe_1+e_{(1,0)}}\partial_1^2\Pi_{x0}+\Pi_{xe_{(1,0)}}\partial_1^2\Pi_{xe_1}
- \Pi_{xe_{(1,0)}} \partial_{a_0}c_{e_1}.
\end{align*}
Together with the BPHZ-choice of renormalization contained in the large-$\sqrt[4]{t}$
behavior imposed on $\Pi_{x\beta}$ through (\ref{eq:pi_generic}), this inductively
determines the functions $c_{\beta}(a_0)$. Equipped with these, (\ref{ao05})
takes the form
\begin{align*}
h(u)=\sum_{\beta:|\beta|<2}c_{\beta}(a(u))
\prod_{k\ge 1}(\frac{1}{k!}\frac{d^ka}{du^k}(u))^{\beta(k)},
\end{align*}
which reproduces \cite[eq. (15)]{OSSW21} in the present paper's notation.

\medskip

It is in this form we may connect to \cite[Assumptions 1 and 2]{OSSW21}. 
Loosely speaking, 
the assumptions in \cite{OSSW21} are contained in the output of Theorem \ref{thm:main},
as upgraded through Proposition \ref{rem:2}.
More precisely, \cite[eq. (5)]{OSSW21} is covered by (\ref{eq:model}) in the form of
(\ref{cw20}), and \cite[eq. (6)]{OSSW21} is covered by (\ref{eq:pi_purely_pol}).
The estimates \cite[eq. (7) and (8)]{OSSW21} are covered by (\ref{eq:pi_minus_generic})
and (\ref{eq:pi_generic}), 
with the difference that in \cite{OSSW21} (like in \cite{Ha14}), 
they are formulated in terms of a general (though fixed) convolution kernel,
and that they are pathwise, with a constant absorbed into a single scaling factor $N_0$ and, 
as mentioned above, locally uniform in $a_0$, cf.~\cite[eq. (30)]{OSSW21}. 
The transformation \cite[eq. (9)]{OSSW21} is reproduced by (\ref{eq:recenter_Pi_specific}).
Estimate \cite[eq. (10)]{OSSW21} is covered by (\ref{eq:gamma}); however, in view of
(\ref{cw24}), the entries of $\Gamma_{xy}$ are differential operators in $a_0$.
Finally \cite[eq. (11)]{OSSW21} follows from evaluating (\ref{eq:def_Pi_minus_a0}) at $x$
while appealing to (\ref{new}). The crucial population condition (\ref{eq:c_pop_cond})
on $c$ is contained in the text just above \cite[eq. (11)]{OSSW21}, and re-formulated  
as $D^{({\bf n})}c=0$ for ${\bf n}\not={\bf 0}$. 
The polynomial correction in (\ref{eq:recenter_Pi_minus}) 
does not appear in \cite{OSSW21}, since there, the model is (implicitly) truncated beyond $|\beta|<2$.
Because of this truncation, only ${\bf n}={\bf 0},(1,0)$ matter in \cite{OSSW21};
however, since \cite{OSSW21} considers $d$ space dimensions, there are $d$ versions
of ${\bf n}=(1,0)$.
There are some differences in notation: When it comes to $\Gamma$, \cite{OSSW21} omits the $*$ 
but exchanges the order in $xy$, while $c$ is called $q$.


\section{Structure of proof}\label{sec:structure}

In this section, all multi-indices $\beta,\beta',\gamma$ 
are implicitly assumed to be populated, cf.~(\ref{ao09}). 
The induction runs over populated multi-indices $\beta$ which are not purely polynomial. 
The purely polynomial multi-indices are treated before the inductive proof: 
If $\beta=e_{\bf n}$ for ${\bf n}\not={\bf 0}$, we must set $c_\beta=0$ 
and $\Pi_{x\beta}^{-}=0$ by (\ref{eq:c_pop_cond}) and (\ref{eq:Pi_minus_def_alt}), 
respectively, and define $\Pi_{x\beta}$ by (\ref{eq:pi_purely_pol}) and
$\pi_{xy\beta}^{(\n)}$ by (\ref{ho06}). 
We note that these definitions are consistent with covariance under shift (\ref{eq:Pi_shift})
and reflection (\ref{eq:Pi_parity}). Finally, 
(\ref{eq:pi_generic}) and (\ref{eq:pi_n}) are satisfied by (\ref{ao20}).  

\subsection{Intertwining of estimates and constructions in induction proof}

Working on the whole space-time $\mathbb{R}^2$ instead of a torus, as we do, has many advantages.
The most obvious is that we do not introduce an artificial scale, namely the size of
the torus, that would break scaling. Another advantage is that 
the inversion of $(\partial_2-\partial_1^2)$
does not require a Fredholm alternative and thus a Lagrange parameter\footnote{
which is polynomial when working on the tensor space of periodic functions and polynomials,
as in \cite[Subsection 3.1]{LO22}}.
However, an inconvenience is that we can't separate the construction from the estimates:
Because of possible infra-red divergences, we need the large-$\sqrt[4]{t}$ part of the estimate
(\ref{eq:pi_minus_generic}) on $\Pi_x^{-}$ to uniquely solve (\ref{eq:model})
for $\Pi_x$ within the growth and anchoring expressed by (\ref{eq:pi_generic}).
In fact, it is not clear whether one can
construct a pre-model in the sense of \cite[Subsection 4.2]{Ha16}
on the whole space.

\medskip

While construction and estimates are logically intertwined, as explained in
Subsection \ref{sec:order} (see also Appendix~\ref{app:tables}), we choose to separate them
in presentation: Section \ref{sec:anal_est} contains the uniform in the ultra-violet cut-off estimates and their proofs,
while Section \ref{sec:model_construction} contains the construction. Section~\ref{sec:proof_rem} 
contains the proof of the divergent (in the ultra-violet cut-off) estimates of Proposition~\ref{rem:1}
and the proof of Proposition~\ref{rem:2}. In fact, also the proof of Malliavin differentiability
is intertwined logically;
it is given in Section~\ref{sec:mall_approx}, while definitions and properties of the
Malliavin–Sobolev spaces can be found in Appendix~\ref{ss:Malle}.
Moreover, the order we present the estimates within the Sections~\ref{sec:anal_est} and \ref{sec:proof_rem}
and the Malliavin differentiability in Section~\ref{sec:mall_approx}
is not strictly by logical order, but according to the objects estimated. 
We explain this inherent structure in Subsections
\ref{sec:loop} and \ref{sec:tasks}. 
Section \ref{sec:algebraic_aspects} establishes the various triangular structures
important for the inductive construction. 


\subsection{The five loops of an induction step: original quantities, expectation, Malliavin de\-rivative, modelled distribution, and back}
\label{sec:loop}

The structure of an induction step requires the distinction of two 
cases:\footnote{even if we were just
interested in singular $\beta$'s, the structure of ${\rm d}\Gamma^*$ requires the estimate
of regular $\beta$'s}
\begin{itemize}
 \item the regular case of $|\beta|\geq 2$, and
 \item the singular case of $|\beta|<2$. 
\end{itemize}
It is thus convenient to introduce the following projection
\begin{align}\label{eq:defQ}
Q\quad\mbox{is the projection onto the direct product indexed by $\beta$'s with $|\beta|<2$}.
\end{align}
The proof of the estimates in Section \ref{sec:anal_est} 
is structured into five subsections, each having the structure of a loop,
and which we order by increasing complexity,
\begin{itemize}
\item Original quantities: 
In Subsection \ref{firstblock}, we estimate $\Gamma_{xy}^*$,
$\Pi_{x}^{-}$, and $\Pi_{x}$, assuming control of $Q\Pi_{x}^{-}$. 
\item Expectation: 
In Subsection \ref{BPHZMall}, we show that
the BPHZ-choice of renormalization gives control of $\mathbb{E}Q\Pi_{x}^{-}$. 
By the SG inequality, this gives control of $\Pi_{x}^{-}$, 
assuming control of the (directional) Malliavin derivative $Q\delta\Pi_{x}^{-}$.
\item Malliavin derivatives: 
In Subsection \ref{secondblock},
we estimate the Malliavin derivatives $Q\delta\Gamma_{xy}^*P$ and $Q\delta\Pi_{x}$, 
assuming control of $Q\delta\Pi_{x}^{-}$.
\item Modelled distribution: 
In Subsection \ref{thirdblock}, we introduce ${\rm d}\Gamma^*_{xz}$ 
$\in{\rm Hom}(\mathsf{T}^*,\mathsf{\tilde T}^*)$, of which we show 
that it is continuous in\footnote{what we call the secondary base point} $z$
in the sense of controlling $Q({\rm d}\Gamma^*_{xy}-{\rm d}\Gamma^*_{xz}\Gamma_{zy}^*)PQ$,
that it describes $Q\delta\Pi^{-}_x$ in terms of $Q\Pi_z^{-}$ in the sense of controlling 
$Q(\delta\Pi^{-}_x-{\rm d}\Gamma^*_{xz}Q\Pi^{-}_z)$,
and that it describes $Q\delta\Pi_x$ in terms of $Q\Pi_z$ in the sense of controlling  
the rough-path increments $Q(\delta\Pi_x-\delta\Pi_x(z)-{\rm d}\Gamma^*_{xz}Q\Pi_z)$.
This subsection is the core of the proof.
\item Back to the estimate of $Q\delta\Pi^{-}_x$ itself.
In Subsection \ref{forthblock}, we provide control of $Q{\rm d}\Gamma^*_{xz}P$
and then of $Q\delta\Pi^{-}_x$, assuming control of
$Q(\delta\Pi^{-}_x-{\rm d}\Gamma^*_{xz}Q\Pi^{-}_z)$.
\end{itemize}


\subsection{The four types of tasks in a loop: algebraic argument, reconstruction, integration, three-point argument}\label{sec:tasks}

Subsections \ref{firstblock}, \ref{secondblock}, \ref{thirdblock}, and,
to some extent, Subsection \ref{forthblock} and Section~\ref{sec:mall_approx} involve the same type of tasks, arranged in a 
similar type of loop (see Figure \ref{fig:four_tasks}). 
An important role is played by the $\pi_{xy}^{({\bf n)}}$'s that determine the $\Gamma_{xy}^*$
via the exponential formula (\ref{exp}), and their counterparts ${\rm d}\pi_{xz}^{({\bf n})}$ for
${\rm d}\Gamma_{xz}^*$, see (\ref{eq:def_dGamma}). The four tasks are:
\begin{itemize}
\item Algebraic argument. All four subsections 
start\footnote{and for Subsection \ref{firstblock} ends} 
with an ``algebraic argument'' 
(called like this because it is based on an exponential-type formula)
to estimate $\Gamma^*_{xy}P$, $Q\delta\Gamma^*_{xy}P$, 
$Q({\rm d}\Gamma^*_{xy}-{\rm d}\Gamma^*_{xz}\Gamma^*_{zy})PQ$, and $Q{\rm d}\Gamma^*_{xz}P$, 
in terms of $\pi_{xy}^{({\bf n})}$, $Q\delta\pi_{xy}^{({\bf n})}$, 
$Q({\rm d}\pi_{xy}^{({\bf n})}-{\rm d}\pi_{xz}^{({\bf n})}
-{\rm d}\Gamma_{xz}^*\pi_{zy}^{({\bf n})})$, and $Q{\rm d}\pi^{({\bf n})}_{xz}$,
respectively. 
\item Reconstruction. Subsections \ref{firstblock} and \ref{thirdblock} feature 
a reconstruction argument
in order to control $({\rm id}-Q)\Pi_x^{-}$ 
and $Q(\delta\Pi_x^{-}-{\rm d}\Gamma_{xz}^*Q\Pi_z^{-})$.
By a reconstruction argument, we understand that for a family $\{F_x\}_{x\in\mathbb{R}^2}$ 
of Schwartz distributions on $\mathbb{R}^2$ that satisfy a continuity condition in 
the base point $x$, 
we estimate $F_{xt}(x)$ in terms of the diagonal\footnote{
in the sense of $\lim_{t\downarrow0} F_{xt}(x)$ with the understanding that this limit exists} 
$F_{x}(x)$.
\item Integration. Subsections \ref{firstblock}, \ref{secondblock}, \ref{thirdblock} involve
an integration argument to pass from $\Pi_x^{-}$, $Q\delta\Pi_x^{-}$, 
and $Q(\delta\Pi_x^{-}-{\rm d}\Gamma_{xz}^*Q\Pi_z^{-})$,
to $\Pi_x$, $Q\delta\Pi_x$, and $Q(\delta\Pi_x-\delta\Pi_x(z)-{\rm d}\Gamma_{xz}^*Q\Pi_z)$,
respectively. By an integration argument, we mean that we pass
an annealed H\"older norm anchored in a base point $x$ through an integral representation.
It amounts to a Schauder estimate.
\item Three-point argument.
All four subsections appeal to a ``three-point argument'' (called like this because
it involves varying an additional third point in order to control polynomial coefficients)
to pass from the estimate
of $\Pi_x$, $Q\delta\Pi_x$, $Q(\delta\Pi_x-\delta\Pi_x(z)-{\rm d}\Gamma_{xz}^*Q\Pi_z)$, 
and $Q{\rm d}\Gamma_{xz}^*P$,
to the estimate of
$\pi_{xy}^{({\bf n})}$, $Q\delta\pi_{xy}^{({\bf n})}$, 
$Q({\rm d}\pi_{xy}^{({\bf n})}-{\rm d}\pi_{xz}^{({\bf n})}
-{\rm d}\Gamma_{xz}^*\pi_{zy}^{({\bf n})})$, and $Q{\rm d}\pi^{({\bf n})}_{xz}$,
respectively.
\end{itemize}


\subsection{The logical order of loops and tasks in one induction step}\label{sec:order}

In the course of Sections \ref{sec:anal_est} and \ref{sec:model_construction}, 
we will add a fairly large number of auxiliary statements.
Some have to be logically included into the induction statement,
because we refer to them as part of the induction hypothesis.
For the convenience of the reader, we list all statements of the induction hypothesis here:
\begin{itemize}
		\item The transitivity of $\Gamma_{xy}^*$ \eqref{ao15} and $\pi_{xy}^{(\mathbf{n})}$ \eqref{rec01}, and the recentering of $\Pi^-_x$ \eqref{eq:recenter_Pi_minus} and $\Pi_x$ \eqref{eq:recenter_Pi_specific},
		\item the estimates of the main objects $\Pi_x$ \eqref{eq:pi_generic}, $\Gamma_{xy}^*$ \eqref{eq:gamma}, $\pi_{xy}^{(\mathbf{n})}$ \eqref{eq:pi_n}, and $\Pi_x^-$ \eqref{eq:pi_minus_generic},
		\item the boundedness of Malliavin derivatives $\delta \Pi_x^-$ \eqref{eq:mal_dual_annealed}, $\delta \pi_{xy}^{(\mathbf{n})}$ \eqref{eq:delta_pi_n}, $\delta\Gamma_{xy}^*$ \eqref{eq:delta_gamma}, and $\delta \Pi_x$ \eqref{eq:delta_pi_generic}, and the modeledness of the latter \eqref{eq:delta_pi_incr_generic},
		\item the boundedness \eqref{eq:d_pi}, \eqref{eq:form_bound}, and continuity \eqref{eq:delta_pi_d_pi_incr}, \eqref{eq:form_cont} of the modelled distribution ${\rm d}\Gamma_{xz}^*$,
		\item the divergent bounds of $c$ \eqref{cw03}, $\partial^\mathbf{m} \Pi_x$ \eqref{cw60}, \eqref{cw60_cont}, \eqref{cw63_generic}, $\Pi^-_x$ \eqref{cw60_minus_cont}, and $\partial^\mathbf{m} \delta \Pi_x$ \eqref{cw60_mal_dual}, \eqref{cw60_cont_mal_dual}, \eqref{cw63_generic_mal_dual},
		\item symmetries, i.~e. shift \eqref{eq:Pi_shift} and reflection \eqref{eq:Pi_parity} covariances of $\Pi_x$,
		\item and the Malliavin differentiability of $\Pi_x$, $\partial_1^2 \Pi_x$, $\partial_2 \Pi_x$, $\Gamma_{xy}^*$, and $\pi_{xy}^{(\mathbf{n})}$. 
\end{itemize}
The remaining additional statements are just used inside one induction step and are
therefore not listed above, for example:
the estimates (\ref{eq:limit}) and (\ref{eq:expect_est}) on the expectation,
the rough-path estimates on $\delta\Pi^{-}$ (\ref{eq:delta_pi_minus_generic}),
the shift and reflection covariances (\ref{eq:Pi_minus_shift}) and
(\ref{eq:Pi_minus_parity}) on the level of $\Pi^{-}$ and Malliavin differentiability of the singular components 
of $\Pi^-$, see Item~\ref{it:1} in Section~\ref{sec:mall_approx} .

\medskip

The logical order of one induction step depends on whether $\beta$ is regular or
singular. For regular $\beta$, we just run through the first Subsection 
\ref{firstblock} (but still most of the construction):
\begin{enumerate}
		\item By the induction hypothesis, $\Gamma_{xy}^* P$ is constructed and estimated via an algebraic argument, see Proposition \ref{prop:gamma_npp}.
		\item Because of $({\rm id} - Q)c = 0$, we construct $\Pi_x^-$ and show
		\begin{enumerate}[a)]
			\item its continuity, see \eqref{cw60_minus_cont} in Proposition \ref{rem:1},
			\item its shift and reflection covariance, see the first part of Proposition \ref{prop:BPHZ},
			\item and the recentering property \eqref{eq:recenter_Pi_minus_specific}, see Proposition \ref{prop:change_of_basepoint}.
		\end{enumerate}
		\item We then estimate $\Pi_x^-$ by regular reconstruction, see Proposition \ref{prop:reconstr_reg}.
		\item We then may 
		\begin{enumerate}[a)]
			\item construct and estimate $\Pi_x$ in an integration step, see Proposition \ref{prop:int_pi_minus},
			\item and show shift and reflection covariance, see Subsection \ref{sec:Pi_construct}.
		\end{enumerate}
		\item Next we construct $\pi_{xy}^{(\mathbf{n})}$, and thus the full $\Gamma_{xy}^*$,
		\begin{enumerate}[a)]
			\item obtain the recentering property \eqref{eq:recenter_Pi_specific}, see Proposition \ref{prop:change_of_basepointII},
			\item  establish \eqref{rec01}, and therefore \eqref{ao15}, see Proposition \ref{prop:change_of_basepointIII},
			\item and estimate $\pi_{xy}^{(\mathbf{n})}$ via a three-point argument, see Proposition \ref{prop:pi_n_three_point}, and consequently $\Gamma_{xy}^*$, see Proposition \ref{prop:gamma_pp}.
		\end{enumerate}
		\item We conclude showing boundedness and continuity of $\partial_1^2 \Pi_x$ and $\partial_2 \Pi_x$, see \eqref{cw60} and \eqref{cw60_cont} in Proposition \ref{rem:1}.		 
	\end{enumerate}
The logical order is much more complex for singular $\beta$:
\begin{enumerate}
		\item Algebraic arguments. By induction hypothesis:
		\begin{enumerate}[a)]
			\item $\Gamma_{xy}^* P$ is constructed and estimated, see Proposition \ref{prop:gamma_npp}.
			\item We establish the Malliavin differentiability of $\Gamma_{xy}^* P$, see Subsection \ref{sec:gamma_three_point_approx}.
			\item This allows to estimate $Q \delta\Gamma_{xy}^* P$, see Proposition \ref{prop:delta_gamma_npp},
			\item $Q {\rm d} \Gamma_{xy}^* P$, see Proposition \ref{prop:form_bound},
			\item and $Q ({\rm d}\Gamma_{xy}^* - {\rm d}\Gamma_{xz}^* \Gamma_{zy}^*) PQ$, see Proposition \ref{prop:algIII}.
		\end{enumerate}
		\item Next we establish some properties which are independent of the specific value of $c$, and thus can be shown before the BPHZ choice of renormalization constants. In particular, we show that for any $c$ satisfying \eqref{eq:c_pop_cond},
		\begin{enumerate}[a)]
			\item $Q\Pi_x^-$ satisfies the recentering property \eqref{eq:recenter_Pi_minus_specific}, see Proposition \ref{prop:change_of_basepoint},
			\item $Q\Pi_x^-$ is shift and reflection covariant, see Subsection \ref{sec:Pi_minus_constr},
			\item and $\frac{d}{dt} \E Q\Pi_{xt}^-$ is estimated in Proposition \ref{prop:limit}.
		\end{enumerate}
		\item We now choose the BPHZ renormalization constant $c$ and estimate $\E \Pi_x^-$.
		\begin{enumerate}[a)]
			\item We show that $\lim_{t\uparrow \infty} \E \Pi^-_{xt}(x)$ exists and therefore we can choose $c$ so that \eqref{eq:BPHZ} holds, see Proposition \ref{prop:BPHZ},
			\item and show the divergent bound \eqref{cw03}, see Step 2 of the proof of Proposition \ref{rem:1}.
			\item We then estimate $\E Q\Pi_x^-$, see Proposition~\ref{prop:expect}.
		\end{enumerate}
		\item Next we study the Malliavin derivative of $\Pi_x^-$.
		\begin{enumerate}[a)]
			\item By the induction hypothesis, we show Malliavin differentiability of $Q\Pi_x^-$, see Subsection~\ref{sec:reconstr_approx}.
			\item We show continuity of $Q \delta \Pi_x^-$, see \eqref{cw60_minus_cont_mal} in Proposition \ref{rem:1_mal}.
			\item Next we estimate $Q(\delta \Pi_x^- - {\rm d}\Gamma_{xz}^* Q \Pi_z^-)$, see Proposition \ref{prop:recIII},
			\item and finally estimate $Q \delta \Pi_x^-$ itself, see Proposition \ref{prop:delta_pi_minus_est}.
		\end{enumerate}
		\item Equipped with the estimates of  $\E Q\Pi_x^-$ and $Q \delta \Pi_x^-$, we use the SG inequality to gain control of $Q \Pi_x^-$. 
		\item At this stage we may show the results of our main theorem.
		\begin{enumerate}[a)]
			\item We construct and estimate $\Pi_x$, see Proposition \ref{prop:int_pi_minus},
			\item and show its shift and reflection covariance, see Subsection \ref{sec:Pi_construct}. 
			\item Next we construct $\pi_{xy}^{(\mathbf{n})}$ and thus the full $\Gamma_{xy}^*$, obtaining the recentering property \eqref{eq:recenter_Pi_specific}, see Proposition \ref{prop:change_of_basepointII},
			\item establish the recentering property \eqref{rec01}, and therefore \eqref{ao15}, see Proposition \ref{prop:change_of_basepointIII},
			\item and finally estimate $\pi_{xy}^{(\mathbf{n})}$, see Proposition \ref{prop:pi_n_three_point}, and consequently $\Gamma_{xy}^*$, see Proposition \ref{prop:gamma_pp}.
		\end{enumerate}
		\item Next we turn to the Malliavin derivatives of our main objects.
		\begin{enumerate}[a)]
			\item We establish Malliavin differentiability of $\Pi_x$, see Subsection \ref{sec:int_approx},
			\item and estimate  $Q\delta\Pi_x$, see Proposition \ref{prop:intII}.
			\item We show that $\pi_{xy}^{(\mathbf{n})}$ is Malliavin differentiable,  see Subsection \ref{sec:gamma_three_point_approx}, and so is $\Gamma_{xy}^*$,
			\item and estimate $\delta \pi_{xy}^{(\mathbf{n})}$, see Proposition \ref{prop:three_pointII}, which in turn implies the estimate of the full $\delta \Gamma_{xy}^*$.
		\end{enumerate}
		\item Next we show the bounds divergent in the ultra-violet cut-off.
		\begin{enumerate}[a)]
			\item We first show the continuity of $\Pi_x^-$, see \eqref{cw60_minus_cont} in Proposition \ref{rem:1}.
			\item Next we establish boundedness and continuity of $\partial_1^2 \Pi_x$ and $\partial_2 \Pi_x$, see \eqref{cw60} and \eqref{cw60_cont} in Proposition \ref{rem:1}.
			\item Finally we obtain boundedness and continuity of $\partial_1^2 \delta \Pi_x$ and $\partial_2 \delta \Pi_x$, see \eqref{cw60_mal_dual} and \eqref{cw60_cont_mal_dual} in Proposition \ref{rem:1_mal}.
		\end{enumerate}
		\item Finally we establish the modeledness estimates.
		\begin{enumerate}[a)]
			\item We construct ${\rm d}\pi_{xy}^{(1,0)}$ and ${\rm d}\Gamma_{xy}^*$, see Subsection \ref{sec:dGamma_construction}.
			\item We estimate $Q (\delta \Pi_x - \delta \Pi_x(z) - {\rm d}\Gamma_{xz}^* Q \Pi_z)$, see Proposition \ref{prop:intIII},
			\item $Q({\rm d}\pi_{xy}^{(\mathbf{n})} - {\rm d} \pi_{xz}^{(\mathbf{n})} - {\rm d}\Gamma_{xz}^* \pi_{zy}^{(\mathbf{n})})$, see Proposition \ref{prop:delta_pi_incr_three_point}, which in turn implies the continuity estimate of the full ${\rm d}\Gamma_{xz}^*$,
			\item and ${\rm d} \pi_{xz}^{(1,0)}$, see Proposition \ref{prop:d_pi}, which in turn implies the boundedness estimate of the full ${\rm d}\Gamma_{xz}^*$.
		\end{enumerate}

	\end{enumerate}	
The logical order of the core estimates, from Proposition~\ref{prop:gamma_npp} to
Proposition \ref{prop:delta_pi_minus_est}, corresponding to the five loops of an induction step
is indicated in Figure \ref{fig:four_tasks} by the small number in the lower right corner of each field.
A detailed overview of the logical order in the regular and singular case including precise input and output of each statement is given in Figure~\ref{fig:regular} and Figure~\ref{fig:singular}, respectively; see Appendix \ref{app:tables}.

\begin{figure}[H]
\small
\begin{center}
\begin{tabular}{|c|c|@{\hspace{0.2ex}}c@{\hspace{0.2ex}}|@{\hspace{0.2ex}}c@{\hspace{0.2ex}}|c|}
\hline 
{\bf Original quantities} & {\bf Expectation} & {\bf Malliavin derivatives} & {\bf Path increments} & {\bf Back to $\delta\Pi^-_{x}$}  \\ 
\shortstack{~\\$\Gamma^*_{xy},\,\Pi^-_x,\,\Pi_x,\,\pi^{(\n)}_{xy}$\\[1.3ex] } & \shortstack{~\\$\mathbb{E}\Pi_x^-$\\[1.3ex]} & \shortstack{$\delta\Gamma^*_{xy},\,\delta\Pi^-_x,\,\delta\Pi_x,$\\[-0.5ex]$\delta\pi^{(\n)}_{xy}$\\[0.ex] } & {\small\shortstack{${\rm d}\Gamma^*_{xy}-{\rm d}\Gamma^*_{xz}\Gamma^*_{zy},$\\ $\delta\Pi^-_x-{\rm d}\Gamma^*_{xz}\Pi^-_z,$\\ $\delta\Pi_x-\delta\Pi_x(z)-{\rm d}\Gamma^*_{xz}\Pi_z$}} & {\small\shortstack{${\rm d}\Gamma^*_{xy},\,{\rm d}\pi^{(\n)}_{xy},$ \\ and averaging\\[0.ex]}}  \\ 
Subsection \ref{firstblock} & Subsection \ref{BPHZMall} & Subsection \ref{secondblock} & Subsection \ref{thirdblock} & Subsection \ref{forthblock} \\
\hhline{|=====|}
Algebraic arg. I\,(i) & & Algebraic arg. II & Algebraic arg. III & Algebraic arg. IV \\ 
\multicolumn{1}{|r@{\hspace{0.5ex}}|}{Proposition \ref{prop:gamma_npp} \hspace{0.4ex} {\scriptsize 1}} & &
\multicolumn{1}{r@{\hspace{0.5ex}}|@{\hspace{.2ex}}}{Proposition \ref{prop:delta_gamma_npp}\hspace{1.3ex} {\scriptsize 2}} & 
\multicolumn{1}{r@{\hspace{0.5ex}}|}{Proposition \ref{prop:algIII}\hspace{1.3ex} {\scriptsize 3}} & 
\multicolumn{1}{r@{\hspace{0.5ex}}|}{Proposition \ref{prop:form_bound}\hspace{0.8ex} {\scriptsize 4}} \\
Algebraic arg. I\,(ii) & & & & \\
\multicolumn{1}{|r@{\hspace{0.5ex}}|}{Proposition \ref{prop:gamma_pp} \hspace{-0.4ex} {\scriptsize 11}} & & & & \\
\hline 
Reconstruction I & & & Reconstruction III &  \\ 
Proposition \ref{prop:reconstr_reg} & & & 
\multicolumn{1}{r@{\hspace{0.5ex}}|}{Proposition \ref{prop:recIII}\hspace{1.3ex} {\scriptsize 7}} & \\
\hline 
Integration I & & Integration II & Integration III & \\ 
\multicolumn{1}{|r@{\hspace{0.5ex}}|}{Proposition \ref{prop:int_pi_minus}\hspace{1.5ex} {\scriptsize 9}} & &
\multicolumn{1}{r@{\hspace{0.5ex}}|@{\hspace{.2ex}}}{Proposition \ref{prop:intII}\hspace{0.5ex} {\scriptsize 12}} & 
\multicolumn{1}{r@{\hspace{0.5ex}}|}{Proposition \ref{prop:intIII}\hspace{0.5ex} {\scriptsize 14}} & \\
\hline 
Three-point arg. I & & Three-point arg. II & Three-point arg. III & Three-point arg. IV \\ 
\multicolumn{1}{|r@{\hspace{0.5ex}}|}{Proposition \ref{prop:pi_n_three_point}\hspace{0.5ex} {\scriptsize 10}} & &
\multicolumn{1}{r@{\hspace{0.5ex}}|@{\hspace{.2ex}}}{Proposition \ref{prop:three_pointII}\hspace{0ex} {\scriptsize 13}} & 
\multicolumn{1}{r@{\hspace{0.5ex}}|}{Proposition \ref{prop:delta_pi_incr_three_point}\hspace{0.5ex} {\scriptsize 15}} & 
\multicolumn{1}{r@{\hspace{0.5ex}}|}{Proposition \ref{prop:d_pi}\hspace{0.1ex} {\scriptsize 16}} \\
\hline 
 & \multicolumn{1}{r@{\hspace{0.5ex}}|@{\hspace{0.2ex}}}{Proposition \ref{prop:limit}\hspace{0.2ex} {\scriptsize 5}} 
  & & & Averaging  \\ 
 & \multicolumn{1}{r@{\hspace{0.5ex}}|@{\hspace{0.2ex}}}{Proposition \ref{prop:expect}\hspace{0.2ex} {\scriptsize 6}} & & & \multicolumn{1}{r@{\hspace{0.5ex}}|}{
Proposition \ref{prop:delta_pi_minus_est} \hspace{0.2ex} {\scriptsize 8} } \\
\hline 
\end{tabular} 
\end{center}
\caption{Columns correspond to the five loops of an induction step, 
cf. Subsection \ref{sec:loop}, rows correspond the four types of tasks in a loop, 
cf. Subsection \ref{sec:tasks}. 
The numbers in the lower right corner indicate the logical order 
within an induction step for singular multi-indices, 
cf. Subsection \ref{sec:order}. The minor tasks of Malliavin differentiability and
divergent bounds in the ultra-violet cut-off are not included.} \label{fig:four_tasks}
\end{figure}


\subsection{The ordering relation \texorpdfstring{$\prec$}{prec} for the induction}\label{induct}

At first sight, one would hope for an induction in the homogeneity $|\beta|$; 
the set $\mathsf{A}$ of homogeneities, being locally finite and bounded from below,
lends itself to an induction argument. 
However, this is not possible because 
as opposed to $\Gamma_{xy}^*$, or the structurally closer $\delta\Gamma_{xy}^*$, 
${\rm d}\Gamma_{xz}^*$ is {\it not} triangular w.~r.~t.~$|\cdot|$. 
As we shall see in Subsection \ref{thirdblock}, this lies in the nature of ${\rm d}\Gamma_{xz}^*$:
$\delta\Pi_{x\beta}$ is modelled (almost) to order $\frac{3}{2}+\alpha$, independently of $\beta$.
Here is a simple example for the failure of triangularity: 
On the one hand we have\footnote{as a consequence of the definition \eqref{eq:def_dGamma} 
of ${\rm d}\Gamma_{xz}$, and using ${\rm d}\pi^{(1,0)}_{xz0}$ $=\partial_1\delta\Pi_{x0}(z)$, 
as a consequence of the definition (\ref{dpi}) of ${\rm d}\pi^{(1,0)}_{xz0}$
and the triangular structure \eqref{eq:triangular_d_Gamma} of ${\rm d}\Gamma^*_{xz}$ 
w.~r.~t.~$\prec$}
$({\rm d}\Gamma^*_{xz})_{e_1}^{e_1+e_{(1,0)}}$ $=\partial_1\delta\Pi_{x0}(z)$,
which does not vanish for generic $z$, 
while on the other hand, $|e_1|=2\alpha<\alpha+1=|e_1+e_{(1,0)}|$.
Even block triangularity with respect to the threshold homogeneity $2$ fails:
$({\rm d}\Gamma^*_{xz})_{2e_1+e_{(1,0)}}^{2e_1+2e_{(1,0)}}$ 
$=2\partial_1\delta\Pi_{x0}(z)\neq0$, while $|2e_1+e_{(1,0)}|$ $=2\alpha+1$
$<2+\alpha$ $=|2e_1+2e_{(1,0)}|$.
This however does not create any problems in the induction: 
Since ${\rm d}\Gamma^*$ is never applied to the objects $\delta\Pi$, $\delta\Pi^-$, $\delta\pi^{(\n)}$, ${\rm d}\pi^{(\n)}$ and $\delta\Gamma^*$, 
we will never appeal to statements involving them for multi-indices $|\beta|\geq2$.

\medskip

Fortunately, it turns out that ${\rm d}\Gamma_{xz}^*$ does have a triangular structure
w.~r.~t.~an ordering that involves $[\beta]$, cf.~(\ref{ao09}), as a ``first digit'' and
$|\beta|_p$, cf.~(\ref{ao12}), as a second digit,
in the spirit of \cite[Subsection 3.5]{LOT21}. In the reconstruction argument,
based on the structure of the term $\sum_{k\ge 0}\mathsf{z}_k\Pi_x^k\partial_1^2\Pi_x$
(or rather its Malliavin derivative), we need a finer ordering, which involves
the component $\beta(0)$ (to which the two other digits are oblivious) as a third digit.
We shall argue in Section \ref{sec:algebraic_aspects} 
that at least the triangular effect of this ordering can be captured by the following ordinal
\begin{align}\label{ord01}
|\beta|_{\prec}:= [\beta] + \frac{1}{2}|\beta|_p+\frac{1}{4}\beta(0),
\end{align}
where the weights in this combination are fixed for convenience but could be
replaced by any three strictly ordered positive numbers.

\medskip

For notational convenience, we define
\begin{align}\label{eq:recIII26}
\beta'\prec\beta\quad&\Leftrightarrow\quad|\beta'|_{\prec}<|\beta|_{\prec},\\
\beta'\preccurlyeq\beta\quad&\Leftrightarrow\quad(\beta'\prec\beta\;\;\mbox{or}\;\;\beta'=\beta).
\end{align}
A further benefit of (\ref{ord01}) is that $|\cdot|_{\prec}$ is coercive, meaning that
the set $\{\,\beta\,|\,|\beta|_{\prec}\le M\,\}$ is finite for every finite $M$
(which would not be true when $|\cdot|_{\prec}$ is replaced by $|\cdot|$ because
both $[\cdot]$ and $|\cdot|_p$ are oblivious to the $\beta(0)$ component).
This is crucial in the induction since at every step, the stochastic integrability
deteriorates due to the unavoidable use of H\"older's inequality in probability
when estimating products of random variables. 

\medskip

While in general, we only\footnote{consider $\beta=e_{(1,0)}$} have
\begin{align}\label{eq:infprec}
[\beta]\ge -1\quad\mbox{and}\quad|\beta|_{\prec}\ge-\frac{1}{2},
\end{align}
it follows from (\ref{ao09}) that
$|\beta|_{\prec}\ge 0$ for all $\beta$ not purely polynomial with equality iff $\beta=0$.
In view of (\ref{eq:recIII26}) and the additivity of $|\cdot|_{\prec}$
this implies compatibility of $\prec$ and summation:
\begin{align}\label{eq:recIII25}
\beta_1\preccurlyeq\beta_1+\beta_2\quad\mbox{provided}\;\beta_2\;
\mbox{is not purely polynomial},
\end{align}
which we shall repeatedly use. In fact, the ordering is de facto irrelevant on the purely
polynomial $\beta$'s, which have been treated.
Among the non-purely polynomial $\beta$'s, the base case is given by $\beta=0$.

\medskip

While $|\cdot|_{\prec}$ is additive (but negative on some purely polynomial indices),
the homogeneity $|\cdot|$ is not (but it is strictly positive, in fact, $\ge\alpha$). We will
often use that
\begin{align}\label{eq:hom_add}
|\cdot|-\alpha\stackrel{\eqref{ao12}}{=}
\alpha[\cdot]+|\cdot|_p\quad\mbox{is additive and non-negative}
\end{align}
on all populated multi-indices.

\medskip

In order to make the value of the multi-index explicit when referring to a
statement like \eqref{eq:pi_generic}, we write $\eqref{eq:pi_generic}_{\beta}$
with the understanding that we refer to the corresponding statement
for the multi-index $\beta$.
When a statement involves two multi-indices like \eqref{eq:gamma},
we write $\eqref{eq:gamma}_{\beta}^\gamma$ when we want to specify also the second multi-index, and $\eqref{eq:gamma}_{\beta}^{\gamma \neq \pp}$ when we only mean it for $\gamma$'s which are not purely polynomial.
All statements of the induction hypothesis will be implicitly assumed to hold
for every integrability exponent $p<\infty$, for every space-time points $x,y,z\in\mathbb{R}^d$,
and every convolution parameter $t\in(0,\infty)$, if applicable.
For example, when we state $\eqref{eq:pi_generic}_{\prec\beta}$,
we mean the estimate for every $p<\infty$ and $x,y\in\mathbb{R}^2$
and for all multi-indices $\beta'\prec\beta$.


\subsection{The base case \texorpdfstring{$\beta=0$}{beta = 0}}

In fact, the argument for the base case w.~r.~t.~the ordering $\prec$, which
reduces to $\beta=0$, is contained in the argument for the induction step,
as we shall explain now, referring to the logical order of the induction 
in the singular case outlined in Subsection \ref{sec:order}.

\medskip

First note that, due to the triangularity properties \eqref{eq:triangular_Gamma}, \eqref{eq:triangular_delta_Gamma}, \eqref{eq:triangular_d_Gamma} and \eqref{eq:dGamma_inc_triangular}, all estimates of Item 1 (namely $\eqref{eq:gamma}_0^{\gamma \neq \pp}$, $\eqref{eq:delta_gamma}_0^{\gamma \neq \pp}$, $\eqref{eq:form_bound}_0^{\gamma \neq \pp}$ and $\eqref{eq:form_cont}_0^{\gamma \neq \pp}$) are void for $\beta = 0$.
	
\medskip
	
In Item 2, the recentering property \eqref{eq:recenter_Pi_minus_specific} is trivial for $\xi_\tau + c_0$ for any $c_0$ $\in$ $\R$, and so is the shift $\eqref{eq:Pi_minus_shift}_0$ and reflection $\eqref{eq:Pi_minus_parity}_0$ covariance. By stationarity, which is contained in Assumption \ref{ass:spectral_gap}, $\E \xi_t (y)$ is constant in $t$ and $y$, and by centeredness (also in Assumption \ref{ass:spectral_gap}) it is equal to $0$; in particular, estimate $\eqref{eq:expect_est}_0$ is trivially satisfied. We choose $c_0 = 0$, so that $\eqref{eq:BPHZ}_0$ holds (Item 3). Note that this choice is consistent with the population condition \eqref{eq:c_pop_cond} and makes the estimate $\eqref{cw03}_0$ hold trivially.

\medskip

Next we turn to Item 4, i.~e. the Malliavin derivative $\delta \xi_\tau$ (see Subsection \ref{sec:base_case_approx} for Malliavin differentiability). Continuity $\eqref{cw60_minus_cont_mal}_0$ is contained in Step~1 of the proof of Proposition~\ref{rem:1_mal},
	modeledness $\eqref{eq:delta_pi_minus_generic}_0$ is established in the proof of Proposition~\ref{prop:recIII} in form of \eqref{eq:base1},
	and boundedness $\eqref{eq:mal_dual_annealed}_0$ is a consequence of modeledness and contained in the proof of Proposition~\ref{prop:delta_pi_minus_est}. The combination of $\E\xi_t(y) = 0$ with $\eqref{eq:mal_dual_annealed}_0$ via the SG inequality yields $\eqref{eq:pi_minus_generic}_0$ (Item 5), which takes the form
	\begin{equation*}
		\E^\frac{1}{p} |\xi_t(y)|^p \lesssim (\sqrt[4]{t})^{\alpha - 2}.
	\end{equation*}
	Moreover, the divergent bound $\eqref{cw60_minus_cont}_0$, i.~e. the continuity of $\xi_\tau$ (Item 8.~(a)), is shown in Step 1 of the proof of Proposition \ref{rem:1}, cf.~Subsection \ref{sec:proof_div} and in particular \eqref{cw06}.
	
	\medskip
	
	Equipped with the estimates of $\xi$ and $\delta \xi$, and the triangularity properties, the rest of the base case follows from the same procedure as in any induction step, namely Items 6 to 9.

\medskip

\section{Estimates} \label{sec:anal_est}

In this section, we establish the stochastic estimates 
(\ref{eq:pi_generic}), (\ref{eq:gamma}), \eqref{eq:pi_n} and (\ref{eq:pi_minus_generic}) for a fixed non-purely 
polynomial multi-index $\beta$.


\subsection{Semi-group convolution}\label{sec:semi_group}
Following \cite[Section 2]{OW19}, we use the space-time elliptic operator
$\partial_1^4-\partial_2^2$ to introduce the family $\{(\cdot)_t\}_{t\in(0,\infty)}$
of convolution operators that respect the parabolic scaling and satisfy the
semi-group property
\begin{align}\label{eq:semi_group}
f_{t+s}=(f_t)_s
\end{align}
convenient for the dyadic nature of reconstruction arguments.
It is given by the convolution with the Schwartz kernel $\psi_t$ defined through
\begin{align}\label{eq:kernel}
\partial_t\psi_t+(\partial_1^4-\partial_2^2)\psi_t=0\quad\mbox{and}\quad
\psi_{t=0}=\mbox{Dirac at origin}.
\end{align}
Note the scaling $\psi_t(x)$ $=(\sqrt[4]{t})^{-3}\,\psi_{t=1}(\frac{x_1}{\sqrt[4]{t}},
\frac{x_2}{\sqrt[2]{t}})$, so that the $x_1$-scale is $\sqrt[4]{t}$,
which explains the appearance of $\sqrt[4]{\tau}$ in Proposition~\ref{rem:1}.  
One reads off (\ref{eq:kernel}) that the Fourier transform of $\psi_{t=1}$ is given by
the Schwartz function $\exp(-k_1^4-k_2^2)$, so that $\psi_{t=1}$ is a Schwartz function
itself. In view of the scaling of $\psi_t$ in terms of $t$, this implies the moment bound
\begin{align}\label{cw61}
\int dz|\partial^{\bf n}\psi_t(z-y)|(\sqrt[4]{t}+|y-z|+|z-x|)^{\theta}\lesssim(\sqrt[4]{t})^{-|{\bf n}|}
(\sqrt[4]{t}+|y-x|)^{\theta},
\end{align}
where we recall that $|z-y|$ is the anisotropic distance function (\ref{ao18}).
We will need that (\ref{cw61}) holds for all $\theta>-3$.

\medskip

Finally, because of the factorization
\begin{align}\label{eq:factorization}
\partial_1^4-\partial_2^2=-(\partial_2-\partial_1^2)(\partial_2+\partial_1^2),
\end{align}
the convolution $(\cdot)_t$ will also be helpful for integration by
providing the kernel representation in form of
\begin{align}\label{eq:representation}
(\partial_2-\partial_1^2)^{-1}f=-\int_0^\infty dt (\partial_2+\partial_1^2)f_t.
\end{align}
%


\subsection{Estimate of the original quantities \texorpdfstring{$\Gamma_{xy}^*$, $\Pi^{-}_x$, $\Pi_x$, and $\pi^{(\bf n)}_{xy}$}{Gamma* xy, Pi minus x, Pi x, and pi n xy}}
\label{firstblock}

The first task is estimate $\eqref{eq:gamma}$ on $\Gamma_{xy}^*$,
based on the exponential formula \eqref{exp}, estimate \eqref{eq:pi_n} of $\pi_{xy}^{(\n)}$ and the population constraint (\ref{eq:pi_n_pop}).
We split this first task into a first half, where we
treat $\Gamma_{xy}^*P$, see Proposition \ref{prop:gamma_npp},
and a second half, where we tackle the full $\Gamma_{xy}^*$, see Proposition \ref{prop:gamma_pp}.
The reason for this splitting is that according to (\ref{ord12}), 
provided $\gamma$ is not purely polynomial, the matrix entry
$(\Gamma_{xy}^*)_{\beta}^\gamma$ depends on $\pi_{xy}^{({\bf n})}$
only through $\pi_{xy\beta'}^{({\bf n})}$ with $\beta'\prec\beta$.
Following the elementary proof of \cite[Lemma 4.3]{LO22}
we obtain from H\"older's inequality in probability:

\begin{proposition}[Algebraic argument I, first half]\label{prop:gamma_npp} Assume that 
$\eqref{eq:pi_n}_{\prec\beta}$ holds. Then 
$\eqref{eq:gamma}_{\beta}^\gamma$ holds for all $\gamma$ not purely polynomial.
\end{proposition}

The second task is the estimate $\eqref{eq:pi_minus_generic}$ of $\Pi_{x}^{-}$,
based on the output of Proposition \ref{prop:gamma_npp}.
Note that by (\ref{new}), definition (\ref{eq:Pi_minus_def_alt}),
when evaluated at the base point $x$, collapses to
\begin{align}\label{eq:Pi_minus_anchor}
\Pi_{x}^{-}(x)=P\mathsf{z}_0\partial_1^2\Pi_x(x)-c+\xi_\tau(x)\mathsf{1}.
\end{align}
From \eqref{eq:pi_generic} we obtain with help
of the semi-group property \eqref{eq:semi_group} and the moment bounds \eqref{cw61}
\begin{equation}\label{mt98}
\mathbb{E}^{\frac{1}{p}}|\partial^\n\Pi_{x\beta\,t}(x)|^p\lesssim(\sqrt[4]{t})^{|\beta|-|\n|}.
\end{equation}
Using this for $\n=(2,0)$ and $\beta-e_0\prec\beta$ we learn from the (annealed)
continuity \eqref{cw60_cont} of $\partial_1^2\Pi_x$
that $\partial_1^2\Pi_{x\,\beta-e_0}(x)=0$ (almost surely) for $|\beta-e_0|=|\beta|>2$. 
Since $c_\beta=0$ for $|\beta|\geq2$, cf.~(\ref{eq:c_pop_cond}), and $|0|=\alpha\le 2$, we
get from (\ref{eq:Pi_minus_anchor}) that $\Pi_{x\beta}^-(x)=0$. 
By the continuity $\eqref{cw60_minus_cont}_\beta$ of
$\Pi_{x\beta}^-$ we may convert this back into the more robust
\begin{align}\label{mt99bis}
\lim_{t\downarrow 0}\E^{\frac{1}{p}}|\Pi_{x\beta\, t}^-(x)|^p = 0\quad\mbox{provided}\;|\beta|>2.
\end{align}
\begin{proposition}[Reconstruction I] \label{prop:reconstr_reg} Assume $|\beta|>2$,
that $\eqref{eq:pi_generic}_{\prec\beta}$, $\eqref{cw60_cont}_{\prec\beta}$, $\eqref{cw60_minus_cont}_{\beta}$, $\eqref{eq:pi_n}_{\prec\beta}$, $\eqref{eq:recenter_Pi_minus_specific}_\beta$ and $\eqref{eq:pi_minus_generic}_{\prec\beta}$ hold,
and that $\eqref{eq:gamma}_{\beta}^\gamma$ holds for all $\gamma$ not purely polynomial.
Then $\eqref{eq:pi_minus_generic}_{\beta}$ holds.
\end{proposition}

\begin{proof}
By general reconstruction, 
see e.~g.~\cite[Proposition 1]{OSSW18} or \cite[Lemma 4.8]{LO22},
the estimate on $\Pi_{x\beta}^{-}$ follows from (\ref{mt99bis}) (which as explained above is a consequence of $\eqref{eq:pi_generic}_{\prec\beta}$, $\eqref{cw60_cont}_{\prec\beta}$ and $\eqref{cw60_minus_cont}_{\beta}$) once we establish
its continuity in the base point $x$ to the order $|\beta|-2>0$:
\begin{align}\label{eq:Pi_minus_cont_base_reg}
\mathbb{E}^\frac{1}{p}|(\Pi^{-}_{y}-\Pi^{-}_{x})_{\beta t}(x)|^p
\lesssim(\sqrt[4]{t})^{\alpha-2}(\sqrt[4]{t}+|y-x|)^{|\beta|-\alpha}.
\end{align}
Estimate (\ref{eq:Pi_minus_cont_base_reg}) in turn relies on 
(\ref{eq:recenter_Pi_minus_specific}), which we rearrange to
\begin{align}\label{eq:Pi_minus_cont_base_inc}
\Pi^{-}_x-\Pi^{-}_y&=(\Gamma_{xy}^*-{\rm id})\Pi^{-}_y
+P \sum_{k\ge 0}\mathsf{z}_k\big(\Gamma_{xy}^*({\rm id}-P)\Pi_y+\pi^{({\bf 0})}_{xy}\big)^k
\partial_1^2\big(\Gamma_{xy}^*({\rm id}-P)\Pi_y+\pi^{({\bf 0})}_{xy}\big),\nonumber\\
&\mbox{where}\quad\Gamma_{xy}^*({\rm id}-P)\Pi_y+\pi^{({\bf 0})}_{xy}
\stackrel{\eqref{eq:pi_purely_pol},\eqref{eq:Gamma_z_n}}{=}
\sum_{{\bf n}}(\mathbf{1}_{{\bf n}\not=\0}\mathsf{z}_{\bf n}+\pi^{({\bf n})}_{xy})
(\cdot-y)^{\bf n}.
\end{align}
Note that by the strict triangularity (\ref{ord12}) of 
$\Gamma_{xy}^*-{\rm id}$ w.~r.~t.~$\prec$,
the first r.~h.~s.~term of (\ref{eq:Pi_minus_cont_base_inc}) involves $\Pi_{y\beta'}^{-}$
only for $\beta'\prec\beta$. We observe that by H\"older's inequality in probability,
the $\mathbb{E}^\frac{1}{p}|\cdot|^p$-norm of each constituent
$(\Gamma_{xy}^*-{\rm id})_{\beta}^{\beta'}\Pi^{-}_{y\beta' t}(x)$ to the matrix-vector
product is estimated by $|y-x|^{|\beta|-|\beta'|}(\sqrt[4]{t})^{\alpha-2}  (\sqrt[4]{t}+|y-x|)^{|\beta'|-\alpha}$,
which because of $|\beta|-|\beta'|\ge0$ (by the triangularity (\ref{ord12}) of
$\Gamma_{xy}^*$ w.~r.~t.~$|\cdot|$) is dominated by the r.~h.~s.~of
(\ref{eq:Pi_minus_cont_base_reg}). By the structure (\ref{eq:triangular_product}) 
of the expression $\sum_{k\ge0}\mathsf{z}_k\pi^k\pi'$,
the second r.~h.~s.~term of (\ref{eq:Pi_minus_cont_base_inc}) 
involves $\mathbf{1}_{{\bf n}\not={\bf 0}}\mathsf{z}_{\bf n}+\pi_{xy\beta'}^{({\bf n})}$
only for $\beta'\prec\beta$. Note that both the population condition (\ref{eq:pi_n_pop}) and 
the estimate (\ref{eq:pi_n})
extend from $\pi^{({\bf n})}_{xy}$ to $\mathbf{1}_{{\bf n}\not={\bf 0}}\mathsf{z}_\n
+\pi^{({\bf n})}_{xy}$,
provided one relaxes $|\beta|>|{\bf n}|$ to $|\beta|\geq|{\bf n}|$.
Hence as a function of the active variable, 
the second r.~h.~s.~of $\eqref{eq:Pi_minus_cont_base_inc}_{\beta}$ 
is a linear combination of monomials $(\cdot-y)^{\bf n}$ with $|{\bf n}|\le|\beta|-2$
with a coefficient estimated by $|y-x|^{|\beta|-|{\bf n}|-2}$, where we used (\ref{ord18}).

\medskip

We now apply the convolution $(\cdot)_t$ to this polynomial in
(\ref{eq:Pi_minus_cont_base_inc}) and evaluate in $x$. 
Using \eqref{cw61} 
in form of $|((\cdot-y)^{\bf n})_t(x)|\lesssim(\sqrt[4]{t}+|y-x|)^{|{\bf n}|}$,
we see that each summand\footnote{summed over $k\ge 0$ and ${\bf n}_1,\dots,{\bf n}_k$} 
of the second r.~h.~s.~term of (\ref{eq:Pi_minus_cont_base_inc})
is estimated by $|y-x|^{|\beta|-|{\bf n}|-2}
(\sqrt[4]{t}+|y-x|)^{|{\bf n}|}$, which obviously is dominated by the
r.~h.~s.~of (\ref{eq:Pi_minus_cont_base_reg}). 
\end{proof}

The third task is the estimate $\eqref{eq:pi_generic}$ of $\Pi_{x}$,
based on the output of Proposition \ref{prop:reconstr_reg}. It relies on 
the construction of $\Pi_{x\beta}$ in terms of $\Pi_{x\beta}^{-}$
with help of the semi-group kernel, see (\ref{eq:representation}):
\begin{align}\label{eq:Pi_minus_to_Pi}
\Pi_{x\beta}=-\int_0^\infty dt(1-{\rm T}_x^{|\beta|})
(\partial_2+\partial_1^2)\Pi^-_{x\beta\,t}.
\end{align}
Here and in the sequel, for $\theta > 0$ and $x \in \mathbb{R}^2$, ${\rm T}_x^{\theta}$ denotes the operation of taking the Taylor polynomial of degree $<\theta$ in the base point $x$, i.~e. ${\rm T}_x^{\theta} f (y)$ $=$ $\sum_{|\mathbf{n}|<\theta} \tfrac{1}{\mathbf{n}!} \partial^\mathbf{n} f(x) (y-x)^\mathbf{n}$.

\medskip

As specific to integration in regularity structures,
the estimate $\eqref{eq:pi_generic}_{\beta}$
can be seen as a Schauder estimate anchored in the base point $x$, 
here on an annealed level. As is typical for Schauder theory, one has to
avoid integer values, which is ensured by (\ref{irrational}).
Incidentally, (\ref{eq:Pi_minus_to_Pi}) reproduces Hairer's form 
of integration \cite[eq.~(8.19)]{Ha14}.
However, since we are working on the whole space instead of the torus so 
that there is no a priori decay in $t$ of the integrand, the polynomial
${\rm T}_x^{|\beta|}\int_0^\infty dt (\partial_2+\partial_1^2)\Pi^-_{x\beta\,t}$
may not be well-defined by itself. 

\begin{proposition}[Integration I]\label{prop:int_pi_minus}
Suppose that $\mbox{(\ref{eq:pi_minus_generic})}_{\beta}$ holds.
Then (\ref{eq:Pi_minus_to_Pi}) defines\footnote{By which we mean that the r.~h.~s.~of \eqref{eq:Pi_minus_to_Pi} makes sense in the Banach space with norm $\sup_{x}|y-x|^{-|\beta|}\mathbb{E}^\frac{1}{p}|\cdot|^p$.} a\footnote{By a Liouville argument similar to the one of Proposition~\ref{prop:change_of_basepointII} this solution can be seen to be unique, cf.~\cite[Lemma~2.1]{OST23}, which however we will not make any use of.} solution of (\ref{eq:model})
and $\mbox{(\ref{eq:pi_generic})}_{\beta}$ holds.
\end{proposition}

\begin{proof}
We start with some preliminary estimates:
Appealing to the semi-group property (\ref{eq:semi_group}) in form of
$\partial^{\bf n}\Pi_{x\beta\,t}^{-}(y)$ $=\int dz\partial^{\bf n}\psi_{\frac{t}{2}}(y-z)
\Pi_{x\beta\,\frac{t}{2}}^{-}(z)$, we gather from $\eqref{eq:pi_minus_generic}_{\beta}$ and (\ref{cw61})
\begin{align}\label{cw64}
\mathbb{E}^\frac{1}{p}|\partial^{\bf n}\Pi_{x\beta\,t}^{-}(y)|^p
\lesssim(\sqrt[4]{t})^{\alpha-2-|{\bf n}|}(\sqrt[4]{t}+|y-x|)^{|\beta|-\alpha}.
\end{align}
Using (\ref{cw64}) for $y=x$ we infer for the Taylor polynomial
of (parabolic) order $<\theta$ that
\begin{align}\label{cw75}
\mathbb{E}^\frac{1}{p}|{\rm T}_x^{\theta}
\partial^{\bf n}\Pi_{x\beta\,t}^{-}(y)|^p\lesssim
\sum_{|{\bf m}|<\theta}(\sqrt[4]{t})^{|\beta|-2-|{\bf n}+{\bf m}|}
|y-x|^{|{\bf m}|}.
\end{align}
Note that by definition, $(1-{\rm T}_x^\theta)\partial^{\bf n}\Pi_{x\beta\,t}^-$
vanishes to order $\theta$ in $s=0$ along the
``parabolic'' curve $[0,1]\ni s\mapsto (sy_1+(1-s)x_1,s^2y_2 + (1-s^2)x_2)$
connecting $x$ to $y$. Hence the ensuing integral representation formula \eqref{eq:Pi_minus_to_Pi}
only involves derivatives of order ${\bf m}$ with $|{\bf m}|\ge\theta$
of $\partial^{\bf n}\Pi_{x\beta\,t}^{-}$, next to being constrained by $m_1+m_2\le\theta+1$.
Using (\ref{cw64}) with $y$ replaced by a point along this curve and ${\bf n}$
replaced by ${\bf n}+{\bf m}$ we obtain
\begin{align}\label{cw76}
\lefteqn{\mathbb{E}^\frac{1}{p}|(1-{\rm T}_x^{\theta})
\partial^{\bf n}\Pi_{x\beta\,t}^{-}(y)|^p}\nonumber\\
&\lesssim
\sum_{\substack{|{\bf m}|\ge\theta \\ m_1+m_2\leq\theta+1}}(\sqrt[4]{t})^{\alpha-2-|{\bf n}+{\bf m}|}
(\sqrt[4]{t}+|y-x|)^{|\beta|-\alpha}
|y-x|^{|{\bf m}|}.
\end{align}

\medskip

With help of these three auxiliary estimates we now derive the central one, namely
\begin{align}\label{cw77}
\int_0^\infty dt\, \mathbb{E}^\frac{1}{p}|(1-{\rm T}_x^{|\beta|})
(\partial_2+\partial_1^2)\Pi_{x\beta\,t}^{-}(y)|^p\lesssim|y-x|^{|\beta|}.
\end{align}
To this purpose, we split the integral into the near-field $t\le|y-x|^4$ and
the far-field $t\ge|y-x|^4$. For the latter, we use (\ref{cw76}) with $\theta=|\beta|$
and ${\bf n}=(0,1),(2,0)$ (and thus $|{\bf n}|=2$), 
so that the growth rate of its r.~h.~s.~in $\sqrt[4]{t}$ is given by  
$-4-|{\bf m}|+|\beta|$. Since by Assumption \ref{ass:spectral_gap} $\alpha$
is irrational, $|\beta|$ is not an integer according to (\ref{irrational}), 
so that the sum in (\ref{cw76}) effectively restricts to $|{\bf m}|>|\beta|$.
Hence the growth rate in $t$ itself is $<-1$, so that 
the integral converges, and is estimated by the r.~h.~s.~of (\ref{cw77}). 

\medskip

We use the triangle inequality to split the near-field part of (\ref{cw77}) into
its contribution from ${\rm T}_x^{|\beta|}$ and from $1$.
On the former, we apply (\ref{cw75}), again for $\theta=|\beta|$
and with $|{\bf n}|=2$. The order of vanishing of its r.~h.~s.~in $\sqrt[4]{t}$ is given by
$-4-|{\bf m}|+|\beta|$. Since the sum is restricted to $|{\bf m}|<|\beta|$,
the order of vanishing in $t$ is $>-1$, so that 
the integral converges at $t=0$, and is also estimated by the r.~h.~s.~of (\ref{cw77}). 
For the contribution from $1$, we directly use (\ref{cw64}) with $|{\bf n}|=2$; the order
of vanishing is now given by $\alpha-4$, so that by $\alpha>0$ the same reasoning applies.

\medskip

We finally argue that (\ref{eq:Pi_minus_to_Pi}) satisfies (\ref{eq:model}).
Estimate (\ref{cw77}) implies that (\ref{eq:Pi_minus_to_Pi}) 
is well-defined in the Banach space
with norm $\sup_{x}|y-x|^{-|\beta|}\mathbb{E}^\frac{1}{p}|\cdot|^p$.
It also implies that the version of (\ref{eq:Pi_minus_to_Pi}) where $\int_0^\infty dt$
is replaced by $\int_s^Tdt$ converges when $s\downarrow0$ and $T\uparrow\infty$
in that topology. On the level of $\int_s^Tdt$, we may exchange the differential
operator $(\partial_2-\partial_1^2)$ with the integration, so that by (\ref{eq:factorization}),
\begin{align*}
-(\partial_2-\partial_1^2)\int_s^T dt(1-{\rm T}_x^{|\beta|})
(\partial_2+\partial_1^2)\Pi_{x\beta\,t}^{-}= 
(1-{\rm T}_x^{|\beta|-2})\Pi_{x\beta\,s}^{-}-(1-{\rm T}_x^{|\beta|-2})\Pi_{x\beta\,T}^{-}.
\end{align*}
We learn from (\ref{cw75}) and (\ref{cw76}), both for 
$\theta=|\beta|-2\not\in\mathbb{Z}$ and ${\bf n}={\bf 0}$,
that ${\rm T}_x^{|\beta|-2}\Pi_{x\beta\,s}^{-}$ vanishes for $s\downarrow 0$,
and that $(1-{\rm T}_x^{|\beta|-2})\Pi_{x\beta\,T}^{-}$ vanishes for $T\uparrow \infty$.
\end{proof}

We now turn to the fourth task of estimating $\pi^{(\n)}_{xy}$. This relies on the identity
\begin{align}\label{eq:three_point}
\sum_{{\bf n}}\pi_{xy}^{({\bf n})}(z-y)^{\bf n}=\Pi_{x}(z)-\Pi_y(z)-(\Gamma_{xy}^*-{\rm id})P\Pi_{y}(z)
\end{align}
involving the three points $x$, $y$, and $z$. Identity (\ref{eq:three_point}) 
follows from (\ref{eq:recenter_Pi_specific}), using (\ref{eq:Gamma_z_n})
and (\ref{eq:pi_purely_pol}) for ${\bf n}\not={\bf 0}$, and (\ref{eq:Pipi}) for ${\bf n}={\bf 0}$. Choosing $\#\{\n\,|\,|\n|<|\beta|\}$ pairwise distinct space-time points in place of $z$ we may interpret the $\beta$ component of the l.~h.~s. of \eqref{eq:three_point} as a Vandermonde matrix applied to the vector $\{\pi_{xy \beta}^{(\mathbf{n})}\}_{|\mathbf{n}|<|\beta|}$. Then the  invertibility of the Vandermonde matrix  yields the equivalence of 
annealed norms
\begin{align*}
\max_{{\bf n}:|{\bf n}|<|\beta|}|y-x|^{|{\bf n}|}\mathbb{E}^\frac{1}{p}|\pi^{({\bf n})}_{xy\beta}|^p
\sim\sup_{z:|z-x|\le|y-x|}\mathbb{E}^\frac{1}{p}\big|\sum_{\bf n}\pi^{({\bf n})}_{xy\beta}
(z-y)^{\bf n}\big|^p;
\end{align*}
formula $\eqref{eq:three_point}_\beta$ allows to estimate $\pi_{xy\beta}^{(\n)}$ 
by the outputs of Propositions~\ref{prop:gamma_npp} and \ref{prop:int_pi_minus}:

\begin{proposition}[Three-point argument I]\label{prop:pi_n_three_point} 
Assume that $\eqref{eq:pi_generic}_{\preccurlyeq\beta}$ and $\eqref{eq:recenter_Pi_specific}_\beta$ hold
and that $\eqref{eq:gamma}_{\beta}^{\gamma}$ holds for all $\gamma$ not purely polynomial. Then
$\eqref{eq:pi_n}_{\beta}$ holds.
\end{proposition}

Equipped with the output of Proposition \ref{prop:pi_n_three_point}, we now may
complete our first task. By the same argument as for Proposition \ref{prop:gamma_npp},
we have

\begin{proposition}[Algebraic argument I, second half]\label{prop:gamma_pp}
$\eqref{eq:pi_n}_{\preccurlyeq\beta}$ implies $\eqref{eq:gamma}_{\beta}$.
\end{proposition}


\subsection{Estimate of the expectation: BPHZ-choice of renormalization, SG inequality, and dualization of Malliavin derivative estimate}\label{BPHZMall}

For this and the next three subsections, we restrict to singular and not purely
polynomial $\beta$
and start addressing the challenging part of the proof, 
namely the estimate $\eqref{eq:pi_minus_generic}_\beta$
of $\Pi_{x\beta}^-$ in this singular case. 
We will use what is called the $\mathbb{L}^p$-version\footnote{In the sequel $\mathbb{L}^p$ denotes the space of
$p$-integrable random variables w.~r.~t.~$\mathbb{E}$.}, for $p\ge 2$, of the SG inequality
\begin{equs}
\E^\frac{1}{p} |F|^p \lesssim |\E F| + \E^{\frac{1}{p}}
\|\frac{\partial F}{\partial\xi}\|_*^p.
\label{eq:spectral_gap_var}
\end{equs}
for all cylindrical functionals $F$ as in \eqref{eq:cyl_fnct}.
It extends by continuity to the classical Malliavin--Sobolev space $\mathbb{H}^p$, see Appendix~\ref{ss:Malle}.
This $\mathbb{L}^p$-version is a simple consequence of \eqref{ao03}, using the chain rule for the Malliavin derivative
and H\"older's estimate in probability, 
and is oblivious to the nature of the underlying Hilbert norm (\ref{as02}), 
see for instance \cite[Step 2 in the proof of Lemma 3.1]{JO22}; the result
is classical in the Gaussian case \cite[Theorem 5.5.11]{Bo98}.
As we argue in Subsection~\ref{sec:reconstr_approx},
for fixed $t$ and $x$, $F=\Pi_{x\beta\,t}^{-}(y)\in\mathbb{H}^p$, to which we will apply
(\ref{eq:spectral_gap_var}).

\medskip

We now argue that the first r.~h.~s.~term $\mathbb{E}\Pi_{x\beta\,t}^-(y)$ 
in (\ref{eq:spectral_gap_var})
is estimated as a consequence of the BPHZ-choice of renormalization
from Subsection \ref{sec:Pi_minus_constr}, cf.~\eqref{eq:expect_est}.
In order to pass from the limit $\lim_{t\uparrow\infty}$ in Proposition \ref{prop:BPHZ}
to a finite value $\sqrt[4]{t}$ of the (spatial) convolution scale,
we need the following proposition; note that the statement (and the proof) is oblivious to the specific value of $c_\beta$ and thus it can be shown before the BPHZ choice $\eqref{eq:BPHZ}_\beta$ has been made.
\begin{proposition}\label{prop:limit}
For $|\beta|<2$, 
suppose that $\eqref{eq:recenter_Pi_minus}_\beta$, $\eqref{eq:pi_minus_generic}_{\prec\beta}$ and 
$\eqref{eq:Pi_minus_shift}_\beta$ holds,
and that $\eqref{eq:gamma}_{\beta}^\gamma$ holds for all $\gamma$ not purely polynomial.
Then we have
\begin{align}\label{eq:limit}
\int_{T}^\infty dt|\frac{d}{dt}\mathbb{E}\Pi_{x\beta \,t}^{-}(y)|
\lesssim (\sqrt[4]{T})^{\alpha-2}(\sqrt[4]{T}+|y-x|)^{|\beta|-\alpha}.
\end{align}
\end{proposition}
\begin{proof}
For $|\beta|<2$, $\eqref{eq:recenter_Pi_minus}_\beta$ reduces to	\begin{equation*}
	{\Pi}_{x\beta}^- =  (\Gamma_{xz}^* \Pi_{z}^-)_\beta.
	\end{equation*}	
Moreover, by $\eqref{eq:Pi_minus_shift}_\beta$, $\E \Pi_{z\beta s}^- (z)$ is independent of $z$. These two facts combined yield
\begin{align}\label{eq:dt_expect}
\frac{d}{dt}\mathbb{E}\Pi_{x\beta \,t}^{-}(y)
=-\int_{\mathbb{R}^2}dz(\partial_1^4-\partial_2^2)\psi_{t-s}(y-z)
\mathbb{E}\big((\Gamma_{xz}^*-{\rm id})\Pi_{z\,s}^{-}\big)_\beta(z),
\end{align}
for all $s\in(0,t)$. 
The merit is that as a consequence of the
strict triangularity (\ref{ord12}) of $\Gamma_{xy}^*-{\rm id}$ w.~r.~t.~$\prec$,
(\ref{eq:dt_expect}) only features $\{\Pi_{z\beta'}^{-}\}_{\beta'\prec\beta}$.
Since by definition \eqref{eq:Pi_minus_def_alt}, $\Pi^{-}_{z}\in\mathsf{\tilde T}^*$, (\ref{eq:dt_expect}) only features
$(\Gamma_{xz}^*-{\rm id})_{\beta}^{\beta'}$ for $\beta'$ not purely polynomial.

\medskip

We use (\ref{eq:dt_expect}) with $s=\frac{t}{2}$. By the Cauchy-Schwarz inequality
in probability,
the contribution from $(\Gamma_{xz}^*-{\rm id})_{\beta}^{\beta'}\Pi_{z\beta' \frac{t}{2}}(z)$ 
is estimated in expectation by $|z-x|^{|\beta|-|\beta'|}(\sqrt[4]{t})^{|\beta'|-2}$
$\lesssim (|y-z|+|y-x|)^{|\beta|-|\beta'|}$ $(\sqrt[4]{t})^{|\beta'|-2}$.
After integration in $z$, by the moment bounds (\ref{cw61}) on $\psi_s$, this contribution
to (\ref{eq:dt_expect}) is controlled by $t^{-1}
(\sqrt[4]{t}+|y-x|)^{|\beta|-|\beta'|}(\sqrt[4]{t})^{|\beta'|-2}$.
Since $|\beta|<2$, integration in $t\ge T$ yields control by
$(\sqrt[4]{T}+|y-x|)^{|\beta|-|\beta'|}(\sqrt[4]{T})^{|\beta'|-2}$.
By $|\beta'|\ge\alpha$, (\ref{eq:limit}) follows.
\end{proof}

Equipped with Proposition \ref{prop:limit} and the choice
$\eqref{eq:BPHZ}_\beta$ of $c_\beta$, the qualitative Proposition \ref{prop:BPHZ}
instantly  
upgrades to 
an estimate of $\mathbb{E}\Pi_{x\beta\,t}^-(x)$;
using $\eqref{eq:recenter_Pi_minus}_{\beta}$
together with H\"older's inequality and $\eqref{eq:gamma}_\beta^\gamma$ for $\gamma$ not purely polynomial and $\eqref{eq:pi_minus_generic}_{\prec\beta}$
yields the desired estimate of $\mathbb{E}\Pi_{x\beta\,t}^-(y)$:

\begin{proposition}\label{prop:expect}
For $|\beta|<2$, suppose
that $\eqref{eq:BPHZ}_\beta$, $\eqref{eq:recenter_Pi_minus}_\beta$, $\eqref{eq:pi_minus_generic}_{\prec\beta}$, and $\eqref{eq:limit}_\beta$ hold,
and that $\eqref{eq:gamma}_{\beta}^\gamma$ holds for all $\gamma$ not purely polynomial. 
Then we have
\begin{align}\label{eq:expect_est}
|\mathbb{E}\Pi_{x\beta\,t}^-(y)| 
\lesssim (\sqrt[4]{t})^{\alpha-2}(\sqrt[4]{t}+|y-x|)^{|\beta|-\alpha}.
\end{align}
\end{proposition}

The remaining task of this and the next three subsections
is thus to estimate the Malliavin derivative of $\Pi_{x\beta\,t}^-(y)$, 
in the norm given by (\ref{eq:spectral_gap_var}), by the r.~h.~s.~of (\ref{eq:expect_est}):
\begin{align}\label{eq:expect2} 
\mathbb{E}^{\frac{1}{p}}\Big|\int_{\mathbb{R}^2}
\big((\partial_1^4- \partial_2^2)^{\frac{1}{4}(\frac{1}{2}-\alpha)}
\frac{\partial}{\partial \xi} \Pi_{x\beta\,t}^-(y)\big)^2\Big|^\frac{p}{2} 
\lesssim(\sqrt[4]{t})^{\alpha-2}(\sqrt[4]{t}+|y-x|)^{|\beta|-\alpha}.
\end{align}
It is convenient to undo the Riesz representation (\ref{ao22}) and to return to the 
derivative $\delta\Pi_{x\beta t}^{-}(y)$ of $\Pi_{x\beta t}^{-}(y)$
in direction of the space-time field $\delta\xi$. 
It is a straightforward consequence of $\mathbb{L}^p$-duality,
with $q$ denoting the conjugate exponent of $p$,
that (\ref{eq:expect2}) is equivalent to
\begin{align*}
|\E\delta \Pi_{x\beta \,t}^-(y)|\lesssim (\sqrt[4]{t})^{\alpha - 2}  
(\sqrt[4]{t}+|y-x|)^{|\beta| - \alpha} \E^{\frac{1}{q}}
\Big|\int_{\mathbb{R}^2}\big(
(\partial_1^4- \partial_2^2)^{\frac{1}{4}\left(\alpha-\frac{1}{2}\right)}\delta\xi\big)^2
\Big|^\frac{q}{2},
\end{align*}
provided that this is established for an arbitrary $\delta\xi$
that is allowed to be random in order to pull the supremum over $\delta\xi$
out of the $\mathbb{L}^q$-norm.
For the base case and when introducing a weight, both in Subsection \ref{thirdblock}, 
it will be convenient to
replace the $L^2(\mathbb{R}^2)$-based fractional Sobolev norm of $\delta\xi$
by its equivalent $L^2(\mathbb{R}^2)$-based Besov norm. This equivalence 
is obvious when the Besov side is formulated in terms of our semi-group,
since it then follows by Plancherel from the elementary scaling identity
$(k_1^4+k_2^2)^{\frac{1}{2}(\alpha-\frac{1}{2})}$
$\sim\int_0^\infty\frac{ds}{s}s^{\frac{1}{2}(\frac{1}{2}-\alpha)}\exp(-2s(k_1^4+k_2^2))$
in terms of the wave vector $k=(k_1,k_2)$ (and relies on $\alpha<\frac{1}{2}$):
\begin{align*}
|\E\delta \Pi_{x\beta t}^-(y)|\lesssim (\sqrt[4]{t})^{\alpha - 2}
(\sqrt[4]{t}+|y-x|)^{|\beta| - \alpha} \E^{\frac{1}{q}}
\Big|\int_0^\infty\frac{ds}{s}(\sqrt[4]{s})^{2(\frac{1}{2}-\alpha)}
\int_{\mathbb{R}^2}(\delta\xi_s)^2
\Big|^\frac{q}{2}.
\end{align*}
In fact, we will establish a stronger version of this estimate: It is strengthened
on the l.~h.~s.~by replacing the expectation by a $\mathbb{L}^{q'}$-norm, and
it is strengthened on the r.~h.~s.~by exchanging the spatial and probabilistic norm
(which by Minkowski's inequality is a strengthening due to $q\le 2$)
\begin{align}\label{eq:mal_dual_annealed}
 \E^{\frac{1}{q'}} |\delta \Pi_{x\beta \,t}^-(y)|^{q'} \lesssim  (\sqrt[4]{t})^{\alpha - 2}  
(\sqrt[4]{t}+|y-x|)^{|\beta| - \alpha}\bar w,
\end{align}
where we introduced the following abbreviation for a norm of $\delta\xi$ 
\begin{align}\label{ao26}
\bar w:=\left(\int_0^\infty \frac{d s}{s} (\sqrt[4]{s})^{2(\frac{1}{2}-\alpha)} \int_{\mathbb{R}^2}\E^{\frac{2}{q}}
|\delta\xi_s|^q \right)^{\frac{1}{2}}.
\end{align}
Estimate (\ref{eq:mal_dual_annealed}) is an annealed estimate, 
where an annealed norm of $\delta\Pi_{x\beta}^{-}$
is controlled by the annealed norm $\bar w$ of $\delta\xi$.
We shall establish (\ref{eq:mal_dual_annealed}) for all $q'<q\le 2$. Hence 
$\lesssim$ now also acquires a dependence on $q'<q\le 2$ (next to $\alpha$ and $\beta$) when
it comes up in the estimate of a Malliavin derivative. 
As for the integrability exponent $2\le p<\infty$, 
all estimates will be implicitly assumed to hold for all $q'<q\le 2$.
The strengthening from $q'=1$ to $q'>1$
is important when using (\ref{eq:mal_dual_annealed}) in the induction:
One often needs to estimate products where (at most) one of the factors comes from a (directional) 
Malliavin derivative of the model 
(as indicated by the appearance of the symbol $\delta$ or ${\rm d}$) 
whereas the other factors are one of the model components (i.~e.~$\Pi_x^{-},\Pi_x,\Gamma_{xy}^*$).
Since the other factors ask for a stochastic $\mathbb{L}^p$-norm with $p<\infty$, 
by H\"older's inequality, we need a stochastic $\mathbb{L}^{q'}$-norm with $q'>1$ on the Malliavin factor,
see for instance (\ref{eq:Hoelder})
in the proof of Proposition \ref{prop:delta_gamma_npp} below. This reflects a deterioration
in the stochastic integrability, which is unavoidable since the
homogeneity of $\Pi_{x\beta}^{-}$ in $\xi$ is $[\beta]+1$. 


\subsection{Estimate of Malliavin derivatives: \texorpdfstring{$\delta\Gamma_{xy}^*$, $\delta\Pi^{-}_x$, $\delta\Pi_x$, and $\delta\pi^{(\bf n)}_{xy}$}{delta Gamma* xy, delta Pi minus x, delta Pi x, and delta pi n xy}}
\label{secondblock}

For this and the next two subsections, we fix a random $\delta\xi$.
This subsection is an interlude: As will become clear only in Subsection \ref{forthblock},
next to $Q\delta\Pi^{-}_{x}$, we also need to estimate 
the directional Malliavin derivatives $Q\delta\Pi_{x}$ and $Q\delta\Gamma_{xy}^*P$. 
In fact, in this subsection, we shall proceed like in Subsection \ref{firstblock} 
and assume the estimate (\ref{eq:mal_dual_annealed})
on $Q\delta\Pi^{-}_{x}$ in order to (inductively) derive the estimates
on the remaining objects $Q\delta\Gamma_{xy}^*P$,
$Q\delta\Pi_{x}$, and $Q\delta\pi^{(\bf n)}_{xy}$.

\medskip

Taking the (directional) Malliavin derivative of the exponential formula (\ref{exp})
applied to $\Gamma^*=\Gamma_{xy}^*$, we obtain by Leibniz' rule, see
Subsection~\ref{sec:gamma_three_point_approx},
\begin{align}\label{deltaGamma}
\delta\Gamma^*_{xy}=\sum_{\bf n}\delta\pi^{(\n)}_{xy}\Gamma^*_{xy} D^{({\bf n})},
\end{align}
which motivates to include the estimates
\begin{align}\label{eq:delta_pi_n}
\mathbb{E}^{\frac{1}{q'}}|\delta\pi_{xy\beta}^{(\n)}|^{q'}
\lesssim |y-x|^{|\beta|-|{\bf n}|} \bar w\quad \text{for all} \ |{\bf n}|<|\beta|.
\end{align}
Note that in view of (\ref{eq:def_pi_n_poly}) we have
\begin{align}\label{eq:algIII7}
\delta\pi_{xy}^{({\bf n})}\in\mathsf{\tilde T}^*.
\end{align}
In analogy to Proposition \ref{prop:gamma_npp} we have

\begin{proposition}[Algebraic argument II]\label{prop:delta_gamma_npp} 
Assume that
$\eqref{eq:delta_pi_n}_{\prec\beta}$ holds,
and that $\eqref{eq:gamma}_{\preccurlyeq\beta}^{\gamma}$ 
holds for all $\gamma$ not purely polynomial.
Then we have for all $\gamma$ not purely polynomial
\begin{equation}\label{eq:delta_gamma}
\E^\frac{1}{q'} |(\delta\Gamma^*_{xy})_{\beta}^{\gamma}|^{q'}
\lesssim|y-x|^{|\beta|-|\gamma|} \bar w.
\end{equation}
\end{proposition}

\begin{proof} 
We distinguish the contributions ${\bf n}={\bf 0}$
and ${\bf n}\not={\bf 0}$ to (\ref{deltaGamma}). By definition (\ref{eq:def_Dnull}) of
$D^{({\bf 0})}$, all contributions to $\eqref{deltaGamma}_\beta^\gamma$
from ${\bf n}={\bf 0}$ are
of the form\footnote{with the implicit understanding that this term vanishes
if $\gamma(k)=0$} 
\begin{align}\label{eq:algII3}
\delta\pi^{({\bf 0})}_{xy\beta_1}(\Gamma_{xy}^*)_{\beta_2}^{\gamma-e_k+e_{k+1}}
\end{align}
for some $k\ge 0$ and (populated) multi-indices $\beta_1,\beta_2$ with 
$\beta_1+\beta_2=\beta$ and $k\ge 0$. Note that by (\ref{eq:algIII7}),
$\beta_1$ is not purely polynomial; likewise, since $\gamma-e_k+e_{k+1}$ is 
obviously neither purely polynomial nor vanishing, 
this transmits to $\beta_2$ by \eqref{eq:algIII5} and (\ref{ord23}).
Hence we may apply (\ref{eq:recIII25}) to the desired effect of
\begin{align}\label{eq:algII4}
\beta_1\prec\beta\;\mbox{and}\;\beta_2\preccurlyeq\beta.
\end{align}
By H\"older's inequality in probability, we estimate
the $\mathbb{E}^\frac{1}{q'}|\cdot|^{q'}$-norm of the product (\ref{eq:algII3}) by the
product of the $\mathbb{E}^\frac{1}{q}|\cdot|^{q}$-norm of the first factor 
and the $\mathbb{E}^\frac{1}{p}|\cdot|^{p}$-norm of the second factor; recall that 
\begin{align}\label{eq:Hoelder}
\frac{1}{q}=\frac{1}{q'}+\frac{1}{p}\quad\mbox{and thus requires}\;q'<q\;
\mbox{because of}\;p<\infty.
\end{align}
By $\eqref{eq:delta_pi_n}_{\prec\beta}$
and $\eqref{eq:gamma}_{\preccurlyeq\beta}^{\gamma}$ for $\gamma$ not purely polynomial,
we thus obtain an estimate by $|y-x|^{|\beta_1|}\bar w$ $|y-x|^{|\beta_2|-|\gamma-e_k+e_{k+1}|}$.
Since $|\gamma-e_k+e_{k+1}|$ $=|\gamma|+\alpha$ by definition (\ref{ao12}),
we learn from (\ref{eq:hom_add}) that as desired
\begin{align}\label{eq:algII1}
|\beta_1|+|\beta_2|-|\gamma-e_k+e_{k+1}|=|\beta|-|\gamma|.
\end{align}

\medskip

We now address the contributions to $\eqref{deltaGamma}_\beta^\gamma$
from some ${\bf n}\not={\bf 0}$, which in view of definition (\ref{eq:def_Dn}) of $D^{({\bf n})}$
are of the form 
\begin{align}\label{eq:algII5}
\delta\pi^{({\bf n})}_{xy\beta_1}(\Gamma_{xy}^*)_{\beta_2}^{\gamma-e_{\bf n}}.
\end{align}
Note that once more, $\gamma-e_{\bf n}$ is neither purely polynomial nor,
by assumption, vanishing. Like above, this transmits to $\beta_2$, and 
yields (\ref{eq:algII4}).
Now (\ref{eq:algII5}) is estimated by $|y-x|^{|\beta_1|-|{\bf n}|} \bar w$ 
$|y-x|^{|\beta_2|-|\gamma-e_{\bf n}|}$. 
Since $|\gamma-e_{\bf n}|=|\gamma|+\alpha-|{\bf n}|$ by definition (\ref{ao12}),
we get once more from (\ref{eq:hom_add})
\begin{align}\label{eq:algII2}
|\beta_1|-|{\bf n}|+|\beta_2|-|\gamma-e_{\bf n}|=|\beta|-|\gamma|.
\end{align}
\end{proof}

We now pass from $\delta\Pi_{x\beta}^{-}$ to $\delta\Pi_{x\beta}$.
To this purpose, we take the Malliavin derivative of (\ref{eq:Pi_minus_to_Pi}), see Subsection~\ref{sec:int_approx}.
By an almost identical 
integration argument to Proposition~\ref{prop:int_pi_minus}, we obtain

\begin{proposition}[Integration II]\label{prop:intII} 
Assume that $\eqref{eq:mal_dual_annealed}_{\beta}$ holds.
Then we have
\begin{align}\label{eq:delta_pi_generic}
\E^{\frac{1}{q'}} |\delta\Pi_{x\beta}(y)|^{q'}\lesssim |y-x|^{|\beta|} \bar w.
\end{align}
\end{proposition}

We finally return from $\delta\Pi_{x\beta}$ to $\delta\pi_{xy}^{({\bf n})}$,
by taking the Malliavin derivative of the three-point identity \eqref{eq:three_point} (which as shown above follows from \eqref{eq:recenter_Pi_specific}),
which by Leibniz' rule, see Subsection~\ref{sec:gamma_three_point_approx}, assumes the form
\begin{align}\label{eq:delta_three_point}
\sum_{{\bf n}}\delta\pi_{xy}^{({\bf n})}(z-y)^{\bf n}=\delta\Pi_{x}(z)
-\Gamma_{xy}^*P\delta\Pi_{y}(z)-\delta\Gamma_{xy}^*P\Pi_{y}(z).
\end{align}
We obtain quite analogously 
to Propositions \ref{prop:pi_n_three_point} and \ref{prop:gamma_pp}, using H\"older's inequality
in probability like in the proof of Proposition \ref{prop:delta_gamma_npp}
 
\begin{proposition}[Three-point argument II]\label{prop:three_pointII}
Assume that $\eqref{eq:gamma}_{\beta}^\gamma$ and $\eqref{eq:delta_gamma}_{\beta}^\gamma$ 
hold for $\gamma$ not purely polynomial, 
and that $\eqref{eq:pi_generic}_{\prec\beta}$, $\eqref{eq:recenter_Pi_specific}_\beta$, and 
$\eqref{eq:delta_pi_generic}_{\preccurlyeq\beta}$ hold. 
Then $\eqref{eq:delta_pi_n}_{\beta}$ and $\eqref{eq:delta_gamma}_\beta$ hold.
\end{proposition}


\subsection{Estimate of modelled distributions: \texorpdfstring{${\rm d}\Gamma_{xy}^*-{\rm d}\Gamma_{xz}^*\Gamma_{zy}^*$, $\delta\Pi^{-}_x-{\rm d}\Gamma_{xz}^*\Pi^{-}_z$, $\delta\Pi_x-\delta\Pi_x(z)-{\rm d}\Gamma_{xz}^*\Pi_z$, and ${\rm d}\pi^{({\bf n})}_{xy}-{\rm d}\pi^{({\bf n})}_{xz}-{\rm d}\Gamma_{xz}^*\pi^{({\bf n})}_{zy}$}{dGamma* xy - dGamma* xz Gamma* zy, deltaPi minus x - dGamma* xz Pi minus z, deltaPi x - deltaPi x(z) - dGamma* xz Pi z, and dpi n xy - dpi n xz - dGamma* xz pin zy}}
\label{thirdblock}

This subsection is at the heart of our proof.
We return to the estimate $\eqref{eq:mal_dual_annealed}_\beta$ on $\delta\Pi^{-}_x$.
Because of a lack of regularity in the singular case, $\eqref{eq:mal_dual_annealed}$ 
cannot be inferred from the estimate
$\eqref{eq:delta_pi_generic}$ on $\delta\Pi_x$
via the Malliavin derivative of the formula \eqref{eq:Pi_minus_def}. 
In addition, such a formula would involve the divergent
constants $c_{\beta'}$, at least for $\beta'\prec\beta$.
Instead, we have to capitalize on the gain in regularity that comes
with the passage from $\Pi_x$ to $\delta\Pi_x$,
which arises from replacing one of the instances of $\xi$ 
in this multi-linear expression by a $\delta\xi$. 
However, this gain is subtle for two reasons:
\begin{itemize}
\item In terms of derivative count, 
the passage from $\xi$ to $\delta\xi$ amounts to a gain in regularity by
$\frac{D}{2}=\frac{3}{2}$, namely from $\alpha-2$ to $\alpha-\frac{1}{2}$. 
However, due to the presence of the other instances of $\xi$
in the multi-linear $\delta\Pi_x$, this does not translate into a plain gain of 
$\frac{3}{2}$ derivatives when passing from $\Pi_x$ to $\delta\Pi_x$.
Still, the degree of modeledness (cf.~\eqref{eq:delta_pi_incr_generic} below) of $\delta\Pi_x$
has a boost from $\alpha$-H\"older continuity to $(\alpha+\frac{3}{2})$-modeledness.
This modeledness w.~r.~t.~$\Pi_{z}$ is described by a modelled distribution\footnote{which
arises from ${\rm d}\pi_{xz}^{(1,0)}$} 
${\rm d}\Gamma_{xz}^*$. Indeed, we shall control the rough-path increment
$\delta\Pi_x(y)-\delta\Pi_x(z)-{\rm d}\Gamma_{xz}^*\Pi_z(y)$ 
and the continuity expression ${\rm d}\Gamma_{xy}^*-{\rm d}\Gamma_{xz}^*\Gamma^*_{zy}$ 
to order $\alpha+\frac{3}{2}-$, the former in the sense of Gubinelli's
controlled rough paths \cite[Definition 1]{Gu04}, see (\ref{eq:delta_pi_incr_generic}),
the latter in the sense of \cite[Definition 3.1]{Ha14}, see (\ref{eq:form_cont}).\footnote{We remark that this separation between a controlled
rough path condition (\ref{eq:delta_pi_incr_generic}) and the
continuity condition (\ref{eq:form_cont}) is once more due to the fact
that our abstract model space $\mathsf{T}$ needs to be complemented by a copy of $\mathbb{R}$
capturing constant functions in order to reproduce Hairer's abstract model space, 
see Subsection~\ref{sec:Hairer}. If this is done,
\cite[Definition 3.1]{Ha14} corresponds to the combination of 
(\ref{eq:form_cont}) and (\ref{eq:delta_pi_incr_generic}).}
\item In terms of scaling, there is -- actually by construction -- 
no difference between the $(\alpha-2)$-H\"older norm relevant for $\xi$ 
and the $L^2$-based Sobolev norm (\ref{ao01}) of (fractional) order $\alpha-2+\frac{D}{2}$ 
$=\alpha-\frac{1}{2}$ relevant for $\delta\xi$, 
or its (scaling-wise identical) annealed Besov version (\ref{ao26}).
Hence we will resort to a trick that appears like a cheat:
In order to control the rough-path increments 
$\delta\Pi_x(y)-\delta\Pi_x(z)-{\rm d}\Gamma_{xz}^*\Pi_z(y)$ 
in terms of the (parabolic) distance $|y-z|$ of the active variable $y$ 
to the secondary base point $z$ to the desired power of $\alpha+\frac{D}{2}$,
we will replace the norm (\ref{ao26}) of $\delta\xi$ by a norm that involves a
weight that diverges in $z$. In order to recover the full (nominal)
gain of order of derivatives of $\frac{D}{2}$,
one would be tempted to replace $\delta\xi_s$ in (\ref{ao26})
by its weighted version $|\cdot-z|^{-\frac{D}{2}}\delta\xi_s$, 
which after squaring would result in the weighted integral $\int_{\mathbb{R}^2}|\cdot-z|^{-D}$.
Recalling the definition (\ref{ao18}) of the Carnot-Carath\'eodory distance, 
we however see that this integral is borderline divergent.
This would, in Subsection \ref{forthblock}, make it impossible to return from the weighted 
to the unweighted norm $\bar w$ by averaging in the base point $z$.
Hence we have to marginally tame the weight by replacing $\frac{D}{2}=\frac{3}{2}$
by some exponent
\begin{align}\label{ao25}
\kappa<\frac{3}{2}\quad(\mbox{and thus in particular $\;\kappa+\alpha<2\;$
by $\;\alpha\le\frac{1}{2}$})
\end{align}
and define
\begin{align}
w(z):=\Big(\int_0^\infty\frac{ds}{s}(\sqrt[4]{s})^{2(\frac{1}{2}-\alpha)} 
\int_{\mathbb{R}^2}dy\,|y-z|^{-2\kappa}\,\mathbb{E}^{\frac{2}{q}}|\delta\xi_s(y)|^q\Big)^{\frac{1}{2}}.
\label{eq:weights}
\end{align}
On the one hand, $w(z)$ is strong enough to control an
only slightly negative H\"older norm (however quenched and localized in $z$) of $\delta\xi_\tau$
\begin{align}\label{eq:base1}
\mathbb{E}^\frac{1}{q}|(\delta\xi_\tau)_{t}(z)|^q
\lesssim(\sqrt[4]{t})^{\alpha-2+\kappa}w(z),
\end{align}
as we shall show in the proof of Proposition \ref{prop:recIII} in the context of the base case.
On the other hand, because of (\ref{ao25}) we have that (even square) averages of $w(z)$ 
reduce to $\bar w$:
\begin{align}\label{eq:3ptIV2}
\fint_{z:|z-x|\le\lambda}dz \, w(z)\le
\Big(\fint_{z:|z-x|\le\lambda}dz \, w^2(z)\Big)^\frac{1}{2}
\lesssim\lambda^{-\kappa}\bar w.
\end{align}
It is conceivable that one could carry out the tasks of this subsection
on the level of Besov spaces, possibly appealing to \cite{HL17}.\footnote{Note added in revision: this has been implemented in \cite{HS24, BOT24}.}
However working on the (positive) H\"older level has the advantage that it is well-behaved
under taking products, which is amply used in reconstruction.
\end{itemize}

This stronger norm \eqref{eq:weights} will indeed result in an (annealed) controlled rough-path estimate
of $\delta\Pi_x$ of order $\kappa+\alpha$, see Proposition \ref{prop:intIII}.
This provides sufficient regularity in reconstruction when passing from
the rough-path increments of $\delta\Pi_x$ to those of $\delta\Pi_x^{-}$,
see Proposition \ref{prop:recIII}.
When it comes to the above-mentioned presence of the divergent $c$ in the
formula relating $\delta\Pi_x^{-}$ to $\delta\Pi_x$, we are saved by the fact that
the $c$ drops out when relating the rough-path increment
$\delta\Pi_x^{-}-{\rm d}\Gamma_{xz}^*\Pi_z^{-}$ of $\delta\Pi^{-}_x$ to the (second derivative of the)
rough-path increment
$\partial_1^2(\delta\Pi_x-\delta\Pi_x(z)-{\rm d}\Gamma_{xz}^*\Pi_z)$ of $\delta\Pi_x$, 
see the crucial formula (\ref{magic}).
We refer the reader to \cite{BOT24} for a more geometric intuition of ${\rm d}\Gamma^*$.

\medskip

We defer the (inductive) construction of 
${\rm d}\Gamma_{xz}^*$ to Subsection \ref{sec:dGamma_construction}
and mention here just what is necessary to explain the estimates:
In terms of its form, ${\rm d}\Gamma_{xz}^*$ is quite similar to
$\delta\Gamma_{xy}^*$, see (\ref{deltaGamma}), but truncated beyond ${\bf n}={\bf 0},(1,0)$, and 
with the Malliavin derivative $\delta\pi^{(1,0)}_{xy}$ replaced by
some\footnote{like for ${\rm d}\Gamma_{xy}^*$, the pre-fix ${\rm d}$ does not refer
to an operation like the directional Malliavin derivative $\delta$, but is part
of the symbol} ${\rm d}\pi_{xz}^{(1,0)}\in Q\mathsf{\tilde T}^*$:
\begin{align}\label{eq:def_dGamma}
{\rm d}\Gamma_{xz}^*=\sum_{{\bf n}={\bf 0},(1,0)}
{\rm d}\pi^{({\bf n})}_{xz}\Gamma_{xz}^*D^{({\bf n})}\quad\mbox{with}\quad
{\rm d}\pi^{({\bf 0})}_{xz}:=\delta\pi^{({\bf 0})}_{xz}.
\end{align}
In line with ${\rm d}\pi_{xz}^{({\bf 0})}=\delta\pi_{xz}^{({\bf 0})}\in \mathsf{\tilde T}^*$,
see (\ref{eq:algIII7}), we impose
\begin{align}\label{eq:algIII6}
{\rm d}\pi_{xz}^{(1,0)}\in Q\mathsf{\tilde T}^*
\quad\mbox{so that by (\ref{eq:algIII5})}\quad
{\rm d}\Gamma_{xz}^*\mathsf{T}^*\subset\mathsf{\tilde T}^*.
\end{align}
In Subsection \ref{sec:dGamma_construction}, we argue that ${\rm d}\pi_{xz}^{(1,0)}$
is determined by imposing qualitative first-order vanishing on every singular component
\begin{align}\label{eq:controlled_path_qual}
\E^\frac{1}{q'}|Q\big(\delta\Pi_{x}(y)-\delta\Pi_{x}(z)-{\rm d}\Gamma_{xz}^* \Pi_{z}(y)\big)|^{q'}
= o(|y-z|).
\end{align}
We note that the population pattern of $({\rm d}\Gamma^*_{xz})_{\beta}^{\gamma}$ quickly 
gains in complexity as the homogeneity of $\beta$ increases, 
see Figure~\ref{fig:dGamma_expl}.
\begin{figure}[H]
\begin{center}
\begin{tabular}{c|l}
$\beta$ & $\gamma$'s for which $({\rm d}\Gamma^*)_\beta^\gamma\neq0$ \\ 
\hline 
$0$ & $e_{(1,0)}$ \\
$e_1$ & $e_0,\,e_{(1,0)},\,e_1+e_{(1,0)}$ \\
$2e_1$ & $e_0,\,2e_0,\,e_0+e_1,\,e_{(1,0)},\,e_1+e_{(1,0)},\,e_0+e_1+e_{(1,0)}, \,
2e_1+e_{(1,0)}$ \\
$3e_1$ & $e_0,\, 2e_0,\,3e_0,\,e_0+e_1,\,2e_0+e_1,\,e_{(1,0)},\,
e_0+2e_1,\,e_1+e_{(1,0)}$, \\
& $e_0+e_1+e_{(1,0)},\,2e_0+e_1+e_{(1,0)},\,2e_1+e_{(1,0)},\,e_0+2e_1+e_{(1,0)},\, 3e_1+e_{(1,0)}$ \\
\end{tabular}
\caption{Population pattern of $({\rm d}\Gamma^*)_\beta$ for $\beta = 0, e_1, 2e_1, 3e_1$.  \label{fig:dGamma_expl}}
\end{center}
\end{figure}

Up to these differences, the type and order of tasks will be as in Subsection \ref{firstblock}. 
The first task
is the algebraic argument relying on the following analogue of (\ref{deltaGamma}):
\begin{align}\label{eq:formula_form_cont}
({\rm d}\Gamma^*_{xy}-{\rm d}\Gamma^*_{xz}\Gamma^*_{zy})Q
&=\sum_{{\bf n}={\bf 0},(1,0)}\big({\rm d}\pi_{xy}^{({\bf n})}-{\rm d}\pi_{xz}^{({\bf n})}
-{\rm d}\Gamma^*_{xz}\pi_{zy}^{({\bf n})}\big)\Gamma^*_{xy} D^{({\bf n})} Q.
\end{align}
This formula, which will be established in the proof of the upcoming
Proposition \ref{prop:algIII}, suggests to introduce the following
estimate on the rough-path increments of ${\rm d}\pi_{xz}^{({\bf n})}$:
\begin{align}\label{eq:delta_pi_d_pi_incr}
\lefteqn{\mathbb{E}^\frac{1}{q'}\big|\big({\rm d}\pi_{xy}^{({\bf n})}-{\rm d}\pi_{xz}^{({\bf n})}
-{\rm d}\Gamma^*_{xz}\pi_{zy}^{({\bf n})}\big)_{\beta}\big|^{q'}}\nonumber\\
&\lesssim |y-z|^{\kappa+\alpha-|{\bf n}|} 
(|y-z|+|z-x|)^{|\beta|-\alpha}(w_x(y)+w_x(z))\quad\mbox{for}\;{\bf n}={\bf 0},(1,0),
\end{align}
which is the analogue of (\ref{eq:pi_n}). We note that (\ref{ao24}), which we need to impose below,
implies in particular\footnote{using $\alpha\le 1$}
\begin{align}\label{eq:algIII8}
\kappa>1-\alpha,
\end{align}
so that the first exponent in (\ref{eq:delta_pi_d_pi_incr}) is strictly positive.
As we shall discuss at the beginning of Subsection \ref{forthblock}, 
in (\ref{eq:delta_pi_d_pi_incr}) and in this subsection, 
we do not just need the norm $w(y)+w(z)$ with a singular weight at the two active points,
$y$ and $z$, but also a contribution from the unweighted norm $\bar w$.
We combine weighted and unweighted norms through
\begin{align}\label{eq:def_w_x_z}
 w_x(z) : = w(z) +|z-x|^{-\kappa} \bar w.
\end{align}
While this inclusion of $\bar w$ is dimensionally correct, it has the irritating
effect of introducing an artificial singularity at $z=x$, which however does
not create problems.

\begin{proposition}[Algebraic argument III]\label{prop:algIII}
Assume that $\eqref{eq:delta_pi_d_pi_incr}_{\prec\beta}$ and $\eqref{rec01}_{\prec\beta}$ hold, 
and that $\eqref{eq:gamma}_{\preccurlyeq\beta}^{\gamma}$ holds for all
$\gamma$ not purely polynomial.
Then we have for all $\gamma$ not purely polynomial
\begin{align}\label{eq:form_cont}
\lefteqn{\mathbb{E}^\frac{1}{q'}\big|\big(({\rm d}\Gamma^*_{xy}-{\rm d}\Gamma^*_{xz}\Gamma^*_{zy} 
)Q\big)_\beta^\gamma\big|^{q'}
\lesssim\Big(\mathbf{1}_{\gamma(1,0)=0}\, 
|y-z|^{\kappa+\alpha}(|y-z|+|z-x|)^{|\beta|-|\gamma|-\alpha}}\nonumber\\  
&+\mathbf{1}_{\gamma(1,0)=1}\, 
|y-z|^{\kappa+\alpha-1}(|y-z|+|z-x|)^{|\beta|-|\gamma|-\alpha+1}\Big)
(w_x(y)+w_x(z)),
\end{align}
with the implicit understanding that all exponents are non-negative unless the l.~h.~s.~vanishes.
\end{proposition}

Before embarking on the proof, let us comment on how to interpret \eqref{eq:form_cont}. For this, it is convenient to
	write the l.~h.~s.~in the component-wise fashion of $({\rm d}\Gamma^*_{xy})_{\beta}^{\gamma}$
	$-\sum_{\beta'}({\rm d}\Gamma^*_{xz})_{\beta}^{\beta'}(\Gamma^*_{zy})_{\beta'}^{\gamma}$,
	to fix $\beta$ and $x$, and to think of $\gamma,\beta'$ and $z,y$ as instances of
	the active index and variable, respectively. First let $y$ and $z$ be such that $|y-z|\ll |z-x|$. Then the r.~h.~s. of $\eqref{eq:form_cont}_\beta$ reduces, up to a multiplicative constant depending on $x$ and $z$, to
$\mathbf{1}_{\gamma(1,0)=0}\, 
	|y-z|^{\kappa+\alpha} +\mathbf{1}_{\gamma(1,0)=1}\, 
	|y-z|^{\kappa+\alpha-1}$, and therefore can be read as a modelled continuity condition \cite[(3.1)]{Ha14} of degree $\kappa + \alpha$, provided the degree of $\gamma$ is given by\footnote{Note that for $|\gamma|<2$, $|\gamma|_p$ $=$ $0$ if $\gamma(1,0) = 0$ and $|\gamma|_p=1$ for $\gamma(1,0) = 1$.} $|\gamma|_p$. There is a second more subtle interpretation which becomes apparent when setting $z=x$ (and ignoring that $w_x(z=x)=\infty$): Then the r.~h.~s. of $\eqref{eq:form_cont}_\beta$ reduces to $|y-x|^{\kappa + |\beta| - |\gamma|}$, which expresses a modelled continuity condition of order $\kappa + |\beta|$ provided $\gamma$ is graded by $|\gamma|$.
\begin{proof}
We start with the argument for formula (\ref{eq:formula_form_cont}): Applying
one of the commuting derivations $D\in\{D^{({\bf m})}\}_{\bf m}$ 
to (\ref{exp}) yields by Leibniz' rule the operator identity
\begin{align*}
D\Gamma^*=\Gamma^*D+\sum_{{\bf n}}(D\pi^{({\bf n})})\Gamma^*D^{({\bf n})}\quad\mbox{and thus}\quad
D\Gamma^*Q=\Gamma^*DQ+\sum_{{\bf n}={\bf 0},(1,0)}(D\pi^{({\bf n})})\Gamma^*D^{({\bf n})}Q.
\end{align*}
Using this for $\Gamma^*=\Gamma_{zy}^*$ and then applying $\Gamma^*_{xz}$ to it,
we obtain by multiplicativity (\ref{eq:mult}), 
which we may apply since $D$ and $\Gamma_{zy}^*D^{({\bf n})}$ 
map $\mathsf{T}^*$ into the sub-algebra $\mathsf{\tilde T}^*$ by (\ref{eq:algIII5}), 
and by transitivity
\eqref{ho03} (which can be applied since once more $D$ maps $\mathsf{T}^*$ into $\mathsf{\tilde T}^*$; furthermore we will need only $\eqref{ho03}_{\beta_2}$ for $\beta_2\preccurlyeq\beta$ as we shall see below, which follows from $\eqref{rec01}_{\prec\beta}$) 
\begin{align*}
\Gamma^*_{xz}D\Gamma^*_{zy}Q=\Gamma^*_{xy}DQ
+\sum_{{\bf n}={\bf 0},(1,0)}(\Gamma_{xz}^*D\pi^{({\bf n})}_{zy})\Gamma^*_{xy}D^{({\bf n})}Q.
\end{align*}
We now specify to $D=D^{({\bf m})}$, (left-)multiply by ${\rm d}\pi^{({\bf m})}_{xz}$,
and sum over ${\bf m}={\bf 0},(1,0)$ to obtain by definition (\ref{eq:def_dGamma}),
\begin{align*}
{\rm d}\Gamma^*_{xz}\Gamma^*_{zy}Q
&=\sum_{{\bf m}={\bf 0},(1,0)}{\rm d}\pi_{xz}^{({\bf m})}\Gamma^*_{xy}D^{({\bf m})}Q
+\sum_{{\bf n}={\bf 0},(1,0)}({\rm d}\Gamma_{xz}^*\pi^{({\bf n})}_{zy})
\Gamma^*_{xy}D^{({\bf n})}Q\nonumber\\
&=\sum_{{\bf n}={\bf 0},(1,0)}\big({\rm d}\pi_{xz}^{({\bf n})}
+({\rm d}\Gamma_{xz}^*\pi^{({\bf n})}_{zy})\big)\Gamma^*_{xy}D^{({\bf n})}Q.
\end{align*}
Subtracting this from (\ref{eq:def_dGamma}) (with $z$ replaced by $y$ and multiplied by $Q$
from the right) yields (\ref{eq:formula_form_cont}).

\medskip

We now turn to the estimate (\ref{eq:form_cont}) proper. Since up to the indicator functions, 
the first r.~h.~s.~term is dominated by the second one, it is enough to establish 
(\ref{eq:form_cont}) without the first indicator function $\mathbf{1}_{\gamma(1,0)=0}$. 
Like in the proof of Proposition \ref{prop:delta_gamma_npp}, we distinguish the contributions
${\bf n}={\bf 0}$ and ${\bf n}=(1,0)$.  
In particular, the $({\bf n}={\bf 0})$-contribution to 
$\eqref{eq:formula_form_cont}_{\beta}^\gamma$ gives rise to terms of the form
\begin{align}\label{eq:algIII2}
\big({\rm d}\pi^{({\bf 0})}_{xy}-{\rm d}\pi^{({\bf 0})}_{xz}
-{\rm d}\Gamma_{xz}^*\pi^{({\bf 0})}_{zy}\big)_{\beta_1}
(\Gamma_{xy}^*)_{\beta_2}^{\gamma-e_k+e_{k+1}}
\end{align}
for some $k\ge 0$ and $\beta_1+\beta_2=\beta$. Since by (\ref{eq:algIII7})
and (\ref{eq:algIII6}), the first factor vanishes when $\beta_1$ is purely polynomial,
like in the proof
of Proposition \ref{prop:delta_gamma_npp}, we effectively have $\beta_1\prec\beta$,
$\beta_2\preccurlyeq\beta$, and $\gamma-e_k+e_{k+1}$ not purely polynomial.
By H\"older's inequality in probability space, 
the $\mathbb{E}^\frac{1}{q'}|\cdot|^{q'}$-norm of (\ref{eq:algIII2}) is estimated by 
\begin{align*}
|y-z|^{\kappa+\alpha}(|y-z|+|z-x|)^{|\beta_1|-\alpha}(w_x(y)+w_x(z))\;
|y-x|^{|\beta_2|-|\gamma-e_k+e_{k+1}|}.
\end{align*}
Using $|y-x|\le|y-z|+|z-x|$ on the last factor, and appealing to (\ref{eq:algII1}), 
we see that this terms is contained in the first r.~h.~s.~term of (\ref{eq:form_cont}). 

\medskip

The terms coming from the $({\bf n}=(1,0))$-contribution to
$\eqref{eq:formula_form_cont}_{\beta}^\gamma$ are of the form
\begin{align}\label{eq:algIII}
\big({\rm d}\pi^{(1,0)}_{xy}-{\rm d}\pi^{(1,0)}_{xz}
-{\rm d}\Gamma_{xz}^*\pi^{(1,0)}_{zy}\big)_{\beta_1}
(\Gamma_{xy}^*)_{\beta_2}^{\gamma-e_{(1,0)}}
\end{align}
for some $\beta_1+\beta_2=\beta$. They are only present for $\gamma(1,0)\ge 1$;
the presence of $Q$ on the l.~h.~s.~of (\ref{eq:form_cont}) 
amounts to the restriction to $|\gamma|<2$, which in view of
(\ref{ao12}) only leaves $\gamma(1,0)=1$, 
giving rise to the characteristic function $\mathbf{1}_{\gamma(1,0)=1}$ 
in the second r.~h.~s.~contribution to (\ref{eq:form_cont}).
Again, as in the proof of Proposition \ref{prop:delta_gamma_npp}, 
we effectively have $\beta_1\prec\beta$, $\beta_2\preccurlyeq\beta$, and $\gamma-e_{\bf n}$ 
not purely polynomial.
By H\"older's inequality in probability space, 
the $\mathbb{E}^\frac{1}{q'}|\cdot|^{q'}$-norm of (\ref{eq:algIII}) is estimated by 
\begin{align*}
|y-z|^{\kappa+\alpha-1}(|y-z|+|z-x|)^{|\beta_1|-\alpha}(w_x(y)+w_x(z))\;
|y-x|^{|\beta_2|-|\gamma-e_{(1,0)}|}.
\end{align*}
As before, this time appealing to (\ref{eq:algII2}), 
we see that this term is contained in the second r.~h.~s.~term of (\ref{eq:form_cont}). 
\end{proof}

\medskip

The second task is to estimate the rough-path
increments of $\delta\Pi_{x}^{-}$ based on the estimate of the rough-path
increments of $\delta\Pi_{x}$ stated in Proposition \ref{prop:intIII}.
It relies on the $c$-free formula
\begin{align}\label{magic}
Q\big(\delta\Pi^-_x-{\rm d}\Gamma^*_{xz}Q\Pi^-_z\big)(z)
=Q\sum_{k\geq0}\mathsf{z}_k\Pi_x^k(z)
\partial_1^2\big(\delta\Pi_x-{\rm d}\Gamma^*_{xz}Q\Pi_z\big)(z)
+ \delta\xi_\tau(z)\mathsf{1},
\end{align}
which is the analogue of (\ref{eq:Pi_minus_anchor}), and the argument for
which will be given in the proof of Proposition \ref{prop:recIII}. 
In order to actually pass from $\eqref{eq:delta_pi_incr_generic}_{\prec\beta}$ to
$\eqref{eq:delta_pi_minus_generic}_\beta$, we need to free ourselves from the evaluation
at $z$ in $\eqref{magic}_\beta$, which will be done by a reconstruction argument.
Reconstruction requires that the sum of $\alpha$ 
(the bare regularity of the first factor $\Pi_{x}$)
and of $\kappa+\alpha$ (the degree of modeledness of the second factor $\delta\Pi_{x}$)
is larger than 2 (due to the presence of the second spatial derivatives). 
This enforces the lower bound 
\begin{align}\label{ao24}
\kappa>2-2\alpha,
\end{align}
which together with the upper bound (\ref{ao25})
is the sole reason for our assumption $\alpha>\frac{1}{4}$.
Incidentally, an identity analogous to (\ref{magic}) would hold for the non-centered model $\Pi$,
i.~e.~for the first base point $x$ omitted, and would presumably allow for reconstruction.
However, it would not allow us to derive the estimates of the right homogeneity, but rather the plain H\"older
estimates quite similar to \cite{Ha18}.

\medskip

In terms of $Q$,
there is a mismatch between the output of Proposition \ref{prop:algIII}
and the ideal input for the upcoming Proposition \ref{prop:recIII}. 
Handling the mismatch requires estimating ${\rm d}\Gamma_{xz}^*({\rm id}-Q)$,
which follows from the boundedness -- as opposed to continuity -- of
${\rm d}\Gamma_{xz}^*$ (see (\ref{eq:form_bound}) in Subsection \ref{forthblock}, 
where it plays a more important role).
Analogously to the second task of Subsection \ref{firstblock} we obtain

\begin{proposition}[Reconstruction III]\label{prop:recIII}
Assume that $\eqref{eq:pi_generic}_{\prec\beta}$, $\eqref{eq:gamma}_{\prec\beta}$, $\eqref{cw60_cont}_{\prec\beta}$, $\eqref{cw60_minus_cont}_{\prec\beta}$, $\eqref{eq:recenter_Pi_minus}_{\prec \beta}$, $\eqref{eq:recenter_Pi_specific}_{\prec\beta}$, 
$\eqref{eq:pi_minus_generic}_{\prec\beta}$, $\eqref{cw60_cont_mal_dual}_{\prec\beta}$, $\eqref{cw60_minus_cont_mal}_\beta$, and $\eqref{eq:delta_pi_incr_generic}_{\prec\beta}$  hold, 
assume that 
$\eqref{eq:form_cont}_{\preccurlyeq\beta}^\gamma$ and 
$\eqref{eq:form_bound}_{\preccurlyeq\beta}^\gamma$ hold, 
both for all $\gamma$ not purely polynomial.
Then we have
\begin{align}\label{eq:delta_pi_minus_generic}
\lefteqn{\mathbb{E}^\frac{1}{q'}|(\delta \Pi^-_{x}-{\rm d}\Gamma^*_{xz}Q\Pi^-_{z})_{\beta t}(y)|^{q'}
}\nonumber\\
&\lesssim(\sqrt[4]{t})^{\alpha-2} (\sqrt[4]{t}+|y-z|)^{\kappa}
(\sqrt[4]{t}+|y-z|+|z-x|)^{|\beta|-\alpha}(w_x(z) + w_x(y)).
\end{align}
\end{proposition}

We remark that these three exponents are natural: The first exponent $\alpha-2$
captures the bare (distributional) regularity of $\delta \Pi^-_{x}-{\rm d}\Gamma^*_{xz} \Pi^-_{z}$
at an arbitrary point, which is not better than the one of $\delta \Pi^-_{x}$ or
$\Pi^{-}_z$, and does not depend on $\beta$. If $y=z$, the sum $\kappa+\alpha-2$
of the two first exponents emerges and describes the regularity of the
expression $\delta \Pi^-_{x}-{\rm d}\Gamma^*_{xz} \Pi^-_{z}$ near the secondary base point $z$,
which does not depend on $\beta$;
it makes the gain of $\kappa$ appear that arises from the weight in $w(z)$.
Finally, the sum $\kappa+|\beta|-2$ is dictated by scaling; 
passing from $w(z)$ back to $\bar w$ removes a length to the power $\kappa$,
leads to $|\beta|-2$, in line with (\ref{eq:mal_dual_annealed}) and ultimately
(\ref{eq:pi_minus_generic}). 

\medskip

As in the proof of Proposition~\ref{prop:reconstr_reg}, for the reconstruction 
of the rough-path increments of $\delta\Pi_x^-$ we require (weak) continuity in the active variable. 
Therefore, we need -- in a purely qualitative way -- the analogue of Proposition~\ref{rem:1} 
on the level of Malliavin derivatives, 
the proof of which follows along the same lines and is postponed to Section~\ref{sec:proof_rem}.

\begin{proposition}[Divergent bounds II]\label{rem:1_mal} Under Assumption~\ref{ass:spectral_gap} the following holds for every populated $|\beta|<2$:
\begin{equation}
 \mathbb{E}^\frac{1}{q'}|\partial_1^2\delta\Pi_{x\beta}(y)|^{q'} + \mathbb{E}^\frac{1}{q'}|\partial_2\delta\Pi_{x\beta}(y)|^{q'}
 \lesssim(\sqrt[4]{\tau})^{\alpha-2}(\sqrt[4]{\tau}+|y-x|)^{|\beta|-\alpha} \bar w.  \label{cw60_mal_dual}
\end{equation}
Furthermore, we have 
\begin{align}
&\mathbb{E}^\frac{1}{q'}|\partial_1^2\delta\Pi_{x\beta}(y)-\partial_1^2\delta\Pi_{x\beta}(z)|^{q'}
+\mathbb{E}^\frac{1}{q'}|\partial_2\delta\Pi_{x\beta}(y) - \partial_2\delta\Pi_{x\beta}(z)|^{q'}
\label{cw60_cont_mal_dual}\\
&+\mathbb{E}^\frac{1}{q'}|\delta\Pi_{x\beta}^-(y) - \delta\Pi_{x\beta}^-(z)|^{q'}\label{cw60_minus_cont_mal}\\
&\lesssim(\sqrt[4]{\tau})^{-2}(\sqrt[4]{\tau}+|y-x|+|z-x|)^{|\beta|-\alpha} |y-z|^\alpha \bar w.\nonumber
\end{align}
\end{proposition}

By duality, \eqref{cw60_cont_mal_dual} implies an annealed and weighted $C^{2,\alpha}$-estimate for the Malliavin derivate 
$\frac{\partial}{\partial \xi}\Pi_{x\beta}$. Kolmogorov's criterion\footnote{for random fields taking values in a Hilbert space} 
then ensures that $\frac{\partial}{\partial \xi}\Pi_{x\beta}\in C^2(H^*)$ almost surely, 
where $H^*$ denotes the Hilbert space with norm
$\|\cdot\|_*$ defined in \eqref{as02}. Likewise, we have 
$\frac{\partial}{\partial \xi}\Pi_{x\beta}^{-}$ $\in C^0(H^*)$.
This justifies evaluating 
$\delta\Pi^-_{x}$ and $\partial_1^2\delta\Pi_x$ in a point $z$, as done in (\ref{magic}).

\begin{proof}[Proof of Proposition \ref{prop:recIII}]
We start with the proof of (\ref{magic}).
On the one hand, we apply $Q$ to \eqref{eq:Pi_minus_anchor}; on the other hand, we take the Malliavin derivative of (\ref{eq:Pi_minus_def_alt}), see Subsection~\ref{sec:reconstr_approx}:
\begin{align}
& Q\Pi_z^{-}(z)=Q\big(P\mathsf{z}_0\partial_1^2\Pi_z-c+\xi_\tau\mathsf{1}\big)(z),\label{eq:recIII27}
\\
& \delta\Pi_x^{-}=P\big(\sum_{k\ge0}\mathsf{z}_k\Pi_x^k\partial_1^2\delta\Pi_x
+\sum_{k\ge 0}(k+1)\mathsf{z}_{k+1}\Pi_x^k\delta\Pi_x\partial_1^2\Pi_x\big)
-\sum_{k\ge 0}\frac{1}{k!}\Pi_x^k\delta\Pi_x(D^{({\bf 0})})^{k+1}c
+\delta\xi_{\tau}\mathsf{1}.\label{eq:recIII30}
\end{align}
Since by definition (\ref{ao12}),
the $\beta(k=0)$-component has no effect on $|\beta|$, we have $QP\mathsf{z}_0=P\mathsf{z}_0Q$, so that in view of
\begin{align}\label{eq:recIII20}
	Q\partial_1^2\Pi_x\in\mathsf{\tilde T}^*,
\end{align}
which follows from \eqref{eq:pi_purely_pol} and \eqref{ao20}, $P$ is inactive in (\ref{eq:recIII27}).
Moreover, by the first item in (\ref{eq:c_pop_cond}), $Q$ is inactive on $c$
and of course on $\mathsf{1}$.  
By (\ref{eq:algIII7}) and (\ref{eq:recIII29}), $P$ is also inactive in (\ref{eq:recIII30}).
Hence we learn (using also ${\rm d}\Gamma_{xz}^*\mathsf{1}=0$) 
that (\ref{magic}) follows from
\begin{align}
Q{\rm d}\Gamma^*_{xz}\mathsf{z}_0Q\partial_1^2\Pi_z(z)
&=Q\Big(\sum_{k\geq0}\mathsf{z}_k\Pi_x^k{\rm d}\Gamma^*_{xz}Q\partial_1^2\Pi_z
+\sum_{k\ge 0}(k+1)\mathsf{z}_{k+1}\Pi_x^k\delta\Pi_x\partial_1^2\Pi_x\Big)(z),\label{eq:recIII31}
\\
{\rm d}\Gamma^*_{xz}c
&=\big(\delta\Pi_x\sum_{k\ge 0}\frac{1}{k!}\Pi_x^k(D^{({\bf 0})})^{k+1}c\big)(z).\label{eq:recIII33}
\end{align}
We start by arguing that (\ref{eq:recIII31}) follows from
\begin{align}\label{eq:recIII32}
{\rm d}\Gamma^*_{xz}\mathsf{z}_0\pi'
&=\sum_{k\geq0}\mathsf{z}_k\Pi_x^k(z){\rm d}\Gamma^*_{xz}\pi'
+\delta\Pi_x(z)\sum_{k\ge 0}(k+1)\mathsf{z}_{k+1}\Pi_x^k(z)\Gamma_{xz}^*\pi'\quad
\mbox{for}\;\pi'\in\mathsf{\tilde T}^*.
\end{align}
Indeed, we use (\ref{eq:recenter_Pi_specific}) in form of $\partial_1^2\Pi_x$ 
$=\Gamma_{xz}^*\partial_1^2\Pi_z$ and the triangularity properties \eqref{ord18} 
and \eqref{eq:triangular_Gamma}
w.~r.~t.~\mbox{$|\cdot|$} to see 
$Q\mathsf{z}_{k+1}\Pi_x^k(z)\delta\Pi_x(z)\Gamma_{xz}^*\partial_1^2\Pi_z(z)$
$=Q\mathsf{z}_{k+1}\Pi_x^k(z)\delta\Pi_x(z)\Gamma_{xz}^*Q\partial_1^2\Pi_z(z)$.
Hence (\ref{eq:recIII31}) indeed follows from (\ref{eq:recIII32}) 
for $\pi'=Q\partial_1^2\Pi_z(z)$.
By (\ref{eq:Pipi}) and the second item in
(\ref{eq:def_dGamma}) in form of $\delta\Pi_x(z)={\rm d}\pi^{({\bf 0})}_{xz}$,
(\ref{eq:recIII33}) and (\ref{eq:recIII32}) take the form of
\begin{align}
{\rm d}\Gamma^*_{xz}\mathsf{z}_0\pi'
&=\sum_{k\geq0}\mathsf{z}_k(\pi^{({\bf 0})}_{xz})^k{\rm d}\Gamma^*_{xz}\pi'
+{\rm d}\pi^{({\bf 0})}_{xz}\sum_{k\ge 0}(k+1)\mathsf{z}_{k+1}
(\pi^{({\bf 0})}_{xz})^k\Gamma^*_{xz}\pi',\label{eq:recIII34}\\
{\rm d}\Gamma^*_{xz}c
&={\rm d}\pi^{({\bf 0})}_{xz}\sum_{k\ge 0}\frac{1}{k!}
(\pi^{({\bf 0})}_{xz})^k(D^{({\bf 0})})^{k+1}c.\label{eq:recIII35}
\end{align}
Because of the second item in the population condition (\ref{eq:c_pop_cond}),
which we may rewrite as $D^{({\bf n})}c=0$ for ${\bf n}\not={\bf 0}$,
it follows from (\ref{exp}) that
\begin{align}\label{eq:Gamma_D_c}
\Gamma^*(D^{({\bf 0})})^{k'}c
=\sum_{k\ge 0}\frac{1}{k!}(\pi^{({\bf 0})})^k(D^{({\bf 0})})^{k+k'}c.
\end{align}
Together with (\ref{eq:Gamma_z_k}) for $k=0,1$, we see that the two identities
(\ref{eq:recIII34}) \& (\ref{eq:recIII35})
can be written as
\begin{align*}
{\rm d}\Gamma^*_{xz}\mathsf{z}_0\pi'
=(\Gamma^*_{xz}\mathsf{z}_0){\rm d}\Gamma^*_{xz}\pi'
+{\rm d}\pi^{({\bf 0})}_{xz}(\Gamma^*_{xz}\mathsf{z}_1)\Gamma^*_{xz}\pi'
\quad\mbox{and}\quad{\rm d}\Gamma^*_{xz}c
={\rm d}\pi^{({\bf 0})}_{xz}\Gamma^*_{xz}D^{({\bf 0})}c.
\end{align*}
By definition (\ref{eq:def_dGamma}), the first identity follows from the fact that
$D^{({\bf n})}$ is a derivation, mapping $\mathsf{T}^*$ into $\mathsf{\tilde T}^*$, 
see (\ref{eq:algIII5}), from how it acts on $\mathsf{z}_k$, see (\ref{eq:def_Dnull}),
and the multiplicativity (\ref{eq:mult}) of $\Gamma_{xz}^*$ together with the fact that 
$\mathsf{\tilde T}^*$ is closed under multiplication. The second identity follows by definition (\ref{eq:def_dGamma})
using once more that $D^{(1,0)}c=0$.

\medskip

Before starting with the estimates, we note that not only $w(y)$, cf.~(\ref{eq:3ptIV2}),
but also $w_x(y)$ behaves well under (square) averaging in $y$, 
in particular by $(\sqrt[4]{t})^{|{\bf n}|}|\partial^{\bf n}\psi_t|$:
\begin{align}\label{eq:intIII9}
(\sqrt[4]{t})^{|{\bf n}|}\int dy'|\partial^{\bf n}\psi_t(y-y')|w_x^2(y')\lesssim w_x^2(z)
\quad\mbox{provided}\;|y-z|\le\sqrt[4]{t}.
\end{align}
Indeed, in view of the definitions (\ref{eq:weights}) and (\ref{eq:def_w_x_z}), (\ref{eq:intIII9})
follows from the elementary fact that also negative moments are preserved by 
averaging with $(\sqrt[4]{t})^{|{\bf n}|}|\partial^{\bf n}\psi_t|$:
\begin{align}\label{eq:3ptIII10}
(\sqrt[4]{t})^{|{\bf n}|}\int dy'|\partial^{\bf n}\psi_t(y-y')| |x'-y'|^{-2\kappa}
&\lesssim (\sqrt[4]{t}+|x'-y|)^{-2\kappa}\nonumber\\
&\sim (2\sqrt[4]{t}+|x'-y|)^{-2\kappa}
\stackrel{|y-z|\le\sqrt[4]{t}}{\lesssim} (\sqrt[4]{t}+|x'-z|)^{-2\kappa}.
\end{align}

\medskip

We now introduce the family $\{F_{xz}\}_{x,z}$ of random space-time Schwartz distributions
\begin{align}\label{eq:recIII2}
F_{xz}:=\big(\delta\Pi_x^{-}-{\rm d}\Gamma_{xz}^*Q\Pi_z^{-}\big)
-\Big(\sum_{k\ge0}\mathsf{z}_k\Pi_x^{k}(z)
\partial_1^2\big(\delta\Pi_x-{\rm d}\Gamma_{xz}^*Q\Pi_z\big)+\delta\xi_\tau\mathsf{1}\Big),
\end{align}
and note that (\ref{magic}) amounts to
\begin{align}\label{eq:recIII16}
F_{xz\beta}(z)=0.
\end{align}
In fact, we need \eqref{eq:recIII16} in the more robust form 
\begin{equs}
\lim_{t\downarrow 0}\E^{\frac{1}{q'}}|F_{xz\beta \, t} (z)|^{q'} = 0. \label{eq:recIII16_weak}
\end{equs}
In order to pass from (\ref{eq:recIII16}) to \eqref{eq:recIII16_weak} 
it is sufficient to argue that the $\beta$-component of the r.~h.~s.~of \eqref{eq:recIII2}
is continuous in the active variable w.~r.~t.~$\E^{\frac{1}{q'}}|\cdot|^{q'}$, 
which we do term by term.
The continuity of $\delta\Pi_{x\beta}^-$ is stated 
in $\eqref{cw60_minus_cont_mal}_\beta$; the continuity
of $\delta\xi_\tau$ amounts to $\eqref{cw60_minus_cont_mal}_{\beta=0}$. 
Appealing to \eqref{eq:triangular_d_Gamma}, the continuity of 
$({\rm d}\Gamma_{xz}^*Q\Pi_z^{-})_\beta$ follows from $\eqref{eq:form_bound}_\beta^{\gamma\neq \pp}$ 
and $\eqref{cw60_minus_cont}_{\prec\beta}$. 
For the term $\big(\sum_{k\ge0}\mathsf{z}_k\Pi_x^{k}(z)
\partial_1^2(\delta\Pi_x-{\rm d}\Gamma_{xz}^*Q\Pi_z)\big)_\beta$ 
we appeal to \eqref{ord24} \& \eqref{eq:triangular_product}. Hence, for the term involving $\partial_1^2\delta\Pi_x$, continuity follows
from boundedness $\eqref{eq:pi_generic}_{\prec\beta}$
and the continuity $\eqref{cw60_cont_mal_dual}_{\prec\beta}$.
For the term involving $\partial_1^2({\rm d}\Gamma_{xz}^*Q\Pi_z)$,
continuity follows from boundedness $\eqref{eq:pi_generic}_{\prec\beta}$, combined with
\eqref{eq:triangular_d_Gamma} \& $\eqref{eq:form_bound}_{\beta'}^{\gamma\neq \pp}$ and 
the continuity $\eqref{cw60_cont}_{\prec \beta}$.

\medskip

We shall establish the following continuity condition of $\{F_{xz}\}_{x,z}$
in the secondary base point $z$:
\begin{align}\label{eq:recIII15}
&\mathbb{E}^\frac{1}{q'}|(F_{xy}-F_{xz})_{\beta t}(y)|^{q'}\nonumber\\
&\lesssim(\sqrt[4]{t})^{\alpha-2}(\sqrt[4]{t}+|y-z|)^{\theta-\alpha}
(\sqrt[4]{t}+|y-z|+|z-x|)^{|\beta|+\kappa-\theta}(w_x(z)+w_x(y)),
\end{align}
where
\begin{align}\label{eq:recIII21}
\theta:=\min\{\kappa+2\alpha,\inf(\mathsf{A}\cap(2,\infty))\}
\stackrel{\eqref{ao24},\;\mathsf{A}\;\text{is locally finite}}{>}2.
\end{align}
Before proceeding, we note that the l.~h.~s.~of (\ref{eq:recIII15}) vanishes 
for $\beta\in\mathbb{N}_0e_0$.
Indeed, we start by observing that $k\not=0$ cannot contribute to such a component 
of (\ref{eq:recIII2}). Because of the difference on the l.~h.~s.~of (\ref{eq:recIII15}), 
only the contributions involving ${\rm d}\Gamma_{x\cdot}^*$ contribute.
Hence in view of (\ref{eq:recIII20}), the claim follows from \eqref{eq:dGamma_e0}.
In particular, by \eqref{eq:beta_e0}, 
the l.~h.~s.~of (\ref{eq:recIII15}) vanishes unless $|\beta|\ge2\alpha$,
which implies for the last exponent in (\ref{eq:recIII15})
that $|\beta|+\kappa-\theta\ge 0$ by definition (\ref{eq:recIII21}) of $\theta$.

\medskip

Since in view of (\ref{eq:recIII21}),
the sum of the first two exponents is positive, (\ref{eq:recIII15}) 
provides a continuity condition of positive order. Using the reconstruction argument in \cite[Lemma 4.8]{LO22} in combination with the averaging property of the weight \eqref{eq:intIII9}, \eqref{eq:recIII16_weak} 
and \eqref{eq:recIII15} upgrade to  
\begin{align}\label{ho27}
\E^\frac{1}{q'} |F_{xz\beta \,t}(z)|^{q'}
\lesssim (\sqrt[4]{t})^{\theta - 2}(\sqrt[4]{t} + |z-x|)^{|\beta|+\kappa -\theta} w_x(z).
\end{align}
Using once more \eqref{eq:recIII15}, this yields
\begin{align*}
\mathbb{E}^\frac{1}{q'}|F_{xz\beta \,t}(y)|^{q'}
&\lesssim
(\sqrt[4]{t})^{\alpha-2}(\sqrt[4]{t}+|y-z|)^{\theta-\alpha}
(\sqrt[4]{t}+|y-z|+|z-x|)^{|\beta|+\kappa-\theta}(w_x(z)+w_x(y)),
\end{align*}
which we just use in the weakened form of 
(since $\kappa\le\theta-\alpha$ 
by (\ref{ao25}) and (\ref{eq:recIII21}))
\begin{align*}
\mathbb{E}^\frac{1}{q'}|F_{xz\beta \,t}(y)|^{q'}
\lesssim(\sqrt[4]{t})^{\alpha-2}(\sqrt[4]{t}+|y-z|)^{\kappa}
(\sqrt[4]{t}+|y-z|+|z-x|)^{|\beta|-\alpha}(w_x(z)+w_x(y)).
\end{align*}
Once we establish the boundedness of the second contribution to (\ref{eq:recIII2}) in form of
\begin{align}\label{eq:recIII17}
\lefteqn{\mathbb{E}^\frac{1}{q'}\big|\big(\sum_{k\ge0}\mathsf{z}_k\Pi_x^{k}(z)
\partial_1^2\big(\delta\Pi_x-{\rm d}\Gamma_{xz}^*Q\Pi_z\big)
+\delta\xi_\tau\mathsf{1}\big)_{\beta \,t}(y)\big|^{q'}}\nonumber\\
&\lesssim(\sqrt[4]{t})^{\alpha-2}(\sqrt[4]{t}+|y-z|)^{\kappa}
(\sqrt[4]{t}+|y-z|+|z-x|)^{|\beta|-\alpha}(w_x(z)+w_x(y)),
\end{align}
we obtain the desired (\ref{eq:delta_pi_minus_generic}). 

\medskip

The remainder of the proof
is devoted to the estimates (\ref{eq:recIII15}) and (\ref{eq:recIII17}).
By the triangle inequality, we split (\ref{eq:recIII15}) into
\begin{align}
\lefteqn{\mathbb{E}^\frac{1}{q'}|({\rm d}\Gamma_{xy}^*Q\Pi_y^{-}
-{\rm d}\Gamma_{xz}^*Q\Pi_z^{-})_{\beta \,t}(y)|^{q'}}\nonumber\\
&\lesssim
(\sqrt[4]{t})^{\alpha-2}(\sqrt[4]{t}+|y-z|)^{\theta-\alpha}
(\sqrt[4]{t}+|y-z|+|z-x|)^{|\beta|+\kappa-\theta}(w_x(z)+w_x(y)),\label{eq:recIII6}\\
\lefteqn{\mathbb{E}^\frac{1}{q'}\big|\Big(\mathsf{z}_k\big(
 \Pi_x^k(y)\partial_1^2(\delta\Pi_x-{\rm d}\Gamma_{xy}^*Q\Pi_y)
-\Pi_x^k(z)\partial_1^2(\delta\Pi_x-{\rm d}\Gamma_{xz}^*Q\Pi_z)\big)\Big)_{\beta \,t}(y)\big|^{q'}
}\nonumber\\
&\lesssim
(\sqrt[4]{t})^{\alpha-2}(\sqrt[4]{t}+|y-z|)^{\theta-\alpha}
(\sqrt[4]{t}+|y-z|+|z-x|)^{|\beta|+\kappa-\theta}(w_x(z)+w_x(y)).\label{eq:recIII14}
\end{align}
Note that as opposed to (\ref{eq:recIII15}), (\ref{eq:recIII17})
does see $\delta\xi$, which we thus split into
\begin{align}\label{eq:recIII19}
\lefteqn{\mathbb{E}^\frac{1}{q'}\big|\big(\mathsf{z}_k\Pi_x^{k}(z)
\partial_1^2(\delta\Pi_x-{\rm d}\Gamma_{xz}^*Q\Pi_z)\big)_{\beta \,t}(y)\big|^{q'}}\nonumber\\
&\lesssim(\sqrt[4]{t})^{\alpha-2}(\sqrt[4]{t}+|y-z|)^{\kappa}
(\sqrt[4]{t}+|y-z|+|z-x|)^{|\beta|-\alpha}(w_x(z)+w_x(y))
\end{align}
and the base case, meaning (\ref{eq:recIII17}) for $\beta=0$.
Due to the presence
of $\mathsf{z}_k$, the l.~h.~s.~of $\eqref{eq:recIII17}_{\beta=0}$ collapses to $(\delta\xi_\tau)_{t}(y)$;
we shall establish the stronger version of (\ref{eq:base1}).
By the semi-group property \eqref{eq:semi_group} and $\alpha-2+\kappa<0$, see \eqref{ao25},
it is enough to establish \eqref{eq:base1} for $\tau=0$.
In order to do so, we appeal again to the semi-group property
in form of $\delta\xi_t(y)$ $=\int dz\psi_{t-s}(y-z)\delta\xi_s(z)$
for $s\in(0,t)$,
so that by the triangle inequality w.~r.~t.~$\mathbb{E}^\frac{1}{q}|\cdot|^q$
we have $\mathbb{E}^\frac{1}{q}|\delta\xi_t(y)|^q$ $\le\int dz|\psi_{t-s}(y-z)|
\mathbb{E}^\frac{1}{q}|\delta\xi_s(z)|^q$, and thus by Cauchy-Schwarz in $z$
(and in view of the obvious sup-bound $|\psi_{t-s}| \lesssim(\sqrt[4]{t-s})^{-3}$)
\begin{align*}
\mathbb{E}^\frac{2}{q}|\delta\xi_t(y)|^q
&\stackrel{\hphantom{\eqref{cw61}}}{\leq}\int dz\psi_{t-s}^2(y-z)|z-y|^{2\kappa}
\int dz|z-y|^{-2\kappa}\mathbb{E}^\frac{2}{q}|\delta\xi_s(z)|^q\nonumber\\
&\stackrel{\eqref{cw61}}{\lesssim}(\sqrt[4]{t-s})^{-3+2\kappa}\int dz|z-y|^{-2\kappa}
\mathbb{E}^\frac{2}{q}|\delta\xi_s(z)|^q.
\end{align*}
Applying $\int_{\frac{t}{4}}^{\frac{t}{2}}\frac{ds}{s}$, so that
$(\sqrt[4]{t-s})^{-3+2\kappa}$ $\sim(\sqrt[4]{t})^{2(\alpha-2+\kappa)}$
$(\sqrt[4]{s})^{2(\frac{1}{2}-\alpha)}$, we obtain the square of (\ref{eq:base1}) with $\tau=0$ by definition
(\ref{eq:weights}).

\medskip

Turning to (\ref{eq:recIII19}) we note that the l.~h.~s.
has the product structure $(\mathsf{z}_k\pi\pi')_\beta$
$=\sum_{\beta_1+\beta_2=\beta}(\mathsf{z}_k\pi)_{\beta_1}\pi'_{\beta_2}$.
In view of the presence of $\mathsf{z}_k$, $\beta_1$ is neither purely polynomial
nor $0$; by (\ref{eq:algIII7}) and (\ref{eq:algIII6}) also $\beta_2$ is not purely polynomial.
Hence by (\ref{eq:recIII25}) we have $\beta_1\preccurlyeq\beta$ and $\beta_2\prec\beta$.
Therefore $\eqref{eq:recIII19}_\beta$ follows from the two estimates 
$\eqref{eq:intIII11}_{\prec\beta}$ and $\eqref{eq:recIII12}_{\preccurlyeq\beta}$
stated below via H\"older's inequality in probability and (\ref{eq:hom_add}).

\medskip

We turn to (\ref{eq:recIII6}), which relies on $\eqref{eq:recenter_Pi_minus}_{\prec\beta}$ in form of
$Q\Pi^{-}_z=Q\Gamma_{zy}^*\Pi^{-}_y$. In preparation for the use of (\ref{eq:form_cont}), we use the triangularity property \eqref{triangle} in form of
\begin{align}\label{eq:recIII4a}
	Q\Gamma^*=Q\Gamma^*Q
\end{align} 
to split the increment:
\begin{align}\label{eq:recIII5}
({\rm d}\Gamma_{xy}^*Q\Pi_y^{-}-{\rm d}\Gamma_{xz}^*Q\Pi_z^{-})_t(y)
\stackrel{\eqref{eq:recIII4a}}{=}
({\rm d}\Gamma_{xy}^*-{\rm d}\Gamma_{xz}^*\Gamma_{zy}^*)Q\Pi_{y\,t}^{-}(y)
+{\rm d}\Gamma_{xz}^*({\rm id}-Q)\Gamma_{zy}^*Q\Pi_{y\,t}^{-}(y).
\end{align}
By the triangular structure \eqref{eq:triangular_Gamma} of $\Gamma^*$,
and the strict triangular structure \eqref{eq:triangular_d_Gamma} of ${\rm d}\Gamma^*$ and \eqref{eq:dGamma_inc_triangular} of its increments, 
we learn that in order to estimate $\eqref{eq:recIII5}_\beta$,
we only need $\eqref{eq:pi_minus_generic}_{\prec\beta}$ and $\eqref{eq:gamma}_{\prec\beta}$.
Because of the presence of $P$ in the definition (\ref{eq:Pi_minus_def_alt})
and by the consistency \eqref{mt44} of elements $\Gamma\in\mathsf{G}$ with the polynomial sector, which for later purpose we write in the compact form
\begin{align}\label{eq:recIII4b}
	\Gamma^*P=P\Gamma^*P,
\end{align}
we only need 
$\eqref{eq:gamma}_{\prec\beta}^{\gamma\neq\pp}$,
$\eqref{eq:form_cont}_{\beta}^{\gamma\neq\pp}$, and
$\eqref{eq:form_bound}_{\beta}^{\gamma\neq \pp}$. We thus have term-by-term
\begin{align}
\lefteqn{\mathbb{E}^\frac{1}{q'}|({\rm d}\Gamma_{xy}^*Q\Pi_y^{-}
-{\rm d}\Gamma_{xz}^*Q\Pi_z^{-})_{\beta \,t}(y)|^{q'}}\nonumber\\
&\lesssim\sum_{\substack{|\gamma|\in\mathsf{A}\cap[\alpha,|\beta|-\alpha] \\ \gamma(1,0)=0}}
|y-z|^{\kappa+\alpha}(|y-z|+|z-x|)^{|\beta|-|\gamma|-\alpha}
(w_x(z)+w_x(y))(\sqrt[4]{t})^{|\gamma|-2}\nonumber\\
&+\sum_{\substack{|\gamma|\in\mathsf{A}\cap[\alpha,|\beta|-\alpha+1] \\ \gamma(1,0)=1}}
|y-z|^{\kappa+\alpha-1}(|y-z|+|z-x|)^{|\beta|-|\gamma|-\alpha+1}
(w_x(z)+w_x(y))(\sqrt[4]{t})^{|\gamma|-2}\nonumber\\
&+\sum_{|\gamma'|\in\mathsf{A}\cap[2,\kappa+|\beta|]}
|z-x|^{\kappa+|\beta|-|\gamma'|}w_x(z)
\sum_{|\gamma|\in\mathsf{A}\cap[\alpha,2]}|y-z|^{|\gamma'|-|\gamma|}(\sqrt[4]{t})^{|\gamma|-2}
\nonumber\\
&\lesssim\Big(\sum_{|\gamma|\in\mathsf{A}\cap[\alpha,|\beta|-\alpha]}
(\sqrt[4]{t})^{|\gamma|-2}(\sqrt[4]{t}+|y-z|)^{\kappa+\alpha}
(\sqrt[4]{t}+|y-z|+|z-x|)^{|\beta|-|\gamma|-\alpha}\label{eq:recIII22}\\
&+\sum_{\substack{|\gamma|\in\mathsf{A}\cap[\alpha,|\beta|-\alpha+1] \\ \gamma(1,0)=1}}
(\sqrt[4]{t})^{|\gamma|-2}
(\sqrt[4]{t}+|y-z|)^{\kappa+\alpha-1}
(\sqrt[4]{t}+|y-z|+|z-x|)^{|\beta|-(|\gamma|-1)-\alpha}\label{eq:recIII23}\\
&+\sum_{|\gamma|\in\mathsf{A}\cap[\alpha,2]}
(\sqrt[4]{t})^{|\gamma|-2}\sum_{|\gamma'|\in\mathsf{A}\cap[2,\kappa+|\beta|]}
(\sqrt[4]{t}+|y-z|)^{|\gamma'|-|\gamma|}
(\sqrt[4]{t}+|y-z|+|z-x|)^{\kappa+|\beta|-|\gamma'|}
\Big)\label{eq:recIII24}\\
&\quad\times(w_x(z)+w_x(y)).\nonumber
\end{align}
The term (\ref{eq:recIII22}) is absorbed into the r.~h.~s.~of (\ref{eq:recIII6}) because 
the first exponent decreases from $|\gamma|-2$ to $\alpha-2$, because the sum of the first two 
exponent decreases from $|\gamma|-2+\kappa+\alpha$ to $\theta-2$ by definition 
(\ref{eq:recIII21}) of $\theta$, and because the sum $|\beta|-2+\kappa$ of the three 
exponents agrees.
The term (\ref{eq:recIII23}) is also absorbed because
once more, the first exponent decreases and the sum of all three exponents agrees,
and because for our not purely polynomial $\gamma$, the constraint $\gamma(1,0)=1$ 
implies $|\gamma|\ge1+\alpha$ by definition (\ref{ao12}) of $|\cdot|$, 
leading to $|\gamma|-2+\kappa+\alpha-1\ge\theta-2$ on the sum of the first two exponents.
Finally, the term (\ref{eq:recIII24}) is also absorbed, since for the not purely polynomial 
$\gamma'$, the constraint $|\gamma'|\ge 2$ implies $|\gamma'|>2$ by (\ref{irrational})
and thus $|\gamma'|\ge\theta$ by definition (\ref{eq:recIII21}) of $\theta$, 
leading once more to $|\gamma'|-2\ge\theta-2$.

\medskip

We now turn to (\ref{eq:recIII14}).
As usual, we will estimate this difference of a product by the sum of two products, 
where each summand contains a difference as one of its factors.
For the same reason as for (\ref{eq:recIII19}) we have the structure $(\mathsf{z}_k\pi\pi')_\beta$ 
$=\sum_{\beta_1+\beta_2=\beta}(\mathsf{z}_k\pi)_{\beta_1}\pi'_{\beta_2}$ with
$\beta_1\preccurlyeq\beta$ and $\beta_2\prec\beta$. 
Hence we see that $\eqref{eq:recIII14}_\beta$ follows from the four estimates
$\eqref{eq:recIII13}_{\prec\beta}$, 
$\eqref{eq:recIII10}_{\preccurlyeq\beta}$, 
$\eqref{eq:intIII11}_{\prec\beta}$, and $\eqref{eq:recIII12}_{\preccurlyeq\beta}$,
stated below,
by the triangle inequality and H\"older's inequality in probability
and (\ref{eq:hom_add}). 

\medskip

As the first ingredient to (\ref{eq:recIII14}), we estimate the components $\prec\beta$ of
${\rm d}\Gamma_{xy}^*Q\partial_1^2\Pi_y-{\rm d}\Gamma_{xz}^*Q\partial_1^2\Pi_z$.
In view of (\ref{eq:recIII20}), the argument is similar to $\eqref{eq:recIII6}$. 
More precisely, considering\footnote{where $\beta$ here is generic and in fact
will be applied to a preceding multi-index}
\begin{align}
\lefteqn{\mathbb{E}^\frac{1}{q'}|({\rm d}\Gamma_{xy}^*Q\partial_1^2\Pi_y
-{\rm d}\Gamma_{xz}^*Q\partial_1^2\Pi_z)_{\beta \,t}(y)|^{q'}}\nonumber\\
&\lesssim
(\sqrt[4]{t})^{\alpha-2}(\sqrt[4]{t}+|y-z|)^{\theta-\alpha}
(\sqrt[4]{t}+|y-z|+|z-x|)^{|\beta|+\kappa-\theta}
(w_x(z)+w_y(z)),\label{eq:recIII13}
\end{align}
$\eqref{eq:recIII13}_\beta$ follows from 
$\eqref{eq:pi_generic}_{\prec\beta}$, 
$\eqref{eq:gamma}_{\prec\beta}^{\gamma\neq \pp}$,
$\eqref{eq:form_cont}_{\beta}^{\gamma\neq \pp}$, and
$\eqref{eq:form_bound}_{\beta}^{\gamma\neq\pp}$.
As for (\ref{eq:recIII15}), we note that the l.~h.~s.~of $\eqref{eq:recIII13}_\beta$
vanishes unless $|\beta|\ge2\alpha$
because of (\ref{eq:dGamma_e0}), so that $|\beta|+\kappa-\theta>0$.

\medskip

As the second ingredient to (\ref{eq:recIII14}), 
we need to estimate all $\preccurlyeq\beta$-components of
the increment of the factor $\mathsf{z}_k\Pi_x^k$. Considering for all $k\ge0$
\begin{align}\label{eq:recIII10}
\mathbb{E}^\frac{1}{p}\big|\big(\mathsf{z}_k\Pi_x^k(y)-\mathsf{z}_k\Pi_x^k(z)
\big)_{\beta}\big|^p\lesssim|y-z|^\alpha(|y-z|+|z-x|)^{|\beta|-2\alpha},
\end{align}
we claim that $\eqref{eq:pi_generic}_{\prec\beta}$, $\eqref{eq:gamma}_{\prec\beta}^\gamma$
(including the purely polynomial $\gamma$'s) 
and $\eqref{eq:recenter_Pi_specific}_{\prec\beta}$ imply $\eqref{eq:recIII10}_\beta$.
We still have the implicit understanding that the l.~h.~s.~vanishes unless
$|\beta|-2\alpha\ge 0$. Indeed, in view of (\ref{eq:beta_e0}), $|\beta|-2\alpha<0$ 
implies $\beta\in\mathbb{N}_0e_0$; due to the presence of $\mathsf{z}_k$,
the term $(\mathsf{z}_k\Pi_x^k)_\beta$ vanishes unless $k=0$, 
in which case the l.~h.~s.~obviously vanishes.
In particular, in establishing (\ref{eq:recIII10}), 
we may restrict to $k\ge 1$ and write with the help of 
(\ref{eq:recenter_Pi_specific}) and (\ref{eq:Pipi})
\begin{align*}
\Pi_x^k(y)-\Pi_x^k(z)=\sum_{k'+k''=k-1}\Pi_x^{k'}(y)\Pi_x^{k''}(z)\Gamma_{xz}^*\Pi_z(y),
\end{align*}
so that we obtain componentwise
\begin{align*}
\lefteqn{\big(\mathsf{z}_k\Pi_x^k(y)-\mathsf{z}_k\Pi_x^k(z)
\big)_{\beta}}\nonumber\\
&=\sum_{k'+k''=k-1}\sum_{e_k+\beta_1+\cdots+\beta_k=\beta}
\Pi_{x\beta_1}(y)\cdots\Pi_{x\beta_{k'}}(y)\Pi_{x\beta_{k'+1}}(z)\cdots\Pi_{x\beta_{k-1}}(z)
\sum_{\gamma}(\Gamma_{xz}^*)_{\beta_k}^\gamma\Pi_{z\gamma}(y).
\end{align*}
By (\ref{eq:triangular_product}) we have $\beta_1,\dots,\beta_k\prec\beta$,
and then by (\ref{eq:triangular_Gamma}) also $\gamma\prec\beta$, so that
 $\eqref{eq:pi_generic}_{\prec\beta}$ and $\eqref{eq:gamma}_{\prec\beta}$ 
are indeed sufficient to conclude by H\"older's inequality in probability
\begin{align*}
\lefteqn{\mathbb{E}^\frac{1}{p}\big|\big(\mathsf{z}_k\Pi_x^k(y)-\mathsf{z}_k\Pi_x^k(z)
\big)_{\beta}\big|^p}\nonumber\\
&\lesssim\sum_{k'+k''=k-1}\sum_{e_k+\beta_1+\cdots+\beta_k=\beta}
|y-x|^{|\beta_1|+\cdots+|\beta_{k'}|}|z-x|^{|\beta_{k'+1}|+\cdots+|\beta_{k-1}|}
\sum_{|\gamma|\in\mathsf{A}\cap[\alpha,|\beta_k|]}|z-x|^{|\beta_k|-|\gamma|}|y-z|^{|\gamma|}.
\end{align*}
Using $|y-x|\le|y-z|+|z-x|$ and (\ref{eq:hom_add}), this collapses to (\ref{eq:recIII10}).

\medskip

As the third ingredient to (\ref{eq:recIII14}), and the first ingredient to (\ref{eq:recIII19}),
we estimate the $\prec\beta$-components of
$\partial_1^2(\delta\Pi_x-{\rm d}\Gamma_{xz}^*Q\Pi_z)_t(y)$. 
Introducing\footnote{once more $\beta$ here denotes a generic multi-index}
\begin{align}\label{eq:intIII11}
&\mathbb{E}^\frac{1}{q'}\big|\partial_1^2(\delta\Pi_x
-{\rm d}\Gamma_{xz}^*Q\Pi_z)_{\beta \,t}(y)\big|^{q'}\nonumber\\
&\stackrel{\hphantom{\eqref{eq:recIII21}}}{\lesssim}(\sqrt[4]{t})^{\alpha-2}(\sqrt[4]{t}+|y-z|)^{\kappa}
(\sqrt[4]{t}+|y-z|+|z-x|)^{|\beta|-\alpha}(w_x(z)+w_x(y))\\
&\stackrel{\eqref{eq:recIII21}}{\le}(\sqrt[4]{t})^{\alpha-2}(\sqrt[4]{t}+|y-z|)^{\theta-2\alpha}
(\sqrt[4]{t}+|y-z|+|z-x|)^{|\beta|+\kappa-\theta+\alpha}(w_x(z)+w_x(y)),\nonumber
\end{align}
we claim that $\eqref{eq:delta_pi_incr_generic}_{\beta}$ implies $\eqref{eq:intIII11}_{\beta}$.
Indeed by (\ref{eq:recIII13}) in the weakened form of
\begin{align*}
\lefteqn{\mathbb{E}^\frac{1}{q'}\big|({\rm d}\Gamma_{xy}^*Q\partial_1^2\Pi_y
-{\rm d}\Gamma_{xz}^*Q\partial_1^2\Pi_z)_{\beta \,t}(y)\big|^{q'}}\nonumber\\
&\lesssim
(\sqrt[4]{t})^{\alpha-2}(\sqrt[4]{t}+|y-z|)^{\kappa}
(\sqrt[4]{t}+|y-z|+|z-x|)^{|\beta|-\alpha}(w_x(z)+w_x(y)),
\end{align*}
it is enough to establish (\ref{eq:intIII11}) for $z=y$,
\begin{align}\label{eq:recIII18}
\mathbb{E}^\frac{1}{q'}\big|\partial_1^2(\delta\Pi_x
-{\rm d}\Gamma_{xy}^*Q\Pi_y)_{\beta \,t}(y)\big|^{q'}\lesssim
(\sqrt[4]{t})^{\alpha-2+\kappa}(\sqrt[4]{t}+|y-x|)^{|\beta|-\alpha}w_x(y).
\end{align}
Writing
\begin{align*}
\partial_1^2(\delta\Pi_x-{\rm d}\Gamma_{xy}^*Q\Pi_y)_{\beta \,t}(y)
=\int dy'\partial_1^2\psi_t(y-y')
\big(\delta\Pi_x-\delta\Pi_x(y)-{\rm d}\Gamma_{xy}^*Q\Pi_y)_{\beta}(y'),
\end{align*}
we obtain $\eqref{eq:recIII18}_\beta$
from $\eqref{eq:delta_pi_incr_generic}_{\beta}$
via H\"older's inequality
in $y'$ and the moment bounds (\ref{cw61}) and (\ref{eq:intIII9}) (both with ${\bf n}=(2,0)$ and the latter with $z=y$). 

\medskip

The last ingredient to (\ref{eq:recIII14}), and the second ingredient for (\ref{eq:recIII19}),
is the estimate of the $\preccurlyeq\beta$-components $\mathsf{z}_k\Pi_x^k(y)$.
Assuming just $\eqref{eq:pi_generic}_{\prec\beta}$ 
and with a subset of the arguments for (\ref{eq:recIII10}), 
we obtain $\eqref{eq:recIII12}_{\beta}$, where
\begin{align}\label{eq:recIII12}
\mathbb{E}^\frac{1}{p}\big|\big(\mathsf{z}_k\Pi_x^k(z)\big)_{\beta}|^p
\lesssim|z-x|^{|\beta|-\alpha}.
\end{align}
\end{proof}

The third task is to pass from the estimate (\ref{eq:delta_pi_minus_generic})
of the rough-path increment of
$\delta\Pi_{x}^{-}$ to the estimate of the rough-path increment of $\delta\Pi_{x}$.
The crucial ingredient is the representation
\begin{align}\label{eq:intIII1}
(\delta \Pi_{x}-\delta\Pi_{x}(z)-{\rm d}\Gamma^*_{xz}Q\Pi_{z})_{\beta}
=-\int_0^\infty dt(1-{\rm T}_z^{2})
(\partial_2+\partial_1^2)(\delta \Pi^-_{x}-{\rm d}\Gamma^*_{xz}Q\Pi^-_{z})_{\beta\,t},
\end{align}
which is the analog of \eqref{eq:Pi_minus_to_Pi} in the integration task of Proposition \ref{prop:int_pi_minus}.
Compared to this previous integration task,
there are now three length scales involved, namely $\sqrt[4]{t}$, $|y-z|$, and $|z-x|$.
In the near-field range $\sqrt[4]{t}\le|y-z|$, we split $1-{\rm T}_z^2$;
in the far-field range $\sqrt[4]{t}\ge\max\{|y-z|,|z-x|\}$, we split
$\delta\Pi_x^{-}-{\rm d}\Gamma_{xz}^*Q\Pi_z^{-}$. Only on the intermediate
range $|y-z|\le\sqrt[4]{t}\le\max\{|y-z|,|z-x|\}$ we use the cancellations
in both the Taylor remainder and the rough-path increment.

\begin{proposition}[Integration III]\label{prop:intIII}
Assume that $\eqref{eq:pi_generic}_{\prec\beta}$, $\eqref{eq:pi_minus_generic}_{\prec\beta}$, 
$\eqref{eq:mal_dual_annealed}_\beta$,
$\eqref{eq:delta_pi_generic}_\beta$,
and $\eqref{eq:delta_pi_minus_generic}_\beta$ hold,
and that $\eqref{eq:form_bound}_{\beta}^\gamma$ holds for all 
$\gamma$ not purely polynomial.
Then $\eqref{eq:intIII1}_\beta$ holds and we have
\begin{align}\label{eq:delta_pi_incr_generic}
\mathbb{E}^\frac{1}{q'} 
|(\delta\Pi_{x}-\delta\Pi_x(z)-{\rm d}\Gamma^*_{xz}Q\Pi_{z})_{\beta}(y)|^{q'}
\lesssim |y-z|^{\kappa+\alpha}(|y-z|+|z-x|)^{|\beta|-\alpha}(w_x(z)+w_x(y)).
\end{align}
\end{proposition}

\begin{proof}
We first show that the r.~h.~s.~of \eqref{eq:intIII1} 
is estimated by the r.~h.~s.~of \eqref{eq:delta_pi_incr_generic}.
In preparation for the near-field range $\sqrt[4]{t}\le|y-z|$, we pre-process
$\eqref{eq:delta_pi_minus_generic}_\beta$.
By the semi-group property (\ref{eq:semi_group}) followed by Jensen's inequality 
we have for any random space-time function $f$ that
$\mathbb{E}^\frac{1}{q'}|\partial^{\bf n}f_t(y)|^{q'}$
$\le\int dy'|\partial^{\bf n}\psi_\frac{t}{2}(y-y')\mathbb{E}^\frac{1}{q'}|f_\frac{t}{2}(y')|^{q'}$.
We use this for $f=(\delta \Pi^-_{x}-{\rm d}\Gamma^*_{xz}Q\Pi^-_{z})_{\beta}$,
insert $\eqref{eq:delta_pi_minus_generic}_\beta$ (with $t$ replaced by $\frac{t}{2}$), 
and appeal to H\"older's inequality on the  
the r.~h.~s.~of $\eqref{eq:delta_pi_minus_generic}_\beta$ (with $y'$ replacing $y$),
in order to use both the positive moment bounds (\ref{cw61}) and
the negative moment bounds (\ref{eq:intIII9}) (for $z=y$ and with $t$ replaced by $\frac{t}{2}$).
This leads to
\begin{align}\label{eq:intIII7}
\lefteqn{\mathbb{E}^\frac{1}{q'}|\partial^{\bf n}
(\delta \Pi^-_{x}-{\rm d}\Gamma^*_{xz}Q\Pi^-_{z})_{\beta \,t}(y)|^{q'}
}\nonumber\\
&\lesssim(\sqrt[4]{t})^{\alpha-2-|{\bf n}|} (\sqrt[4]{t}+|y-z|)^{\kappa}
(\sqrt[4]{t}+|y-z|+|z-x|)^{|\beta|-\alpha}(w_x(z) + w_x(y)).
\end{align}
We restrict (\ref{eq:intIII7}), once to $y=z$, and once to the near-field range:
\begin{align*}
\lefteqn{\mathbb{E}^\frac{1}{q'}|\partial^{\bf n}
(\delta \Pi^-_{x}-{\rm d}\Gamma^*_{xz}Q\Pi^-_{z})_{\beta \,t}(z)|^{q'}
\lesssim(\sqrt[4]{t})^{\alpha-2-|{\bf n}|+\kappa}
(\sqrt[4]{t}+|z-x|)^{|\beta|-\alpha}w_x(z),}\\
\lefteqn{\mathbb{E}^\frac{1}{q'}|\partial^{\bf n}
(\delta \Pi^-_{x}-{\rm d}\Gamma^*_{xz}Q\Pi^-_{z})_{\beta \,t}(y)|^{q'}
}\nonumber\\
&\lesssim(\sqrt[4]{t})^{\alpha-2-|{\bf n}|}|y-z|^{\kappa}
(|y-z|+|z-x|)^{|\beta|-\alpha}(w_x(z)+w_x(y))\quad\mbox{provided}\;
\sqrt[4]{t}\le|y-z|.
\end{align*}
We use this in two ways:
\begin{align}
\lefteqn{\mathbb{E}^\frac{1}{q'}|{\rm T}_z^2(\partial_1^2+\partial_2)
(\delta \Pi^-_{x}-{\rm d}\Gamma^*_{xz}Q\Pi^-_{z})_{\beta \,t}(y)|^{q'}
}\nonumber\\
&\lesssim t^{-1}\sum_{{\bf n}={\bf 0},(1,0)}|y-z|^{|{\bf n}|}
(\sqrt[4]{t})^{\alpha-|{\bf n}|+\kappa}
(\sqrt[4]{t}+|z-x|)^{|\beta|-\alpha}w_x(z),\label{aux20}\\
\lefteqn{\mathbb{E}^\frac{1}{q'}|(\partial_1^2+\partial_2)
(\delta \Pi^-_{x}-{\rm d}\Gamma^*_{xz}Q\Pi^-_{z})_{\beta \,t}(y)|^{q'}
}\nonumber\\
&\lesssim t^{-1}(\sqrt[4]{t})^{\alpha}|y-z|^{\kappa}
(|y-z|+|z-x|)^{|\beta|-\alpha}(w_x(z)+w_x(y))\quad\mbox{provided}\;
\sqrt[4]{t}\le|y-z|.\nonumber
\end{align}
Applying $\int_0^{|y-z|^4}dt$, using $\alpha-1+\kappa>0$
by (\ref{eq:algIII8}) on the first integral, and $\alpha>0$ on the second, we obtain
\begin{align}
\lefteqn{\mathbb{E}^\frac{1}{q'}\big|\int_0^{|y-z|^4}dt{\rm T}_z^2(\partial_1^2+\partial_2)
(\delta \Pi^-_{x}-{\rm d}\Gamma^*_{xz}Q\Pi^-_{z})_{\beta \,t}(y)\big|^{q'}}\nonumber\\
&\lesssim|y-z|^{\alpha+\kappa}
(|y-z|+|z-x|)^{|\beta|-\alpha}w_x(z),\label{eq:intIII3}\\
\lefteqn{\mathbb{E}^\frac{1}{q'}\big|\int_0^{|y-z|^4}dt(\partial_1^2+\partial_2)
(\delta \Pi^-_{x}-{\rm d}\Gamma^*_{xz}Q\Pi^-_{z})_{\beta \,t}(y)\big|^{q'}}\nonumber\\
&\lesssim |y-z|^{\alpha+\kappa}
(|y-z|+|z-x|)^{|\beta|-\alpha}(w_x(z)+w_x(y)),\label{eq:intIII4}
\end{align}
which takes care of the near-field contribution.

\medskip

We now turn to the far-field contribution $\sqrt[4]{t}\ge\max\{|y-z|,|z-x|\}$,
which we split into the one coming from
${\rm d}\Gamma_{xz}^* Q\Pi_z^{-}$ and the one from $\delta\Pi_x^{-}$.
For the first one, we note that by $\Pi^{-}_z\in\mathsf{\tilde T}^*$,
see (\ref{eq:Pi_minus_def_alt}), and the strict triangularity (\ref{eq:triangular_d_Gamma}) of
${\rm d}\Gamma^*_{xz}$, only 
$\eqref{eq:form_bound}_\beta^{\gamma\neq \pp}$ 
and $\eqref{eq:pi_minus_generic}_{\prec\beta}$ are needed for
\begin{align*}
\mathbb{E}^\frac{1}{q'}|({\rm d}\Gamma_{xz}^*Q\Pi_z^{-})_{\beta \,t}(y)|^{q'}
\lesssim\sum_{|\gamma|\in\mathsf{A}\cap(-\infty,2)\cap[\alpha,\kappa+|\beta|]}
|z-x|^{\kappa+|\beta|-|\gamma|}w_x(z)(\sqrt[4]{t})^{\alpha-2}
(\sqrt[4]{t}+|y-z|)^{|\gamma|-\alpha},
\end{align*}
which, by a similar but simpler argument as for (\ref{eq:intIII7}), we pre-process to 
\begin{align}\label{eq:intIII20}
\lefteqn{\mathbb{E}^\frac{1}{q'}|\partial^{\bf n}
({\rm d}\Gamma_{xz}^*Q\Pi_z^{-})_{\beta \,t}(y)|^{q'}}\nonumber\\
&\lesssim\sum_{|\gamma|\in\mathsf{A}\cap(-\infty,2)\cap[\alpha,\kappa+|\beta|]}
(\sqrt[4]{t})^{|\gamma|-2-|{\bf n}|}
|z-x|^{\kappa+|\beta|-|\gamma|}w_x(z)
\quad\mbox{provided}\;|y-z|\le\sqrt[4]{t}.
\end{align}
Representing Taylor's remainder in a way suitable for our parabolic scaling,
namely 
\begin{align}\label{eq:Taylor_re}
(1-{\rm T}_z^2)f(y)=\int_0^1ds(1-s)\frac{d^2h}{ds^2}(s)\quad
\mbox{with}\quad h(s)=f(sy_1+(1-s)z_1,s^2y_2+(1-s^2)z_2),
\end{align}
so that the l.~h.~s.~involves the four partial derivatives $\partial^{\bf n}f$
with $|{\bf n}|\ge 2$ and $n_1+n_2\le 2$. Applying this to
$f=(\partial_1^2+\partial_2)({\rm d}\Gamma_{xz}^*Q\Pi_z^{-})_{\beta \,t}$,
we learn from (\ref{eq:intIII20}) that
\begin{align*}
\lefteqn{\mathbb{E}^\frac{1}{q'}|(1-{\rm T}_z^2)(\partial_1^2+\partial_2)
({\rm d}\Gamma_{xz}^*Q\Pi_z^{-})_{\beta \,t}(y)|^{q'}}\nonumber\\
&\lesssim t^{-1}\sum_{\substack{|{\bf n}|\ge 2 \\ n_1+n_2\le 2}}
\sum_{|\gamma|\in\mathsf{A}\cap(-\infty,2)\cap[\alpha,\kappa+|\beta|]}
|y-z|^{|{\bf n}|}(\sqrt[4]{t})^{|\gamma|-|{\bf n}|}
|z-x|^{\kappa+|\beta|-|\gamma|}w_x(z)
\quad\mbox{provided}\;|y-z|\le\sqrt[4]{t}.
\end{align*}
Applying $\int_{\max\{|y-z|^4,|z-x|^4\}}^\infty dt$ we obtain because of $|\gamma|-|{\bf n}|<2-2=0$
\begin{align}\label{eq:intIII5}
\lefteqn{\mathbb{E}^\frac{1}{q'}\big|\int_{\max\{|y-z|^4,|z-x|^4\}}^\infty dt
(1-{\rm T}_z^2)(\partial_1^2+\partial_2)
({\rm d}\Gamma_{xz}^*Q\Pi_z^{-})_{\beta \,t}(y)\big|^{q'}}\nonumber\\
&\lesssim \sum_{\substack{|{\bf n}|\ge 2 \\ n_1+n_2\le 2}}
\sum_{|\gamma|\in\mathsf{A}\cap(-\infty,2)\cap[\alpha,\kappa+|\beta|]}
|y-z|^{|{\bf n}|}(|y-z|+|z-x|)^{|\gamma|-|{\bf n}|}
|z-x|^{\kappa+|\beta|-|\gamma|}w_x(z)\nonumber\\
&\lesssim 
|y-z|^{\kappa+\alpha}(|y-z|+|z-x|)^{|\beta|-\alpha}w_x(z)
\quad\mbox{by}\quad|{\bf n}|\ge 2\stackrel{\eqref{ao25}}{\ge}
\kappa+\alpha.
\end{align}

\medskip

For the second part of the integrand, we pre-process $\eqref{eq:mal_dual_annealed}_\beta$ to
\begin{align}\label{aux10}
\mathbb{E}^\frac{1}{q'}|\partial^{\bf n}\delta\Pi^{-}_{x\beta \,t}(y)|^{q'}
\lesssim(\sqrt[4]{t})^{\alpha-2-|{\bf n}|}(\sqrt[4]{t}+|y-x|)^{|\beta|-\alpha}\bar w,
\end{align}
which in turn implies by Taylor and $|y-x|+|z-x|\lesssim|y-z|+|z-x|$
\begin{align*}
\mathbb{E}^\frac{1}{q'}|(1-{\rm T}_z^2)(\partial_1^2+\partial_2)\delta\Pi^{-}_{x\beta \,t}(y)|^{q'}
\lesssim t^{-1}\sum_{\substack{|{\bf n}|\ge 2 \\ n_1+n_2\le 2}}
|y-z|^{|{\bf n}|}(\sqrt[4]{t})^{\alpha-|{\bf n}|}(\sqrt[4]{t}+|y-z|+|z-x|)^{|\beta|-\alpha}\bar w.
\end{align*}
Applying $\int_{\max\{|y-z|^4,|z-x|^4\}}^\infty dt$ we obtain because of
$|\beta|-|{\bf n}|<2-2=0$ and as in (\ref{eq:intIII5})
\begin{align}\label{eq:intIII6}
\lefteqn{\mathbb{E}^\frac{1}{q'}\big|\int_{\max\{|y-z|^4,|z-x|^4\}}^\infty dt
(1-{\rm T}_z^2)(\partial_1^2+\partial_2)\delta\Pi^{-}_{x\beta \,t}(y)\big|^{q'}}\nonumber\\
&\stackrel{\hphantom{\eqref{eq:def_w_x_z}}}{\lesssim}\sum_{\substack{|{\bf n}|\ge 2 \\ n_1+n_2\le 2}}
|y-z|^{|{\bf n}|}
(|y-z|+|z-x|)^{\alpha-|{\bf n}|}
(|y-z|+|z-x|)^{|\beta|-\alpha}\bar w\nonumber\\
&\stackrel{\eqref{eq:def_w_x_z}}{\lesssim}
|y-z|^{\kappa+\alpha}(|y-z|+|z-x|)^{|\beta|-\alpha} w_x(z).
\end{align}

\medskip

In view of (\ref{eq:intIII3}), (\ref{eq:intIII4}), (\ref{eq:intIII5}), 
and (\ref{eq:intIII6}), it remains to consider the case $|y-z|\le|z-x|$ and to estimate
the intermediate range
\begin{align}\label{eq:intIII8}
\mathbb{E}^\frac{1}{q'}\big|\int_{|y-z|^4}^{|z-x|^4}dt(1-{\rm T}_z^2)
(\partial_1^2+\partial_2)(\delta\Pi_{x}^- -{\rm d}\Gamma_{xz}^*Q\Pi^-_z)_{\beta \,t}(y)\big|^{q'}
\lesssim|y-z|^{\kappa+\alpha}|z-x|^{|\beta|-\alpha}w_x(z).
\end{align}
To this purpose, we pre-process (\ref{eq:intIII7}):
We apply the semi-group property (\ref{eq:semi_group})
in form of $\mathbb{E}^\frac{1}{q'}|f_t(y)|^{q'}$ $\le\int dy'|\psi_\frac{t}{2}(y-y')|
\mathbb{E}^\frac{1}{q'}|f_\frac{t}{2}(y')|^{q'}$ to
$f=\partial^{\bf n}(\delta\Pi^-_{x}-{\rm d}\Gamma^*_{xz}Q\Pi^-_{z})_{\beta}$,
insert (\ref{eq:intIII7}), use H\"older's inequality in $y'$ in order to access
(\ref{cw61}) and (\ref{eq:intIII9}) (both with ${\bf n}={\bf 0}$
and $t$ replaced by $\frac{t}{2}$), thereby obtaining
\begin{align}\label{eq:pre_process}
\lefteqn{\mathbb{E}^\frac{1}{q'}|\partial^{\bf n}
(\delta \Pi^-_{x}-{\rm d}\Gamma^*_{xz}Q\Pi^-_{z})_{\beta \,t}(y)|^{q'}
}\nonumber\\
&\lesssim(\sqrt[4]{t})^{\alpha-2-|{\bf n}|+\kappa}
(\sqrt[4]{t}+|z-x|)^{|\beta|-\alpha}w_x(z)\quad\mbox{provided}\;|y-z|\le\sqrt[4]{t}.
\end{align}
By Taylor, this implies
\begin{align*}
\lefteqn{\mathbb{E}^\frac{1}{q'}|(1-{\rm T}_z^2)(\partial_1^2+\partial_2)
(\delta \Pi^-_{x}-{\rm d}\Gamma^*_{xz}Q\Pi^-_{z})_{\beta \,t}(y)|^{q'}
}\nonumber\\
&\lesssim t^{-1}\sum_{\substack{|{\bf n}|\ge 2 \\ n_1+n_2\le 2}}
|y-z|^{|{\bf n}|}(\sqrt[4]{t})^{\alpha-|{\bf n}|+\kappa}
|z-x|^{|\beta|-\alpha}w_x(z)\quad\mbox{provided}\;
|y-z|\le\sqrt[4]{t}\le|z-x|.
\end{align*}
Integration gives (\ref{eq:intIII8}).

\medskip

We now argue in favor of \eqref{eq:intIII1}.
We obtain from $\eqref{eq:model}_{\preccurlyeq\beta}$ and its Malliavin derivative that
\begin{align}\label{ao23}
(\partial_2 - \partial_1^2)
Q(\delta \Pi_{x}-\delta\Pi_{x}(z)-\dd \Gamma^*_{xz}Q\Pi_{z})
=Q(\delta \Pi^-_{x}-\dd\Gamma^*_{xz}Q\Pi^-_{z}).
\end{align}
By (\ref{eq:controlled_path_qual}),
$Q(\delta\Pi_{x}-\delta\Pi_{x}(z)-{\rm d}\Gamma^*_{xz}Q\Pi_{z})_{\beta}$ 
vanishes to first order in $z$, where we could smuggle in the second $Q$ by \eqref{eq:pi_generic}.
Thanks to the presence of $Q$, and the control of $\eqref{eq:pi_generic}_{\prec\beta}$,
$({\rm d}\Gamma^*_{xz}Q\Pi_{z})_{\beta}$ grows sub-quadratically at infinity,
by the control $\eqref{eq:delta_pi_generic}_{\beta}$, 
the same applies to $\delta\Pi_{x\beta}$. 
By the above established estimate, 
the $t$-integral in \eqref{eq:intIII1} vanishes to first order in $z$.
Hence the annealed Liouville argument from the proof of Proposition \ref{prop:change_of_basepointII} yields \eqref{eq:intIII1}, 
provided the $t$-integral grows sub-quadratically at infinity as well, and is a solution to \eqref{ao23}.
To prove it is a solution to \eqref{ao23}, we proceed as in the proof of Proposition \ref{prop:int_pi_minus}, 
and replace $\int_0^\infty dt$ by $\int_s^T dt$ to obtain
\begin{align*}
-(\partial_2-\partial_1^2) \int_s^T dt(1-{\rm T}_z^{2})
(\partial_2+\partial_1^2)(\delta \Pi^-_{x}-{\rm d}\Gamma^*_{xz}Q\Pi^-_{z})_{\beta\,t}\\
= (\delta \Pi^-_{x}-{\rm d}\Gamma^*_{xz}Q\Pi^-_{z})_{\beta\,s}
- (\delta \Pi^-_{x}-{\rm d}\Gamma^*_{xz}Q\Pi^-_{z})_{\beta\,T}.
\end{align*}
The first term coincides with the r.~h.~s.~of \eqref{ao23} as $s\downarrow0$, 
the second term vanishes as $T\uparrow\infty$ by $\eqref{eq:mal_dual_annealed}_\beta$, 
$\eqref{eq:form_bound}_\beta^{\gamma\neq\pp}$ and $\eqref{eq:pi_minus_generic}_{\prec\beta}$ due to the presence of $Q$.

\medskip

The rest of the proof is dedicated to establish sub-quadratic growth of the $t$-integral in \eqref{eq:intIII1}, where we distinguish the near-field range $\sqrt[4]{t}\leq|y-z|$ 
from the far-field range $\sqrt[4]{t}\geq|y-z|$. 
In the far field, we break up $\delta\Pi^-_x$ and ${\rm d}\Gamma^*_{xz}Q\Pi^-_z$. 
The intermediate estimate in \eqref{eq:intIII5} shows growth in $y$ of order $|\gamma|<2$ of the latter, 
and the intermediate estimate in \eqref{eq:intIII6} shows growth in $y$ of order $|\beta|<2$ of the former. 
We turn to the near field, where we break up $1$ and ${\rm T}_z^2$. 
For the former, we further break up $\delta\Pi^-_x$ and ${\rm d}\Gamma^*_{xz}Q\Pi^-_z$.
By \eqref{aux10} we obtain 
\begin{align*}
\E^\frac{1}{q'}\big| \int_0^{|y-z|^4} dt\, (\partial_1^2+\partial_2)\delta\Pi^-_{x\beta\,t}(y)\big|^{q'}
&\lesssim \int_0^{|y-z|^4} dt\, (\sqrt[4]{t})^{\alpha-4}(\sqrt[4]{t}+|x-y|)^{|\beta|-\alpha}\bar{w}\\
&\lesssim |y-z|^\alpha (|y-z|+|x-y|)^{|\beta|-\alpha}\bar{w},
\end{align*}
which grows in $y$ with order $|\beta|<2$.
Similarly, by $\eqref{eq:form_bound}_\beta^{\gamma\neq \pp}$ 
and $\eqref{eq:pi_minus_generic}_{\prec\beta}$ (actually in its strengthened version $\eqref{cw64}_{\prec\beta}$) 
we obtain 
\begin{align*}
\E^\frac{1}{q'} & \big| \int_0^{|y-z|^4} dt\, (\partial_1^2+\partial_2) ({\rm d}\Gamma^*_{xz}\Pi^-_{z})_{\beta\,t}(y)\big|^{q'}\\
&\lesssim \int_0^{|y-z|^4} dt \sum_{|\gamma|\in\mathsf{A}\cap(-\infty,2)} |x-z|^{|\beta|-|\gamma|+\kappa} w_x(z) (\sqrt[4]{t})^{\alpha-4}(\sqrt[4]{t}+|y-z|)^{|\gamma|-\alpha} \\
&\lesssim \sum_{|\gamma|\in\mathsf{A}\cap(-\infty,2)} |x-z|^{|\beta|-|\gamma|+\kappa} w_x(z) |y-z|^{|\gamma|} ,
\end{align*}
which once more grows sub-quadratically in $y$. 
For the contributions from ${\rm T}_z^2$ we first consider the regime $\sqrt[4]{t}\leq|x-z|$ (recall that we are interested in large $y$ and we may therefore assume w.~l.~o.~g.~that $|x-z|<|y-z|$). 
By \eqref{aux20} we obtain 
\begin{align*}
\E^\frac{1}{q'} & \big| \int_0^{|x-z|^4} dt\, {\rm T}_z^2 (\partial_1^2+\partial_2) (\delta\Pi^-_x-{\rm d}\Gamma^*_{xz}\Pi^-_{z})_{\beta\,t}(y)\big|^{q'}\\
&\lesssim \int_0^{|x-z|^4} dt\, \sum_{\n=\0,(1,0)} |y-z|^{|\n|} (\sqrt[4]{t})^{\alpha-|\n|-4+\kappa}(\sqrt[4]{t}+|x-z|)^{|\beta|-\alpha}w_x(z)\\
&\lesssim  \sum_{\n=\0,(1,0)} |y-z|^{|\n|} |x-z|^{|\beta|-|\n|+\kappa}w_x(z) ,
\end{align*}
where in the last inequality we used $\kappa+\alpha-1>0$ (see \eqref{eq:algIII8}),
and hence obtain again sub-quadratic growth in $y$. 
In the regime $|x-z|\leq \sqrt[4]{t} \leq |y-z|$, we split $\delta\Pi^-_x$ and ${\rm d}\Gamma^*_{xz}\Pi^-_z$. 
On the one hand, we have by \eqref{aux10} 
\begin{align*}
\E^\frac{1}{q'} & \big| \int_{|x-z|^4}^{|y-z|^4} dt\, {\rm T}_z^2 (\partial_1^2+\partial_2) \delta\Pi^-_{x\beta\,t}(y)\big|^{q'} \\
&\lesssim \int_{|x-z|^4}^{|y-z|^4} dt\, \sum_{\n=\0,(1,0)} |y-z|^{|\n|} (\sqrt[4]{t})^{\alpha-|\n|-4}
((\sqrt[4]{t})^{|\beta|-\alpha}+|x-z|^{|\beta|-\alpha})\bar{w} \\
&\lesssim \sum_{\n=\0,(1,0)} |y-z|^{|\n|} (|y-z|^{|\beta|-|\n|}+ |x-z|^{|\beta|-|\n|}+|y-z|^{\alpha-|\n|}|x-z|^{|\beta|-\alpha})\bar{w},
\end{align*}
which grows again sub-quadratically in $y$.
On the other hand, by $\eqref{eq:form_bound}_\beta^{\gamma\neq \pp}$ 
and $\eqref{eq:pi_minus_generic}_{\prec\beta}$ (actually in its strengthened version $\eqref{cw64}_{\prec\beta}$), we have 
\begin{align*}
\E^\frac{1}{q'} & \big| \int_{|x-z|^4}^{|y-z|^4} dt\, {\rm T}_z^2 (\partial_1^2+\partial_2) ({\rm d}\Gamma^*_{xz}\Pi^-_{z})_{\beta\,t}(y)\big|^{q'} \\
&\lesssim \int_{|x-z|^4}^{|y-z|^4} dt \sum_{|\gamma|\in\mathsf{A}\cap(-\infty,2)} |x-z|^{|\beta|-|\gamma|+\kappa} w_x(z) \sum_{\n=\0,(1,0)} |y-z|^{|\n|} (\sqrt[4]{t})^{|\gamma|-4-|\n|} \\
&\lesssim \sum_{|\gamma|\in\mathsf{A}\cap(-\infty,2)} |x-z|^{|\beta|-|\gamma|+\kappa} w_x(z) \sum_{\n=\0,(1,0)} |y-z|^{|\n|} (|y-z|^{|\gamma|-|\n|}+|x-z|^{|\gamma|-|\n|}),
\end{align*}
which is of sub-quadratic growth in $y$.
\end{proof}

The fourth task is to pass from the output 
$\eqref{eq:delta_pi_incr_generic}_{\beta}$ of Proposition \ref{prop:intIII}
to $\eqref{eq:delta_pi_d_pi_incr}_{\beta}$.
The proof is based on formula
\begin{align}\label{eq:3ptIII2}
&\big({\rm d}\pi^{({\bf 0})}_{xy}-{\rm d}\pi^{({\bf 0})}_{xz}
-{\rm d}\Gamma^*_{xz}Q\pi^{({\bf 0})}_{zy}\big)
+\big({\rm d}\pi^{(1,0)}_{xy}-{\rm d}\pi^{(1,0)}_{xz}-{\rm d}\Gamma^*_{xz}Q\pi^{(1,0)}_{zy}\big)
(\cdot -y)_1\nonumber\\
&\quad=\big(\delta\Pi_x-\delta\Pi_x(z)-{\rm d}\Gamma^*_{xz}Q\Pi_z\big)
      -\big(\delta\Pi_x-\delta\Pi_x(y)-{\rm d}\Gamma^*_{xy}Q\Pi_y\big)
-\big({\rm d}\Gamma^*_{xy}Q-{\rm d}\Gamma^*_{xz}Q\Gamma^*_{zy}\big)P\Pi_{y},
\end{align}
which we shall establish in the proof of the upcoming Proposition
\ref{prop:delta_pi_incr_three_point}. Its merit is 
that it connects an affine polynomial with the coefficients 
given by the rough-path increments of $\{{\rm d}\pi^{({\bf n})}_{xz}\}_{{\bf n}={\bf 0},(1,0)}$ to
the rough-path increments of $\delta\Pi_x$ (in the secondary base points $y$ and $z$)
and the continuity expression for ${\rm d}\Gamma_{xz}^*$. 
In analogy to the fourth task in Subsection \ref{firstblock} we have

\begin{proposition}[Three-point argument III]\label{prop:delta_pi_incr_three_point} 
Assume that $\eqref{eq:pi_generic}_{\prec\beta}$,
$\eqref{eq:pi_n}_{\prec\beta}$, $\eqref{eq:recenter_Pi_specific}_{\prec\beta}$, and 
$\eqref{eq:delta_pi_incr_generic}_\beta$ hold,
and that $\eqref{eq:gamma}_{\prec\beta}^{\gamma}$, $\eqref{eq:form_cont}_\beta^{\gamma}$,
and $\eqref{eq:form_bound}_\beta^\gamma$
hold for all $\gamma$ not purely polynomial.
Then $\eqref{eq:delta_pi_d_pi_incr}_{\beta}$ holds.
\end{proposition}

\begin{proof}
The formula (\ref{eq:3ptIII2}) follows from substituting $\Pi_z$
according to (\ref{eq:recenter_Pi_specific}) by $\Gamma_{zy}^*\Pi_y+\pi_{zy}^{({\bf 0})}$,
and then appealing to
\begin{align}
(\delta\Pi_x(y),\delta\Pi_x(z))&\stackrel{\eqref{eq:Pipi}}{=}
({\rm d}\pi_{xy}^{(\bf 0)},{\rm d}\pi_{xz}^{(\bf 0)}),\label{eq:deltaPipi}\\
Q\Gamma_{zy}^*({\rm id}-P)\Pi_y&
\stackrel{\eqref{eq:Gamma_z_n},\eqref{eq:pi_purely_pol}}{=}
(\mathsf{z}_{(1,0)}+Q\pi^{(1,0)}_{zy})(\cdot-y)_1,\nonumber\\
\big({\rm d}\Gamma^*_{xy}Q-{\rm d}\Gamma^*_{xz}Q\Gamma^*_{zy}\big)({\rm id}-P)\Pi_{y}
&\stackrel{\eqref{eq:def_dGamma}}{=}
\big({\rm d}\pi^{(1,0)}_{xy}-{\rm d}\pi^{(1,0)}_{xz}-{\rm d}\Gamma^*_{xz}Q\pi^{(1,0)}_{zy}\big)
(\cdot -y)_1.\nonumber
\end{align}
Evaluating at $y$ makes formula (\ref{eq:3ptIII2}) collapse to 
\begin{align}\label{eq:3ptIII}
{\rm d}\pi^{({\bf 0})}_{xy}-{\rm d}\pi^{({\bf 0})}_{xz}
-{\rm d}\Gamma^*_{xz}Q\pi^{({\bf 0})}_{zy}
=\delta\Pi_x(y)-\delta\Pi_x(z)-{\rm d}\Gamma^*_{xz}Q\Pi_z(y).
\end{align}
Since ${\rm d}\Gamma^*_{xz}$ vanishes on $\mathsf{z}_{\bf n}$ unless ${\bf n}=(1,0)$,
we may rewrite (\ref{eq:3ptIII}) as
\begin{align*}
{\rm d}\pi^{({\bf 0})}_{xy}-{\rm d}\pi^{({\bf 0})}_{xz}
-{\rm d}\Gamma^*_{xz}\pi^{({\bf 0})}_{zy}
\stackrel{\eqref{eq:deltaPipi}}{=}
\big(\delta\Pi_x(y)-\delta\Pi_x(z)-{\rm d}\Gamma^*_{xz}Q\Pi_z(y)\big)
-{\rm d}\Gamma^*_{xz}P({\rm id}-Q)\Pi_{z}(y).
\end{align*}
Taking the $\mathbb{E}^\frac{1}{q'}|\cdot|^{q'}$ of the $\beta$-component
of this identity, appealing to the strictly triangular structure (\ref{eq:triangular_d_Gamma})
of ${\rm d}\Gamma^*_{xz}$ w.~r.~t.~$\prec$, we obtain by our assumptions
$\eqref{eq:pi_generic}_{\prec\beta}$,
$\eqref{eq:delta_pi_incr_generic}_\beta$, and $\eqref{eq:form_bound}_\beta^{\gamma\not=\pp}$,
\begin{align*}
\mathbb{E}^\frac{1}{q'}\big|{\rm d}\pi^{({\bf 0})}_{xy}-{\rm d}\pi^{({\bf 0})}_{xz}
-{\rm d}\Gamma^*_{xz}\pi^{({\bf 0})}_{zy}\big|^{q'}
&\lesssim|y-z|^{\kappa+\alpha}(|y-z|+|z-x|)^{|\beta|-\alpha}(w_x(y)+w_x(z))\nonumber\\
&+\sum_{|\gamma|\in\mathsf{A}\cap[2,\infty)}|z-x|^{\kappa+|\beta|-|\gamma|}w_x(z)
|y-z|^{|\gamma|}.
\end{align*}
By the second item in (\ref{ao25}),
the second r.~h.~s.~term
can be absorbed into the first. This establishes $\eqref{eq:delta_pi_d_pi_incr}_\beta$ 
for ${\bf n}={\bf 0}$.

\medskip

We now address the ${\bf n}=(1,0)$ contribution to (\ref{eq:3ptIII2}). 
Appealing to (\ref{eq:recIII4a}), (\ref{eq:recIII4b}) and (\ref{eq:3ptIII}), we may rewrite (\ref{eq:3ptIII2}) as
\begin{align*}
&\big({\rm d}\pi^{(1,0)}_{xy}-{\rm d}\pi^{(1,0)}_{xz}-{\rm d}\Gamma^*_{xz}\pi^{(1,0)}_{zy}\big)
(\cdot -y)_1\nonumber\\
&=\big(\delta\Pi_x-\delta\Pi_x(z)-{\rm d}\Gamma^*_{xz}Q\Pi_z\big)
-\big(\delta\Pi_x-\delta\Pi_x(z)-{\rm d}\Gamma^*_{xz}Q\Pi_z\big)(y)
-\big(\delta\Pi_x-\delta\Pi_x(y)-{\rm d}\Gamma^*_{xy}Q\Pi_y\big)\nonumber\\
&-\big({\rm d}\Gamma^*_{xy}-{\rm d}\Gamma^*_{xz}\Gamma^*_{zy}\big)PQ\Pi_{y}
-{\rm d}\Gamma_{xz}^*P({\rm id}-Q)\Gamma^*_{zy}QP\Pi_y
-{\rm d}\Gamma^*_{xz}P({\rm id}-Q)\pi^{(1,0)}_{zy}(\cdot-y)_1.
\end{align*}
We then take the $\mathbb{E}^\frac{1}{q'}|\cdot|^{q'}$ of the $\beta$-component of this identity.
The first three r.~h.~s.~terms are controlled by $\eqref{eq:delta_pi_incr_generic}_{\beta}$.
For the fourth term we use the strict triangularity (\ref{eq:dGamma_inc_triangular}),
so that $\eqref{eq:pi_generic}_{\prec\beta}$ is sufficient,
next to $\eqref{eq:form_cont}_\beta^{\gamma\neq \pp}$.
For the fifth term  we use the strict triangularity of ${\rm d}\Gamma_{xz}^*$
and the triangularity of $\Gamma_{zy}^*$, so that 
$\eqref{eq:gamma}_{\prec\beta}^{\gamma\neq \pp}$ and 
once more $\eqref{eq:pi_generic}_{\prec\beta}$ are sufficient, 
next to $\eqref{eq:form_bound}_\beta^{\gamma\neq \pp}$.
For the sixth term we just use the strict triangularity of ${\rm d}\Gamma_{xz}^*$, so that $\eqref{eq:pi_n}_{\prec\beta}$ is sufficient,
once more next to $\eqref{eq:form_bound}_\beta^{\gamma\neq \pp}$.
We obtain term by term, using that $\gamma(1,0)\not=0$ implies $|\gamma|\ge 1$
on the fifth r.~h.~s.~term below,
\begin{align*}
\lefteqn{|(\cdot-y)_1|\mathbb{E}^\frac{1}{q'}\big|{\rm d}\pi^{(1,0)}_{xy}-{\rm d}\pi^{(1,0)}_{xz}
-{\rm d}\Gamma^*_{xz}Q\pi^{(1,0)}_{zy}\big|^{q'}}\nonumber\\
&\lesssim|\cdot-z|^{\kappa+\alpha}(|\cdot-z|+|z-x|)^{|\beta|-\alpha}
(w_x(\cdot)+w_x(z))\nonumber\\
&+|y-z|^{\kappa+\alpha}(|y-z|+|z-x|)^{|\beta|-\alpha}
(w_x(y)+w_x(z))\nonumber\\
&+|\cdot-y|^{\kappa+\alpha}(|\cdot-y|+|y-x|)^{|\beta|-\alpha}
(w_x(\cdot)+w_x(y))\nonumber\\
&+\sum_{|\gamma|\in\mathsf{A}\cap[0,|\beta|-\alpha]}
|y-z|^{\kappa+\alpha}
(|y-z|+|z-x|)^{|\beta|-|\gamma|-\alpha}(w_x(z)+w_x(y))|\cdot-y|^{|\gamma|}\nonumber\\
&+\sum_{|\gamma|\in\mathsf{A}\cap[1,|\beta|-\alpha+1]}
|y-z|^{\kappa+\alpha-1}
(|y-z|+|z-x|)^{|\beta|-|\gamma|-\alpha+1}(w_x(z)+w_x(y))|\cdot-y|^{|\gamma|}\nonumber\\
&+\sum_{|\gamma|\in\mathsf{A}\cap[2,\kappa+|\beta|]}
|z-x|^{\kappa+|\beta|-|\gamma|}w_x(z)
\sum_{|\gamma'|\in\mathsf{A}\cap[0,|\gamma|]}|y-z|^{|\gamma|-|\gamma'|}|\cdot-y|^{|\gamma'|}
\nonumber\\
&+\sum_{|\gamma|\in\mathsf{A}\cap[2,\kappa+|\beta|]}
|z-x|^{\kappa+|\beta|-|\gamma|}w_x(z)|y-z|^{|\gamma|-1}|\cdot-y|.
\end{align*}
Restricting the active variable to the (parabolic) ball $|\cdot-y|\le|y-z|$, 
and using the second item in (\ref{ao25}) on the last two terms, this estimate collapses to
\begin{align}\label{eq:3ptIII4}
\lefteqn{|(\cdot-y)_1|\mathbb{E}^\frac{1}{q'}\big|{\rm d}\pi^{(1,0)}_{xy}-{\rm d}\pi^{(1,0)}_{xz}
-{\rm d}\Gamma^*_{xz}Q\pi^{(1,0)}_{zy}\big|^{q'}}\nonumber\\
&\lesssim|y-z|^{\kappa+\alpha}(|y-z|+|z-x|)^{|\beta|-\alpha}
(w_x(\cdot)+w_x(y)+w_x(z)).
\end{align}
We now average the active variable over this ball in order to recover
$\eqref{eq:delta_pi_d_pi_incr}_\beta$ for ${\bf n}=(1,0)$. 
Indeed, for the l.~h.~s.~of (\ref{eq:3ptIII4}) we appeal to the obvious
$\fint_{|\cdot-y|\le|y-z|}|(\cdot-y)_1|$ $\sim|y-z|$. For the r.~h.~s.~of (\ref{eq:3ptIII4}),
by definition (\ref{eq:def_w_x_z}) of $w_x$, it suffices to establish
\begin{align}\label{eq:3ptIII3}
\fint_{|\cdot-y|\le\lambda}|\cdot-x|^{-\kappa}\lesssim|y-x|^{-\kappa}
\quad\mbox{for}\;\lambda=|y-z|,
\end{align}
which is an easy version of (\ref{eq:3ptIII10}).
\end{proof}


\subsection{From \texorpdfstring{$\delta\Pi^{-}_x-{\rm d}\Gamma^*_{xz}\Pi^{-}_z$}{delta Pi minus x - dGamma* xz Pi minus z} back to \texorpdfstring{$\delta\Pi^{-}_x$}{deltaPi minus x} via boundedness of \texorpdfstring{${\rm d}\Gamma^*_{xz}$ and ${\rm d}\pi_{xz}^{({\bf n})}$}{dGamma* xz and dpi n xz}, and by averaging in the secondary base point \texorpdfstring{$z$}{z}}
\label{forthblock}

The aim of this last subsection is to pass from the estimate
(\ref{eq:delta_pi_minus_generic}) of the rough-path increment of $\delta\Pi_{x}^{-}$
to the estimate (\ref{eq:mal_dual_annealed}) of $\delta\Pi_{x}^{-}$ itself. 
Clearly, in view of the structure of the rough-path increment, 
this will require an estimate of ${\rm d}\Gamma_{xz}^*P$, see Proposition \ref{prop:form_bound}.
The proof of Proposition \ref{prop:form_bound} will be similar to the first task of 
Subsections \ref{firstblock} and \ref{thirdblock},
and rely on the estimate of ${\rm d}\pi_{xz}^{(\bf n)}$ for ${\bf n}={\bf 0},(1,0)$.
By (\ref{eq:deltaPipi}) the estimate of ${\rm d}\pi_{xz}^{(\bf 0)}$ is already part of the 
induction hypothesis.
However, the estimate of ${\rm d}\pi_{xz}^{(1,0)}$ needs to be included into the induction:
\begin{align}\label{eq:d_pi}
\mathbb{E}^\frac{1}{q'}|{\rm d}\pi^{(1,0)}_{xz\beta}|^{q'}
\lesssim |z-x|^{\kappa+|\beta|-1} w_x(z).
\end{align}
Note that the exponent is strictly positive, see (\ref{eq:algIII8}).
The first task of this subsection is based on the formula 
(\ref{eq:def_dGamma}) for ${\rm d}\Gamma^*_{xz}$. 

\begin{proposition}[Algebraic argument IV]\label{prop:form_bound}
Assume that $\eqref{eq:delta_pi_generic}_{\prec\beta}$ and $\eqref{eq:d_pi}_{\prec\beta}$ hold,
and assume that $\eqref{eq:gamma}_{\preccurlyeq\beta}^{\gamma}$ holds
for all $\gamma$ not purely polynomial.
Then we have for $\gamma$ not purely polynomial\footnote{with the understanding that
the l.~h.~s.~vanishes when the exponent is non-positive}
\begin{align}\label{eq:form_bound}
\E^\frac{1}{q'}|({\rm d}\Gamma^*_{xz})_\beta^{\gamma}|^{q'}
\lesssim|z-x|^{\kappa+|\beta|-|\gamma|}w_x(z).
\end{align}
This includes the statement that $({\rm d}\Gamma^*_{xz}P)_\beta$ depends only on 
${\rm d}\pi^{(1,0)}_{xz\beta'}$ and $\delta\Pi_{x\beta'}$ with $\beta'\prec\beta$,
and on $(\Gamma_{xz}^*P)_{\beta'}$ with $\beta'\preccurlyeq\beta$.
\end{proposition}

\begin{proof} The structure is very similar to the proof of Proposition \ref{prop:algIII}:
Again, we distinguish the contributions from ${\bf n}={\bf 0}$ and ${\bf n}=(1,0)$
to $\eqref{eq:def_dGamma}_{\beta}^\gamma$. In view of (\ref{eq:deltaPipi}), all
terms in the $({\bf n}={\bf 0})$-contribution are of the form
\begin{align}\label{eq:algIV1}
\delta\Pi_{x\beta_1}(z)(\Gamma_{xz}^*)_{\beta_2}^{\gamma-e_k+e_{k+1}},
\end{align}
for some $k\ge 0$ and multi-indices $\beta_1,\beta_2$ 
constrained by $\beta_1\prec\beta$, $\beta_2\preccurlyeq\beta$, $\beta_1+\beta_2=\beta$,
and noting that $\gamma-e_k+e_{k+1}$ is not purely polynomial. 
Hence under our assumptions, the $\mathbb{E}^\frac{1}{q'}|\cdot|^{q'}$-norm of 
(\ref{eq:algIV1}) is estimated by
\begin{align*}
|z-x|^{|\beta_1|}\bar w\;|z-x|^{|\beta_2|-|\gamma-e_k+e_{k+1}|}
\stackrel{\eqref{eq:def_w_x_z}}{\le} 
|z-x|^{\kappa+|\beta_1|}w_x(z)\;|z-x|^{|\beta_2|-|\gamma-e_k+e_{k+1}|}.
\end{align*}
It follows from (\ref{eq:algII1}) that this is contained in the r.~h.~s.~of 
(\ref{eq:form_bound}).

\medskip

All terms in the $({\bf n}=(1,0))$-contribution are of the form
\begin{align}\label{eq:algIV2}
{\rm d}\pi_{xz\beta_1}^{(1,0)}(\Gamma_{xz}^*)_{\beta_2}^{\gamma-e_{(1,0)}},
\end{align}
for some multi-indices $\beta_1,\beta_2$
constrained by $\beta_1\prec\beta$, $\beta_2\preccurlyeq\beta$, $\beta_1+\beta_2=\beta$,
and noting that $\gamma-e_{(1,0)}$ is not purely polynomial.
Hence under our assumptions, the $\mathbb{E}^\frac{1}{q'}|\cdot|^{q'}$-norm of
(\ref{eq:algIV2}) is estimated by
\begin{align*}
|z-x|^{\kappa+|\beta_1|-1}w_x(z)\;|z-x|^{|\beta_2|-|\gamma-e_{(1,0)}|}.
\end{align*}
It follows from (\ref{eq:algII2}) that also this is contained 
in the r.~h.~s.~of (\ref{eq:form_bound}).
\end{proof}


We will establish (\ref{eq:d_pi}) based on the three-point formula
\begin{align}\label{eq:3ptIV}
(y-z)_1{\rm d}\pi_{xz}^{(1,0)}=
-\big(\delta\Pi_x(y)-\delta\Pi_x(z)-{\rm d}\Gamma_{xz}^*Q\Pi_z(y)\big)
+\delta\Pi_x(y)-\delta\Pi_x(z)-{\rm d}\Gamma_{xz}^*PQ\Pi_z(y),
\end{align}
which follows from (\ref{eq:pi_purely_pol}) in form of $({\rm id}-P)Q\Pi_z(y)$ 
$=(y-z)_1\mathsf{z}_{(1,0)}$ and (\ref{eq:dGamma_z_10}).
By an argument similar to the fourth task of Subsections 
\ref{firstblock} and \ref{thirdblock} (see Propositions~\ref{prop:pi_n_three_point}, \ref{prop:gamma_pp} and \ref{prop:delta_pi_incr_three_point}) we obtain

\begin{proposition}[Three-point argument IV]\label{prop:d_pi} 
Assume that $\eqref{eq:pi_generic}_{\prec\beta}$, $\eqref{eq:delta_pi_generic}_\beta$ 
and $\eqref{eq:delta_pi_incr_generic}_\beta$ hold,
and that $\eqref{eq:form_bound}_\beta^{\gamma}$ holds for all $\gamma$ not purely polynomial.
Then $\eqref{eq:d_pi}_\beta$ and $\eqref{eq:form_bound}_\beta$ hold.
\end{proposition}

\begin{proof} Taking the $\mathbb{E}^\frac{1}{q'}|\cdot|^{q'}$ of
$\eqref{eq:3ptIV}_\beta$, and appealing to the strict triangularity (\ref{eq:triangular_d_Gamma})
of ${\rm d}\Gamma^*$, we obtain by our assumptions
\begin{align}\label{eq:3ptIV4}
|(y-z)_1|\mathbb{E}^\frac{1}{q'}|{\rm d}\pi_{xz\beta}^{(1,0)}|^{q'}
&\stackrel{\hphantom{(\ref{eq:def_w_x_z})}}{\lesssim}|y-z|^{\kappa+\alpha}(|y-z|+|z-x|)^{|\beta|-\alpha}(w_x(y)+w_x(z))\nonumber\\
&\stackrel{\hphantom{\eqref{eq:def_w_x_z}}}{+}|y-x|^{|\beta|}\bar w+|z-x|^{|\beta|}\bar w+
\sum_{|\gamma|\in\mathsf{A}\cap[|\alpha|,\kappa+|\beta|)}
|z-x|^{\kappa+|\beta|-|\gamma|}w_x(z)|y-z|^{|\gamma|}\nonumber\\
&\stackrel{\eqref{eq:def_w_x_z}}{\lesssim}
(|y-z|+|z-x|)^{\kappa+|\beta|}\big(w(y)+|y-x|^{-\kappa}\bar w+w_x(z)\big).
\end{align}
We now average over all $y$ with $|y-z|\le\frac{1}{2}|z-x|$, where the factor
of $\frac{1}{2}$ ensures that $|y-x|\ge\frac{1}{2}|z-x|$, so that in this range 
(\ref{eq:3ptIV4}) simplifies to
\begin{align}\label{eq:3ptIV3}
|(y-z)_1|\mathbb{E}^\frac{1}{q'}|{\rm d}\pi_{xz\beta}^{(1,0)}|
\lesssim|z-x|^{\kappa+|\beta|}\big(w(y)+|z-x|^{-\kappa}\bar w+w_x(z)\big).
\end{align}
The averaging of (\ref{eq:3ptIV3}) over $\{y:|y-z|\le\frac{1}{2}|z-x|\}$
ensures that on the one hand, we have for the
l.~h.~s.~that $\fint_{y:|y-z|\le\frac{1}{2}|z-x|}dy\,|(y-z)_1|\sim|z-x|$,
and that on the other hand, we may use (\ref{eq:3ptIV2}) for the r.~h.~s.~to
the effect of $\fint_{y:|y-z|\le\frac{1}{2}|z-x|}dy \,w(y)$
$\lesssim |z-x|^{-\kappa}\bar w$. Hence, once more appealing to definition
(\ref{eq:def_w_x_z}) of $w_x(z)$, we see that (\ref{eq:3ptIV3}) turns into $\eqref{eq:d_pi}_\beta$.
\end{proof}


Passing from (\ref{eq:delta_pi_minus_generic}) to (\ref{eq:mal_dual_annealed})
also means replacing the weighted norms of $\delta\xi$ by the original norm $\bar w$. 
This will be done starting from the obvious identity 
\begin{align}\label{eq:twisted_aver}
\delta \Pi_{x}^-=
(\delta\Pi_{x}^{-}-{\rm d}\Gamma^*_{xz}Q\Pi_z^-)
+{\rm d}\Gamma^*_{xz}Q\Pi_z^-
\end{align}
and averaging in the secondary base point $z$. In particular, we have the following proposition

\begin{proposition} \label{prop:delta_pi_minus_est}
Assume that $\eqref{eq:recenter_Pi_minus}_{\beta}$, $\eqref{eq:pi_minus_generic}_{\prec\beta}$,
$\eqref{eq:mal_dual_annealed}_{\prec \beta}$ and $\eqref{eq:delta_pi_minus_generic}_\beta$ 
hold, assume that $\eqref{eq:gamma}_\beta^\gamma$, $\eqref{eq:delta_gamma}_\beta^\gamma$ and $\eqref{eq:form_bound}_\beta^\gamma$ for $\gamma$ not 
purely polynomial hold.
Then $\eqref{eq:mal_dual_annealed}_\beta$ holds.
\end{proposition}

\begin{proof}
We apply $(\cdot)_t$ to $\eqref{eq:twisted_aver}_\beta$ and then the 
$\mathbb{E}^\frac{1}{q'}|\cdot|^{q'}$-norm;
on the first term, we use the upgrade $\eqref{eq:pre_process}_\beta$ (with ${\bf n}={\bf 0}$)
of $\eqref{eq:delta_pi_minus_generic}_\beta$;
on the second term, we use $\eqref{eq:form_bound}_\beta^{\gamma\neq\pp}$, and
$\eqref{eq:pi_minus_generic}_{\prec\beta}$ (thanks to the 
strict triangularity
(\ref{eq:triangular_d_Gamma}) of ${\rm d}\Gamma_{xz}^*$ w.~r.~t.~$\prec$), obtaining
\begin{align*}
\lefteqn{\mathbb{E}^\frac{1}{q'}|\delta\Pi_{x\beta t}^{-}(y)|^{q'}
\lesssim(\sqrt[4]{t})^{\alpha+\kappa-2} (\sqrt[4]{t}+|z-x|)^{|\beta|-\alpha} w_x(z)}\nonumber\\
&+\sum_{|\gamma|\in\mathsf{A}\cap[\kappa+|\beta|,2]}
|z-x|^{\kappa+|\beta|-|\gamma|}w_x(z)(\sqrt[4]{t})^{|\gamma|-2}\quad\mbox{provided}\;|y-z|\leq \sqrt[4]{t}.
\end{align*}
By definition (\ref{eq:def_w_x_z}), this yields
\begin{equs}
\mathbb{E}^\frac{1}{q'}|\delta\Pi_{x\beta t}^{-}(y)|^{q'}
\lesssim(\sqrt[4]{t})^{|\beta|+\kappa-2} (w(z)+(\sqrt[4]{t})^{-\kappa}\bar w)
\quad\mbox{provided}\;|y-z|\leq \sqrt[4]{t}
\;\mbox{and}\;\frac{1}{2}\sqrt[4]{t}\le|z-x|\le\sqrt[4]{t}.
\end{equs}
Using this for $y=x$ and averaging over the $z$-annulus 
$\frac{1}{2}\sqrt[4]{t}\le|z-x|\le\sqrt[4]{t}$ yields by \eqref{eq:3ptIV2}
\begin{equs}
\mathbb{E}^\frac{1}{q'}|\delta\Pi_{x\beta t}^{-}(x)|^{q'} \lesssim 
(\sqrt[4]{t})^{|\beta|-2} \bar w.
\label{eq:mal_dual_annealed_almost}
\end{equs}

\medskip

In order to replace in (\ref{eq:mal_dual_annealed_almost}) the evaluation at $x$
by a generic point $y$, we take the Malliavin derivative of 
$\eqref{eq:recenter_Pi_minus}_\beta$, see Subsection~\ref{sec:reconstr_approx}, noting that there is no polynomial defect because of 
$|\beta|<2$:
\begin{equs}\label{eq:delta_recenter_pi_minus}
\delta\Pi_{x\beta}^- = 
\big(\delta \Gamma^*_{xy} \Pi_{y}^-+\Gamma_{xy}^* \delta\Pi_{y}^-\big)_\beta. 
\end{equs}
Applying $(\cdot)_t$, evaluating at $y$, and taking the $\mathbb{E}^\frac{1}{q'}|\cdot|^{q'}$-norm
we obtain 
using $\eqref{eq:delta_gamma}_\beta^{\gamma\not=\pp}$,
$\eqref{eq:pi_minus_generic}_{\prec\beta}$,
$\eqref{eq:gamma}_\beta^{\gamma\not=\pp}$, 
$\eqref{eq:mal_dual_annealed}_{\prec\beta}$,
and $\eqref{eq:mal_dual_annealed_almost}_{\beta}$ (with $x$ replaced by $y$) that
\begin{align*}
\mathbb{E}^{\frac{1}{q'}}|\delta\Pi_{x\beta t}^-(y)|^{q'}& 
\lesssim \sum_{|\gamma|\in \mathsf{A}\cap[\alpha,|\beta|]}
|x-y|^{|\beta|-|\gamma|} \bar w (\sqrt[4]{t})^{|\gamma|-2} 
+ \sum_{|\gamma|\in \mathsf{A}\cap[\alpha,|\beta|]} 
|x-y|^{|\beta|-|\gamma|} (\sqrt[4]{t})^{|\gamma|-2} \bar w 
 \\
 & \lesssim (\sqrt[4]{t})^{\alpha-2} (\sqrt[4]{t} + |y-x|)^{|\beta|-\alpha} \bar w.\qedhere
\end{align*} 
\end{proof}


\section{Constructions} \label{sec:model_construction}

In this section, we carry out one step of the inductive construction
of $c$, $\Pi^-_x$, $\Pi_x$, and $\Gamma_{xy}^*$ such that, next to the population properties, 
the axioms
(\ref{eq:Pi_minus_def}), (\ref{eq:recenter_Pi}), (\ref{ao15}), 
(\ref{eq:model}), and (\ref{eq:recenter_Pi_minus}) are satisfied. 
In fact, instead of 
(\ref{eq:recenter_Pi}) and (\ref{eq:recenter_Pi_minus}) we shall establish the stronger
(\ref{eq:recenter_Pi_specific}) and (\ref{eq:recenter_Pi_minus_specific}).
In line with the logical order of an induction step, see Subsection \ref{sec:order},
we proceed as follows
\begin{itemize}
\item In Subsection \ref{sec:Pi_minus_constr}, 
we construct $\Pi^{-}_x$ by the choice of $c$
that amounts to the BPHZ-choice of renormalization.
\item In Subsection \ref{sec:Pi_construct}, 
we construct $\Pi_x$ by integration of $\Pi_x^{-}$.
\item In Subsection \ref{sec:Gamma_construct}, we construct 
$\Gamma_{xy}^*$ via the re-centering when passing from $\Pi_y$ to $\Pi_x$ 
by a polynomial with coefficients $\pi_{xy}^{({\bf n})}$.
\item In Subsection \ref{sec:dGamma_construction}, 
we construct ${\rm d}\Gamma_{xz}^*$ via 
${\rm d}\pi_{xz}^{({\bf n})}$ (with ${\bf n}=(1,0)$).
\end{itemize}


\subsection{Construction of \texorpdfstring{$c$}{c} and thus \texorpdfstring{$\Pi^{-}_{x}$}{Pi minus x} via BPHZ-choice of renormalization}
\label{sec:Pi_minus_constr}

According to \eqref{eq:triangular_c} and \eqref{eq:triangular_product},
the r.~h.~s.~of
(\ref{eq:Pi_minus_def}) depends on $\Pi_{x}$ only through $\Pi_{x\beta'}$ with $\beta'\prec\beta$, and on $c$ only through $c_{\beta'}$ with $\beta' \preccurlyeq\beta$. In addition, the levels $c_{\beta'}$ with $\beta'\prec \beta$ have stabilized; our task now it to choose $c_\beta$ such that $\eqref{eq:BPHZ}_\beta$ is satisfied.

\medskip

As alluded to at the end of Subsection \ref{sec:model}, we have to include
shift covariance and reflection parity of $\Pi_x$ into the induction in
order for $\eqref{eq:BPHZ}_\beta$ to hold with a $c_\beta$ independent of $x$ and meeting the population condition (\ref{eq:c_pop_cond}). By covariance
under a space-time shift $z+\cdot$, we mean 
\begin{align}\label{eq:Pi_shift}
	\Pi_{z+x\,\beta}[\xi](z+y)=\Pi_{x\,\beta}[\xi(z+\cdot)](y).
\end{align}
By parity, we understand the following covariance under the spatial reflection $Rx=(-x_1,x_2)$:
\begin{align}\label{eq:Pi_parity}
	\Pi_{Rx\,\beta}[\xi](Ry)=(-1)^{\sum_{{\bf n}}n_1\beta({\bf n})}\Pi_{x\,\beta}[\xi(R\cdot)](y).
\end{align}
Note that in view of (\ref{eq:pi_purely_pol}), 
$\eqref{eq:Pi_shift}_\beta$ and $\eqref{eq:Pi_parity}_\beta$ are tautologically satisfied for 
$\beta$ purely polynomial. For any $c_\beta$ satisfying the population constraint $\eqref{eq:c_pop_cond}_\beta$, 
	properties $\eqref{eq:Pi_shift}_{\prec \beta}$ and $\eqref{eq:Pi_parity}_{\prec \beta}$ automatically upgrade to $\eqref{eq:Pi_minus_shift}_\beta$ and $\eqref{eq:Pi_minus_parity}_\beta$, respectively, where
	\begin{align}
		\Pi_{z+x\,\beta}^{-}[\xi](z+y)&=\Pi_{x\,\beta}^{-}[\xi(z+\cdot)](y),\label{eq:Pi_minus_shift}\\
		\Pi_{Rx\,\beta}^{-}[\xi](Ry)&
		=(-1)^{\sum_{{\bf n}}n_1\beta({\bf n})}\Pi_{x\,\beta}^{-}[\xi(R\cdot)](y)
		\label{eq:Pi_minus_parity}.
	\end{align}
\begin{proposition}\label{prop:BPHZ}
Assume that $\eqref{eq:limit}_\beta$, $\eqref{eq:Pi_minus_shift}_\beta$, and $\eqref{eq:Pi_minus_parity}_\beta$ hold\footnote{for any $c$ constrained by \eqref{eq:c_pop_cond}}. Then there exists a choice of $c_\beta$ satisfying $\eqref{eq:c_pop_cond}_\beta$ such that $\eqref{eq:BPHZ}_\beta$ holds.
\end{proposition}
\begin{proof}
Let us denote\footnote{the notation $\tilde\Pi_{x}^{-}$ differs from the one used in Subsection~\ref{divc}}
	\begin{align}\label{cw66}
		\tilde\Pi_{x}^-:=P\sum_{k\ge 1}\mathsf{z}_k\Pi_x^k\partial_1^2\Pi_x
		-\sum_{l\ge 1}\frac{1}{l!}\Pi_x^l(D^{({\bf 0})})^l c,
	\end{align}
	which is defined such that
	\begin{align}\label{cw68}
		\Pi_{x\beta}^{-}=\partial_1^2\Pi_{x\beta-e_0}+\tilde\Pi_{x\beta}^{-}-c_\beta+\xi_\tau\delta_{\beta}^0,
	\end{align}
	with the understanding that the first r.~h.~s.~term is absent unless $\beta(0)\ge 1$. We use \eqref{cw68} in form of $\frac{d}{d t}\Pi_{x\beta t}^-(x)$
	$=\frac{d}{dt}(\partial_1^2\Pi_{x\beta-e_0}$ $+\tilde \Pi^-_{x\beta}+\xi \delta_\beta^0)_t(x)$. 
	By stationarity of $\xi$, we may rewrite
	\begin{equs}\label{aux56}
		\lim_{t\uparrow \infty} \E(\partial_1^2\Pi_{x\beta-e_0} + \tilde \Pi^-_{x\beta})_t(x) = \E(\partial_1^2\Pi_{x\beta-e_0} + \tilde \Pi^-_{x\beta})_\tau(x) + \int_\tau^\infty d t \frac{d}{d t} \E\Pi_{x\beta t}^-(x).
	\end{equs}
	Here the last term of the r.~h.~s. is well-defined thanks to $\eqref{eq:limit}_\beta$. Therefore, $\lim_{t\uparrow \infty} \E \Pi_{x \beta t}^- (x)$ exists for any $c_\beta \in \R$. By $\eqref{eq:Pi_minus_shift}_\beta$ and the invariance in law 
	of $\xi$ under space-time shift in Assumption~\ref{ass:spectral_gap}, 
	it is independent of $x$; in addition, by $\eqref{eq:Pi_minus_parity}_\beta$ and the invariance in law of $\xi$ 
	under spatial reflection in Assumption \ref{ass:spectral_gap}, 
	$\lim_{t\uparrow \infty} \E \Pi_{x \beta t}^- (x)=0$ if $\sum_{\n} n_1 \beta(\mathbf{n})$ is odd (recall $c_\beta$ satisfies \eqref{eq:c_pop_cond} by assumption). If $\sum_{\n}n_1\beta(\n)$ is even, then in view of $|\beta|<2$ we necessarily have $\beta(\n)=0$ for all $\n\neq\0$. 
	Consequently, we may choose $c_\beta$  
	such that $\eqref{eq:BPHZ}_\beta$ holds (note that shifting the value of $c_\beta$ by a constant is the same as shifting the value of the r.~h.~s.~of $\eqref{eq:BPHZ}_\beta$ by the same constant), still being consistent with \eqref{eq:c_pop_cond}.
\end{proof}


\subsection{Construction of \texorpdfstring{$\Pi_x$}{Pi x} via integration}
\label{sec:Pi_construct}

The construction of $\Pi_{x\beta}$ given $\Pi_{x\beta}^{-}$ is provided by
formula $\eqref{eq:Pi_minus_to_Pi}_\beta$, the convergence of which is part of Proposition 
\ref{prop:int_pi_minus}. It follows from the representation $\eqref{eq:Pi_minus_to_Pi}_\beta$ 
that $\eqref{eq:Pi_minus_shift}_\beta$ and $\eqref{eq:Pi_minus_parity}_\beta$ transmit to $\eqref{eq:Pi_shift}_\beta$ and $\eqref{eq:Pi_parity}_\beta$.


\subsection{Construction of \texorpdfstring{$\Gamma_{xy}^*$}{Gamma* xy} via re-centering encoded through \texorpdfstring{$\pi^{({\bf n})}_{xy}$}{pi n xy}}
\label{sec:Gamma_construct}

Via the exponential formula (\ref{exp}), 
$\Gamma_{xy}^*$ is determined by $\{\pi^{({\bf n})}_{xy}\}_{\bf n}$, 
so that the task is to construct the latter. For the induction step $\beta$,
this means constructing $\{\pi^{({\bf n})}_{xy\beta}\}_{{\bf n}:|{\bf n}|<|\beta|}$.
It follows from (\ref{ord12}) and (\ref{ord12b}) that
\begin{align}
(\Gamma_{xy}^*P)_\beta\quad\mbox{depends on}\;\pi_{xy\beta'}^{({\bf n})}\;
\mbox{only for}\;\beta'\prec\beta,\label{ho02}\\
(\Gamma_{xy}^*)_\beta\quad\mbox{depends on}\;\pi_{xy\beta'}^{({\bf n})}\;
\mbox{only for}\;\beta'\preccurlyeq\beta.\label{ho01}
\end{align}
In particular, at the beginning of the induction step $\beta$,
the $\beta$-th row of $\Gamma_{xy}^*P$ and all rows $\beta'\prec\beta$
of $\Gamma_{xy}^*$ have already ``stabilized''.

\medskip

The task is to construct $\pi^{({\bf n})}_{xy\beta}$
such that re-centering $\eqref{eq:recenter_Pi_specific}_{\preccurlyeq\beta}$
\& $\eqref{eq:recenter_Pi_minus_specific}_{\preccurlyeq\beta}$ and transitivity
$\eqref{ao15}_{\preccurlyeq\beta}$ hold.
By (\ref{ho01}), it is enough to verify the current component
$\eqref{eq:recenter_Pi_minus_specific}_{\beta}$
\& $\eqref{eq:recenter_Pi_specific}_{\beta}$ for re-centering.
By the triangularity (\ref{eq:triangular_Gamma}) of $\Gamma_{xy}^*$ 
w.~r.~t.~$\prec$, we likewise see that it is enough to establish
the current component $\eqref{ao15}_{\beta}$ for transitivity. 

\medskip

By the composition rule \eqref{eq:monoid}, transitivity \eqref{ao15} is a consequence of
\begin{align}\label{rec01}
\pi_{xy}^{(\n)} - \pi_{xz}^{(\n)}- \Gamma_{xz}^* \pi_{zy}^{(\n)} = 0
\quad\mbox{and}\quad\pi^{({\bf n})}_{xx}=0.
\end{align}
More precisely, in view of (\ref{ho01}),
transitivity $\eqref{ao15}_{\beta}$ is a consequence of $\eqref{rec01}_{\preccurlyeq\beta}$.
Hence we include $\eqref{rec01}_{\beta}$ into the induction step $\beta$,
and note that by (\ref{ho02}), the induction hypothesis $\eqref{rec01}_{\prec\beta}$
implies
\begin{align}\label{ho03}
\Gamma_{xy}^*P=\Gamma_{xz}^*\Gamma_{zy}^*P\quad\mbox{and}\quad \Gamma_{xx}^*P=P.
\end{align}
It follows from (\ref{ao21}) and the binomial formula
that (\ref{rec01}) holds for purely polynomial $\beta$.
The argument for the base case $\beta=0$ 
is covered by the induction step below.

\medskip

We proceed in three steps:
\begin{itemize}
\item Algebraic argument. Proposition \ref{prop:change_of_basepoint} passes from
the $\Pi$-statement $\eqref{eq:recenter_Pi_specific}_{\prec\beta}$ to the $\Pi^{-}$-statement
$\eqref{eq:recenter_Pi_minus_specific}_{\beta}$.
It is based on the definition (\ref{eq:Pi_minus_def_alt}) of $\Pi^{-}$ in terms
of $\Pi$ and on the multiplicativity (\ref{eq:mult}) of $\Gamma^*$.
The renormalization term transforms as desired because $c$ is independent
of $x$ and $\mathsf{z}_{\bf n}$. 
A difficulty in appealing to multiplicativity
for (\ref{eq:Pi_minus_def}) lies in the presence of the projection $P$
in formula (\ref{eq:Pi_minus_def_alt}). 
\item Integration. Proposition \ref{prop:change_of_basepointII} passes from the
$\Pi^{-}$-statement $\eqref{eq:recenter_Pi_minus_specific}_{\beta}$, or rather the weaker
$\eqref{eq:recenter_Pi_minus}_{\beta}$, 
to the $\Pi$-statement $\eqref{eq:recenter_Pi_specific}_{\beta}$. 
It is based on a Liouville argument for the PDE (\ref{eq:model}), which implies that under the 
growth condition (\ref{eq:pi_generic}), $(\Gamma_{xy}^*P\Pi_{y}-\Pi_{x})_{\beta}$
is a polynomial of degree $<|\beta|$, so that we may define
\begin{equation}\label{eq:new_pi_n}
(\Gamma_{xy}^*P\Pi_{y}-\Pi_{x})_{\beta}=
-\sum_{\n:|\n|<|\beta|} \pi_{xy \beta}^{(\n)} (\cdot - y)^\n.
\end{equation}
\item Three-point argument.
Proposition \ref{prop:change_of_basepointIII} uses 
$\mbox{(\ref{ho03})}_\beta$ and $\mbox{(\ref{eq:new_pi_n})}_\beta$ 
to establish $\eqref{rec01}_{\beta}$, and thus $\eqref{ao15}_{\beta}$. 
\end{itemize}

Before we address Proposition \ref{prop:change_of_basepoint},
we need to argue that the r.~h.~s.~term in $\eqref{eq:recenter_Pi_minus_specific}_\beta$
is well-defined at this stage of the induction step (regular case 2c or singular case 2a). Indeed, because of 
$\Pi^{-}\in\mathsf{\tilde T}^*$ and by the triangular structure (\ref{eq:triangular_Gamma}) 
of $\Gamma^*$ w.~r.~t.~$\prec$, the first r.~h.~s.~term $(\Gamma_{xy}^*\Pi_y^{-})_{\beta}$ involves 
$(\Gamma_{xy}^*)_\beta^\gamma$ only for $\gamma$ not purely polynomial, which has stabilized
by (\ref{ho02}), 
and $\Pi_{y\beta'}^{-}$ for $\beta'\preccurlyeq\beta$, which has been constructed.
Because of the triangular structure \eqref{eq:triangular_product}
of $\sum_{k\ge0}\mathsf{z}_k\pi^k\pi'$,
the second r.~h.~s.~term of $\eqref{eq:recenter_Pi_minus_specific}_\beta$,
namely the product $\big(\mathsf{z}_k(\Gamma^*_{xy}\Pi_y+\pi^{({\bf 0})}_{xy})^k
\partial_1^2(\Gamma^*_{xy}\Pi_y+\pi^{({\bf 0})}_{xy})\big)_\beta$,
depends on its factors only in terms of 
$(\Gamma^*_{xy}\Pi_y+\pi^{({\bf 0})}_{xy})_{\beta'\prec\beta}$.
We note that $(\Gamma_{xy}^*)_{\beta'\prec\beta}^\gamma$ has stabilized for all $\gamma$
by (\ref{ho01}).
Again, by the triangular structure (\ref{eq:triangular_Gamma}) of $\Gamma_{xy}^*$,
$\Pi_{y\beta'}$ is involved only for $\beta'\preccurlyeq\beta$,
which has been constructed.

\begin{proposition}\label{prop:change_of_basepoint}
Suppose $\eqref{eq:recenter_Pi_specific}_{\prec\beta}$ holds. 
Then $\eqref{eq:recenter_Pi_minus_specific}_{\beta}$ holds. Moreover, $\eqref{eq:recenter_Pi_minus_specific}_{\beta}$ is of the form $\eqref{eq:recenter_Pi_minus}_{\beta}$.
\end{proposition}
We highlight that this proposition is independent of the specific value of $c$ and therefore can be obtained before the BPHZ choice of renormalization $\eqref{eq:BPHZ}_\beta$.

\begin{proof}
We deal with the three terms on the r.~h.~s.~of (\ref{eq:Pi_minus_def_alt}) one by one.
In preparation of the first one, we claim that (\ref{eq:mult}) implies for any 
$\pi,\pi'\in\mathsf{T}^*$
\begin{align}\label{eq:mult4}
P\sum_{k\ge 0}\mathsf{z}_k\big(\Gamma^*\pi+\pi^{({\bf 0})}\big)^k(\Gamma^*\pi')
=\Gamma^*P\sum_{k\ge 0}\mathsf{z}_k\pi^k\pi'
+P\sum_{k\ge 0}\mathsf{z}_k\big(\Gamma^*({\rm id}-P)\pi+\pi^{({\bf 0})}\big)^k
(\Gamma^*({\rm id}-P)\pi'),
\end{align}
where $\pi^{({\bf 0})}$ is related to $\Gamma^*$ by (\ref{exp}).
Note that $P$ plays the role of the projection of $\mathbb{R}[[\mathsf{z}_k,\mathsf{z}_{\bf n}]]$
onto $\mathsf{\tilde T}^*$, and ${\rm id}-P$ the one of $\mathsf{T}^*$ onto $\mathsf{\bar T}^*$.
Here comes the argument for (\ref{eq:mult4}):
The $(k+1)$-fold iteration of (\ref{eq:mult}) component-wise reads\footnote{with
the understanding that all multi-indices are populated}
\begin{align*}
(\Gamma^*)^{\gamma}_\beta
=\sum_{\beta_0+\cdots+\beta_{k+1}=\beta}(\Gamma^*)_{\beta_0}^{\gamma_0}
\cdots(\Gamma^*)_{\beta_{k+1}}^{\gamma_{k+1}}
\quad\mbox{provided}\;\gamma=\gamma_0+\cdots+\gamma_{k+1}.
\end{align*}
This allows to characterize the commutator of $\Gamma^*$ and $P$ on a
product of arbitrary $\pi^{(0)}\in\mathsf{\tilde T}^*$, 
$\pi^{(1)},\dots,$ $\pi^{(k+1)}\in\mathsf{T}^*$ (so that below, $P$ acts like a projection from
$\mathbb{R}[[\mathsf{z}_k,\mathsf{z}_{\bf n}]]$ onto $\mathsf{T}^*$):
\begin{align}\label{eq:mult8}
\big(P(\Gamma^*\pi^{(0)})\cdots(\Gamma^*\pi^{(k+1)})\big)_\beta
&=\big(\Gamma^*P\pi^{(0)}\cdots\pi^{(k+1)}\big)_\beta\nonumber\\
&+\sum_{\beta_0+\cdots+\beta_{k+1}=\beta}\sum_{\gamma_0+\cdots+\gamma_{k+1}\;\mbox{not populated}}
(\Gamma^*)^{\gamma_0}_{\beta_0}\pi^{(0)}_{\gamma_0}\cdots
(\Gamma^*)^{\gamma_{k+1}}_{\beta_{k+1}}\pi^{(k+1)}_{\gamma_{k+1}}.
\end{align}
We use (\ref{eq:mult8}) for $\pi^{(0)}=\mathsf{z}_k$, $\pi^{(1)}=\cdots=\pi^{(k)}=\pi$,
and $\pi^{(k+1)}=\pi'$, and combine it with the following
consequence\footnote{recalling our implicit assumption that $\gamma_1,\dots,\gamma_{k+1}$
are populated} of the definitions (\ref{defpp}) \& (\ref{ao09})
\begin{align*}
e_k+\gamma_1+\cdots+\gamma_{k+1}\;\;\mbox{not populated}\quad\Longleftrightarrow\quad
\gamma_1,\dots,\gamma_{k+1}\;\;\mbox{purely polynomial}.
\end{align*}
Hence (\ref{eq:mult8}) assumes the form of
\begin{align*}
P(\Gamma^*\mathsf{z}_k)(\Gamma^*\pi)^k(\Gamma^*\pi')
=\Gamma^*P\mathsf{z}_k\pi^k\pi'
+P(\Gamma^*\mathsf{z}_k)(\Gamma^*({\rm id}-P)\pi)^k(\Gamma^*({\rm id}-P)\pi'),
\end{align*}
which we sum in $k$, 
\begin{align}\label{eq:mult2}
P\sum_{k\ge 0}(\Gamma^*\mathsf{z}_k)(\Gamma^*\pi)^k(\Gamma^*\pi')
=\Gamma^*P\sum_{k\ge 0}\mathsf{z}_k\pi^k\pi'
+P\sum_{k\ge 0}(\Gamma^*\mathsf{z}_k)(\Gamma^*({\rm id}-P)\pi)^k(\Gamma^*({\rm id}-P)\pi').
\end{align}
By (\ref{eq:Gamma_z_k}) followed by the binomial formula, we obtain
\begin{align*}
\sum_{k\ge 0}(\Gamma^*\mathsf{z}_k)(\Gamma^*\pi)^k(\Gamma^*\pi')
=\sum_{k\ge 0}\mathsf{z}_{k}(\Gamma^*\pi+\pi^{({\bf 0})})^{k}
(\Gamma^*\pi')
\end{align*}
and the same identity with $\pi,\pi'$ replaced by $({\rm id}-P)\pi,({\rm id}-P)\pi'$. 
Inserting these two identities in (\ref{eq:mult2}) yields (\ref{eq:mult4}).
We apply (\ref{eq:mult4}) with $\Gamma^*=\Gamma_{xy}^*$ (noting that the
corresponding $\pi_{xy}^{({\bf 0})}$ is a constant in space-time), $\pi=\Pi_{y}$, and
$\pi'=\partial_1^2\Pi_y$, which results in
\begin{align}\label{eq:mult3}
\lefteqn{P\sum_{k\ge 0}\mathsf{z}_k(\Gamma^*_{xy}\Pi_y+\pi_{xy}^{({\bf 0})})^k
\partial_1^2(\Gamma^*_{xy}\Pi_y+\pi_{xy}^{({\bf 0})})}\nonumber\\
&=\Gamma^*_{xy}P\sum_{k\ge 0}\mathsf{z}_k\Pi_y^k\partial_1^2\Pi_y
+P\sum_{k\ge 0}\mathsf{z}_k(\Gamma^*_{xy}({\rm id}-P)\Pi_y+\pi_{xy}^{({\bf 0})})^k
\partial_1^2(\Gamma^*_{xy}({\rm id}-P)\Pi_y+\pi_{xy}^{({\bf 0})}).
\end{align}

\medskip

We now turn to the second r.~h.~s.~contribution to (\ref{eq:Pi_minus_def_alt}), 
and claim that for any $\pi\in\mathsf{T}^*$ and $c\in\mathsf{\tilde T}^*$ satisfying the
population condition (\ref{eq:c_pop_cond}), we have 
\begin{align}\label{eq:mult5}
\sum_{k\ge 0}\frac{1}{k!}(\Gamma^*\pi+\pi^{({\bf 0})})^k(D^{({\bf 0})})^kc
=\Gamma^*\sum_{k\ge0}\frac{1}{k!}\pi^k(D^{({\bf 0})})^kc.
\end{align}
Because of (\ref{eq:mult}), it remains to show
\begin{align}\label{eq:mult6}
\sum_{k\ge 0}\frac{1}{k!}(\Gamma^*\pi+\pi^{({\bf 0})})^k(D^{({\bf 0})})^kc
=\sum_{k\ge0}\frac{1}{k!}(\Gamma^*\pi)^k(\Gamma^*(D^{({\bf 0})})^kc).
\end{align}
Note that the second item in the population condition (\ref{eq:c_pop_cond}) on $c$
can be re-expressed as $D^{({\bf n})}c=0$ for ${\bf n}\not={\bf 0}$,
cf.~(\ref{eq:def_Dn}). We now appeal to (\ref{eq:Gamma_D_c}),
which implies (\ref{eq:mult6}), once more by the binomial formula.
We apply (\ref{eq:mult5}) to $\Gamma^*=\Gamma^*_{xy}$ and $\pi=\Pi_y$, to the effect of
\begin{align}\label{eq:mult7}
\sum_{k\ge 0}\frac{1}{k!}(\Gamma^*_{xy}\Pi_y+\pi^{({\bf 0})}_{xy})^k(D^{({\bf 0})})^kc
=\Gamma^*_{xy}\sum_{k\ge0}\frac{1}{k!}\Pi_y^k(D^{({\bf 0})})^kc.
\end{align}

\medskip

We finally turn to the third r.~h.~s.~contribution to (\ref{eq:Pi_minus_def_alt}) 
and note that by the second item in (\ref{eq:mult}) we have
\begin{align*}
\xi\mathsf{1}=\Gamma^*_{xy}\xi\mathsf{1}.
\end{align*}
The sum of (\ref{eq:mult3}), (\ref{eq:mult7}) and the last identity yields
(\ref{eq:recenter_Pi_minus_specific}) by definition (\ref{eq:Pi_minus_def}). 

\medskip

Finally, in order to pass from (\ref{eq:recenter_Pi_minus_specific}) to (\ref{eq:recenter_Pi_minus}),
we need to argue that the $\beta$-component
of the second r.~h.~s.~term in (\ref{eq:recenter_Pi_minus_specific})
is a polynomial of degree $\le|\beta|-2$.
In view of the definition (\ref{eq:Tbar_def}) of $P$ and the definition
(\ref{eq:pi_purely_pol}) of $\Pi_y$,
the second r.~h.~s.~term in (\ref{eq:recenter_Pi_minus_specific}) is (componentwise) a polynomial.
In view of the triangularity \eqref{eq:triangular_Gamma} of $\Gamma^*$ w.~r.~t.~$|\cdot |$ and the structure (\ref{ord18})
of $\sum_{k\ge 0}\mathsf{z}_k\pi^k\pi'$ w.~r.~t.~$|\cdot|$,
combined with (\ref{ao20}), the degree of the $\beta$-component of the polynomial
is $\le|\beta|-2$. Hence (\ref{eq:recenter_Pi_minus_specific}) is indeed a stronger version
of (\ref{eq:recenter_Pi_minus}).
\end{proof}


Before addressing Proposition \ref{prop:change_of_basepointII},
we need to argue that $(\Gamma_{xy}^*P\Pi_y)_\beta$ appearing in
(\ref{eq:new_pi_n}) is well-defined at this stage of the induction step (regular case 5a or singular case 6c):
By (\ref{ho02}), $(\Gamma_{xy}^*P)_\beta$ has stabilized.
By the triangularity of  
(\ref{eq:triangular_Gamma}) of $\Gamma_{xy}^*$ w.~r.~t.~$\prec$, 
$(\Gamma_{xy}^*P\Pi_y)_\beta$
only involves $\Pi_{y\gamma}$'s with $\gamma\preccurlyeq\beta$, and thus already constructed.
As always in an integration step, both for the estimates in Proposition \ref{prop:int_pi_minus} 
and for uniqueness here, we have to stay away from integers, cf.~(\ref{irrational}).

\begin{proposition}[Liouville]\label{prop:change_of_basepointII}
Suppose $\eqref{eq:pi_generic}_{\preccurlyeq\beta}$ and
$\eqref{eq:recenter_Pi_minus}_\beta$ hold, and that $\eqref{eq:gamma}_\beta^\gamma$ holds for all $\gamma$ not purely polynomial. 
There exist $\{\pi_{xy\beta}^{({\bf n})}\}_{|{\bf n}|<|\beta|}$ such that
$\eqref{eq:recenter_Pi_specific}_\beta$ holds.
\end{proposition}


\begin{proof}
Applying $\partial_2-\partial_1^2$ to $(\Gamma_{xy}^* P\Pi_{y})_\beta$ results in
\begin{align}\label{eq:transeqn}
(\partial_2-\partial_1^2)(\Gamma_{xy}^* P\Pi_{y})_{\beta}
&\stackrel{\hphantom{\eqref{eq:model}_{\preccurlyeq\beta}}}{=}(\Gamma_{xy}^* (\partial_2-\partial_1^2)P\Pi_{y})_\beta\nonumber\\
&\stackrel{\eqref{eq:model}_{\preccurlyeq\beta}}{=}
(\Gamma_{xy}^*\Pi_{y}^{-})_\beta \nonumber\\
&\, \stackrel{\eqref{eq:recenter_Pi_minus}_{\beta}}{=} \,
\Pi_{x\beta}^{-}+\mbox{polynomial of degree $\le|\beta|-2$},
\end{align}
where we used the triangularity w.~r.~t.~both $\prec$ and homogeneity $|\cdot|$,
the latter in order to control the degree of the polynomial.
Hence $(\Gamma_{xy}^* P\Pi_{y})_{\beta}$ solves the the same PDE (\ref{eq:model}) as $\Pi_{x\beta}$ modulo a polynomial of degree $|\beta|-2$, so that
\begin{align}\label{eq:quenched4}
(\partial_2-\partial_1^2)\partial^{\bf n}(\Gamma_{xy}^* P\Pi_{y}-\Pi_{x})_{\beta}
=0\quad\mbox{provided $|{\bf n}|>|\beta|-2$}.
\end{align}
In view of the kernel estimate (\ref{cw61}), the quantitative
$\eqref{eq:pi_generic}_{\preccurlyeq\beta}$ yields 
the qualitative $\eqref{eq:quenched}_{\preccurlyeq\beta}$, where
\begin{align}\label{eq:quenched}
\lim_{t\uparrow\infty}\mathbb{E}^\frac{1}{p}|\partial^{\bf n}\Pi_{x\beta\,t}(\cdot)|^p=0
\quad\mbox{provided $|{\bf n}|>|\beta|$}.
\end{align}
Hence by the triangularity of $\Gamma_{xy}^*$ w.~r.~t.~both $\prec$
and $|\cdot|$, and using $\eqref{eq:gamma}_\beta^{\gamma\neq\pp}$, this implies
\begin{align}\label{eq:quenched2}
\lim_{t\uparrow\infty}\mathbb{E}^\frac{1}{p}|\partial^{\bf n}
(\Gamma_{xy}^* P\Pi_{y}-\Pi_{x})_{\beta t}(\cdot)|^p=0
\quad\mbox{provided $|{\bf n}|>|\beta|$}.
\end{align}
We now argue that (\ref{eq:quenched4}) and (\ref{eq:quenched2}) imply
\begin{align}\label{eq:Liou}
\partial^{\bf n}(\Gamma_{xy}^*P\Pi_{y}-\Pi_{x})_{\beta}=0\quad\mbox{provided}\;
|{\bf n}|>|\beta|.
\end{align}
Indeed, according to the factorization (\ref{eq:factorization}) and the definition of $(\cdot)_t$, 
(\ref{eq:quenched4}) implies 
\begin{align*}
\partial_t\partial^{\bf n}(\Gamma_{xy}^*P\Pi_{y}-\Pi_{x})_{\beta t}
=0\quad\mbox{provided}\;|{\bf n}|>|\beta|.
\end{align*}
This permits to pass from (\ref{eq:quenched2}) for $t\uparrow\infty$ to (\ref{eq:Liou})
for $t\downarrow0$.

\medskip

As a consequence of \eqref{eq:Liou}, $(\Gamma_{xy}^*P\Pi_{y}-\Pi_{x})_{\beta}$ is a 
polynomial of degree $\leq |\beta|$, which by $|\beta|\notin \N$, cf. \eqref{irrational}, 
allows to define $\pi_{xy\beta}^{(\n)}$ by (\ref{eq:new_pi_n}).
By definition (\ref{eq:pi_purely_pol}), and because of
$\Gamma^*\mathsf{z}_{\bf n}$ $\stackrel{\eqref{exp}}{=}\mathsf{z}_{\bf n}+\pi^{({\bf n})}$
in form of $(\Gamma^*)_{\beta}^{e_{\bf n}}$ $=\pi^{({\bf n})}_\beta$, 
this turns into the desired $\eqref{eq:recenter_Pi_specific}_\beta$.
\end{proof}

\begin{proposition}\label{prop:change_of_basepointIII}
$\eqref{rec01}_{\prec\beta}$ and $\eqref{eq:recenter_Pi_specific}_{\preccurlyeq\beta}$ 
imply $\eqref{ao15}_\beta$ and $\eqref{rec01}_{\beta}$. 
\end{proposition}


\begin{proof}
Recall that $\eqref{ao15}_\beta$ is a consequence of $\eqref{rec01}_{\preccurlyeq\beta}$ by the composition rule \eqref{eq:monoid}.
Furthermore, 
the second item in $\mbox{(\ref{rec01})}_\beta$ 
is a consequence of the first, and the induction hypothesis $\mbox{(\ref{rec01})}_{\prec\beta}$:
Indeed, the first item in $\mbox{(\ref{rec01})}_\beta$
for $x=y=z$ assumes the form of $\pi_{xx\beta}^{({\bf n})}$
$=(({\rm id}-\Gamma_{xx}^*)\pi_{xx}^{({\bf n})})_\beta$,
so that it remains to appeal to the strict triangularity of $\Gamma_{xx}^*-{\rm id}$
w.~r.~t.~$\prec$.

\medskip

Akin to the three-point argument, we will establish the first item in 
$\mbox{(\ref{rec01})}_\beta$ in form of
\begin{align}\label{ho08}
\sum_{{\bf n}:|{\bf n}|<|\beta|}(\pi_{xy}^{(\n)}- \pi_{xz}^{(\n)}- 
\Gamma_{xz}^*\pi_{zy}^{(\n)})_\beta (\cdot - y)^{\n}=0.
\end{align}
The representation $({\rm id}-P)\pi$ 
$=\sum_{{\bf m}\not={\bf 0}}\pi_{e_{\bf m}}\mathsf{z}_{\bf m}$
allows us to rewrite (\ref{ho06}) as $\pi_{zy}^{({\bf n})}=P\pi_{zy}^{({\bf n})}$
$+\sum_{{\bf m}>{\bf n}}\binom{\bf m}{\bf n}(y-z)^{{\bf m}-{\bf n}}\mathsf{z}_{\bf m}$.
Hence by (\ref{eq:Gamma_z_n}) we obtain
$\Gamma_{xz}^*\pi_{zy}^{({\bf n})}=\Gamma_{xz}^*P\pi_{zy}^{({\bf n})}$
$+\sum_{{\bf m}>{\bf n}}$ $\binom{\bf m}{\bf n}$ $(y-z)^{{\bf m}-{\bf n}}$
$(\mathsf{z}_{\bf m}+\pi_{xz}^{({\bf m})})$. Since $\beta$ is not purely polynomial,
this implies
\begin{align*}
(\pi_{xz}^{({\bf n})}+\Gamma_{xz}^*\pi_{zy}^{({\bf n})})_\beta=(\Gamma_{xz}^*P\pi_{zy}^{({\bf n})})_\beta
+\sum_{{\bf m}\ge{\bf n}}\tbinom{\bf m}{\bf n}(y-z)^{{\bf m}-{\bf n}}\pi_{xz\beta}^{({\bf m})}.
\end{align*}
Since by (\ref{eq:pi_n_pop}) the sum is restricted to $|{\bf m}|<|\beta|$
this yields by the binomial formula
\begin{align*}
\lefteqn{\sum_{{\bf n}:|{\bf n}|<|\beta|}
(\pi_{xz}^{({\bf n})}+\Gamma_{xz}^*\pi_{zy}^{({\bf n})})_\beta(\cdot - y)^{\n}}\nonumber\\
&=\sum_{{\bf n}:|{\bf n}|<|\beta|}(\Gamma_{xz}^*P\pi_{zy}^{({\bf n})})_\beta(\cdot - y)^{\n}
+\sum_{{\bf m}:|{\bf m}|<|\beta|}\pi_{xz\beta}^{({\bf m})}(\cdot-z)^{{\bf m}}.
\end{align*}
In view of the triangularity (\ref{eq:triangular_Gamma}) of $\Gamma_{xz}^*$ w.~r.~t.~$|\cdot|$
and once more (\ref{eq:pi_n_pop}), we may rewrite the first r.~h.~s.~term to obtain
\begin{align}\label{ho07}
\lefteqn{\sum_{{\bf n}:|{\bf n}|<|\beta|}
(\pi_{xz}^{({\bf n})}+\Gamma_{xz}^*\pi_{zy}^{({\bf n})})_\beta(\cdot - y)^{\n}}\nonumber\\
&=\sum_{\gamma\not=\pp}(\Gamma_{xz}^*)_\beta^\gamma
\sum_{{\bf n}:|{\bf n}|<|\gamma|}\pi_{zy\gamma}^{({\bf n})}(\cdot - y)^{\n}
+\sum_{{\bf m}:|{\bf m}|<|\beta|}\pi_{xz\beta}^{({\bf m})}(\cdot-z)^{{\bf m}}.
\end{align}
By the triangularity (\ref{eq:triangular_Gamma}) of $\Gamma_{xz}^*$ w.~r.~t.~$\prec$,
we may appeal to $\mbox{(\ref{eq:new_pi_n})}_{\preccurlyeq\beta}$ (which we recall is a consequence of $\eqref{eq:recenter_Pi_specific}_{\preccurlyeq\beta}$) to see that
(\ref{ho08}), into which we insert (\ref{ho07}), reduces to
\begin{align*}
(\Gamma_{xy}^*P\Pi_y-\Pi_x)_\beta=
(\Gamma_{xz}^*P(\Gamma_{zy}^*P\Pi_y-\Pi_z)\big)_\beta
+(\Gamma_{xz}^*P\Pi_z-\Pi_x)_\beta.
\end{align*}
This identity follows from (\ref{eq:recIII4b}) and the first item in $\mbox{(\ref{ho03})}_\beta$,
which as explained is a consequence of $\eqref{rec01}_{\prec\beta}$.
\end{proof}

\subsection{Construction of \texorpdfstring{${\rm d}\Gamma_{xz}^*$}{d Gamma* xz} through \texorpdfstring{${\rm d}\pi^{({\bf n})}_{xz}$}{d pi n xz}}
\label{sec:dGamma_construction}

Recall the Ansatz (\ref{eq:def_dGamma}) for ${\rm d}\Gamma_{xz}^*$.
We seek ${\rm d}\pi^{(1,0)}_{xz}\in Q\mathsf{\tilde T}^*$ 
such that (\ref{eq:controlled_path_qual}) holds, which since $\partial_1\Pi_z$ and $\partial_1\delta\Pi_x$ are well-defined\footnote{ see the discussions before Proposition \ref{rem:1} and after Proposition \ref{rem:1_mal}, respectively} amounts to
\begin{align}\label{eq:partial_dGamma}
Q\big(\partial_1\delta\Pi_{x}(z)-{\rm d}\Gamma^*_{xz}\partial_1\Pi_z(z)\big)=0.
\end{align}
Evaluating (\ref{eq:def_dGamma}) on $\mathsf{z}_{(1,0)}$, and
using the definitions (\ref{eq:def_Dnull}) \& (\ref{eq:def_Dn}) of $D^{({\bf n})}$
(together with the last item in (\ref{eq:mult})), we obtain
${\rm d}\Gamma_{xz}^*\mathsf{z}_{(1,0)}={\rm d}\pi^{(1,0)}_{xz}$.
Applying $Q$ and feeding in the postulate ${\rm d}\pi^{(1,0)}_{xz}\in Q\mathsf{\tilde T}^*$,
we have
\begin{align}\label{eq:dGamma_z_10}
{\rm d}\pi^{(1,0)}_{xz}=Q{\rm d}\Gamma_{xz}^*\mathsf{z}_{(1,0)}.
\end{align}
Hence in view of $({\rm id}-P)\partial_1\Pi_z(z)=\mathsf{z}_{(1,0)}$,
which follows from (\ref{eq:pi_purely_pol}), (\ref{eq:partial_dGamma}) can be re-arranged to 
\begin{align}\label{dpi}
{\rm d} \pi_{xz}^{(1,0)}=Q\big(\partial_1 \delta\Pi_{x}(z)
-{\rm d}\Gamma^*_{xz}P\partial_1 \Pi_{z}(z)\big).
\end{align}

\medskip

According to the order within an induction step stated in Subsection~\ref{sec:order} (singular case 9a),
$(\Gamma_{xz}^*P)_{\preccurlyeq\beta}$ has been constructed
(Proposition \ref{prop:gamma_npp}), and based on this and the induction
hypothesis, $({\rm d}\Gamma_{xz}^*P)_\beta$
has been constructed (Proposition \ref{prop:form_bound}). 
Also $(Q\Pi_x)_\beta$ and thus $(Q\delta\Pi_x)_\beta$ have been constructed 
(Proposition \ref{prop:int_pi_minus}).
In view of the strict triangularity (\ref{eq:triangular_d_Gamma}) of ${\rm d}\Gamma_{xz}^*$
w.~r.~t.~$\prec$, 
(\ref{dpi}) involves only $(P\Pi_{z})_{\prec\beta}$,
which has been constructed by induction hypothesis.
Hence ${\rm d}\pi_{xz\beta}^{(1,0)}$ is well-defined through (\ref{dpi}).


\section{\texorpdfstring{Divergent bounds and analyticity}{Divergent bounds and analyticity}}\label{sec:proof_rem}
\subsection{\texorpdfstring{Divergent bounds: Proof of Propositions~\ref{rem:1} and \ref{rem:1_mal}}{Divergent bounds: Proof of Propositions 2.3 and 4.13}}\label{sec:proof_div}
\begin{proof}[Proof of Proposition \ref{rem:1}]
We now embark on the proof of $(\ref{cw03})_\beta$, $\eqref{cw60}_\beta$, $(\ref{cw60_cont})_\beta$, and 
$\eqref{cw60_minus_cont}_\beta$.
In view of (\ref{eq:c_pop_cond}) and (\ref{eq:pi_purely_pol}), it is enough to consider 
(populated) $\beta$ that are not purely polynomial. 
We shall frequently appeal to $\tilde{\Pi}_x^-$ as defined in \eqref{cw66}.

\medskip

\textbf{Step 1.}
We first quantify the continuity of $\xi_\tau$, which amount to the base case $\beta=0$ in $\eqref{cw60_minus_cont}_\beta$.
By a scaling argument and the fact that $\psi_{t=1}$ is a Schwartz
function, we get $\|\partial^{\bf n}\psi_\tau\|_*$ 
$\lesssim(\sqrt[4]{\tau})^{-2-|{\bf n}|+\alpha}$, see (\ref{as02}) for the definition
of $\|\cdot\|_*$. By the mean-value theorem and translation invariance of the norm, this implies
$\|\psi_\tau(y-\cdot)-\psi_\tau(x-\cdot)\|_*$ $\lesssim (\sqrt[4]{\tau})^{-2}|y-x|^\alpha$.
Since $\psi_\tau(y-\cdot)-\psi_\tau(x-\cdot)$ is the Malliavin derivative of
the cylindrical random variable $\xi_\tau(y)-\xi_\tau(x)$, we obtain by the $\mathbb{L}^p$-version
(\ref{eq:spectral_gap_var}) of the SG inequality (and since by stationarity, 
$\mathbb{E}\xi_\tau(y)$ does not depend on $y$) 
\begin{align}\label{cw06}
\mathbb{E}^\frac{1}{p}|\xi_\tau(y)-\xi_\tau(x)|^p\lesssim(\sqrt[4]{\tau})^{-2}|y-x|^\alpha.
\end{align}

\medskip

\textbf{Step 2.}
$\eqref{eq:pi_generic}_{\prec\beta}$ \&
$\eqref{eq:BPHZ}_\beta$ \&
$\eqref{cw03}_{\prec\beta}$ \&
$\eqref{cw60}_{\prec\beta}$ \&
$\eqref{eq:limit}_\beta$ 
$\implies \eqref{cw03}_\beta$. In fact, we also need the following estimate on $\tilde \Pi_{x\beta}^-$,    
\begin{align}\label{cw65}
\mathbb{E}^\frac{1}{p}|\tilde\Pi_{x\beta}^{-}(y)|^p
\lesssim(\sqrt[4]{\tau})^{\alpha-2}(\sqrt[4]{\tau}+|y-x|)^{|\beta|-2\alpha}|y-x|^\alpha.
\end{align}
which is meant to include the statement that $\tilde\Pi_{x\beta}^{-}$ vanishes unless $|\beta|\ge 2\alpha$. 
Note that the presence of the bare factor $|y-x|^\alpha$ will be important for integration below. The proof of 
$\eqref{cw65}_\beta$ is similar to the one of $\eqref{cw65_cont}_\beta$ below, using only the estimates 
$\eqref{eq:recIII12}_{\preccurlyeq\beta}$,  $\eqref{cw60}_{\prec\beta}$, $\eqref{eq:recIII12}_{e_k+\beta_1}$ and $\mbox{(\ref{cw03})}_{\prec\beta}$.

\medskip

We now turn to the estimate $\eqref{cw03}_\beta$. For this we note that $\eqref{eq:BPHZ}_\beta$ implies
	\begin{equs}\label{fo01}
		c_\beta = \lim_{t\uparrow \infty} \E(\partial_1^2\Pi_{x\beta-e_0} + \tilde \Pi^-_{x\beta})_t(x). 
\end{equs}
Recall now the representation \eqref{aux56}. Using that $\beta - e_0\prec \beta$, we learn from $(\ref{mt98})_{\prec \beta}$ (which we recall is a consequence of $\eqref{eq:pi_generic}_{\prec\beta}$) that the first term on the 
r.~h.~s.~is estimated 
by $\sqrt[4]{\tau}^{|\beta|-2}$, as desired. 
Rewriting the second r.~h.~s.~term as $\int dy\psi_\tau(x-y)\mathbb{E}\tilde\Pi_{x\beta}^{-}(y)$, we
get the same estimate from $\mbox{(\ref{cw65})}_{\beta}$ with help
of the moment bounds \eqref{cw61}. For the third term, the estimate follows from 
$\eqref{eq:limit}_\beta$ (for $T=\tau$ and $y=x$).

\medskip

\textbf{Step 3.} 
$\eqref{eq:pi_generic}_{\prec\beta}$ \&
$\eqref{eq:gamma}_{\prec\beta}$ \&
$\eqref{cw03}_{\prec \beta}$ \&
$\eqref{cw60}_{\prec\beta}$ \& 
$\eqref{cw60_cont}_{\prec \beta}$ \&
$\eqref{eq:recenter_Pi_specific}_{\prec\beta}$
$\implies \eqref{cw60_minus_cont}_\beta$.
By \eqref{cw68}, noting that $\beta-e_0\prec \beta$, the 
required estimate follows by $\eqref{cw60_cont}_{\prec \beta}$ and the following continuity estimate on $\tilde \Pi_{x\beta}^-$,
\begin{equs}
\E^{\frac{1}{p}} |\tilde \Pi_{x\beta}^-(y) - \tilde \Pi_{x\beta}^-(z)|^p 
\lesssim(\sqrt[4]{\tau})^{-2} (\sqrt[4]{\tau}+|y-z|+|z-x|)^{|\beta|-\alpha} 
|y-z|^\alpha \label{cw65_cont}.
\end{equs}
In order to prove $\eqref{cw65_cont}_\beta$ we consider (\ref{cw66}) componentwise and prove the corresponding continuity 
estimate for each summand separately. Note that each summand has the product structure 
$(\pi\pi')_\beta = \sum_{\beta_1+\beta_2 = \beta} \pi_{\beta_1} \pi_{\beta_2}'$. In particular, summands
in the first term on the r.~h.~s.~of (\ref{cw66}) correspond to 
$\pi=\z_k\Pi_x^k$ and $\pi' = \partial_1^2\Pi_x$, 
where purely polynomial $\beta_1$ as well as $\beta_1=0$ do not contribute, 
to the effect of $\beta_2\prec\beta$ by \eqref{eq:recIII25}.
Similarly, since $\beta_2=e_{(1,0)}$ does not contribute, 
a careful inspection of \eqref{ord01} shows that $\beta_1\preccurlyeq\beta$.
Therefore, the desired estimate 
for these terms follows via H\"older's inequality in probability from the boundedness 
$\eqref{eq:recIII12}_{\preccurlyeq\beta}$ (which we recall is a consequence of $\eqref{eq:pi_generic}_{\prec\beta}$) \& $\eqref{cw60}_{\prec\beta}$
and continuity $\eqref{eq:recIII10}_{\preccurlyeq\beta}$ 
(which we recall is a consequence of $\eqref{eq:pi_generic}_{\prec\beta}$ \& $\eqref{eq:gamma}_{\prec\beta}$ \& $\eqref{eq:recenter_Pi_specific}_{\prec\beta}$) 
\& $\eqref{cw60_cont}_{\prec\beta}$
of $\z_k\Pi_x^k$ and $\partial_1^2\Pi_x$. 
The summands in the second term on the r.~h.~s.~of (\ref{cw66}) correspond 
to $\pi=\Pi_x^k$ and $\pi' = (D^{({\bf 0})})^kc$, 
where $\pi'_{\beta_2}$ depends only on $c_\gamma$ for $\gamma\prec\beta$ by \eqref{eq:triangular_c}. 
Moreover, from \eqref{Dopop} and \eqref{eq:c_pop_cond} we infer 
that $\pi'_{\beta_2}\neq0$ only for $k\leq[\beta_2]$, 
which by \eqref{ord01} implies $e_k\preccurlyeq\beta_2$ 
and hence $e_k+\beta_1\preccurlyeq\beta$.
Therefore, the desired estimate follows by continuity\footnote{
note that $\eqref{eq:recIII10}_{e_k+\beta_1}$ yields an estimate 
on $(\Pi_x^k)_{\beta_1}$} 
$\eqref{eq:recIII10}_{e_k+\beta_1\preccurlyeq\beta}$ of $\Pi_x^k$ and
the estimate $|((D^{({\bf 0})})^kc)_{\beta_2}| = \sum_\gamma |((D^{({\bf 0})})^k)_{\beta_2}^\gamma| |c_\gamma| \lesssim
(\sqrt[4]{\tau})^{|\beta_2|-\alpha k-2}$ which follows from \eqref{Dopop} and 
$\eqref{cw03}_{\prec \beta}$.

\medskip

\textbf{Step 4.}
$\eqref{eq:pi_generic}_{\preccurlyeq\beta}$ \&
$\eqref{eq:gamma}_\beta$ \&
$\eqref{cw03}_{\prec\beta}$ \&
$\eqref{cw60}_{\prec\beta}$ \&
$\eqref{eq:recenter_Pi_specific}_\beta$ \& 
$(\ref{eq:pi_minus_generic})_\beta$ \&
$\eqref{cw63_generic}_{\prec\beta}$ 
$\implies \eqref{cw60_cont}_\beta$.
We start by recalling a version of Campanato's argument, namely
that the pointwise $C^{2+\alpha}$-estimate $\eqref{cw60_cont}_\beta$ may be encoded in terms
of averages.
More precisely, we argue that $\mbox{(\ref{cw63_generic})}_{\beta}$ 
implies the pointwise control $\eqref{cw60_cont}_\beta$, where 
\begin{align}\label{cw63_generic}
\mathbb{E}^\frac{1}{p}|\partial^{\bf n}\partial^{\mathbf{m}}\Pi_{x\beta\,t}(y)|^p
\lesssim(\sqrt[4]{\tau})^{-2} (\sqrt[4]{\tau+t} + |y-x|)^{|\beta|-\alpha} (\sqrt[4]{t})^{\alpha-|{\bf n}|} 
\quad\mbox{for}\;{\bf n}\not={\bf 0},\; {\bf m}\in\{(2,0),(0,1)\}, 
\end{align}
which is part of the induction hypothesis.
For simplicity, let us restrict to the case ${\bf m} = (2,0)$.
To this purpose, we split the first term on the l.~h.~s.~of $(\ref{cw60_cont})_\beta$
into $(\partial_1^2\Pi_{x\beta}- \partial_1^2\Pi_{x\beta \, t})(y)$
$+(\partial_1^2\Pi_{x\beta \, t}(y) - \partial_1^2\Pi_{x\beta \, t}(z))$
$+(\partial_1^2\Pi_{x\beta \, t} - \partial_1^2\Pi_{x\beta})(z)$, where the choice
$t=|y-z|^4$ will turn out to be natural. Based on (\ref{eq:kernel}) we rewrite the first and last contribution as
\begin{align}\label{cw72}
\partial_1^2\Pi_{x\beta }-\partial_1^2\Pi_{x\beta \, t}
=\int_0^t ds(-\partial_2^2+\partial_1^4)\partial_1^2\Pi_{x\beta\,s}.
\end{align}
Hence on these contributions, we may use $(\ref{cw63_generic})_\beta$ 
with $|{\bf n}|=4$ and $t$ replaced by $s$, yielding control by
$(\sqrt[4]{\tau})^{-2} (\sqrt[4]{\tau+t} + |y-x|+|z-x|)^{|\beta|-\alpha} 
\int_0^t ds (\sqrt[4]{s})^{\alpha - 4} $. Note that thanks to $\alpha>0$, the last
integral converges; by the choice of $t=|y-z|^4$, and by $|y-x|\le|y-z|+|z-x|$,
this gives rise to the r.~h.~s.~of $(\ref{cw60_cont})_\beta$.
We rewrite the remaining middle term 
$\partial_1^2\Pi_{x\beta \, t}(y)-\partial_1^2\Pi_{x\beta \, t}(z)$ 
as an integral of the derivative along the connecting segment.
Since this involves $\partial^{\bf n}\partial_1^2\Pi_{x\beta \, t}$ for
${\bf n}\in\{(1,0),(0,1)\}$ we obtain from $(\ref{cw63_generic}) _\beta$ an estimate by
\begin{equs}
(\sqrt[4]{\tau})^{-2} (\sqrt[4]{\tau+t} + |y-x|+|z-x|)^{|\beta|-\alpha} \big( (\sqrt[4]{t})^{\alpha-1} |y-z| + (\sqrt[4]{t})^{\alpha-2} |y-z|^2\big). 
\end{equs}
By our choice of $t= |y-z|^4$ also this term can be subsumed into to the r.~h.~s.~of 
$(\ref{cw60_cont})_\beta$.

\medskip

We now turn to integration proper and argue
that $\mbox{(\ref{cw65})}_{\beta}$ (which we recall is a consequence of $\eqref{eq:pi_generic}_{\prec\beta}$ \&
$\eqref{cw03}_{\prec\beta}$ \&
$\eqref{cw60}_{\prec\beta}$) 
and $\mbox{(\ref{cw63_generic})}_{\prec\beta}$
imply $(\mbox{\ref{cw63_generic})}_{\beta}$. 
Again, for simplicity we restrict ourselves to the case 
${\bf m} = (2,0)$.
We first argue that it is sufficient to establish $(\mbox{\ref{cw63_generic})}_{\beta}$ for $y=x$, which assumes the form
\begin{align}\label{cw63}
\mathbb{E}^\frac{1}{p}|\partial^{\bf n}\partial_1^2\Pi_{x\beta\,t}(x)|^p
\lesssim(\sqrt[4]{\tau})^{-2} (\sqrt[4]{\tau+t})^{|\beta|-\alpha} (\sqrt[4]{t})^{\alpha-|{\bf n}|}
\quad\mbox{for}\;{\bf n}\not={\bf 0}.
\end{align}
Indeed, in order to pass from the generic point $y$ to the base point $x$, 
we appeal again to (\ref{eq:recenter_Pi_specific}) in form of $\partial^\n\partial_1^2\Pi_{x\beta}
=\sum_\gamma(\Gamma^*_{xy})_\beta^\gamma \partial^\n\partial_1^2\Pi_{y\gamma}$
and to the triangularity \eqref{eq:triangular_Gamma}.
Hence by $(\ref{eq:gamma})_\beta$,
$(\mbox{\ref{cw63_generic})}_{\beta}$ is a consequence
of $(\mbox{\ref{cw63})}_{\preccurlyeq\beta}$:
\begin{align*}
\E^{\frac{1}{p}}|\partial^\n\partial_1^2 \Pi_{x \beta \, t}(y)|^p 
& \lesssim \sum_{|\gamma|\in\mathsf{A}\cap(-\infty,|\beta|]} |y-x|^{|\beta|-|\gamma|} 
(\sqrt[4]{\tau})^{-2}  (\sqrt[4]{\tau+t})^{|\gamma|-\alpha} (\sqrt[4]{t})^{\alpha-|\n|} \\
& \lesssim (\sqrt[4]{\tau})^{-2} (\sqrt[4]{\tau+t} +|y-x|)^{|\beta|-\alpha} (\sqrt[4]{t})^{\alpha-|\n|}.
\end{align*}

\medskip

We also note that for $(\ref{cw63})_\beta$, it is enough to deal with the case of $t\le\tau$.
Indeed, we obtain from $\eqref{mt98}_\beta$ (with ${\bf n}$ replaced by ${\bf n} + (2,0)$, which we recall is a consequence of $\eqref{eq:pi_generic}_\beta$) that $\E^{\frac{1}{p}}|\partial^\n\partial_1^2 \Pi_{x \beta \, t}(x)|^p
\lesssim (\sqrt[4]{t})^{|\beta|-2-|\n|}$, which for $t\ge\tau$ is dominated by
the r.~h.~s.~of $(\ref{cw63})_\beta$.
Therefore, it remains to prove $(\ref{cw63}) _\beta$ in the reduced form of
\begin{align}\label{cw35}
\mathbb{E}^\frac{1}{p}|\partial^{\bf n}\partial_1^2\Pi_{x\beta\,t}(x)|^p
\lesssim(\sqrt[4]{\tau})^{-2+|\beta|-\alpha} (\sqrt[4]{t})^{\alpha-|{\bf n}|}
\quad\mbox{for}\;{\bf n}\not={\bf 0}\;\mbox{and}\;t\le\tau.
\end{align}

\medskip

In establishing $(\ref{cw35})_\beta$, 
we distinguish the cases $|\beta|<3$ and $|\beta|\ge 3$, starting with the former.
Upon applying $\partial^\n\partial_1^2$ to the integral representation (\ref{eq:Pi_minus_to_Pi}), 
the Taylor polynomial drops out because of $|\beta|<3$ and $\n\neq {\bf 0}$.
Also applying $(\cdot)_t$, we obtain with help of (\ref{eq:semi_group})
\begin{align*}
\partial^\n\partial_1^2\Pi_{x\beta\,t}
=-\int_t^\infty ds \partial^\n\partial_1^2(\partial_2+\partial_1^2)\Pi_{x\beta\,s}^{-}.
\end{align*}
We split the integral at $s=\tau\ge t$, substitute (\ref{cw68})
in the first part, and we evaluate at $x$:
\begin{align}\label{cw70}
\partial^{\bf n}\partial_1^2\Pi_{x\beta\,t}(x)
=&-\int_t^\tau ds\partial^{\bf n}\partial_1^2(\partial_2+\partial_1^2)
(\partial_1^2\Pi_{x\beta-e_0}+\tilde\Pi_{x\beta}^{-}+\xi_\tau\delta_\beta^0)_{s}(x)\nonumber\\
&-\int_\tau^\infty ds\partial^{\bf n}\partial_1^2(\partial_2+\partial_1^2)\Pi_{x\beta\,s}^{-}(x).
\end{align}

\medskip

We start with the first r.~h.~s.~contribution to (\ref{cw70}), 
$\int_t^\tau ds$ $\partial^{\bf n}\partial_1^2(\partial_2+\partial_1^2)
\partial_1^2\Pi_{x\beta-e_0\,s}(x)$. Our ordering 
$\beta-e_0 \prec \beta$
allows us to appeal to $(\mbox{\ref{cw35})}_{\prec\beta}$, which we use 
with ${\bf n}$ replaced by $(2,1)+{\bf n}$ and $(4,0)+{\bf n}$.
Because of $|\beta-e_0|=|\beta|$, this yields control by
$(\sqrt[4]{\tau})^{|\beta|-2-\alpha}$ $\int_t^\tau ds$ $(\sqrt[4]{s})^{\alpha-4-|{\bf n}|}$,
which due to $\alpha<1\leq|{\bf n}|$ behaves as the term on the r.~h.~s.~$(\ref{cw35})_\beta$. 
We rewrite the second r.~h.~s.~contribution to (\ref{cw70}) as
$\int_t^\tau ds$ $\partial_1^{\bf n}\partial_1^2(\partial_2+\partial_1^2)
\tilde\Pi_{x\beta\,s}^-(x)$
$=\int_t^\tau ds\int dy$ $\partial^{\bf n}\partial_1^2(\partial_2+\partial_1^2)\psi_{s}(x-y)$
$\tilde\Pi_{x\beta}^-(y)$, we appeal to $(\mbox{\ref{cw65})}_{\beta}$ (which we recall is a consequence of $\eqref{eq:pi_generic}_{\prec\beta}$ \&
$\eqref{cw03}_{\prec\beta}$ \&
$\eqref{cw60}_{\prec\beta}$) 
and the moment bounds (\ref{cw61})
to obtain control by $(\sqrt[4]{\tau})^{\alpha-2}$ $\int_t^\tau ds$
$(\sqrt[4]{s})^{\alpha-|{\bf n}|-4}$ $(\sqrt[4]{\tau+s})^{|\beta|-2\alpha}$.
Because of $s\le\tau$, the last factor
is $\sim(\sqrt[4]{\tau})^{|\beta|-2\alpha}$,
so that by $\alpha>0$ we again arrive at $(\sqrt[4]{\tau})^{|\beta|-\alpha-2}$
$(\sqrt[4]{t})^{\alpha-|{\bf n}|}$.
For the third r.~h.~s.~contribution to (\ref{cw70}),
$\int_t^\tau ds$ $\partial^{\bf n}\partial_1^2(\partial_2+\partial_1^2)
\xi_{\tau+s}(x)$, we appeal to (\ref{cw06}) to obtain an estimate by
$(\sqrt[4]{\tau})^{-2}$ $\int_t^\tau ds$
$(\sqrt[4]{s})^{\alpha-|{\bf n}|-4}$,
which acts as the previous term because of $|\beta|=\alpha$ for $\beta=0$.

\medskip

We now turn to the last contribution in (\ref{cw70}), namely $\int_\tau^\infty ds \partial^{\bf n}\partial_1^2(\partial_2+\partial_1^2)\Pi_{x\beta\,s}^{-}(x)$, 
and appeal to the standard estimates $(\ref{eq:pi_minus_generic})_\beta$ in their upgraded form
$(\ref{cw64})_\beta$.
Equipped with $(\ref{cw64})_\beta$, which we use for $y=x$ and ${\bf n}$ replaced by ${\bf n} + (2,1)$ and ${\bf n} + (4,0)$, 
we obtain control by $\int_\tau^\infty ds(\sqrt[4]{s})^{|\beta|-6-|{\bf n}|}$.
Since $|\beta|<3$, therefore $-2-|\n|+|\beta|<0$ for $\n\neq {\bf 0}$, 
this integral is $\sim(\sqrt[4]{\tau})^{-2-|{\bf n}|+|\beta|}$,
which due to $\alpha<|{\bf n}|$ for $\n\neq {\bf 0}$ and $t\le\tau$ is dominated 
by the r.~h.~s.~of $(\ref{cw35})_\beta$.

\medskip

We conclude integration proper by dealing with the case of $|\beta|\ge 3$.
In this case, we may directly infer $(\ref{cw35})_\beta$ from $(\ref{mt98})_\beta$
(with ${\bf n}$ replaced by ${\bf n} + (2,0)$, which we recall is a consequence of $\eqref{eq:pi_generic}_\beta$), gaining
control by $(\sqrt[4]{t})^{-|{\bf n}|-2+|\beta|}$, which is dominated by
the r.~h.~s.~of $(\ref{cw35})_\beta$ because $|\beta|\ge 3\ge 2+\alpha$ and $t\le\tau$.

\medskip

\textbf{Step 5.} $\eqref{eq:pi_generic}_\beta$ \& $\eqref{cw60_cont}_\beta \implies \eqref{cw60}_\beta$.
We pass from the homogeneous $C^{2+\alpha}$-estimate
$(\ref{cw60_cont})_\beta$ to the $C^2$-estimate $(\ref{cw60})_\beta$, namely by interpolation
with the $C^{\alpha-2}$-estimate $\eqref{mt98}_\beta$ (which we recall is a consequence of $\eqref{eq:pi_generic}_\beta$).
To this purpose, we write $\partial^\n\Pi_{x\beta}(y)$
$=\partial^\n\Pi_{x\beta\,\tau}(y)+\int dz\psi_\tau(y-z)\big(\partial^\n\Pi_{x\beta}(y)$
$-\partial^\n\Pi_{x\beta}(z)\big)$ for $\n\in\{(2,0),(0,1)\}$. On the first term, 
we use $\eqref{mt98}_\beta$ with $t=\tau$.
On the second term, we use $(\ref{cw60_cont})_\beta$ 
and then the moment bound (\ref{cw61}) for $t=\tau$. 
\end{proof}


\begin{proof}[Proof of Proposition~\ref{rem:1_mal}] 
The arguments follow those of the proof 
of Proposition~\ref{rem:1}, and should be read in parallel.

\medskip

\textbf{Step 1.}
$\eqref{eq:pi_generic}_{\prec\beta}$ \& 
$\eqref{eq:gamma}_{\prec\beta}$ \& 
$\eqref{cw03}_{\prec \beta}$ \&
$\eqref{cw60}_{\prec\beta}$ \& 
$\eqref{cw60_cont}_{\prec \beta}$ \&
$\eqref{eq:recenter_Pi_specific}_{\prec\beta}$ \&
$\eqref{eq:delta_gamma}_{\prec\beta}$ \& 
$\eqref{eq:delta_pi_generic}_{\prec \beta}$ \&
$\eqref{cw60_mal_dual}_{\prec \beta}$ \& 
$\eqref{cw60_cont_mal_dual}_{\prec \beta}$ 
$\implies \eqref{cw60_minus_cont_mal}_\beta$.
In order to establish (\ref{cw60_minus_cont_mal}),
we take the Malliavin derivative of \eqref{cw68}:
\begin{equs}
\delta\Pi_{x\beta}^- = \partial_1^2 \delta\Pi_{x\beta-e_0} 
+ \delta \tilde\Pi_{x\beta}^- + \delta\xi_\tau \delta_0^\beta.  
\end{equs}
The desired estimate on $\delta\xi_\tau$ follows analogously to \eqref{cw06}.
The estimate
on the first r.~h.~s.~follows from the induction hypothesis 
$\eqref{cw60_cont_mal_dual}_{\prec \beta}$.
In order to estimate the middle r.~h.~s.~term  
we take the Malliavin derivative 
of \eqref{cw66} which, similarly to \eqref{eq:recIII30}, gives
\begin{align}
 \delta\Pi_x^{-}=P\big(\sum_{k\ge1}\mathsf{z}_k\Pi_x^k\partial_1^2\delta\Pi_x
+\sum_{k\ge 0}(k+1)\mathsf{z}_{k+1}\Pi_x^k\delta\Pi_x\partial_1^2\Pi_x\big)
-\sum_{k\ge 1}\frac{1}{k!}\Pi_x^k\delta\Pi_x(D^{({\bf 0})})^{k+1}c. \nonumber
\end{align} 
Therefore, the additional ingredients w.~r.~t.~the proof of $\eqref{cw65_cont}_\beta$ are 
$\eqref{eq:delta_pi_generic}_{\prec \beta}$, $\eqref{cw60_mal_dual}_{\prec \beta}$, 
$\eqref{cw60_cont_mal_dual}_{\prec \beta}$,
and $\eqref{eq:delta_pi_generic_cont}_{\prec \beta}$, where
\eqref{eq:delta_pi_generic_cont} is given by  
\begin{equs}
\E^{\frac{1}{q'}}|\delta\Pi_{x\beta}(y) - \delta\Pi_{x\beta}(z)|^{q'}\lesssim |y-z|^\alpha (|y-x|+|z-x|)^{|\beta|-\alpha} \bar w. 
\label{eq:delta_pi_generic_cont}
\end{equs}
In order to prove $\eqref{eq:delta_pi_generic_cont}_\beta$ we take the Malliavin derivative of $\eqref{eq:recenter_Pi_specific}_{\beta}$
and evaluate at $z$, which leads to the identity 
$\delta\Pi_{x\beta}(y) - \delta\Pi_{x\beta}(z)$ 
$= -\sum_\gamma (\delta \Gamma^*_{xy})_{\beta}^\gamma \Pi_{y\gamma}(z) 
-\sum_\gamma (\Gamma^*_{xy})_{\beta}^\gamma \delta\Pi_{y\gamma}(z)$, where due to
\eqref{eq:triangular_Gamma} and \eqref{eq:triangular_delta_Gamma} the sums are 
restricted to $\gamma\preccurlyeq \beta$. Therefore, $\mbox{(\ref{eq:delta_pi_generic_cont})}_\beta$
follows from $\eqref{eq:delta_gamma}_\beta$ \& $\eqref{eq:pi_generic}_{\preccurlyeq\beta}$ and $\eqref{eq:gamma}_\beta$ \& $\eqref{eq:delta_pi_generic}_{\preccurlyeq\beta}$. 

\medskip

\textbf{Step 2.}
$\eqref{eq:pi_generic}_{\prec\beta}$ \&
$\eqref{cw03}_{\prec\beta}$ \&
$\eqref{cw60}_{\prec\beta}$ \&
$\eqref{eq:delta_pi_generic}_{\prec \beta}$ \& 
$\eqref{cw60_mal_dual}_{\prec \beta}$ \&
$\eqref{cw63_generic_mal_dual}_{\prec\beta}$ 
$\implies \eqref{cw60_cont_mal_dual}_\beta$.
We proceed as for (\ref{cw60_cont}) in Step~4 of the
proof of Proposition~\ref{rem:1}: We derive the pointwise $\eqref{cw60_cont_mal_dual}_\beta$
from the analogue of the weak $\eqref{cw63_generic}_\beta$ for $\delta \Pi_{x\beta}$, as given by
\begin{align}
&\mathbb{E}^\frac{1}{q'}|\partial^{\bf n}\partial^{\bf m}\delta \Pi_{x\beta\,t}(y)|^{q'} \nonumber\\
&\lesssim(\sqrt[4]{\tau})^{-2} (\sqrt[4]{\tau+t} + |y-x|)^{|\beta|-\alpha} (\sqrt[4]{t})^{\alpha-|{\bf n}|} \bar w
\quad
\mbox{for}\;{\bf n}\not={\bf 0},\;{\bf m}\in\{(2,0), (0,1)\}, \quad \label{cw63_generic_mal_dual}
\end{align}
which we state here since it is part of the induction hypothesis. 
The analogue of $\eqref{cw65}_\beta$ (which we recall is a consequence of $\eqref{eq:pi_generic}_{\prec\beta}$ \&
$\eqref{cw03}_{\prec\beta}$ \&
$\eqref{cw60}_{\prec\beta}$) for $\delta 
\tilde\Pi_{x\beta}^-$, from which we derive (\ref{cw63_generic_mal_dual}) by integration, 
relies on the additional ingredients
$\eqref{eq:delta_pi_generic}_{\prec \beta}$ and $\eqref{cw60_mal_dual}_{\prec \beta}$.

\medskip

\textbf{Step 3.}
$\eqref{eq:delta_pi_generic}_\beta$ \& $\eqref{cw60_cont_mal_dual}_\beta$ $\implies \eqref{cw60_mal_dual}_\beta$.
We can pass from $\eqref{cw60_cont_mal_dual}_\beta$ to $\eqref{cw60_mal_dual}_\beta$ using
the analogue of $\eqref{mt98}_\beta$ for $\partial^\n \delta\Pi_{x\beta}$, which follows from 
$\eqref{eq:delta_pi_generic}_\beta$ and \eqref{cw61}. 
\end{proof}


\subsection{\texorpdfstring{Analyticity: Proof of Proposition~\ref{rem:2}}{Analyticity: Proof of Proposition 2.7}}

\begin{proof}[Proof of Proposition \ref{rem:2}]
The first statement is easy to check:
The most substantial change is the representation (\ref{eq:representation})
of the solution operator in terms of the convolution semi-group, see
Subsection \ref{sec:semi_group}. This requires
changing the generator of the latter to $(\partial_2-a_0\partial_1^2)^*$
$(\partial_2$ $-a_0\partial_1^2)$ $=-\partial_2^2+|a_0|^2\partial_1^4$. 
Hence the new convolution kernel is a simple spatial rescaling of the standard one:
\begin{align}\label{cw17}
\psi_t(a_0,x)=\psi_t(\frac{x_1}{\sqrt{|a_0|}},x_2),
\end{align}
and thus satisfies the relevant moment bounds (\ref{cw61}) locally uniformly in $a_0$.

\medskip

For convenience of the reader, we list the annealed H\"older-type norms with respect to which we
will establish complex differentiability and thus analyticity:
the norm on random functions given by (\ref{eq:pi_generic}) of Theorem \ref{thm:main},
\begin{align}\label{cw26}
\sup_{y}|y-x|^{-|\beta|}\mathbb{E}^\frac{1}{p}|\Pi_{x\beta}(y)|^p,
\end{align}
the norm on random distributions given by \eqref{eq:pi_minus_generic} of Proposition \ref{rem:5}
\begin{align}\label{cw25}
\sup_{y,t}(\sqrt[4]{t})^{2-\alpha}(\sqrt[4]{t}+|y-x|)^{\alpha-|\beta|}
\mathbb{E}^\frac{1}{p}|\Pi_{x\beta\,t}^{-}(y)|^p,
\end{align}
and finally the stronger norm given by (\ref{cw60}) in the proof of Proposition \ref{rem:1}
\begin{align}\label{cw30}
\sup_y(\sqrt[4]{\tau}+|y-x|)^{\alpha-|\beta|}\mathbb{E}^\frac{1}{p}
|\partial_1^2\Pi_{x\beta}(y)|^p.
\end{align}
Note that in view of the argument for (\ref{eq:pi_minus_generic}), we may
revert from (\ref{cw17}) to the standard convolution kernel for (\ref{cw25}).

\medskip

Over the next several paragraphs, we establish (\ref{cw58ter}), (\ref{cw58bis}), and 
(\ref{cw44}) for $\pi\in\{c,\Pi_x\}$
by induction in $\beta$ with respect to the ordering $\prec$. 
In fact, we rephrase these statements as
\begin{align}
\partial_{a_0}c_\beta&=(\beta(0)+1)c_{\beta+e_0},\label{cw80}\\
\partial_{a_0}\Pi_{x\beta}&=(\beta(0)+1)\Pi_{x\beta+e_0}\quad\mbox{w.~r.~t.~(\ref{cw26})}
,\label{cw81}
\end{align}
with the understanding that these imply that
the functions are complex differentiable -- and thus analytic -- in $a_0$.

\medskip

We start with a reconstruction argument showing that $\mbox{(\ref{cw80})}_{\prec\beta}$ and
$\mbox{(\ref{cw81})}_{\prec\beta}$ imply $\mbox{(\ref{cw80})}_{\beta}$
and $\mbox{(\ref{cw46})}_{\beta}$, where
\begin{align}\label{cw46}
\partial_{a_0}\Pi_{x\beta}^{-}+\partial_1^2\Pi_{x\beta}
=(\beta(0)+1)\Pi_{x\beta+e_0}^{-}\quad\mbox{w.~r.~t.~(\ref{cw25})}.
\end{align}
In preparation for this, we note that in view of the locally uniform estimate (\ref{cw60}), the 
analyticity expressed by $\mbox{(\ref{cw81})}_{\beta}$ transmits from $\Pi_{x\beta}$
to $\partial_1^2\Pi_{x\beta}$ in form of $\mbox{(\ref{cw45})}_{\beta}$, where
\begin{align}\label{cw45}
\partial_{a_0}\partial_1^2\Pi_{x\beta}=(\beta(0)+1)\partial_1^2\Pi_{x\beta+e_0}
\quad\mbox{w.~r.~t.~(\ref{cw30})}.
\end{align}
Like in the proof of Proposition \ref{rem:1} we work with $\tilde\Pi_{x\beta}^{-}$, see \eqref{cw66}, 
which componentwise takes the form 
\begin{align}
\tilde\Pi_{x\beta}^{-}
&=\sum_{k\ge 1}\sum_{e_k+\beta_1+\cdots+\beta_{k+1}=\beta}
\Pi_{x\beta_1}\cdots\Pi_{x\beta_k}\partial_1^2\Pi_{x\beta_{k+1}}\nonumber\\
&-\sum_{k\ge 1}\frac{1}{k!}\sum_{\beta_1+\cdots+\beta_{k+1}=\beta}
\Pi_{x\beta_1}\cdots\Pi_{x\beta_k}\sum_{\gamma}((D^{({\bf 0})})^k)_{\beta_{k+1}}^\gamma c_\gamma.
\end{align}
By \eqref{eq:triangular_product} and \eqref{eq:triangular_c}
the two sums are restricted  to $\beta_1,\dots,\beta_{k+1},\gamma\prec\beta$. 
Therefore we obtain from the analyticity expressed 
in $\mbox{(\ref{cw80})}_{\prec\beta}$, $\mbox{(\ref{cw81})}_{\prec\beta}$, and $\mbox{(\ref{cw45})}_{\prec\beta}$ 
via Leibniz' rule
\begin{align}\label{cw84}
\partial_{a_0}\tilde\Pi_{x\beta}^{-}
&=\sum_{k\ge 1}\sum_{\substack{e_k+\beta_1+\cdots+\beta_{k+1}=\beta \\ l_1+\cdots+l_{k+1}=1}}
\partial_{a_0}^{l_1}\Pi_{x\beta_1}\cdots\partial_{a_0}^{l_k}\Pi_{x\beta_k}
\partial_1^2\partial_{a_0}^{l_{k+1}}\Pi_{x\beta_{k+1}}\nonumber\\
&-\sum_{k\ge 1}\frac{1}{k!}\sum_{\substack{\beta_1+\cdots+\beta_{k+1}=\beta \\ l_1+\cdots+l_{k+1}=1}}
\partial_{a_0}^{l_1}\Pi_{x\beta_1}\cdots\partial_{a_0}^{l_k}\Pi_{x\beta_k}
\sum_{\gamma}((D^{({\bf 0})})^k)_{\beta_{k+1}}^\gamma \partial_{a_0}^{l_{k+1}}c_\gamma.
\end{align}
Inserting the corresponding formulas, we obtain
\begin{align}\label{cw36}
\partial_{a_0}\tilde\Pi_{x\beta}^{-}=(\beta(0)+1)\tilde\Pi_{x\beta+e_0}^-
\quad\mbox{w.~r.~t.~(\ref{cw30})}.
\end{align}

\medskip

We now argue that $\mbox{(\ref{cw36})}_{\beta}$ implies $\mbox{(\ref{cw80})}_{\beta}$, 
where w.~l.~o.~g.~we restrict to $|\beta|<2$.
Rewriting the first r.~h.~s.~term of \eqref{fo01} as 
$\int dy\partial_1^2\psi_t(x-y)\mathbb{E}\Pi_{x\beta-e_0}(y)$, 
we learn from (\ref{eq:pi_generic}) that its limit vanishes.
This leaves us with 
\begin{align}\label{cw55}
c_{\beta}=\lim_{t\uparrow\infty}\mathbb{E}\tilde\Pi_{x\beta\,t}^{-}(x).
\end{align}
Since (\ref{cw36}) implies that $\mathbb{E}\tilde\Pi_{x\beta\,t}^{-}(x)$ is analytic,
this transmits to $c_{\beta}$ via the locally uniform convergence (\ref{cw55}).  

\medskip

We now may turn to (\ref{cw46}), based on the identity (\ref{cw68}) to which we
apply $\partial_{a_0}$, noting that there is no contribution from $\xi$. 
According to $\mbox{(\ref{cw45})}_{\prec\beta}$ (recall that $\beta-e_0\prec\beta$),
to $\mbox{(\ref{cw36})}_{\beta}$, and to $\mbox{(\ref{cw80})}_{\beta}$,
we obtain $\mbox{(\ref{cw46})}_{\beta}$, at first w.~r.~t.~to the stronger
topology (\ref{cw30}) and then to the weaker one (\ref{cw25}).

\medskip

We now turn to integration,
arguing that $\mbox{(\ref{cw46})}_{\beta}$ implies $\mbox{(\ref{cw81})}_{\beta}$.
We consider the (what would be) first-order Taylor terms
\begin{align*}
\Delta\Pi_{x\beta}&:=\Pi_{x\beta}(a_0')-\Pi_{x\beta}(a_0)
-(a_0'-a_0)(\beta(0)+1)\Pi_{x\beta+e_0}(a_0),\\
\Delta\Pi_{x\beta}^{-}&:=\Pi_{x\beta}^{-}(a_0')-\Pi_{x\beta}^{-}(a_0)
-(a_0'-a_0)\partial_{a_0}\Pi_{x\beta}^{-}(a_0).
\end{align*}
We now appeal to (\ref{cw20}), which we use for $\beta$, also with $a_0$ replaced by $a_0'$,
and for $\beta+e_0$. We then obtain from (\ref{cw46})
\begin{align}\label{cw19}
(\partial_2-a_0\partial_1^2)&\Delta\Pi_{x\beta}
=\Delta\Pi_{x\beta}^{-}
+(a_0'-a_0)\partial_1^2(\Pi_{x\beta}(a_0')-\Pi_{x\beta}(a_0)).
\end{align}
We note that $\Delta\Pi_{x\beta}$ 
(qualitatively) inherits the growth and anchoring estimate (\ref{eq:pi_generic})
from $\Pi_{x\beta}$ and $\Pi_{x\beta+e_0}$.
Hence by the Liouville argument from the proof of Proposition \ref{prop:change_of_basepointII}
we obtain from (\ref{cw19}) an integral representation of $\Delta\Pi_{x\beta}$ in terms of
$\Delta\Pi_{x\beta}^{-}$ $+(a_0'-a_0)\partial_1^2(\Pi_{x\beta}(a_0')-\Pi_{x\beta}(a_0))$
in form of (\ref{eq:Pi_minus_to_Pi}).
The argument of Proposition \ref{prop:int_pi_minus} then shows that
\begin{align*}
\lefteqn{\mbox{norm (\ref{cw26}) of}\;\Delta\Pi_{x\beta}(y)
\;\lesssim\;\mbox{norm (\ref{cw25}) of}\;\Delta\Pi_{x\beta\,t}^{-}}\nonumber\\
&+\;|a_0'-a_0|\times\quad\mbox{norm (\ref{cw26}) of}\;(\Pi_{x\beta}(a_0')-\Pi_{x\beta}(a_0)).
\end{align*}
Writing $\Pi_{x\beta}(a_0')-\Pi_{x\beta}(a_0)$ $=\Delta\Pi_{x\beta}$ $+(a_0'-a_0)(\beta(0)+1)
\Pi_{x\beta+e_0}(a_0)$, we absorb the second r.~h.~s.~term for $|a_0'-a_0|\ll 1$ to obtain
\begin{align*}
\lefteqn{\mbox{norm (\ref{cw26}) of}\;\Delta\Pi_{x\beta}(y)}\nonumber\\
&\lesssim\mbox{norm (\ref{cw25}) of}\;\Delta\Pi_{x\beta\,t}^{-}
\;+\;|a_0'-a_0|^2\times\;\mbox{norm (\ref{cw26}) of}\;\Pi_{x\beta+e_0}(a_0).
\end{align*}
This estimate shows that (\ref{cw46}) transmits to (\ref{cw81}).

\medskip

Let us now treat the base case $\eqref{cw80}_{\beta=0}$ and $\eqref{cw81}_{\beta=0}$.
For the former we note that $\Pi^-_{x\,0}=\xi_\tau-c_0$ and $\Pi^-_{x\,e_0}=\partial_1^2\Pi_{x\,0}-c_{e_0}$;
since $\xi$ is centered by Assumption~\ref{ass:spectral_gap}
and $\E|\partial_1^2\Pi_{x\,0\,t}(x)|\lesssim(\sqrt[4]{t})^{\alpha-2}$ by $\eqref{eq:pi_generic}_{\beta=0}$, 
the BPHZ-choice \eqref{eq:BPHZ} yields $c_0=0=c_{e_0}$.
For the latter we note that $\eqref{cw46}_{\beta=0}$ holds by the just said, 
and thus the integration argument established above shows $\eqref{cw81}_{\beta=0}$.

\medskip

We now turn to (\ref{cw59bis}). 
By (\ref{eq:def_pi_n_poly}), there is nothing to show
for purely polynomial $\beta$, so that we turn to not purely polynomial $\beta$,
which we treat by induction in $|\beta|$.  
According to (\ref{eq:pi_purely_pol}) and (\ref{eq:Gamma_z_n}), 
the $\beta$-component of (\ref{eq:recenter_Pi_specific}) can be rewritten as
\begin{align}\label{cw83}
\sum_{{\bf n}}(z-y)^{\bf n}\pi^{({\bf n})}_{xy\beta}
+\sum_{\gamma\not=\pp}(\Gamma_{xy}^*)_{\beta}^\gamma\Pi_{y\gamma}(z)
=\Pi_{x\beta}(z).
\end{align}
According to \eqref{ord12}, for not purely polynomial $\gamma$, $(\Gamma_{xy}^*)_{\beta}^{\gamma}$
depends on $\pi_{xy\beta'}^{({\bf n})}$ via (\ref{exp}) only through 
$|\beta'|<|\beta|$.
Hence by induction hypothesis, $(\Gamma_{xy}^*)_{\beta}^{\gamma}$ is analytic in $a_0$
w.~r.~t.~$\mathbb{E}^\frac{1}{p}|\cdot|^p$. Together with the
analyticity of $\Pi_{y\gamma}$ and $\Pi_{x\beta}$ w.~r.~t.~(\ref{cw26}), see (\ref{cw81}),
we obtain analyticity of $\sum_{{\bf n}}(\cdot-y)^{\bf n}\pi^{({\bf n})}_{xy\beta}$
w.~r.~t.~(\ref{cw26}). Since by (\ref{eq:pi_n_pop}), the sum is effectively
constrained to $|{\bf n}|<|\beta|$, the invertibility of the Vandermonde
matrix yields analyticity of the coefficients $\pi_{xy\beta}^{({\bf n})}$.

\medskip

It remains to establish (\ref{cw44}) for $\pi_{xy\beta}^{({\bf n})}$,
which again we do by induction in $|\beta|$.
With help of (\ref{exp}), we rewrite (\ref{cw83}) as
\begin{align*}
\Pi_{x\beta}(z)&=\sum_{{\bf n}}(z-y)^{\bf n}\pi^{({\bf n})}_{xy\beta}
+\sum_{\gamma\not=\pp}\sum_{k\ge 0}\frac{1}{k!}\sum_{{\bf n}_1,\cdots,
{\bf n}_k}\sum_{\beta_1+\cdots+\beta_{k+1}=\beta}\nonumber\\
&\pi_{xy\beta_1}^{({\bf n}_1)}\cdots\pi_{xy\beta_k}^{({\bf n}_k)}
(D^{({\bf n}_k)}\cdots D^{({\bf n}_1)})_{\beta_{k+1}}^\gamma\Pi_{y\gamma}(z).
\end{align*}
We now apply $\partial_{a_0}$, as in (\ref{cw84}), use once more the
triangular structure \eqref{ord12} and \eqref{eq:triangular_Gamma}, to feed in the induction
hypothesis of (\ref{cw44}). 

\medskip

The same argumentation applies to establish the base case $\eqref{cw59bis}_{\beta=0}$ and $\eqref{cw44}_{\beta=0}$ for $\pi_{xy}^{(\n)}$.
\end{proof}


\section{Malliavin differentiability}\label{sec:mall_approx}

In this section we prove that the quantities of interest, 
namely $\Pi_x$, $\Pi_{x}^-$, $\Gamma_{xy}^*$ and $\pi^{(\mathbf{n})}_{xy}$,
are Malliavin differentiable, so that Leibniz rule can indeed be applied.
For the definitions and basic properties of the Malliavin--Sobolev spaces $\HH^p$ for 
random variables ($\Gamma_{xy}^*$ and $\pi^{(\mathbf{n})}_{xy}$) and $C^0(\HH^p)$ for 
random fields ($\Pi_x$ and $\Pi_{x}^-$) we refer to Appendix~\ref{ss:Malle}.
As before, $p<\infty$ denotes a generic
exponent and its value may change from line to line, which is convenient when
applying Leibniz rule.

\medskip  

We proceed in line with the logical order of the induction, see Subsection~\ref{sec:order}.
\begin{enumerate}[1] 
\item \label{it:0} In Subsection~\ref{sec:base_case_approx} we establish
\footnote{always with the understanding that this holds for all $p<\infty$} $\xi_\tau\in C^0(\HH^p)$.
\item \label{it:1} Then, for any populated multi-index $\beta$ with $|\beta|<2$,
in Subsection~\ref{sec:reconstr_approx} we establish $\Pi^-_{x\beta}\in C^0(\HH^p)$.
\item \label{it:2} In Subsection~\ref{sec:int_approx} we prove $\Pi_{x\beta}$, $\partial_1^2\Pi_{x\beta}$, $\partial_2\Pi_{x\beta}$,
$\in C^0(\HH^p)$ via integration and relate $\delta\Pi_{x\beta}$ 
to $\delta\Pi^-_{x\beta}$ through the representation \eqref{eq:Pi_minus_to_Pi}.
\item \label{it:3} In Subsection~\ref{sec:gamma_three_point_approx} we establish
$\Gamma_{xy\beta}^*,\pi_{xy\beta}^{(\mathbf{n})}\in\HH^p$ by the algebraic and three-point argument.
\end{enumerate}
The identities \eqref{deltaGamma}, \eqref{eq:delta_three_point}, \eqref{eq:recIII30}
and \eqref{eq:delta_recenter_pi_minus} then follow by a now fully justified
application of Leibniz' rule.
Statements~\ref{it:2} and \ref{it:3} are ``officially'' included into the induction since
we need them in form of an induction hypothesis. Statement~\ref{it:1} is established 
inside the induction step.

\subsection{The base case: Malliavin differentiability of \texorpdfstring{$\Pi_{x0}^{-}=\xi_\tau$}{Pi minus x 0 = xi tau}}\label{sec:base_case_approx}

From the argument used to derive \eqref{cw06} we learn that for every $y$, the cylindrical random variable $\xi_\tau(y)$ is in $\mathbb{H}^p$,
with (deterministic) Malliavin derivative $\psi_\tau(y-\cdot)$. In order to argue that the random field $\xi_\tau$ is in 
$C^0(\HH^p)$, we need to show that $\xi_\tau(y)$ can be approximated w.~r.~t.~$\E^{\frac{1}{p}}|\cdot|^p$ by 
functionals of the form (\ref{cw51}) that form a Cauchy sequence w.~r.~t.~the family of semi-norms (\ref{cw50}). 
To this purpose, we note that for fixed $\tau>0$, $\psi_\tau(y-\cdot)$ can be approximated, 
locally uniformly in $y$ and w.~r.~t.~to the Schwartz topology (in the active variable),
by functions $\psi_N(y,\cdot)$ that are finite sums of functions of the form
$\psi(y)\psi'(\cdot)$, where $\psi$ is smooth and $\psi'$ a Schwartz function. 
This structure implies that $\big(\xi,\psi_N(y,\cdot)\big)$ is of the form (\ref{cw51}).
On the one hand, the convergence implies that $\xi_\tau(y)$ is even the pathwise limit of
$\big(\xi,\psi_N(y,\cdot)\big)$, locally uniformly in $y$. 
Since convergence in the Schwartz topology implies convergence w.~r.~t.~$\|\cdot\|_*$,
$\|\psi_\tau(y-\cdot)-\psi_N(y,\cdot)\|_*$ is small locally uniformly in $y$.
Since $\psi_N(y,\cdot)$ is the Malliavin derivative of $\big(\xi,\psi_N(y,\cdot)\big)$,
we obtain from the $\mathbb{L}^p$-version (\ref{eq:spectral_gap_var}) of the SG inequality
that $\big(\xi,\psi_N(y,\cdot)\big)$ forms a Cauchy sequence in $C^0(\HH^p)$,
locally uniformly in $y$. 


\subsection{Reconstruction}\label{sec:reconstr_approx}

In this subsection, we establish $\Pi_{x\beta}^{-}\in C^0(\HH^p)$. 
It follows from \eqref{eq:triangular_product} and \eqref{eq:triangular_c}
that $\Pi_{x\beta}^{-}$ depends on $\Pi_{\beta'}$ only for $\beta'\prec\beta$
and on $c_{\beta'}$ only for $\beta'\preccurlyeq\beta$.
At this stage of the induction step (singular case 4a), 
we already know that the deterministic constant $c_\beta$ is finite.
Hence by the induction hypothesis and (\ref{cw53bis}) we obtain 
$\Pi_{x\beta}^{-}\in C^0(\HH^p)$.
In more quantitative terms, we obtain from $\eqref{eq:pi_generic}_{\prec\beta}$ \& 
$\eqref{cw60}_{\prec\beta}$ on the $\mathbb{L}^p$-norms and 
$\eqref{eq:delta_pi_generic}_{\prec\beta}$ \& $\eqref{cw60_mal_dual}_{\prec\beta}$ on the
$\mathbb{L}^p(H^*)$-norms that
\begin{align}\label{ho09}
\sup_{y} \frac{1}{1+|y|^{|\beta|}} \|\Pi^-_{x\beta}(y)\|_{\HH^p} <\infty.
\end{align}

\subsection{Integration}\label{sec:int_approx} 

We now turn to integration. By (\ref{ho09}), the first part of Lemma~\ref{lem:Malle} implies
that $(\partial_1^2+\partial_2)\Pi_{x\beta\,t}^{-}(y)\in\mathbb{H}^p$ is
continuous in $(t,y)\in(0,\infty)\times\mathbb{R}^2$.
In order to deduce that $\Pi_{x\beta}\in C^0(\mathbb{H}^p)$
from the integral representation (\ref{eq:Pi_minus_to_Pi}),
we shall apply the second part of Lemma~\ref{lem:Malle} to
$F(t,y)=(1-{\rm T}_x^{|\beta|})(\partial_1^2+\partial_2)\Pi_{x\beta\,t}^{-}(y)$.
In order to check assumption (\ref{cw57}), 
we revisit the proof of Proposition \ref{prop:int_pi_minus},
this time starting from (\ref{eq:pi_minus_generic})
and the un-dualized version of (\ref{eq:mal_dual_annealed}), which combine to
\begin{align*}
\|\Pi_{x\beta\,t}^{-}(y)\|_{\mathbb{H}^p}
\lesssim(\sqrt[4]{t})^{\alpha-2}(\sqrt[4]{t}+|y-x|)^{|\beta|-\alpha}.
\end{align*}
As for (\ref{cw64}), this can be upgraded to
\begin{align}\label{ho21}
\|\partial^{\bf n}\Pi_{x\beta\,t}^{-}(y)\|_{\mathbb{H}^p}
\lesssim(\sqrt[4]{t})^{\alpha-2-|{\bf n}|}(\sqrt[4]{t}+|y-x|)^{|\beta|-\alpha}.
\end{align}
This allows us to derive, in analogy to (\ref{cw75}) and (\ref{cw76}),
\begin{align}\label{ho22}
\|{\rm T}_x^{|\beta|}\partial^{\bf n}\Pi_{x\beta\,t}^{-}(y)\|_{\mathbb{H}^p}
\lesssim\sum_{|{\bf m}|<|\beta|}
(\sqrt[4]{t})^{|\beta|-2-|{\bf n}|-|{\bf m}|}|y-x|^{|{\bf m}|}
\end{align}
and
\begin{align}\label{ho20}
\|(1-{\rm T}_x^{|\beta|})\partial^{\bf n}\Pi_{x\beta\,t}^{-}(y)\|_{\mathbb{H}^p}
&\lesssim\sum_{\substack{|{\bf m}|\ge|\beta| \\ m_1+m_2<|\beta|+1}}
(\sqrt[4]{t})^{\alpha-2-|{\bf n}|-|{\bf m}|}(\sqrt[4]{t}+|y-x|)^{|\beta|-\alpha}|y-x|^{|{\bf m}|}.
\end{align}
The far-field part of (\ref{cw57}), i.~e.~the integral over $(N,\infty)$, now follows 
directly from (\ref{ho20}) with ${\bf n}\in\{(2,0),(0,1)\}$ and thus $|{\bf n}|=2$,
because thanks to (\ref{irrational}) we have $|{\bf m}|>|\beta|$, so that all
exponents on $\sqrt[4]{t}$ are strictly less than $-4$. Hence the far-field integral
decays as some negative power of $N$, uniformly for bounded $|y-x|$.
For the near-field part of (\ref{cw57}), i.~e.~the integral over $(0,N^{-1})$, we combine (\ref{ho21}) and (\ref{ho22})
by the triangle inequality and note that all exponents on $\sqrt[4]{t}$ are
strictly larger than $-4$.

\medskip

We now turn to the argument for $\partial^{\bf m}\Pi_{x\beta}\in C^0(\mathbb{H}^p)$
for ${\bf m}\in\{(2,0),(0,1)\}$. We appeal to the integral representation (\ref{cw72}) in form of
\begin{align}\label{ho23}
\partial^{\bf m}\Pi_{x\beta}=\partial^{\bf m}\Pi_{x\beta\,\tau}
+\int_0^\tau dt(\partial_1^4-\partial_2^2)\partial^{\bf m}\Pi_{x\beta\,t}.
\end{align}
By the first part of Lemma~\ref{lem:Malle} we have 
$\partial^{\bf m}\Pi_{x\beta\,\tau}\in C^0(\mathbb{H}^p)$, and that
$(\partial_1^4-\partial_2^2)\partial^{\bf m}\Pi_{x\beta\,t}(y)\in\mathbb{H}^p$ 
is continuous in $(t,y)\in(0,\infty)\times\mathbb{R}^2$. It remains to appeal to the second part
of Lemma~\ref{lem:Malle} with $F(t,y)=(\partial_1^4-\partial_2^2)\partial^{\bf m}\Pi_{x\beta\,t}$ for $t<\tau$
(and vanishing for $t\ge\tau$)
in order to infer that the integral on the r.~h.~s.~of (\ref{ho23}) is also in $C^0(\mathbb{H}^p)$.
For this, we need to verify the near-field part of (\ref{cw57}).
We therefore combine (\ref{cw63_generic}) and (\ref{cw63_generic_mal_dual}) to
\begin{align*}
\|\partial^{\bf n}\partial^{\bf m}\Pi_{x\beta\,t}(y)\|_{\mathbb{H}^p}
\lesssim(\sqrt[4]{\tau})^{-2}(\sqrt[4]{\tau+t}+|y-x|)^{|\beta|-\alpha}
(\sqrt[4]{t})^{\alpha-|{\bf n}|},
\end{align*}
which we use for ${\bf n}\in\{(4,0),(0,2)\}$. Hence the exponent on $\sqrt[4]{t}$
is $\alpha-4>-4$, so that (\ref{cw57}) decays as $(N^{-1})^\alpha$ uniformly for 
bounded $|y-x|$, as desired.


\subsection{Algebraic and three-point arguments} \label{sec:gamma_three_point_approx}

For $(\Gamma^*_{xy}P)_\beta\in\HH^p$ we appeal to the exponential formula \eqref{exp}, 
and the fact that it involves $\pi^{(\mathbf{n})}_{x\beta'}$ only for $\beta'\prec\beta$, 
see \eqref{ord12}.
Hence we may appeal to the induction hypothesis Statement~\ref{it:3}, and use the Leibniz rule \eqref{cw53bis}. 

\medskip

For $\pi^{(\mathbf{n})}_{xy\beta}\in\HH^p$ we appeal to the three-point identity 
\eqref{eq:three_point}. By the just established $(\Gamma^*_{xy}P)_\beta\in\HH^p$ and
the previously established $\Pi_x,\Pi_y\in C^0(\HH^p)$ we obtain with help of 
Leibniz rule \eqref{cw53bis}
that $\sum_{{\bf n}}\pi_{xy}^{({\bf n})}(\cdot-y)^{\bf n}\in C^0(\HH^p)$.
Since the latter is a polynomial of order $<|\beta|$, this implies the desired statement on its 
coefficients.

\medskip

The Malliavin differentiability of the full row $\Gamma_{xy\beta}^*$ now follows from
the identity (\ref{eq:Gamma_z_n}) in form of 
$(\Gamma^*_{xy})_{\beta}^{e_\mathbf{n}}$ $=\pi^{(\mathbf{n})}_{xy\beta}$. 


\section{Triangular structures and dependencies}\label{sec:algebraic_aspects}

We start this section by providing the (elementary) arguments from Subsection
\ref{sec:model}. We first address (\ref{eq:def_Dnull}).
First of all, since $\mathbb{R}[[u]]\ni a\mapsto a(\cdot+v)\in\mathbb{R}[[u]]$
is linear, $v\mapsto\pi[a(\cdot+v)]$ is polynomial for $\pi\in\mathbb{R}[\mathsf{z}_k]$,
so that indeed, $\frac{d}{dv}_{|v=0}\pi[a(\cdot+v)]$
is well-defined as a linear form on the algebra $\mathbb{R}[\mathsf{z}_k]$
that satisfies Leibniz' rule. The same is true for the l.~h.~s.~of (\ref{eq:def_Dnull}),
when applied to $\pi$ and evaluated at $a$. Hence it is enough to check (\ref{eq:def_Dnull})
on the coordinates (\ref{ao07a}). Since $\mathsf{z}_k$ is a linear functional we learn from 
(\ref{ao07a}) that 
$\frac{d}{dv}_{|v=0}\mathsf{z}_k[a(\cdot+v)]$ 
$=\mathsf{z}_k[\frac{da}{du}]$ $=(k+1)\mathsf{z}_{k+1}[a]$,
in agreement with the r.~h.~s.~of (\ref{eq:def_Dnull}).

\medskip

We now turn to the argument for (\ref{mapping}). The r.~h.~s.~inclusion is automatically satisfied 
for the last contribution to (\ref{eq:Pi_minus_def}).
By \eqref{ord19}, the $\beta$-component of the middle term in \eqref{eq:Pi_minus_def} 
is only non-vanishing for $[\beta]\ge 0$.
Only for the first r.~h.~s.~term in (\ref{eq:Pi_minus_def})
we need the additional component in (\ref{mapping}): It is the linear combination of terms
of the form $\mathsf{z}_k\Pi_{x}^k\partial_1^2\Pi_{x}$ where $k\ge 0$.
In view of \eqref{eq:recIII29}, 
its $\beta$-component is non-vanishing if either $[\beta]\ge 0$, 
or if all its constituents are purely polynomial. 
By \eqref{eq:pi_purely_pol}, the latter case corresponds to a space-time polynomial.

\medskip

We now address (\ref{ao21}). By the binomial formula, the
postulate $(\cdot-y)^{\bf n}\Gamma_{yx}$ $=(\cdot-x)^{\bf n}$ 
translates into $(\cdot)^{\bf n}\Gamma_{yx}$
$=\sum_{{\bf m}}\binom{{\bf n}}{{\bf m}}$ $(y-x)^{{\bf n}-{\bf m}}$ $(\cdot)^{\bf m}$,
with the understanding that the binomial coefficient vanishes unless ${\bf m}\le{\bf n}$
componentwise. By definition (\ref{ao07p}), this implies
$\mathsf{z}_{\bf m}.\Gamma_{yx}(\cdot)^{\bf n}$
$=\binom{{\bf n}}{{\bf m}}(y-x)^{{\bf n}-{\bf m}}$.
By duality\footnote{with our abuse of notation} this implies on the component-wise level
$(\Gamma_{xy}^*)^{e_{\bf m}}_{e_{\bf n}}$ $=\binom{{\bf n}}{{\bf m}}(y-x)^{{\bf n}-{\bf m}}$.
This yields (\ref{ao21}).

\medskip

Finally, we turn to the (strict) triangular structure of $\Gamma^*-{\rm id}$ 
and ${\rm d}\Gamma^*$ w.~r.~t.~$|\cdot|$ and $|\cdot|_{\prec}$,
on which the entire induction argument relies on.
Equally important are the strict triangular dependencies of $\Gamma^*$ and ${\rm d}\Gamma^*$
on $\pi^{({\bf n})}$ and ${\rm d}\pi^{({\bf n})}$, respectively.
The same applies to the dependencies of the expressions $\mathsf{z}_k\pi^k\pi'$
and $\pi^k(D^{({\bf 0})})^kc$ on $\pi,\pi',c$.

\begin{lemma}[Triangular dependencies]\label{lem:dependencies}
For $\gamma$ not purely polynomial,
\begin{align}
(\Gamma^*)_\beta^\gamma 
\quad\mbox{does not depend on}\quad
\pi_{\beta'}^{(\n)}
\quad\mbox{unless}\quad
\beta'\prec\beta, \label{ord12}
\end{align}
and, for arbitrary $\gamma$,
\begin{align}
(\Gamma^*)_\beta^\gamma 
\quad\mbox{does not depend on}\quad
\pi_{\beta'}^{(\n)}
\quad\mbox{unless}\quad
\beta'\preccurlyeq\beta. \label{ord12b}
\end{align}
For $k\geq1$, $\pi\in\mathsf{T}^*$ and $c\in\tilde{\mathsf{T}}^*$,
\begin{align}
\pi^k(D^{(\0)})^kc\in\tilde{\mathsf{T}}^*, \label{ord19} \hspace{3.5cm}\\
(\pi^{k}(D^{(\0)})^{k}c)_\beta
\quad\mbox{does not depend on}\quad
\pi_{\beta'},\; c_{\beta'}
\quad\mbox{unless}\quad
\beta'\prec\beta. \label{eq:triangular_c}
\end{align}
For $k\geq0$ and $\pi^{(1)},\dots,\pi^{(k+1)}\in\mathsf{T}^*$,
\begin{align}\label{ord24}
(\z_k\pi^{(1)}\cdots\pi^{(k+1)})_\beta
= \sum_{e_k+\beta_1+\cdots+\beta_{k+1}=\beta} \pi^{(1)}_{\beta_1}\cdots\pi^{(k+1)}_{\beta_{k+1}} 
\end{align}
involves only multi-indices $\beta_1,\dots,\beta_{k+1}$ satisfying
\begin{align}
\beta_1,\dots,\beta_{k+1} &\prec\beta, \label{eq:triangular_product} \\ 
|\beta_1|+\cdots+|\beta_{k+1}|&=|\beta|, \label{ord18}
\end{align}
and for $\pi^{(k+1)}\in\tilde{\mathsf{T}}^*$,
\begin{align}\label{eq:recIII29}
\mathsf{z}_k \pi^{(1)} \cdots \pi^{(k+1)} \in \tilde{\mathsf{T}}^*.
\end{align}
\end{lemma}
\begin{lemma}[Triangular structures]\label{lem:triangular}
For any $\Gamma^*\in\mathsf{G}^*$,
\begin{align}
(\Gamma^*-{\rm id})_\beta^\gamma \neq 0 
&\quad\implies\quad {\gamma}\prec{\beta} \quad\mbox{and}\quad |\gamma|<|\beta|, \label{eq:triangular_Gamma} \\
(\delta\Gamma^*)_\beta^\gamma \neq 0
&\quad\implies\quad {\gamma}\prec{\beta} \quad\mbox{and}\quad |\gamma|<|\beta|, \label{eq:triangular_delta_Gamma} \\
({\rm d}\Gamma^*)_\beta^\gamma \neq 0 
&\quad\implies\quad {\gamma}\prec{\beta} \quad\mbox{and}\quad |\gamma| \leq |\beta|+1-\alpha. \label{eq:triangular_d_Gamma}
\end{align}
In particular, 
\begin{align}
(\Gamma^*)_0^\gamma\neq0\quad&\implies\quad\gamma=0, \label{ord23} \\
({\rm d}\Gamma_{xy}^*-{\rm d}\Gamma_{xz}^*\Gamma_{zy}^*)_\beta^\gamma
\neq0\quad&\implies\quad \gamma\prec\beta.\label{eq:dGamma_inc_triangular}
\end{align}
\end{lemma}
We state (and instantly prove) one further property:  
\begin{align}\label{eq:dGamma_e0}
({\rm d}\Gamma_{xz}^*)_{\beta}^\gamma=0\quad\mbox{for}\;\beta\in \mathbb{N}_0e_0
\;\mbox{and}\;\gamma\;\mbox{not purely polynomial}.
\end{align}
By definition (\ref{eq:def_dGamma}), for (\ref{eq:dGamma_e0}) it suffices to 
show that $(\Gamma^*D^{({\bf n})})_{\beta}^\gamma$ vanishes for such $\beta,\gamma$.
Since by (\ref{ao12}) and $\alpha\in(0,\frac{1}{2}]$, 
\begin{align}\label{eq:beta_e0}
\beta\in\mathbb{N}_0e_0\quad\Longleftrightarrow\quad
|\beta|=\min\mathsf{A}\quad\Longleftrightarrow\quad|\beta|<2\alpha,
\end{align}
the latter follows
from the triangularity of $\Gamma^*$ w.~r.~t.~$|\cdot|$ and the fact that
$(D^{({\bf n})})_{\beta}^{\gamma}$ vanishes for such $\beta,\gamma$,
see (\ref{eq:def_Dnull}) and (\ref{eq:def_Dn}).

\medskip

\begin{proof}[Proof of Lemma \ref{lem:dependencies}]
Proof of \eqref{ord12} \& \eqref{ord12b}.
Indeed, \eqref{ord12b} is an immediate consequence of \eqref{ord12} by \eqref{eq:Gamma_z_n}.
By \eqref{exp} in its component-wise version, we see that for \eqref{ord12} we have to show
\begin{equation}\label{ord13}
\pi_{\beta_1}^{({\bf n}_1)}\cdots \pi_{\beta_k}^{({\bf n}_k)} \big( D^{({\bf n}_1)}\cdots D^{({\bf n}_k)} \big)_{\beta_{k+1}}^\gamma \neq 0
\quad \implies \quad
\beta_1,\dots,\beta_k \prec\beta,
\end{equation}
where $k\ge 1$ and $\beta_1,\dots,\beta_{k+1}$ satisfy
\begin{equs}\label{mt33}
\sum_{k'=1}^{k+1}\beta_{k'}=\beta
\quad\mbox{and}\quad
|\n_{k'}|<|\beta_{k'}| \quad\mbox{for}\quad k'=1,\dots,k,
\end{equs}
where the second part comes from (\ref{eq:pi_n_pop}). From the definitions \eqref{eq:def_Dnull} 
and \eqref{eq:def_Dn} we infer
\begin{equs}\label{ord08}
(D^{(\0)})_\beta^\gamma&\not=0\implies
\Big( [\gamma]=[\beta]-1\mbox{ and }|\gamma|_p=|\beta|_p \mbox{ and } \gamma(0)\leq\beta(0)+1\Big),
\end{equs}
and for $\n\neq\0$
\begin{equs}\label{ord09}
(D^{({\bf n})})_\beta^\gamma&\not=0\implies
\Big([\gamma]=[\beta]-1\mbox{ and }|\gamma|_p
=|\beta|_p+|{\bf n}|\mbox{ and }\gamma(0)=\beta(0)\Big).
\end{equs}
This yields by iteration
\begin{align}
\lefteqn{\big(D^{({\bf n}_1)} \cdots D^{({\bf n}_k)} \big)_{\beta_{k+1}}^\gamma \neq 0
\quad\Longrightarrow\quad}\nonumber\\
&[\gamma]=[\beta_{k+1}]-k 
\quad\mbox{and}\quad
|\gamma|_p=|\beta_{k+1}|_p+\sum_{k'=1}^{k}|{\bf n}_{k'}|\quad\mbox{and}\quad
\gamma(0)\leq\beta_{k+1}(0)+\sum_{k'=1}^k \delta^{\n_{k'}}_{\0}, \label{mt35}
\end{align}
which by definition (\ref{ord01}) of $|\cdot|_{\prec}$ implies
\begin{align}
\big( D^{({\bf n}_1)} \cdots D^{({\bf n}_k)} \big)_{\beta_{k+1}}^\gamma \neq 0
\quad\implies \quad
|\gamma|_{\prec}\le|\beta_{k+1}|_{\prec}+\sum_{k'=1}^k 
(\tfrac{1}{2} |\n_{k'}| + \tfrac{1}{4} \delta_\0^{\n_{k'}} -1).
\end{align}
Hence by the first item in (\ref{mt33}) and the
additivity of $|\cdot|_{\prec}$, the l.~h.~s.~of \eqref{ord13} yields
\begin{equation}\label{ord16}
\length{\beta}
=\sum_{k'=1}^{k+1} \length{\beta_{k'}}
\geq \length{\gamma}+\sum_{k'=1}^k\big(\length{\beta_{k'}}-
(\tfrac{1}{2} |\n_{k'}| + \tfrac{1}{4} \delta_\0^{\n_{k'}} -1)\big).
\end{equation}
We now argue that each summand in the last term of (\ref{ord16}) is positive
\begin{equation}\label{ord15}
\tfrac{1}{2} |\n_{k'}| +\tfrac{1}{4} \delta^{\n_{k'}}_\0 -1< \length{\beta_{k'}}
\quad\mbox{for}\quad k'=1,\dots,k.
\end{equation}
Indeed, for $\n_{k'}=\0$ we have as desired
\begin{align*}
\tfrac{1}{2} |\n_{k'}| +\tfrac{1}{4} \delta^{\n_{k'}}_\0  -1
= \tfrac{1}{4}  -1 <-\tfrac{1}{2}
\stackrel{\eqref{eq:infprec}}{\leq} \length{\beta_{k'}}.
\end{align*}
For $\n_{k'}\neq\0$, we have
\begin{align*}
\tfrac{1}{2} |\n_{k'}| +\tfrac{1}{4} \delta^{\n_{k'}}_\0-1
= \tfrac{1}{2} |\n_{k'}| -1
\stackrel{\eqref{mt33}}{<}\tfrac{1}{2} |\beta_{k'}| -1,
\end{align*}
and conclude by
\begin{align*}
\tfrac{1}{2} |\beta_{k'}|-1\stackrel{\eqref{ao12}}{=}
\tfrac{1}{2}\big(\alpha([\beta_{k'}]+1)+|\beta_{k'}|_p\big)-1
\stackrel{\eqref{eq:infprec}}{\le}[\beta_{k'}]+\tfrac{1}{2}|\beta_{k'}|_p+\tfrac{1}{4}\beta_{k'}(0)
\stackrel{\eqref{ord01}}{=}|\beta_{k'}|_{\prec}.
\end{align*}

\medskip

Using \eqref{ord15}, we obtain from \eqref{ord16}
\begin{equation}\label{ord14}
\length{\beta} \geq \length{\beta_{k'}} + \length{\gamma} - 
(\tfrac{1}{2} |\n_{k'}| + \tfrac{1}{4} \delta_\0^{\n_{k'}} -1)
\quad\mbox{for any}\quad k'=1,\dots,k.
\end{equation}
In case $\n_{k'}=\0$, we obtain from \eqref{ord14} as desired 
\begin{equation*}
\length{\beta} 
\geq \length{\beta_{k'}} + \length{\gamma} - \tfrac{1}{4} +1
\stackrel{\eqref{eq:infprec}}{>}\length{\beta_{k'}}.
\end{equation*}
In case $\n_{k'}\neq\0$ we obtain from \eqref{ord14}
\begin{equation*}
\length{\beta} \geq \length{\beta_{k'}} + \length{\gamma} - \tfrac{1}{2} |\n_{k'}| +1
\stackrel{\eqref{ord01}}{=} 
\length{\beta_{k'}} + [\gamma] + \tfrac{1}{2} (|\gamma|_p-|\n_{k'}|) +\tfrac{1}{4}\gamma(0) +1.
\end{equation*}
Since $|\gamma|_p- |\n_{k'}|\geq0$ by \eqref{mt35}, $[\gamma]\geq0$ by the assumption 
that $\gamma$ is not purely polynomial, and $\gamma(0)\geq0$, this yields again 
$\length{\beta}>\length{\beta_{k'}}$, which finishes the argument for \eqref{ord12}.

\medskip

Proof of \eqref{ord19}. 
The $\beta$-component of \eqref{ord19} equals 
\begin{equation}\label{ord20}
\sum_{\substack{\beta_1+\dots+\beta_{k+1}=\beta \\ \gamma}} 
\pi_{\beta_1}\cdots\pi_{\beta_{k}} ((D^{(\0)})^{k})_{\beta_{k+1}}^{\gamma} \;c_{\gamma},
\end{equation}
where by additivity of $[\cdot]$ and the first item of \eqref{mt35} the multi-indices are restricted to
$$
[\beta]=[\beta_1]+\cdots+[\beta_{k+1}] = [\beta_1]+\cdots+[\beta_{k}]+[\gamma]+k.
$$
Since $[\cdot]\geq -1$ we obtain $[\beta]\geq[\gamma]$, and by $c\in\tilde{\mathsf{T}}^*$ we have $[\gamma]\geq0$ which yields as desired $[\beta]\geq0$.

\medskip

Proof of \eqref{eq:triangular_c}.
We use again that the $\beta$-component of \eqref{ord19} equals \eqref{ord20}.
From \eqref{mt35} we obtain $\length{\beta_{k+1}}\geq\length{\gamma}+\tfrac{3}{4}k$, and hence by additivity of $\length{\cdot}$
$$
\length{\beta}=\length{\beta_1}+\cdots+\length{\beta_{k+1}}
\geq \length{\beta_1}+\cdots+\length{\beta_{k}}+\length{\gamma}+\tfrac34 k.
$$
Since $\length{\cdot}\geq-1/2$, we obtain for any $\beta'\in\{\beta_1,\dots,\beta_{k},\gamma\}$
$$
\length{\beta}\geq \length{\beta'}+\tfrac14 k >\length{\beta'},
$$
which finishes the proof of \eqref{eq:triangular_c}.

\medskip

Proof of \eqref{eq:triangular_product} and \eqref{ord18}.
We have to show that
\begin{align*}
e_k+\beta_1+\dots+\beta_{k+1}=\beta \quad\implies\quad 
\left\{
\begin{array}{c}
\beta_1 ,\dots,\beta_{k+1} \prec \beta,\\
|\beta_1|+\dots+|\beta_{k+1}|=|\beta|. 
\end{array}\right.
\end{align*}
For the upper item we distinguish $k=0$ from $k\neq0$.
In the first case we have $\length{\beta} = \length{\beta_1} + \tfrac{1}{4} > \length{\beta_1}$. 
In the latter one, we have by additivity 
$\length{\beta}=k+\length{\beta_1}+\dots+\length{\beta_{k+1}}$, 
and since $\length{\cdot}\geq -1/2$, 
we obtain $\length{\beta}\geq k/2+\length{\beta_i}>\length{\beta_i}$ for all $i=1,\dots,k+1$.
The lower item is an immediate consequence of the definition of $|\cdot|$, cf. \eqref{ao12}.

\medskip

Proof of \eqref{eq:recIII29}.
The $\beta$-component of \eqref{eq:recIII29} equals \eqref{ord24}, hence $[\beta]=[e_k]+[\beta_1]+\cdots+[\beta_{k+1}]$. Since $[e_k]=k$, $[\cdot]\geq-1$ and $[\beta_{k+1}]\geq0$ by $\pi_{k+1}\in\tilde{\mathsf{T}}^*$, see \eqref{ao09}, we obtain as desired $[\beta]\geq0$.
\end{proof}
\begin{proof}[Proof of Lemma \ref{lem:triangular}]
Proof of \eqref{eq:triangular_Gamma}.
Recall that $(\Gamma^*-{\rm id})_\beta^\gamma$ is a linear combination of terms of the form of the l.~h.~s.~of \eqref{ord13}, involving multi-indices $\beta_1,\dots,\beta_{k+1}$ for $k\geq1$ subject to \eqref{mt33}.
Putting together \eqref{ord16} and \eqref{ord15}, we obtain $\gamma\prec\beta$.
From \eqref{mt35} and
\begin{equs}
\sum_{k'=1}^{k+1}[\beta_{k'}]=[\beta],\quad
\sum_{k'=1}^{k+1}|\beta_{k'}|_p=|\beta|_p,
\end{equs}
which is an immediate consequence of \eqref{mt33}, we see that the l.~h.~s.~of \eqref{ord13} implies
\begin{align*}
[\gamma]=[\beta]-\sum_{k'=1}^k [\beta_{k'}] -k
\quad\mbox{and}\quad
|\gamma|_p<|\beta|_p+\sum_{k'=1}^{k}(|\beta_{k'}|-|\beta_{k'}|_p).
\end{align*}
This yields
$$
|\gamma|<|\beta|+\sum_{k'=1}^k \big(-\alpha(1+[\beta_{k'}]) + |\beta_{k'}|-|\beta_{k'}|_p\big) = |\beta|,
$$
establishing the last item in \eqref{eq:triangular_Gamma}.

\medskip

Proof of \eqref{eq:triangular_delta_Gamma}.
This is an immediate consequence of \eqref{eq:triangular_Gamma}.

\medskip

Proof of \eqref{eq:triangular_d_Gamma}.				
We first note that
\begin{equation*}
D_{\gamma'}^\gamma\neq 0 \quad\implies\quad 
{\gamma}\prec{\gamma'} 
\quad\mbox{and}\quad
|\gamma|\leq|\gamma'|+1-\alpha
\quad\mbox{for}\quad
D\in\{D^{(\0)},D^{(1,0)}\},
\end{equation*}
which is an immediate consequence of \eqref{ord08} and \eqref{ord09}.
By \eqref{eq:triangular_Gamma} we therefore obtain
\begin{equation*}
(\Gamma^* D)_{\beta''}^\gamma\neq0\quad\implies\quad 
{\gamma}\prec{\beta''} 
\quad\mbox{and}\quad
|\gamma|\leq|\beta''|+1-\alpha
\quad\mbox{for}\quad
D\in\{D^{(\0)},D^{(1,0)}\}.
\end{equation*}
To establish the last item in \eqref{eq:triangular_d_Gamma}, we also note that by $\beta'+\beta''=\beta$ and $|\cdot|\geq\alpha$, we have $|\beta''|=|\beta|-|\beta'|+\alpha\leq|\beta|$, which yields as desired $|\gamma|\leq|\beta|+1-\alpha$.
The first item in \eqref{eq:triangular_d_Gamma} follows from $\beta''\preccurlyeq \beta$, which we shall establish now.
From the definition \eqref{eq:def_dGamma} of ${\rm d}\Gamma^*_{xz}$, we see that $({\rm d}\Gamma^*_{xz})_\beta^\gamma$ is a linear combination of terms of the form ${\rm d}\pi^{(\n)}_{xz\beta'}(\Gamma^*_{xz} D^{(\n)})_{\beta''}^\gamma$ with $\beta'+\beta''=\beta$ and $\n=\0,(1,0)$.
By \eqref{eq:algIII7} and the first item in \eqref{eq:algIII6} we see that purely polynomial $\beta'$ do not contribute, hence $\beta''\preccurlyeq\beta$ by \eqref{eq:recIII25}.
\end{proof}

\appendix

\section{Malliavin--Sobolev spaces for random variables and random fields} \label{ss:Malle}

In this section we recall the definitions of the classical Malliavin--Sobolev spaces and extend it to the case of random fields. 
We also show that these definitions are stable under the operations used to construct the model, which allows us to prove 
Malliavin differentiability in Section~\ref{sec:mall_approx}.

\medskip

We start with some notation. We denote by $H^*$ the Hilbert space with norm $\|\cdot\|_*$ given by \eqref{as02}. 
The space of $p$-integrable random variables is denoted by $\mathbb{L}^p$ and the space of $p$-integrable random variables with values in $H^*$ by $\mathbb{L}^p(H^*)$. The natural norms on these spaces are given by $\E^{\frac{1}{p}}|\cdot|^p$ and 
$\E^{\frac{1}{p}}\|\cdot\|_*^p$ respectively. 

\medskip

The classical Malliavin--Sobolev space\footnote{see \cite[Section~1.2]{Nu06} for the Gaussian case} $\HH^p$ is given by the completion of cylindrical functionals \eqref{eq:cyl_fnct} w.~r.~t.~the stochastic norm 
\begin{equs}
\|F\|_{\HH^p} := \E^\frac{1}{p}\left(|F|^p+\|\frac{\partial F}{\partial\xi}\|_{*}^p\right). 
\end{equs}
Since by Assumption~\ref{ass:spectral_gap} the Malliavin derivate $\frac{\partial}{\partial \xi}$ is 
closable from $\mathbb{L}^2$ to $\mathbb{L}^2(H^*)$, it is also closable from $\mathbb{L}^p$ to $\mathbb{L}^2(H^*)$ for
$p\geq 2$, and naturally extends to a bounded linear operator from $\mathbb{L}^p$ to $\mathbb{L}^p(H^*)$.
Therefore, if a sequence of cylindrical functionals $\{F_N\}_N$ of the form \eqref{eq:cyl_fnct} converges to 
$F$ in $\mathbb{L}^p$ and it is Cauchy w.~r.~t.~$\|\cdot\|_{\HH^p}$, then $F\in \HH^p$ and 
$\frac{\partial F}{\partial \xi} = \lim_{N\uparrow \infty} \frac{\partial F_N}{\partial \xi}$ in $\mathbb{L}^p(H^*)$.

\medskip

Since we deal with random functionals of the noise which are continuous in an 
annealed sense, cf. \eqref{cw60_cont}, it is convenient to work with the space $C^0(\HH^p)$. We stress that by 
$C^0(\mathbb{H}^p)$ we do not mean the Banach space endowed with the norm 
$\sup_{y\in\mathbb{R}^2}\|F(y)\|_{\mathbb{H}^p}$, but the linear space of continuous (and possibly unbounded)
functions from $\mathbb{R}^2$ with values in $\mathbb{H}^p$, endowed with the topology given by the family of semi-norms 
\begin{align}\label{cw50}
\sup_{y\in K} \|F(y)\|_{\HH^p}\quad\mbox{for all}\;K\subset \R^2 \ \text{compact}.
\end{align}
Clearly $C^0(\HH^p)$ has the property 
\begin{align}\label{cw53}
C^0(\HH^p)\ni F\mapsto F(y)\in\mathbb{H}^p\quad\mbox{is bounded for all}\;y\in\mathbb{R}^2,
\end{align}
and for a compactly supported continuous function $\psi$,\footnote{with the understanding that
this linear map is well-defined}
\begin{align}\label{cw52}
C^0(\HH^p)\ni F\mapsto \psi*F\in C^0(\HH^p)\quad\mbox{is continuous},
\end{align}
see Lemma \ref{lem:Malle} for a refinement. The latter statement follows from the fact that functionals 
$F$ of the form
\begin{equs}\label{cw51}
F[\xi](y) = \bar F \big(y;(\xi,\zeta_1),\ldots,(\xi,\zeta_N)\big),
\end{equs}
where $N\in \N$, $\bar F$ is a smooth function on $\R^2\times \R^N$ and $\zeta_1,\ldots,\zeta_N$ are Schwartz functions, 
are dense in $C^0(\HH^p)$ w.~r.~t.~the semi-norms \eqref{cw50} and (\ref{cw52}) preserves their form and is bounded  under 
(\ref{cw50}). Here comes the argument for density: By Cantor's 
diagonal sequence argument, it is enough to establish the approximation for fixed $K$ in (\ref{cw50}). Because $K$ is 
compact and $F\in C^0(K;\mathbb{H}^p)$, given $\delta>0$, there exist finitely many open sets 
$U_1,\dots,U_M\subset\mathbb{R}^2$ covering $K$ such that the oscillation of $F$ on each $U_m$ 
w.~r.~t.~$\|\cdot\|_{\mathbb{H}^p}$ is $\le\delta$. For $m=1,\dots,M$, pick an $x_m\in U_m$. Since 
$F(x_m)\in \mathbb{H}^p$ there exists a cylindrical functional $F_m$ of the form 
(\ref{cw56}), such that $\|F(x_m)-F_m\|_{\mathbb{H}^p}\le\delta$. Pick a smooth partition of unity 
$\eta_1,\dots,\eta_M\ge 0$ subordinate to the covering $\{U_i\}_{i=1}^M$. Then $\tilde F(x)=\sum_{m=1}^M\eta_m(x)F_m$ is of the
form (\ref{cw51}) and we have $\sup_{x\in K}\|F(x)-\tilde F(x)\|_{\mathbb{H}^p}$ $\le2\delta$.

\medskip

Appealing to closability we have

\begin{remark}\label{rem:3}
If a sequence $\{F_N\}_N$ in $C^0(\HH^p)$ converges to a random field $F$, 
pointwise in $y$ w.~r.~t.~$\E^{\frac{1}{p}}|\cdot|^p$, and is a Cauchy sequence w.~r.~t.~the family of semi-norms (\ref{cw50}),
then $F\in C^0(\HH^p)$ with $\frac{\partial F(y)}{\partial \xi} = \lim_{N\uparrow \infty} \frac{\partial F_N(y)}{\partial \xi}$. Furthermore, 
convergence of $\{F_N\}_N$ to $F$ takes place w.~r.~t.~these semi-norms.
\end{remark}

\medskip 

For reconstruction, and also the algebraic and the three-point argument, we need
Leibniz rule, meaning that the two bilinear maps
\begin{equation}
\left\{
\begin{array}{c}
\mathbb{H}^{p_1}\times\mathbb{H}^{p_2}\ni(F_1,F_2)\mapsto F_1F_2\in\mathbb{H}^p 
\\
C^0(\HH^{p_1})\times C^0(\HH^{p_2})\ni(F_1,F_2)\mapsto F_1F_2\in C^0(\HH^p) 
\end{array}
\right\} 
\text{ are continuous provided }\;\frac{1}{p}=\frac{1}{p_1}+\frac{1}{p_2},
\label{cw53bis}
\end{equation}
where the relation between the exponents is dictated by H\"older's inequality (in probability).
This follows from the fact that the maps (\ref{cw53bis}) preserve (\ref{cw56})
and (\ref{cw51}), respectively, and that the latter is bounded as a bilinear map under (\ref{cw50}).

\medskip

The refinement of (\ref{cw52}) we need in the integration step is given by the following lemma,

\begin{lemma}\label{lem:Malle}
Suppose for some $\theta>0$ that $F\in C^0(\HH^p)$ satisfies 
\begin{align}\label{cw58}
\sup_{y}\frac{1}{1+|y|^{\theta}}\|F(y)\|_{\mathbb{H}^p}<\infty.
\end{align}
Then we have $\psi*F\in C^0(\HH^p)$ for any Schwartz kernel $\psi$,
with the natural formula for the Malliavin derivative.

\medskip

Moreover, suppose that $(0,\infty)\times\mathbb{R}^2\ni(t,y)\mapsto F(t,y)\in\mathbb{H}^p$
is continuous and satisfies
\begin{align}\label{cw57}
\lim_{N\uparrow\infty}\int_{(0,N^{-1})\cup(N,\infty)} dt\,\|F(t,y)\|_{\mathbb{H}^p}=0\quad\mbox{locally uniformly in}\;y.
\end{align}
Then $\int_0^\infty dt \, F(t,\cdot)\in C^0(\mathbb{H}^p)$ with the natural formula for the Malliavin derivative.
\end{lemma}

\begin{proof}
We start by addressing the first statement in the lemma.
To this purpose, we select a sequence $\psi_N$ of compactly supported kernels that approximate $\psi$ in Schwartz space. 
According to (\ref{cw52}), we have $\psi_N*F\in C^0(\HH^p)$.
By Remark \ref{rem:3} it suffices to argue that
$\psi_N*F$ is a Cauchy sequence in $C^0(\HH^p)$. Indeed, using the
triangle inequality in $C^0(\HH^p)$ in form of
\begin{align*}
\|(\psi_{N}*F-\psi_{N'}*F)(y)\|_{\mathbb{H}^p} \le\sup_{z\in\mathbb{R}^2}\frac{1}{1+|z|^{\theta}}\|F(z)\|_{\mathbb{H}^p}
\int dz \, |(\psi_N-\psi_{N'})(y-z)|(1+|z|^{\theta}),
\end{align*}
we see that (\ref{cw58}) is sufficient to convert the convergence of
the kernels in Schwartz space to the Cauchy criterion w.~r.~t.~the family of semi-norms (\ref{cw50}).

\medskip

We now turn to the second statement in this lemma.
By Remark \ref{rem:3}, it is enough to establish that
$\int_{N^{-1}}^Ndt\, F(t,\cdot)$, which automatically is an element of
$C^0(\mathbb{H}^p)$, is in fact a Cauchy sequence in $C^0(\mathbb{H}^p)$.
This is precisely what assumption (\ref{cw57}) ensures.
\end{proof}
\section{\texorpdfstring{Summary of the logical order of the proof}{Summary of the logical order of the proof}}\label{app:tables}
To assist the reader, here we include a detailed description of the inductive proof according to the steps outlined in Subsection \ref{sec:order}, both for regular (Figure~\ref{fig:regular}) and singular (Figure~\ref{fig:singular}) multi-indices. We include all the statements, their inputs and outputs.	
	\begin{figure}[H]
		\small
		\begin{center}
			\begin{tabular}{|c|c|@{\hspace{0.2ex}}c@{\hspace{0.2ex}}|@{\hspace{0.2ex}}c@{\hspace{0.2ex}}|c|}
					\hline 
					{\bf Step} & {\bf Statement} & {\bf Input} & {\bf Output}\\  
					\hhline{|====|}
					1 & Proposition \ref{prop:gamma_npp} & $\eqref{eq:pi_n}_{\prec \beta}$  & $\eqref{eq:gamma}_\beta^{\gamma \neq \pp}$ \\ 
					\hline
					2.~a) & Proposition \ref{rem:1} & $\eqref{eq:pi_generic}_{\prec \beta}$, $\eqref{eq:gamma}_{\prec \beta}$, $\eqref{cw03}_{\prec \beta}$, & $\eqref{cw60_minus_cont}_\beta$\\
					& & $\eqref{cw60}_{\prec \beta}$, $\eqref{cw60_cont}_{\prec \beta}$, $\eqref{eq:recenter_Pi_specific}_{\prec\beta}$& \\
					\hline
					2.~b) & Subsection \ref{sec:Pi_minus_constr} & $\eqref{eq:Pi_shift}_{\prec \beta}$, $\eqref{eq:Pi_parity}_{\prec \beta}$ & $\eqref{eq:Pi_minus_shift}_\beta$, $\eqref{eq:Pi_minus_parity}_\beta$\\
					\hline
					2.~c) & Proposition \ref{prop:change_of_basepoint} & $\eqref{eq:recenter_Pi_specific}_{\prec \beta}$ & $\eqref{eq:recenter_Pi_minus}_\beta$, $\eqref{eq:recenter_Pi_minus_specific}_\beta$\\
					\hline
					3 & Proposition \ref{prop:reconstr_reg} & $\eqref{eq:pi_generic}_{\prec\beta}$, $\eqref{eq:gamma}_\beta^{\gamma \neq \pp}$, $\eqref{cw60_cont}_{\prec\beta}$, $\eqref{cw60_minus_cont}_\beta$ & $\eqref{eq:pi_minus_generic}_\beta$\\
					& & $\eqref{eq:pi_n}_{\prec \beta}$, $\eqref{eq:recenter_Pi_minus_specific}_\beta$, $\eqref{eq:pi_minus_generic}_{\prec \beta}$ & \\
					\hline
					4.~a) & Proposition \ref{prop:int_pi_minus} & $\eqref{eq:pi_minus_generic}_\beta$ & $\eqref{eq:pi_generic}_\beta$, $\eqref{eq:Pi_minus_to_Pi}_\beta$ \\
					\hline
					4.~b) & Subsection \ref{sec:Pi_construct} & $\eqref{eq:Pi_minus_to_Pi}_\beta$, $\eqref{eq:Pi_minus_shift}_\beta$, $\eqref{eq:Pi_minus_parity}_\beta$ & $\eqref{eq:Pi_shift}_\beta$, $\eqref{eq:Pi_parity}_\beta$\\
					\hline 
					5.~a) & Proposition \ref{prop:change_of_basepointII} & $\eqref{eq:pi_generic}_{\preccurlyeq \beta}$, $\eqref{eq:gamma}_\beta^{\gamma\neq\pp}$, $\eqref{eq:recenter_Pi_minus}_\beta$ & $\eqref{eq:recenter_Pi_specific}_\beta$ \\
					\hline
					5.~b) & Proposition \ref{prop:change_of_basepointIII} & $\eqref{eq:recenter_Pi_specific}_{\preccurlyeq\beta}$, $\eqref{rec01}_{\prec \beta}$   & $\eqref{ao15}_\beta$, $\eqref{rec01}_\beta$\\
					\hline
					5.~c) & Proposition \ref{prop:pi_n_three_point} & $\eqref{eq:pi_generic}_{\preccurlyeq\beta}$, $\eqref{eq:gamma}_{\beta}^{\gamma \neq \pp}$, $\eqref{eq:recenter_Pi_specific}_{\beta}$ & $\eqref{eq:pi_n}_\beta$ \\
					\hhline{*{1}{|~}*{3}{|-}|~|}
					& Proposition \ref{prop:gamma_pp} & $\eqref{eq:pi_n}_{\preccurlyeq\beta}$ & $\eqref{eq:gamma}_\beta$\\
					\hline
					6 & Proposition \ref{rem:1} & $\eqref{eq:pi_generic}_{\preccurlyeq \beta}$, $\eqref{eq:gamma}_\beta$, $\eqref{cw03}_{\prec\beta}$, &  $\eqref{cw60}_{\beta}$, $\eqref{cw60_cont}_\beta$, $\eqref{cw63_generic}_{\beta}$\\
					& & $\eqref{cw60}_{\prec \beta}$, $\eqref{eq:recenter_Pi_specific}_\beta$, $\eqref{eq:pi_minus_generic}_\beta$, $\eqref{cw63_generic}_{\prec \beta}$ & \\
					\hline
			\end{tabular} 
		\end{center}
		\caption{Logic of an induction step for regular $\beta$. The attentive reader may have noticed that $\eqref{cw03}_{\prec\beta}$ is contained in several hypotheses, but $\eqref{cw03}_\beta$ is not part of any output; the reason is that $\eqref{cw03}_\beta$ trivially holds by the population constraint \eqref{eq:c_pop_cond} for these multi-indices.}\label{fig:regular}
	\end{figure}
	
	\begin{figure}[H]
		\small
		\begin{center}
			\begin{tabular}{|c|c|@{\hspace{0.2ex}}c@{\hspace{0.2ex}}|@{\hspace{0.2ex}}c@{\hspace{0.2ex}}|c|}
					\hline 
					{\bf Step} & {\bf Statement} & {\bf Input} & {\bf Output}\\  
					\hhline{|====|}
					1.~a) & Proposition \ref{prop:gamma_npp} & $\eqref{eq:pi_n}_{\prec \beta}$  & $\eqref{eq:gamma}_\beta^{\gamma \neq \pp}$ \\ 
					\hline
					1.~c) & Proposition \ref{prop:delta_gamma_npp}& $\eqref{eq:gamma}_{\preccurlyeq \beta}^{\gamma \neq \pp}$, $\eqref{eq:delta_pi_n}_{\prec \beta}$ & $\eqref{eq:delta_gamma}_\beta^{\gamma\neq \pp}$\\
					\hline
					1.~d) & Proposition \ref{prop:form_bound} & $\eqref{eq:gamma}_{\preccurlyeq \beta}^{\gamma \neq \pp}$, $\eqref{eq:delta_pi_generic}_{\prec \beta}$, $\eqref{eq:d_pi}_{\prec \beta}$ & $\eqref{eq:form_bound}_\beta^{\gamma \neq \pp}$\\
					\hline
					1.~e) & Proposition \ref{prop:algIII} & $\eqref{eq:gamma}_{\preccurlyeq \beta}^{\gamma\neq \pp}$, $\eqref{eq:delta_pi_d_pi_incr}_{\prec \beta}$, $\eqref{rec01}_{\prec \beta}$ & $\eqref{eq:form_cont}_{\beta}^{\gamma \neq \pp}$\\
					\hline
					2.~a) & Proposition \ref{prop:change_of_basepoint} &  $\eqref{eq:recenter_Pi_specific}_{\prec \beta}$ & $\eqref{eq:recenter_Pi_minus}_\beta$, $\eqref{eq:recenter_Pi_minus_specific}_\beta$\\
					\hline
					2.~b) & Subsection \ref{sec:Pi_minus_constr}  & $\eqref{eq:Pi_shift}_{\prec \beta}$, $\eqref{eq:Pi_parity}_{\prec \beta}$ & $\eqref{eq:Pi_minus_shift}_\beta$, $\eqref{eq:Pi_minus_parity}_\beta$\\
					\hline
					2.~c) & Proposition \ref{prop:limit} & $\eqref{eq:gamma}_{\beta}^{\gamma \neq \pp}$, $\eqref{eq:pi_minus_generic}_{\prec \beta}$, $\eqref{eq:Pi_minus_shift}_\beta$ & $\eqref{eq:limit}_\beta$\\
					\hline
					3.~a) & Proposition \ref{prop:BPHZ} & $\eqref{eq:limit}_\beta$, $\eqref{eq:Pi_minus_shift}_\beta$, $\eqref{eq:Pi_minus_parity}_\beta$ & $\eqref{eq:BPHZ}_\beta$\\
					\hline
					3.~b) & Proposition \ref{rem:1} & $\eqref{eq:pi_generic}_{\prec \beta}$, $\eqref{eq:BPHZ}_\beta$, $\eqref{cw03}_{\prec \beta}$, $\eqref{cw60}_{\prec \beta}$, $\eqref{eq:limit}_\beta$ & $\eqref{cw03}_\beta$\\
					\hline
					3.~c) & Proposition \ref{prop:expect} & $\eqref{eq:gamma}_\beta^{\gamma \neq \pp}$, $\eqref{eq:BPHZ}_\beta$, $\eqref{eq:recenter_Pi_minus}_\beta$, $\eqref{eq:pi_minus_generic}_{\prec \beta}$, $\eqref{eq:limit}_\beta$ & $\eqref{eq:expect_est}_\beta$\\
					\hline
					& & $\eqref{eq:pi_generic}_{\prec\beta}$, $\eqref{eq:gamma}_{\prec\beta}$, $\eqref{cw03}_{\prec\beta}$, & \\
					4.~b) & Proposition \ref{rem:1_mal} & $\eqref{cw60}_{\prec\beta}$, $\eqref{cw60_cont}_{\prec\beta}$, $\eqref{eq:recenter_Pi_specific}_{\prec \beta}$, & $\eqref{cw60_minus_cont_mal}_\beta$ \\
					& & $\eqref{eq:delta_gamma}_{\prec\beta}$, $\eqref{eq:delta_pi_generic}_{\prec \beta}$, $\eqref{cw60_mal_dual}_{\prec \beta}$, $\eqref{cw60_cont_mal_dual}_{\prec \beta}$ & \\
					\hline
					&  & $\eqref{eq:pi_generic}_{\prec \beta}$, $\eqref{eq:gamma}_{\prec \beta}$, $\eqref{cw60_cont}_{\prec\beta}$, $\eqref{cw60_minus_cont}_{\prec\beta}$,& \\
					4.~c) & Proposition \ref{prop:recIII} &  $\eqref{eq:recenter_Pi_minus}_{\prec\beta}$, $\eqref{eq:recenter_Pi_specific}_{\prec \beta}$, $\eqref{eq:pi_minus_generic}_{\prec \beta}$, $\eqref{eq:form_cont}_{\preccurlyeq \beta}^{\gamma \neq \pp}$, & $\eqref{eq:delta_pi_minus_generic}_\beta$\\
					& & $\eqref{cw60_cont_mal_dual}_{\prec\beta}$, $\eqref{cw60_minus_cont_mal}_\beta$, $\eqref{eq:delta_pi_incr_generic}_{\prec \beta}$, $\eqref{eq:form_bound}_{\preccurlyeq \beta}^{\gamma \neq \pp}$ & \\
					\hline 
					4.~d) & Proposition \ref{prop:delta_pi_minus_est} & $\eqref{eq:gamma}_\beta^{\gamma \neq \pp}$, $\eqref{eq:recenter_Pi_minus}_\beta$, $\eqref{eq:pi_minus_generic}_{\prec \beta}$,  & $\eqref{eq:mal_dual_annealed}_\beta$\\
					& & $\eqref{eq:mal_dual_annealed}_{\prec \beta}$, $\eqref{eq:delta_gamma}_\beta^{\gamma \neq \pp}$, $\eqref{eq:delta_pi_minus_generic}_\beta$, $\eqref{eq:form_bound}_\beta^{\gamma \neq \pp}$ & \\
					\hline
					5 & SG inequality & $\eqref{eq:expect_est}_\beta$, $\eqref{eq:mal_dual_annealed}_\beta$ & $\eqref{eq:pi_minus_generic}_\beta$\\
					\hline
					6.~a) & Proposition \ref{prop:int_pi_minus} & $\eqref{eq:pi_minus_generic}_\beta$ & $\eqref{eq:pi_generic}_\beta$, $\eqref{eq:Pi_minus_to_Pi}_\beta$\\
					\hline
					6.~b) & Subsection \ref{sec:Pi_construct} & $\eqref{eq:Pi_minus_to_Pi}_\beta$, $\eqref{eq:Pi_minus_shift}_\beta$, $\eqref{eq:Pi_minus_parity}_\beta$ & $\eqref{eq:Pi_shift}_\beta$, $\eqref{eq:Pi_parity}_\beta$\\
					\hline
					6.~c) & Proposition \ref{prop:change_of_basepointII} & $\eqref{eq:pi_generic}_{\preccurlyeq \beta}$, $\eqref{eq:gamma}_\beta^{\gamma\neq\pp}$, $\eqref{eq:recenter_Pi_minus}_\beta$ & $\eqref{eq:recenter_Pi_specific}_\beta$\\
					\hline
					6.~d) & Proposition \ref{prop:change_of_basepointIII} & $\eqref{eq:recenter_Pi_specific}_{\preccurlyeq\beta}$, $\eqref{rec01}_{\prec \beta}$ & $\eqref{ao15}_\beta$, $\eqref{rec01}_\beta$\\
					\hline
					6.~e) & Proposition \ref{prop:pi_n_three_point} & $\eqref{eq:pi_generic}_{\preccurlyeq \beta}$, $\eqref{eq:gamma}_{\beta}^{\gamma \neq \pp}$, $\eqref{eq:recenter_Pi_specific}_\beta$ & $\eqref{eq:pi_n}_\beta$\\
					\hhline{*{1}{|~}*{3}{|-}|~|}
					& Proposition \ref{prop:gamma_pp} & $\eqref{eq:pi_n}_{\preccurlyeq \beta}$ & $\eqref{eq:gamma}_\beta$\\
					\hline
					7.~b) & Proposition \ref{prop:intII} & $\eqref{eq:mal_dual_annealed}_\beta$ & $\eqref{eq:delta_pi_generic}_\beta$\\
					\hline 
					7.~d) & Proposition \ref{prop:three_pointII} & $\eqref{eq:pi_generic}_{\prec \beta}$, $\eqref{eq:gamma}_\beta^{\gamma \neq \pp}$, & $\eqref{eq:delta_pi_n}_\beta$, $\eqref{eq:delta_gamma}_\beta$\\
					& & $\eqref{eq:recenter_Pi_specific}_\beta$, $\eqref{eq:delta_gamma}_\beta^{\gamma \neq \pp}$, $\eqref{eq:delta_pi_generic}_{\preccurlyeq \beta}$ & \\
					\hline
					8.~a) & Proposition \ref{rem:1} & $\eqref{eq:pi_generic}_{\prec \beta}$, $\eqref{eq:gamma}_{\prec \beta}$, $\eqref{cw03}_{\prec \beta}$, & $\eqref{cw60_minus_cont}_\beta$\\
					& & $\eqref{cw60}_{\prec \beta}$, $\eqref{cw60_cont}_{\prec \beta}$, $\eqref{eq:recenter_Pi_specific}_{\prec\beta}$ & \\
					\hline
					8.~b) & Proposition \ref{rem:1} & $\eqref{eq:pi_generic}_{\preccurlyeq \beta}$, $\eqref{eq:gamma}_\beta$, $\eqref{cw03}_{\prec \beta}$,  & $\eqref{cw60}_{\beta}$, $\eqref{cw60_cont}_\beta$, $\eqref{cw63_generic}_{\beta}$\\
					& & $\eqref{cw60}_{\prec \beta}$, $\eqref{eq:recenter_Pi_specific}_\beta$, $\eqref{eq:pi_minus_generic}_\beta$ $\eqref{cw63_generic}_{\prec \beta}$ & \\
					\hline
					8.~c) & Proposition \ref{rem:1_mal} & $\eqref{eq:pi_generic}_{\prec \beta}$, $\eqref{cw03}_{\prec \beta}$, $\eqref{cw60}_{\prec \beta}$, $\eqref{eq:recenter_Pi_specific}_\beta$,  &  $\eqref{cw60_mal_dual}_{\beta}$, $\eqref{cw60_cont_mal_dual}_\beta$, $\eqref{cw63_generic_mal_dual}_{\beta}$\\
					& & $\eqref{eq:delta_pi_generic}_{\preccurlyeq \beta}$, $\eqref{cw60_mal_dual}_{\prec \beta}$, $\eqref{cw63_generic_mal_dual}_{\prec \beta}$ & \\
					\hline
					9.~b) & Proposition \ref{prop:intIII} & $\eqref{eq:pi_generic}_{\prec \beta}$, $\eqref{eq:pi_minus_generic}_{\prec \beta}$,  & $\eqref{eq:intIII1}_\beta$, $\eqref{eq:delta_pi_incr_generic}_\beta$\\
					& & $\eqref{eq:mal_dual_annealed}_\beta$, $\eqref{eq:delta_pi_generic}_\beta$, $\eqref{eq:delta_pi_minus_generic}_\beta$, $\eqref{eq:form_bound}_\beta^{\gamma \neq \pp}$ & \\
					\hline
					9.~c) & Proposition \ref{prop:delta_pi_incr_three_point} & $\eqref{eq:pi_generic}_{\prec \beta}$, $\eqref{eq:gamma}_{\prec \beta}^{\gamma \neq \pp}$, $\eqref{eq:pi_n}_{\prec \beta}$, $\eqref{eq:recenter_Pi_specific}_{\prec \beta}$,  & $\eqref{eq:delta_pi_d_pi_incr}_\beta$\\
					& & $\eqref{eq:form_cont}_\beta^{\gamma \neq \pp}$, $\eqref{eq:delta_pi_incr_generic}_\beta$, $\eqref{eq:form_bound}_\beta^{\gamma \neq \pp}$ & \\
					\hline
					9.~d) & Proposition \ref{prop:d_pi} & $\eqref{eq:pi_generic}_{\prec \beta}$, $\eqref{eq:delta_pi_generic}_\beta$, $\eqref{eq:delta_pi_incr_generic}_\beta$, $\eqref{eq:form_bound}_\beta^{\gamma \neq \pp}$ & $\eqref{eq:d_pi}_\beta$, $\eqref{eq:form_bound}_\beta$\\
					\hline
			\end{tabular} 
		\end{center}
		\caption{Logic of an induction step for singular $\beta$. Steps 1b, 4a, 7a, 7c and 9a, involving only constructions or Malliavin differentiability, have been omitted.}\label{fig:singular}
	\end{figure}	
\section*{Acknowledgments}
The authors thank Scott Smith and Lucas Broux for discussions, comments and suggestions, and the referees for their careful proofreading and numerous suggestions, which improved the readability of the paper. 
PT, PL, and MT thank
the Max Planck Institute for Mathematics in the Sciences for its excellent working conditions, where this research was carried out.

\pdfbookmark{References}{references}
\addtocontents{toc}{\protect\contentsline{section}{References}{\thepage}{references.0}}

\bibliographystyle{plainurl}
\bibliography{stoch_est_qsl}{}

\medskip
\begin{flushleft}
	\footnotesize \normalfont
	\textsc{Pablo Linares\\
		Department of Mathematics, Imperial College London\\
		SW7 2AZ London, United Kingdom}\\
	\texttt{\textbf{p.linares-ballesteros@imperial.ac.uk}}
\end{flushleft}

\begin{flushleft}
\footnotesize \normalfont
\textsc{Felix Otto\\
Max Planck Institute for Mathematics in the Sciences\\ 
04103 Leipzig, Germany}\\
\texttt{\textbf{felix.otto@mis.mpg.de}}
\end{flushleft}

\begin{flushleft}
\footnotesize \normalfont
\textsc{Markus Tempelmayr\\
Mathematics M\"unster, University of M\"unster\\ 
48149 M\"unster, Germany}\\
\texttt{\textbf{markus.tempelmayr@uni-muenster.de}}
\end{flushleft}

\begin{flushleft}
\footnotesize \normalfont
\textsc{Pavlos Tsatsoulis\\
Faculty of Mathematics, University of Bielefeld\\
33615 Bielefeld, Germany}\\
\texttt{\textbf{ptsatsoulis@math.uni-bielefeld.de}}
\end{flushleft}

\end{document}